\documentclass[11pt]{article}

\usepackage{amssymb,amsmath,amsthm,enumitem,mathrsfs,mathabx,accents,xy,float}
\xyoption{all}
\usepackage{tikz}
\usetikzlibrary{cd}

\makeatletter
\renewenvironment{proof}[1][\proofname]{\par
  \pushQED{\qed}%
  \normalfont \topsep-5\p@\@plus6\p@\relax
  \trivlist
  \item[\hskip\labelsep
    #1\@addpunct{.}]\ignorespaces
}{%
  \popQED\endtrivlist\@endpefalse
}
\makeatother

\usepackage{titlesec}
\titlespacing\section{0pt}{-3pt plus 2pt minus 1pt}{-7pt plus 2pt minus 1pt}
\titlespacing\subsection{0pt}{-3pt plus 2pt minus 1pt}{-7pt plus 2pt minus 1pt}
\titlespacing\subsubsection{0pt}{10pt plus 4pt minus 2pt}{1pt plus 2pt minus 2pt}

\numberwithin{equation}{section}

\newtheoremstyle{reduced_space}{}{-0.5\baselineskip}{}{}{\bfseries}{}{.5em}{}
\theoremstyle{reduced_space}

\newtheorem{thm}{Theorem}[section]
\newtheorem{prp}[thm]{Proposition}
\newtheorem{lmm}[thm]{Lemma}

\newtheorem{rmk}[thm]{Remark}
\newtheorem{dfn}[thm]{Definition}

\def\BE#1{\begin{equation}\label{#1}}
\def\EE{\end{equation}}
\def\eref#1{(\ref{#1})}

\def\BEnum#1{\begin{enumerate}[label=#1,leftmargin=*,topsep=-10pt,itemsep=-3pt]}
\def\EEnum{\end{enumerate}}

\def\ov#1{\overline{#1}}
\def\sf#1{\textsf{#1}}
\def\wt#1{\widetilde{#1}}
\def\tn#1{\textnormal{#1}} 
\def\lr#1{\langle{#1}\rangle}
\def\blr#1{\big\langle{#1}\big\rangle}

\def\wch#1{\widecheck{#1}}
\def\wh#1{\widehat{#1}}

\def\sm#1{\begin{small}#1\end{small}}

\def\lra{\longrightarrow}
\def\xlra#1{\xrightarrow{{#1}}}
\def\lhra{\lhook\joinrel\xrightarrow}

\def\cA{\mathcal A}
\def\C{\mathbb C}
\def\cC{\mathcal C}
\def\bfC{\mathbf C}
\def\D{\mathbb D}

\def\cH{\mathcal H}

\def\cJ{\mathcal J}
\def\cL{\mathcal L}
\def\M{\mathfrak M}
\def\cM{\mathcal M}
\def\cN{\mathcal N}

\def\P{\mathbb P}
\def\R{\mathbb R}
\def\Q{\mathbb Q}
\def\cS{\mathcal S}

\def\cZ{\mathcal Z}
\def\Z{\mathbb Z}

\def\al{\alpha}
\def\be{\beta}
\def\de{\delta}
\def\ep{\varepsilon}
\def\ga{\gamma}
\def\io{\iota}

\def\la{\lambda}
\def\si{\sigma}
\def\om{\omega}
\def\th{\theta}
\def\ve{\varepsilon}

\def\vph{\varphi}
\def\ze{\zeta}

\def\De{\Delta}
\def\Ga{\Gamma}
\def\La{\Lambda}
\def\Om{\Omega}
\def\Si{\Sigma}
\def\Th{\Theta}
\def\Ups{\Upsilon}

\def\fc{\mathfrak c}
\def\fd{\mathfrak d}
\def\ff{\mathfrak f}
\def\fI{\mathfrak i}
\def\fj{\mathfrak j}
\def\fo{\mathfrak o}

\def\os{\mathfrak{os}}
\def\fs{\mathfrak s}

\def\bh{\mathbf h}
\def\bp{\mathbf p}
\def\bt{\mathbf t}
\def\u{\mathbf u}
\def\y{\mathbf y}

\def\ne{\textnormal{e}}

\def\Aut{\tn{Aut}}

\def\codim{\tn{codim}}

\def\nd{\tn{d}}

\def\dim{\tn{dim}}
\def\ev{\tn{ev}}
\def\id{\tn{id}}

\def\Lk{\tn{Lk}}
\def\nod{\tn{nd}}
\def\PD{\tn{PD}}

\def\pt{\tn{pt}}

\def\PSL{\tn{PSL}}
\def\rk{\tn{rk}}

\def\OSpin{\tn{OSpin}}
\def\top{\tn{top}}

\def\node{\tn{node}}

\def\i{\infty}

\def\eset{\emptyset}
\def\prt{\partial}
\def\dbar{\ov\partial}
\def\st{\bigstar}

\def\bu{\bullet}

\def\oGa{\Ga}
\def\OGa{\accentset{\circ}{\Ga}}
\def\oS{S}

\begin{document}

\title{WDVV-Type Relations for\\ Disk Gromov-Witten Invariants in Dimension 6}
\author{Xujia Chen and Aleksey Zinger\thanks{Supported by NSF grant DMS 1500875
and Simons collaboration grant 587036}}
\date{\today}

\maketitle

\begin{abstract}
\noindent
The first author's previous work established Solomon's WDVV-type relations
for Welschinger's invariant curve counts in real symplectic fourfolds
by lifting geometric relations over possibly unorientable morphisms.
We apply her framework to obtain WDVV-style relations for the disk invariants
of real symplectic sixfolds with some symmetry, in particular confirming 
Alcolado's prediction for~$\P^3$ and extending it to other spaces.
These relations reduce the computation of  Welschinger's invariants of 
many real symplectic sixfolds
to invariants in small degrees and provide lower bounds for counts of real rational curves 
with positive-dimensional insertions in some cases.
In the case of~$\P^3$, our lower bounds fit perfectly with Koll\'ar's vanishing results.
\end{abstract}

\tableofcontents
\setlength{\parskip}{\baselineskip}

\section{Introduction}
\label{intro_sec}

The WDVV relation~\cite{KM,RT} for genus~0 Gromov-Witten invariants completely solves 
the classical problem of enumerating complex rational curves in the complex projective space~$\P^n$.
Invariant counts of real rational $J$-holomorphic curves with point insertions
in compact real symplectic fourfolds and sixfolds, now known as \sf{Welschinger's invariants},
were defined in~\cite{Wel4,Wel6} and interpreted in terms of counts of $J$-holomorphic maps
from the disk~$\D^2$ in~\cite{Jake}.
Two WDVV-type relations for Welschinger's invariants in dimension~4 were predicted 
in~\cite{Jake2} and established in~\cite{RealWDVV}.
Similarly to the WDVV relation of~\cite{KM,RT}, these relations completely 
determine Welschinger's invariants of many real symplectic fourfolds from very basic input;
see~\cite{RealWDVVapp}.
Methods for computing Welschinger's invariants of the projective space~$\P^3$
were introduced in~\cite{BM07,BG}.
A WDVV-type relation for counts of real rational curves without real constraints
was obtained in~\cite{RealEnum}.
The existence of WDVV-type relations for Welschinger's invariants in dimension~6
was announced in~\cite{Jake2}, but without specifying their statements.

The present paper applies the approach of~\cite{RealWDVV} to obtain two
relations for Welschinger's invariants of real symplectic sixfolds with symmetry
as in Definition~\ref{aveG_dfn}.
These relations yield the two WDVV-type ODEs of Theorem~\ref{WDVVdim3_thm} for generating functions
for the disk and complex Gromov-Witten invariants.
Our ODEs~\eref{WDVVodeM12_e} and \eref{WDVVodeM03_e} in the case of~$\P^3$
agree with the ODEs~(4.82) and~(4.76), respectively, in~\cite{Adam},
but correct their structure for more general spaces.
Unlike the relations alluded to in \cite[Rem.~4]{Jake2}, 
our relations do not involve linking numbers.

The first author showed in~\cite{RealWDVV} that
the disk counts of~\cite{Jake} in real symplectic fourfolds can be viewed
as the degrees of relatively oriented pseudocycles from open subspaces
of the moduli spaces $\ov\M_{k,l}(B;J)$ of real rational $J$-holomorphic maps 
constructed in~\cite{Penka2}.
She then established Solomon's relations for Welschinger's invariants of real symplectic 
fourfolds in~\cite{RealWDVV} by lifting
\BEnum{($\R\arabic*$)}

\item\label{M12rel_it} a zero-dimensional homology relation on the moduli space 
$\R\ov\cM_{0,1,2}\!\approx\!\R\P^2$ of stable real
genus~0 curves with 1~real marked point and 2~conjugate pairs of marked points and

\item\label{M03rel_it} 
the one-dimensional homology relation on the moduli space $\R\ov\cM_{0,0,3}$ of stable real
genus~0 curves with 3~conjugate pairs of marked points discovered in~\cite{RealEnum}

\EEnum
{\it along with} suitably chosen bounding cobordisms~$\Ups$ for them to $\wh\M_{k,l;l^*}(B;J)$,
a cut of
$\ov\M_{k,l}(B;J)$ along hypersurfaces that obstruct the relative orientability of the forgetful morphisms
\BE{ffdfn_e0}\ff_{1,2}\!:\ov\M_{k,l}(B;J)\lra\R\ov\cM_{0,1,2} \quad\hbox{and}\quad  
\ff_{0,3}\!:\ov\M_{k,l}(B;J)\lra\R\ov\cM_{0,0,3}.\EE
The intersections of the boundary of $\wh\M_{k,l;l^*}(B;J)$ with~$\Ups$ then determine 
the wall-crossing effects on the lifted relations in $\ov\M_{k,l}(B;J)$;
see \cite[Lemma~3.5]{RealWDVV}. 

The WDVV-type relations of Theorem~\ref{WDVVdim3_thm} for the disk counts of~\cite{Jake}
arise from relations between counts of two- and three-component real curves
obtained by lifting~\ref{M12rel_it} and~\ref{M03rel_it} exactly as in~\cite{RealWDVV};
see Section~\ref{LiftRel_subs}.
The lifted relations, depicted in Figure~\ref{LiftedRel_fig} on page~\pageref{LiftedRel_fig}, 
have the exact same form as in~\cite{RealWDVV} and
hold without any symmetry assumption on the target, but 
for {\it fixed} collections of constraints for the curves; see Proposition~\ref{LiftedRel_prp}.
The counts of curves represented by the individual diagrams 
in Figure~\ref{LiftedRel_fig}  generally depend on the choices of the constraints.
We eliminate this dependance by averaging these counts over the action of 
a finite group~$G$ of symmetries as in Definition~\ref{aveG_dfn},
if such a group exists, and then split the averaged counts into
invariant counts of irreducible real and complex curves;
see Section~\ref{SplitAxiom_subs} and 
Propositions~\ref{Rdecomp_prp} and~\ref{Cdecomp_prp}.

As noted in \cite[Prop.~17]{Adam},
the WDVV-type relations of Theorem~\ref{WDVVdim3_thm} are very effective in 
computing the disk invariants of some real symplectic sixfolds~$(X,\om,\phi)$.
The disk invariants with only point constraints agree with Welschinger's
invariants up to~sign.
As only some elements of $H_2(X\!-\!X^{\phi})$ can be represented by holomorphic curves
in a real projective manifold~$(X,\om,\phi)$, 
Theorem~\ref{WDVVdim3_thm} and Proposition~\ref{SNXphi_prp} lead to lower bounds 
for counts of real algebraic curves in some real algebraic threefolds
through curve constraints; see Section~\ref{LowBnd_subs}.
In light of~\cite{Kollar13}, there can be no non-trivial lower bounds of this kind
for real lines and conics in~$\P^3$.
However, Theorem~\ref{WDVVdim3_thm} and Proposition~\ref{SNXphi_prp} provide such
bounds for real cubic curves in all cases not precluded by~\cite{Kollar13};
they are shown in boldface in Table~\ref{P3nums_tbl} on page~\pageref{P3nums_tbl}.

\begin{rmk}\label{JS_rmk}
Two months after the present paper was posted on arXiv,
\cite{JS3} provided WDVV-type relations for the open GW-invariants constructed
in the authors' earlier papers based on algebraic considerations.
These invariants, which in general count $J$-holomorphic multi-disks meeting
the given cycles and some auxiliary bordered chains,
reduce to Welschinger's real invariants in the setting of~\cite{Wel6}.
In the presence of a symmetry as in Definition~\ref{aveG_dfn},
the relations of~\cite{JS3} in turn reduce to those of Theorem~\ref{WDVVdim3_thm},
which we establish based on self-contained geometric considerations.
\end{rmk}

\subsection{The real symplectic setting}
\label{RealSympl_subs}

Let $(X,\om,\phi)$ be a compact \sf{real symplectic manifold}, i.e.~$\om$ is a symplectic form
on~$X$ and $\phi$ is an involution on~$X$ so that $\phi^*\om\!=\!-\om$.
The fixed locus~$X^{\phi}$ of~$\phi$ is then a Lagrangian submanifold of~$(X,\om)$.
An \sf{automorphism} of $(X,\om,\phi)$ is a diffeomorphism~$\psi$ of~$X$ 
such~that 
$$\psi^*\om=\om \qquad\hbox{and}\qquad \psi\!\circ\!\phi=\phi\!\circ\!\psi.$$
We denote by $\Aut(X,\om,\phi)$ the group of automorphisms of $(X,\om,\phi)$.
Let 
\begin{alignat*}{2}
H_*(X)&=H_*(X;\Q), &\quad
H_*(X)^{\phi}_{\pm}&=\big\{\be\!\in\!H_*(X)\!:\phi_*\be\!=\!\pm\be\big\}, \\
H^*(X)&=H^*(X;\Q),&\quad
H^*(X)^{\phi}_{\pm}&=\big\{\mu\!\in\!H^*(X)\!:\phi^*\mu\!=\!\pm\mu\big\}.
\end{alignat*}
For a connected component $\wch{X}^{\phi}$ of~$X^{\phi}$, we denote by
$$\Aut\big(X,\om,\phi;\wch{X}^{\phi}\big)\subset \Aut(X,\om,\phi)$$
the subgroup of automorphisms of $(X,\om,\phi)$ mapping~$\wch{X}^{\phi}$ to itself.
Let
\begin{alignat*}{2}
H_*(X\!-\!\wch{X}^{\phi})&=H_*(X\!-\!\wch{X}^{\phi};\Q), &\quad
H_*(X\!-\!\wch{X}^{\phi})^{\phi}_{\pm}
&=\big\{\be\!\in\!H_*(X\!-\!\wch{X}^{\phi})\!:\phi_*\be\!=\!\pm\be\big\},\\
H^*(X,\wch{X}^{\phi})&=H^*(X,\wch{X}^{\phi};\Q), &\quad
H^*(X,\wch{X}^{\phi})^{\phi}_{\pm}
&=\big\{\mu\!\in\!H^*(X,\wch{X}^{\phi})\!:\phi^*\mu\!=\!\pm\mu\big\}.
\end{alignat*}

Every element of $H_2(X,\wch{X}^{\phi};\Z)$ can be represented by a continuous map
$$f\!:\big(\Si,\prt\Si\big)\lra\big(X,\wch{X}^{\phi}\big)$$ 
from a compact oriented surface with boundary; see \cite[Lem.~4.3(b)]{SpinPin}.
The continuous map \hbox{$\wh{f}\!:\wh\Si\!\lra\!X$} obtained by gluing $f$ with 
the map~$\phi\!\circ\!f$ from~$\Si$ with the opposite orientation along~$\prt\Si$
then represents an element~$[\wh{f}]$ of~$H_2(X;\Z)$.
It depends only on the element~$[f]$ of $H_2(X,\wch{X}^{\phi};\Z)$ represented by~$f$; 
see \cite[Lem.~4.3(c)]{SpinPin}.
Let 
\BE{fdwchXdfn_e}\fd_{\wch{X}^{\phi}}\!:H_2\big(X,\wch{X}^{\phi};\Z\big) 
\lra H_2(X;\Z)\lra H_2(X)\EE
be the composition of the resulting homomorphism with 
the obvious homomorphism to~$H_2(X)$.
We denote~by 
\BE{Z2bnddfn_e}
\prt_{\wch{X}^{\phi};\Z_2}\!: H_2\big(X,\wch{X}^{\phi};\Z\big)\lra H_1\big(\wch{X}^{\phi};\Z\big)
\lra H_1\big(\wch{X}^{\phi};\Z_2\big)\EE
the composition of the boundary homomorphism of the relative exact sequence for the pair 
$(X,\wch{X}^{\phi})$ with the mod~2 reduction of the coefficients.
We call an element $B\!\in\!H_2(X)$ \sf{$(\wch{X}^{\phi},\Z_2)$-trivial} if
$$\prt_{\wch{X}^{\phi};\Z_2}\big(\fd_{\wch{X}^{\phi}}^{-1}(B)\!\big)=\big\{0\big\}
\subset H_1\big(\wch{X}^{\phi};\Z_2\big).$$

We denote by $\cJ_{\om}$ the space of $\om$-compatible (or -tamed) 
almost complex structures~$J$ on~$X$ and by
$\cJ_{\om}^{\phi}\!\subset\!\cJ_{\om}$ the subspace of almost complex structures~$J$
such that \hbox{$\phi^*J\!=\!-J$}.
Let 
$$c_1(X,\om)\equiv c_1(TX,J)\in H^2(X)$$
be the first Chern class of $TX$ with respect to some $J\!\in\!\cJ_{\om}$;
it is independent of such a choice.
For \hbox{$B\!\in\!H_2(X)$}, define
$$\ell_{\om}(B)=\blr{c_1(X,\om),B}\in\Q. $$
If $B$ is in the image of the second homomorphism in~\eref{fdwchXdfn_e},
then $\ell_{\om}(B)\!\in\!\Z$.
If $B$ is in the image of the composite homomorphism in~\eref{fdwchXdfn_e} and 
$\wch{X}^{\phi}$ is orientable, then $\ell_{\om}(B)\!\in\!2\Z$;
see \cite[Prop.~4.1]{BHH}.

For $J\!\in\!\cJ_{\om}$ and $B\!\in\!H_2(X)$,
a subset $C\!\subset\!X$ is a \sf{genus~0} (or \sf{rational}) 
\sf{irreducible degree~$B$ $J$-holomorphic  curve}
if there exists a simple (not multiply covered) \hbox{$J$-holomorphic} map 
\BE{Cudfn_e}u\!:\P^1\lra X \qquad\hbox{s.t.}\quad  C=u(\P^1),~~u_*[\P^1]=B.\EE
If in addition $\mu_1,\ldots,\mu_l\!\in\!H^*(X)$, we denote by 
\BE{CGWdfn_e}\lr{\mu_1,\ldots,\mu_l}_B^X\in\Q\EE
the (complex) GW-invariant of $(X,\om)$ enumerating rational degree~$B$ 
$J$-holomorphic curves $C\!\subset\!X$ through generic representatives of 
the Poincare duals of $\mu_1,\ldots,\mu_l$.
This invariant vanishes unless $B$ lies in the image of the second 
homomorphism in~\eref{fdwchXdfn_e}.

A curve $C\!\subset\!X$ as in~\eref{Cudfn_e} is called \sf{real} if \hbox{$\phi(C)\!=\!C$}. 
In such a case, $u$ in~\eref{Cudfn_e} can be chosen so that it intertwines $\phi$
with one of the two standard involutions on~$\P^1$,
$$\tau\!:\P^1\lra\P^1, ~~\tau(z)=\frac{1}{\ov z},  \qquad\hbox{or}\qquad
\eta\!:\P^1\lra\P^1, ~~\eta(z)=-\frac{1}{\ov z},$$
i.e.~either $u\!\circ\!\tau\!=\!\phi\!\circ\!u$ or $u\!\circ\!\eta\!=\!\phi\!\circ\!u$.
We call such maps~$u$ \sf{$(\phi,\tau)$-real} and \sf{$(\phi,\eta)$-real}, respectively,
and the images of $(\phi,\tau)$-real maps \sf{$(\phi,\tau)$-real curves}.
For a $(\phi,\tau)$-real curve \hbox{$C\!\subset\!X$}, we denote by \hbox{$\R C\!\subset\!X^{\phi}$} 
the image of the $\tau$-fixed locus $S^1\!\subset\!\P^1$ under 
a $(\phi,\tau)$-real map $u$ as in~\eref{Cudfn_e}. 
The degree~$B$ of a $(\phi,\tau)$-real map lies 
in the image of the composite homomorphism in~\eref{fdwchXdfn_e}
if $\R C\!\subset\!\wch{X}^{\phi}$.

For a subgroup $G$ of $\Aut(X,\om,\phi;\wch{X}^{\phi})$, we denote by 
$\cJ_{\om;G}^{\phi}\!\subset\!\cJ_{\om}^{\phi}$ the subspace of 
$G$-invariant almost complex structures.
Let
\begin{equation*}\begin{split}
H_*^G(X\!-\!\wch{X}^{\phi})^{\phi}_{\pm}&=
\big\{\be\!\in\!H_*(X\!-\!\wch{X}^{\phi})^{\phi}_{\pm}\!:
\psi_*\be\!=\!\be\,\forall\psi\!\in\!G\big\}\,,\\
H^*_G(X,\wch{X}^{\phi})^{\phi}_{\pm}&=
\big\{\mu\!\in\!H^*(X,\wch{X}^{\phi})^{\phi}_{\pm}\!:
\psi^*\mu\!=\!\mu\,\forall\psi\!\in\!G\big\}.
\end{split}\end{equation*}

\subsection{Disk invariants under symmetries}
\label{DiskGW_subs}

From now on, suppose that the (real) dimension of~$X$ is~6.
The tangent bundle of an orientable connected component $\wch{X}^{\phi}$ of~$X^{\phi}$
is then trivializable and thus admits a Spin-structure~$\fs$ for any choice of
orientation~$\fo$ in~$\wch{X}^{\phi}$.
We call such a pair $\os\!\equiv\!(\fo,\fs)$ an \sf{OSpin-structure on~$\wch{X}^{\phi}$}.
For $B\!\in\!H_2(X)$ and a tuple $(\mu_1,\ldots,\mu_l)$ of homogeneous elements 
of $H^{2*}(X)$ and $H^{2*}(X,\wch{X}^{\phi})$, let 
\BE{dimcond_e} 
k= k_B(\mu_1,\ldots,\mu_l)
\equiv\frac12\Big(\ell_{\om}(B)\!+\!2l-\sum_{i=1}^l\deg\mu_i\Big).\EE
Under certain conditions on $B$ and $\mu_1,\ldots,\mu_l$,
an OSpin-structure $\os$ on~$\wch{X}^{\phi}$ determines an \sf{open GW-invariant} 
\BE{RGWdfn_e}\blr{\mu_1,\ldots,\mu_l}_{B;\wch{X}^{\phi}}^{\phi,\os}\in\Q\EE
of $(X,\om,\phi)$
enumerating $(\phi,\tau)$-real degree~$B$ $J$-holomorphic curves $C\!\subset\!X$ 
with $\R C\!\subset\!\wch{X}^{\phi}$ that
pass through generic representatives for the preimages of $\mu_1,\ldots,\mu_l$
under the Poincare duality isomorphisms
\BE{PDisomdfn_e}\PD_X\!: H_p(X)\lra H^{6-p}(X) \quad\hbox{and}\quad
\PD_{X,\wch{X}^{\phi}}\!: H_p\big(X\!-\!\wch{X}^{\phi}\big)\lra H^{6-p}\big(X,\wch{X}^{\phi}\big)\EE
as in \cite[Theorems~67.1,70.2]{Mu2}
and through~$k$ points in~$\wch{X}^{\phi}$.
These conditions are recalled in the next paragraph.
If $\wch{X}^{\phi}$ is orientable and such curves exist, then $\ell_{\om}(B)$ is even
 and thus $k\!\in\!\Z$.
The number~\eref{RGWdfn_e} is defined to be~0 if $k\!<\!0$.

Invariant signed counts~\eref{RGWdfn_e} were first defined in~\cite{Wel6} under the assumptions 
that 
\BE{taucond_e1}\ell_{\om}(B)>0, \qquad \mu_i\!\in\!H^2(X)\!\cup\!H^6(X)~~\forall\,i,\EE
i.e.~each $\mu_i$ represents a Poincare dual of a ``complex" hypersurface or a point, 
and either 
\BE{noS2bub_e} k\!>\!0 \qquad\hbox{or}\qquad
0\not\in \prt_{\wch{X}^{\phi};\Z_2}\big(\fd_{\wch{X}^{\phi}}^{-1}(B)\!\big)
\subset H_1\big(\wch{X}^{\phi};\Z_2\big);\EE
these counts are now known as \sf{Welschinger's invariants}.
The interpretation of these counts in terms of $J$-holomorphic maps from disks 
in~\cite{Jake} dropped the first restriction in~\eref{taucond_e1} and later led to
Solomon's observation that the counts~\eref{RGWdfn_e} are also well-defined with 
$\mu_i\!\in\!H^4(X,\wch{X}^{\phi})$ for some~$i$;
see Section~\ref{DiskInvDfn_subs}.

The now standard way to drop the restriction~\eref{noS2bub_e} under certain topological conditions 
on $(X,\phi)$ is to include counts of $(\phi,\eta)$-curves; see~\cite{Teh}.
Another way to do so 
is to count only $(\phi,\tau)$-real degree~$B$ $J$-holomorphic curves $C\!\subset\!X$ 
with $\R C\!\subset\!\wch{X}^{\phi}$ such that $\R C$ does not vanish 
in~$H_1(\wch{X}^{\phi};\Z_2)$.
While both approaches are suitable for the purposes of Proposition~\ref{LiftedRel_prp}, 
neither appears to lead to splitting formulas as in 
Propositions~\ref{Rdecomp_prp} and~\ref{Cdecomp_prp} by itself.
We instead pursue a different approach.

\begin{dfn}\label{aveG_dfn}
Let $(X,\om,\phi)$ be a real symplectic manifold and 
$\wch{X}^{\phi}$ be a connected component of~$X^{\phi}$. 
A finite subgroup~$G$ of $\Aut(X,\om,\phi;\wch{X}^{\phi})$ is 
an \sf{averager for $(X,\om,\phi;\wch{X}^{\phi})$} 
if $G$ acts trivially on~$H_2(X)^{\phi}_-$ and
some element $\psi\!\in\!G$ restricts to an orientation-reversing diffeomorphism of~$\wch{X}^{\phi}$.
\end{dfn}

As explained in Section~\ref{sphbub_subs} and summarized by {Proposition~\ref{GtauGW_prp}, 
an averager leads to pairwise cancellations of certain curve counts and yields 
well-defined counts of $(\phi,\tau)$-curves even if the condition~\eref{noS2bub_e}
does not hold.
An averager also leads to the splitting formulas
of Propositions~\ref{Rdecomp_prp} and~\ref{Cdecomp_prp} and 
thus to the real WDVV equations of Theorem~\ref{WDVVdim3_thm} in Section~\ref{WDVV_subs}.
An averager for~$\P^3$ is generated by a real hyperplane reflection; see Section~\ref{sphbub_subs}.

\begin{prp}\label{GtauGW_prp}
Suppose $(X,\om,\phi)$ is a compact real symplectic sixfold,
$\os$ is an OSpin-structure on 
a connected component~$\wch{X}^{\phi}$ of~$X^{\phi}$,  
$G$ is an averager for $(X,\om,\phi;\wch{X}^{\phi})$,
and $J\!\in\!\cJ_{\om;G}^{\phi}$.
\BEnum{(\arabic*)}

\item\label{GtauGWdfn_it} For all $B\!\in\!H_2(X)$ and $l\!\in\!\Z^{\ge0}$, 
there is a multilinear symmetric functional 
\BE{GtauGWdfn_e0}\blr{\cdot,\ldots,\cdot}_{B;\wch{X}^{\phi};G}^{\phi,\os}\!: 
H^{2*}(X)^{\oplus l}\lra \Q\EE
enumerating $(\phi,\tau)$-real degree~$B$ $J$-holomorphic curves $C\!\subset\!X$ 
with $\R C\!\subset\!\wch{X}^{\phi}$ and satisfying
\begin{gather}\label{GtauGWdfn_e0a} 
\blr{\cdot,\ldots,\cdot}_{B;\wch{X}^{\phi};G}^{\phi,\os}=0
\quad\forall\,B\!\not\in\!H_2(X)_-^{\phi},\\
\label{GtauGWdfn_e0b} 
\blr{\mu,\mu_1,\ldots,\mu_l}_{B;\wch{X}^{\phi};G}^{\phi,\os}
=\lr{\mu,B}\blr{\mu_1,\ldots,\mu_l}_{B;\wch{X}^{\phi};G}^{\phi,\os}
\quad\forall\,\mu\!\in\!H^2(X).
\end{gather}

\item\label{GtauGWvan_it} Let $B\!\in\!H_2(X)$ and $l\!\in\!\Z^{\ge0}$.
If either 
\BE{GtauGWvan_e0}
B~\hbox{is}~(\wch{X}^{\phi},\Z_2)-\hbox{trivial} \quad\hbox{and}\quad
\frac12\ell_{\om}(B)\equiv\big|\{i\!:\mu_i\!\in\!H^4(X)\}\big|\mod2\EE
or $\mu_i\!\in\!H^2(X)^{\phi}_+\!\oplus\!H^4(X)^{\phi}_-$ for some~$i$, then
$$\blr{\mu_1,\ldots,\mu_l}_{B;\wch{X}^{\phi};G}^{\phi,\os}=0.$$

\EEnum
\end{prp}

\begin{rmk}\label{GtauGW_rmk0}
The $\OSpin$-structure~$\os$ for an orientation~$\fo$ on~$\wch{X}^{\phi}$ naturally 
determines an $\OSpin$-structure~$\ov\os$ for the opposite orientation~$\ov\fo$;
see \cite[Section~1.2]{SpinPin}.
By the proof of Proposition~\ref{GtauGW_prp} in Section~\ref{orient_subs},
the first condition in~\eref{GtauGWvan_e0} can be replaced by 
the existence of $\psi\!\in\!G$ restricting to an orientation-reversing 
diffeomorphism of~$\wch{X}^{\phi}$ such that $\psi^*\!\os\!=\!\ov\os$.
Such a $\psi$ does not exist in the case of~$\P^3$, but does exist in the case of $(\P^1)^3$
with the two natural involutions specified in \cite[Section~5]{RealWDVVapp}.
By~\eref{dimcond_e}, the second condition in~\eref{GtauGWvan_e0} means
that the number~$k$ of real point insertions is even.
\end{rmk}

As explained at the end of Section~\ref{Compute_subs}, the condition of Definition~\ref{aveG_dfn}
implies that the natural homomorphisms
\BE{aveG_e0}\io_*\!:H_2\big(X\!-\!\wch{X}^{\phi}\big)\lra H_2(X) 
\quad\hbox{and}\quad
r\!:H^4\big(X,\wch{X}^{\phi}\big)\lra H^4(X)\EE
restrict to isomorphisms
\BE{aveG_e1}
\io_*^G\!:H_2^G(X\!-\!\wch{X}^{\phi})^{\phi}_- \stackrel{\approx}{\lra} H_2(X)^{\phi}_-
\quad\hbox{and}\quad
r_G\!:H^4_G(X,\wch{X}^{\phi})^{\phi}_+\stackrel{\approx}{\lra} H^4(X)^{\phi}_+\,,\EE
respectively.
For homogeneous elements $\mu_1,\ldots,\mu_l$ of $H^{2*}(X)$, let
$$\wt\mu_i=\begin{cases}r_G^{-1}(\mu_i), &\hbox{if}~\mu_i\!\in\!H^4(X)^{\phi}_+;\\
0, &\hbox{if}~\mu_i\!\in\!H^4(X)^{\phi}_-;\\
\mu_i,&\hbox{if}~\mu_i\!\not\in\!H^4(X).
\end{cases}$$
We define 
\BE{Gnumsdfn_e}\blr{\mu_1,\ldots,\mu_l}_{B;\wch{X}^{\phi};G}^{\phi,\os}
=\blr{\wt\mu_1,\ldots,\wt\mu_l}_{B;\wch{X}^{\phi}}^{\phi,\os}\,.\EE

The numbers~\eref{Gnumsdfn_e} depend on the $G$-invariant subspace 
$$H^4_G(X,\wch{X}^{\phi})_+^{\phi}\subset H^4(X,\wch{X}^{\phi})_+^{\phi}$$ 
or equivalently on the $G$-invariant subspace
$$H_2^G(X\!-\!\wch{X}^{\phi})^{\phi}_-\subset H_2(X\!-\!\wch{X}^{\phi})^{\phi}_-.$$
However, these numbers do not depend on the choice of an averager~$G$ 
which acts trivially on~$H^2(X)$ if the subspace of~$H^4(X)$ spanned by 
the cup products of the elements of~$H^2(X)$ contains~$H^4(X)_+^{\phi}$.
In particular, the disk invariants of $(\P^3,\om_3,\tau_3)$ provided by
Proposition~\ref{GtauGW_prp}\ref{GtauGWdfn_it} are independent of the choice of~$G$.

\subsection{Main theorem}
\label{WDVV_subs}

Let $(X,\om,\phi)$ be a connected compact real symplectic sixfold.
Define 
\begin{gather*}
\fd\!:H_2(X)\lra H_2(X)_-^{\phi}, \qquad \fd(B)=B\!-\!\phi_*(B),\\
\La_{\om}^{\phi}=
\big\{(\Psi\!:H_2(X)_-^{\phi}\!\lra\!\Q)\!:
\big|\{B\!\in\!H_2(X)_-^{\phi}\!:
\Psi(B)\!\neq\!0,\,\om(B)\!<\!E\}\big|\!<\!\i~\forall\,E\!\in\!\R\big\}.
\end{gather*}
We write an element $\Psi$ of $\La_{\om}^{\phi}$ as 
$$\Psi=\sum_{B\in H_2(X)_-^{\phi}}\!\!\!\!\!\!\!\!\Psi(B)q^B$$
and multiply two such elements as powers series in $q$ with the exponents in $H_2(X)_-^{\phi}$.
Choose a basis $\mu_1^{\st},\ldots,\mu_N^{\st}$ for 
\BE{Hpart_e}H^0(X)\!\oplus\!H^2(X)^{\phi}_-\!\oplus\!H^4(X)^{\phi}_+\!\oplus\!H^6(X)\EE
consisting of homogeneous elements.
Let $(g_{ij})_{i,j}$ be the $N\!\times\!N$-matrix given~by
$$g_{ij}=\blr{\mu_i^{\st}\mu_j^{\st},[X]}$$
and $(g^{ij})_{ij}$ be its inverse.
For a tuple $\bt\!\equiv\!(t_1,\ldots,t_N)$ of formal variables, let
$$\mu_{\bt}^{\st}=\mu_1^{\st}t_1\!+\!\ldots\!+\!\mu_N^{\st}t_N\,.$$

Suppose in addition that $\os$ is an OSpin-structure on 
a connected component~$\wch{X}^{\phi}$ of~$X^{\phi}$ and
$G$ is an averager for $(X,\om,\phi;\wch{X}^{\phi})$.
For $B\!\in\!H_2(X)$, $k,l\!\in\!\Z^{\ge0}$, and homogeneous elements 
$\mu_1,\ldots,\mu_l$ of $H^{2*}(X)$, define
$$\blr{\mu_1,\ldots,\mu_l}_{B,k;G}^{\phi,\os}
=\begin{cases}\lr{\mu_1,[\pt]},&\hbox{if}\,(B,k,l)\!=\!(0,1,1),\,\mu_1\!\in\!H^0(X);\\
\lr{\mu_1,\ldots,\mu_l}_{B;\wch{X}^{\phi};G}^{\phi,\os},&
\hbox{if}~B\!\neq\!0,~\mu_i\!\not\in\!H^0(X)\,\forall\,i,~
k\!=\!k_B(\mu_1,\ldots,\mu_l);\\
0,&\hbox{otherwise}.\end{cases}$$
We extend the GW-functional $\lr{\ldots}_B^X$ and the open GW-functional
$\lr{\ldots}_{B,k;G}^{\phi,\os}$ linearly over the formal variables~$t_i$.

We define $\Phi_{\om}^{\phi}\!\in\!\La_{\om}^{\phi}[[t_1,\ldots,t_N]]$
and $\Om_{\om;G}^{\phi,\os}\!\in\!\La_{\om}^{\phi}[[t_1,\ldots,t_N,u]]$ by
\begin{gather*}
\Phi_{\om}^{\phi}(t_1,\ldots,t_N)=
\sum_{\begin{subarray}{c}B\in H_2(X)_-^{\phi}\\ l\in\Z^{\ge0}\end{subarray}}\!\!\!
\Bigg(\sum_{\begin{subarray}{c}B'\in H_2(X)\\
\fd(B')=B\end{subarray}}\!\!\!\!\!\blr{
\underset{l}{\underbrace{\mu_{\bt}^{\st},\ldots,\mu_{\bt}^{\st}}}}_{B'}^X\!\Bigg)
\frac{q^B}{l!}\,,\\
\Om_{\om;G}^{\phi,\os}(t_1,\ldots,t_N,u)=
\sum_{\begin{subarray}{c}B\in H_2(X)_-^{\phi}\\ k,l\in\Z^{\ge0}\end{subarray}}\!\!\!
\blr{
\underset{l}{\underbrace{\mu_{\bt}^{\st},\ldots,\mu_{\bt}^{\st}}}}_{B,k;G}^{\phi,\os}
\frac{2^{1-l}q^Bu^k}{k!l!}\,.
\end{gather*}
By Gromov's Compactness Theorem and the assumption that $\phi^*\om\!=\!-\om$,
the inner sum in the definition of~$\Phi_{\om}^{\phi}$ has finitely nonzero terms.
For the same reason, the coefficients of the powers of $t_1,\ldots,t_N,u$ in
$\Phi_{\om}^{\phi}$ and $\Om_{\om;G}^{\phi,\os}$ lie in~$\La_{\om}^{\phi}$.

\begin{thm}\label{WDVVdim3_thm}
Suppose $(X,\om,\phi)$ is a connected compact real symplectic sixfold,
$\os$ is an OSpin-structure on a connected component~$\wch{X}^{\phi}$ of~$X^{\phi}$, and 
$G$ is an averager for $(X,\om,\phi;\wch{X}^{\phi})$.
For all \hbox{$a,b,c\!=\!1,\ldots,N$}, 
\begin{gather}
\label{WDVVodeM12_e} 
\sum_{1\le i,j\le N}\!\!\!\!\!\!\!\big(\prt_{t_a}\prt_{t_b}\prt_{t_i}\Phi_{\om}^{\phi}\big)g^{ij}
\big(\prt_u\prt_{t_j}\Om_{\om;G}^{\phi,\os}\big)+
\big(\prt_{t_a}\prt_{t_b}\Om_{\om;G}^{\phi,\os}\big)\!
\big(\prt_u^2\Om_{\om;G}^{\phi,\os}\big)
=\big(\prt_{t_a}\prt_u\Om_{\om;G}^{\phi,\os}\big)\!
\big(\prt_{t_b}\prt_u\Om_{\om;G}^{\phi,\os}\big),\\
\label{WDVVodeM03_e} 
\begin{split}
&\sum_{1\le i,j\le N}\!\!\!\!\!\!\!\big(\prt_{t_a}\prt_{t_b}\prt_{t_i}\Phi_{\om}^{\phi}\big)g^{ij}
\big(\prt_{t_c}\prt_{t_j}\Om_{\om;G}^{\phi,\os}\big)+
\big(\prt_{t_a}\prt_{t_b}\Om_{\om;G}^{\phi,\os}\big)\!
\big(\prt_{t_c}\prt_u\Om_{\om;G}^{\phi,\os}\big)\\
&\hspace{1.4in}=
\sum_{1\le i,j\le N}\!\!\!\!\!\!\!\big(\prt_{t_a}\prt_{t_c}\prt_{t_i}\Phi_{\om}^{\phi}\big)g^{ij}
\big(\prt_{t_b}\prt_{t_j}\Om_{\om;G}^{\phi,\os}\big)+
\big(\prt_{t_a}\prt_{t_c}\Om_{\om;G}^{\phi,\os}\big)\!
\big(\prt_{t_b}\prt_u\Om_{\om;G}^{\phi,\os}\big).
\end{split}\end{gather}
\end{thm}

\vspace{.1in}

The paper is organized as follows. 
Sections~\ref{DiskInvDfn_subs}-\ref{SplitAxiom_subs} outline the main steps in the proofs 
of Proposition~\ref{GtauGW_prp} and Theorem~\ref{WDVVdim3_thm}, 
pointing out the key similarities and differences with the fourfold case
treated in~\cite{RealWDVV}.
Section~\ref{Compute_subs} discusses the implications of Theorem~\ref{WDVVdim3_thm}
for the computability of the disk GW-invariants of real symplectic sixfolds.
Section~\ref{LowBnd_subs} obtains non-trivial lower bounds for counts 
of real rational curves in~$\P^3$, including with line constraints.
The relevant intersection theoretic notation and conventions are specified 
in Section~\ref{TopolPrelim_sec}.
Section~\ref{DM_sec} describes orientations for subspaces of 
the Deligne-Mumford spaces $\ov\cM_{k,l}^{\tau}$ and 
properties of the hypersurfaces~$\Ups$ in $\ov\cM_{1,2}^{\tau}$ and $\ov\cM_{0,3}^{\tau}$
that play a key role in the proof of Theorem~\ref{WDVVdim3_thm}.
Section~\ref{RealGWs_sec} sets up the notation for moduli spaces of stable maps
and their subspaces, 
states the propositions that are the main steps in the proof of Theorem~\ref{WDVVdim3_thm},
and deduces this theorem from them.
The proofs of most of these propositions are deferred to Section~\ref{proofs_sec}. 
We focus on the geometric situations when virtual techniques are not needed,
but the reasoning fits with all standard VFC constructions and thus extends 
to real symplectic threefolds $(X,\om,\phi)$ with spherical classes~$B$ 
such that $\ell_{\om}(B)\!=\!0$.

\section{Outline of the proof and applications}
\label{outline_sec}

\subsection{Disk invariants}
\label{DiskInvDfn_subs}

Let $(X,\om,\phi)$ be a real symplectic manifold, 
$\wch{X}^{\phi}$ be a connected component of~$X^{\phi}$,
$k,l\!\in\!\Z^{\ge0}$, $B\!\in\!H_2(X)$, and $J\!\in\!\cJ_{\om}^{\phi}$.
We denote~by 
$$\M_{k,l}(B;J;\wch X^\phi)\subset \ov\M_{k,l}(B;J;\wch X^\phi)$$
the moduli space of $(\phi,\tau)$-real rational degree~$B$  $J$-holomorphic maps 
with $k$ real marked points and $l$ conjugate pairs of marked points which take
the $\tau$-fixed locus $S^1\!\subset\!\P^1$ to~$\wch{X}^{\phi}$
and the stable map compactification of this moduli space, respectively.
Let 
\BE{mstdfn_e}\M^{\st}_{k,l}\big(B;J;\wch X^\phi\big)\subset\M_{k,l}\big(B;J;\wch X^\phi\big)\EE
be the subspace parametrizing $(\phi,\tau)$-real maps whose restrictions
to the disks $\D^2_{\pm}\!\subset\!\P^1$ cut out by the $\tau$-fixed locus $S^1\!\subset\!\P^1$
do not represent elements in the kernel of~\eref{Z2bnddfn_e}.
The stable map compactification
$$\ov\M^{\st}_{k,l}\big(B;J;\wch X^\phi\big)\subset\ov\M_{k,l}\big(B;J;\wch X^\phi\big)$$
of this subspace is a union of connected components of $\ov\M_{k,l}(B;J;\wch X^\phi)$
and has no boundary.
The codimension~1 strata of $\ov\M^{\st}_{k,l}(B;J;\wch X^\phi)$ consist of maps 
from two-component domains with a real node.
An OSpin-structure~$\os$ on $\wch X^\phi$ induces an orientation~$\fo_{\M}$ of 
$\M_{k,l}^{\st}(B;J;\wch X^\phi)$, but $\ov\M_{k,l}(B;J;\wch X^\phi)$
is generally unorientable.
The orientation~$\fo_{\M}$ does not extend through some codimension~1 strata
of $\ov\M_{k,l}^{\st}(B;J;\wch X^\phi)$; we call them \sf{bad strata}. 

If the domain and target of the evaluation morphism 
\BE{evtotmap_e0}\ev\!: \ov\M_{k,l}^{\st}(B;J;\wch X^\phi)\lra 
\wch{X}_{k,l}\!\equiv\!(\wch X^{\phi})^k\!\times\!X^l\EE
are of the same dimension, a generic path between two generic points in~$\wch{X}_{k,l}$ avoids 
the images of the bad strata and thus $\ev$ has a well-defined degree.
This fundamental insight of~\cite{Jake}, formulated in terms of disk moduli spaces
instead of the real map spaces introduced in~\cite{Penka2},
provided a moduli space interpretation
of Welschinger's invariants with the potential for applications of techniques
of complex GW-theory to study these invariants.
In a more standard perspective of symplectic topology, 
the restriction of~$\ev$ to the complement of the bad codimension~1 strata
is a codimension~0 pseudocycle; 
the degree of~$\ev$ is simply the degree of this pseudocycle.
This perspective on the insight of~\cite{Jake}
provides the intersection-theoretic setting 
for the proof of Theorem~\ref{WDVVdim3_thm} in the present paper.
An analogue of this perspective plays a similar role in~\cite{RealWDVV},
which established the WDVV-type relations for real symplectic fourfolds
foreseen in~\cite{Jake2}.

More generally, suppose that 
\BE{bfCdfn_e}
\bfC=q_1\!\times\!q_2\!\times\!\ldots\!\times\!q_k\!\times\!H_1\!\times\!\ldots\!\times\!H_l
\subset \wch{X}_{k,l}\EE
is a generic constraint consisting of $q_1,\ldots,q_k\!\in\!\wch{X}^{\phi}$ and oriented
submanifolds \hbox{$H_1,\ldots,H_l\!\subset\!X$} of real codimensions 2, 4, and~6 
(we call them \sf{divisor}, \sf{curve}, and \sf{point} constraints, respectively) so~that
\BE{bfCcond_e}\dim\,\ov\M_{k,l}\big(B;J;\wch X^\phi\big)\!+\!\dim\,\bfC=\dim\,\wch{X}_{k,l}\,.\EE
This implies that the intersection 
\BE{intersectionC_e}
\ov\M_{k,l}^{\st}(B;J;\wch X^\phi)\xlra{~\ev~}\wch{X}_{k,l}\longleftarrow\!\rhook\bfC\EE
of the two maps to $\wch{X}_{k,l}$ is a finite set of signed points.
For a suitable choice of the orientation~$\fo_{\M}$, a generic path between 
two generic constraints as in~\eref{bfCdfn_e} avoids the bad strata 
of the left-hand side of~\eref{intersectionC_e} except for 
the strata consisting of maps that are constant on a component of the domain 
which carries only a conjugate pair~$z_i^{\pm}$ of marked points 
with $H_i$ being a curve~class; 
see the proof of Proposition~\ref{JakePseudo_prp} in Section~\ref{orient_subs}.
This observation, which follows from the reasoning in~\cite{Jake}, 
implies that the intersection number of~\eref{intersectionC_e} does not depend on 
the choices of generic point constraints in~$\wch{X}^{\phi}$ and~$X$, 
divisor constraints representing fixed elements of~$H_4(X)$, 
curve constraints representing fixed elements of $H_2(X\!-\!\wch X^\phi)$, and 
$J\!\in\!\cJ_{\om}^{\phi}$. 

The intersection number in~\eref{intersectionC_e} is the \sf{open GW-invariant}~\eref{RGWdfn_e}
informally introduced by J.~Solomon after~\cite{Jake} under the assumption
that~\eref{noS2bub_e} holds.
If this is the case, the two spaces in~\eref{mstdfn_e} are the same.
Thus, the superscript~$\st$ can be dropped from the left-hand side in~\eref{intersectionC_e}
and the resulting invariants~\eref{RGWdfn_e}  
count all $(\phi,\tau)$-real degree~$B$ curves through the constraint~$\bfC$.

\subsection{Cancellations under symmetry}
\label{sphbub_subs}

If \eref{noS2bub_e} does not hold, the spaces 
\begin{equation*}\begin{split}
\M'_{k,l}\big(B;J;\wch X^\phi\big)&\equiv
\M_{k,l}\big(B;J;\wch X^\phi\big)-\M_{k,l}^{\st}\big(B;J;\wch X^\phi\big)
\qquad\hbox{and} \\
\ov\M'_{k,l}\big(B;J;\wch X^\phi\big)&\equiv
\ov\M_{k,l}\big(B;J;\wch X^\phi\big)-\ov\M_{k,l}^{\st}\big(B;J;\wch X^\phi\big)
\end{split}\end{equation*}
may be nonempty.
The boundary of $\ov\M_{k,l}(B;J;\wch X^\phi)$ is contained in $\ov\M'_{k,l}(B;J;\wch X^\phi)$.

Let $G$ be an averager for $(X,\om,\phi;\wch{X}^{\phi})$ as in Definition~\ref{aveG_dfn}.
Suppose $J\!\in\!\cJ_{\om;G}^{\phi}$ and
$\bfC$ as in~\eref{bfCdfn_e} and~\eref{bfCcond_e} is $G$-invariant.
Each $\psi\!\in\!G$ induces automorphisms
$$\psi_{k,l}\!:\wch{X}_{k,l}\lra\wch{X}_{k,l}
\qquad\hbox{and}\qquad
\Psi\!:\M'_{k,l}\big(B;J;\wch X^\phi\big)\lra \M'_{k,l}\big(B;J;\wch X^\phi\big)$$
by acting on each component of $\wch{X}_{k,l}$ and by
composing each map $u\!:\P^1\!\lra\!X$ with~$\psi$.
The first induced automorphism preserves the orientation if and only if
either $k\!\in\!2\Z$ or $\psi$ preserves the orientation of~$\wch X^\phi$.
The second induced isomorphism preserves the orientation~$\fo_{\M}$ if and only if
$\psi$ preserves the orientation of~$\wch X^\phi$.
Since the right vertical arrow in the diagram
\BE{GCancelDiag_e}\begin{split}
\xymatrix{\M'_{k,l}(B;J;\wch X^\phi)\ar[rr]^<<<<<<<<<{\ev}\ar[d]_{\Psi}^{\approx}&&
\wch{X}_{k,l}\ar[d]_{\psi_{k,l}}^{\approx} && 
\,\bfC\ar@{_(->}[ll]\ar[d]_{\psi}^{\approx}\\
\M'_{k,l}(B;J;\wch X^\phi)\ar[rr]^<<<<<<<<<{\ev}&&
\wch{X}_{k,l} && \,\bfC\ar@{_(->}[ll]}
\end{split}\EE
is orientation-preserving, it follows that $\psi$ induces a sign-reversing bijection
on the intersection set of the two maps in~\eref{GCancelDiag_e} 
if $\psi$ reverses the orientation of~$\wch X^\phi$ and $k\!\in\!2\Z$.
The intersection number of these two maps is zero then.
Along with the last paragraph of Section~\ref{DiskInvDfn_subs},
this implies that the superscript~$\st$ can be dropped from the left-hand side in~\eref{intersectionC_e}
and that the resulting invariants~\eref{RGWdfn_e} and~\eref{Gnumsdfn_e} 
count all $(\phi,\tau)$-real degree~$B$ curves through the constraint~$\bfC$
(provided $J\!\in\!\cJ_{\om;G}^{\phi}$ and $\bfC$ is $G$-invariant),
whether or not \eref{noS2bub_e} holds. 

If $B$ is $(\wch{X}^{\phi},\Z_2)$-trivial, then
$$\M'_{k,l}\big(B;J;\wch X^\phi\big)=\M_{k,l}\big(B;J;\wch X^\phi\big).$$
If $k\!\in\!2\Z$, which is equivalent to the last condition in~\eref{GtauGWvan_e0}, 
the reasoning in the previous paragraph yields the vanishing of the numbers~\eref{Gnumsdfn_e}
in the first case of Proposition~\ref{GtauGW_prp}\ref{GtauGWvan_it}.

The prototypical example of a real symplectic sixfold is the complex projective space~$\P^3$
with the Fubini-Study symplectic form~$\om_3$ and the anti-symplectic involution
$$\tau_3\!: \P^3\lra\P^3, \qquad 
\tau_3\big([Z_0,Z_1,Z_2,Z_3]\big)=\big[\ov{Z_0},\ov{Z_1},\ov{Z_2},\ov{Z_3}\big].$$
An averager~$G$ in this case is generated by the reflection about a $\tau_3$-invariant 
complex hyperplane such~as
\BE{psi3dfn_3}\psi_3\!:\P^3\lra\P^3, \qquad  
\psi_3\big([Z_0,Z_1,Z_2,Z_3]\big)=\big[Z_0,Z_1,Z_2,-Z_3\big].\EE
According to \cite[Rem.~2.4(2)]{Wel6}, it was observed by G.~Mikhalkin in the early 2000s that 
even-degree curves passing through collections of points
interchanged by such reflections have opposite signs in the sense of~\cite{Wel6}.
The resulting vanishing of Welschinger's invariants is precisely 
the $(\P^3,\tau_3)$-case of the first case of Proposition~\ref{GtauGW_prp}\ref{GtauGWvan_it}
with all $\mu_i\!\in\!H^6(\P^3)$.
Reflections analogous to~\eref{psi3dfn_3} can be readily defined for $(\P^1)^3$ with 
two different conjugations (with fixed loci $(S^1)^3$ and $S^1\!\times\!S^2$).

\subsection{Lifting relations from Deligne-Mumford spaces}
\label{LiftRel_subs}

Let $\bfC$ be a generic constraint as in~\eref{bfCdfn_e} so that 
\BE{bfCcond_e2}
\dim\,\ov\M_{k,l}\big(B;J;\wch X^\phi\big)\!+\!\dim\,\bfC=\dim\,\wch{X}_{k,l}\!+\!2\,.\EE
We cut $\ov\M_{k,l}^{\st}(B;J;\wch X^\phi)$ along all bad strata and   
obtain an oriented moduli space $\wh\M_{k,l}^{\st}(B;J;\wch X^\phi)$ with boundary. 
The forgetful morphisms~\eref{ffdfn_e0} we encounter take values
in the subspaces $\ov\cM_{k',l'}^{\tau}$ of $\R\ov\cM_{0,k',l'}$ of 
real curves with non-empty fixed locus.
We denote the induced morphisms from the cut moduli spaces also by~$\ff_{k',l'}$.
Let $\Ups\!\subset\!\ov\cM_{k',l'}^{\tau}$ be a co-oriented bordered hypersurface  
whose boundary consists of curves with three components and a conjugate pair of nodes.

By the assumptions above, the intersection of the~maps
\BE{intersectionCY_e}
\wh\M_{k,l}^{\st}\big(B;J;\wch X^\phi\big)\xlra{~\ev\times\ff_{k',l'}~}
\wch{X}_{k,l}\!\times\!\ov\cM_{k',l'}^{\tau}\longleftarrow\!\rhook\bfC\!\times\!\Ups\EE
is a one-dimensional manifold with boundary.
Thus, 
\BE{bigeq_e}
\wh\M_{k,l}^{\st}\big(B;J;\wch X^\phi\big)\cdot\big(\bfC\!\times\!\prt\Ups\big)
=\pm~\prt\wh\M_{k,l}^{\st}\big(B;J;\wch X^\phi\big)\cdot\big(\bfC\!\times\!\Ups\big),\EE
where $\cdot$ denotes the signed counts of intersection points in 
$X_{k,l}\!\times\!\ov\cM_{k',l'}^{\tau}$. 
For $\Ups\!\subset\!\ov\cM_{k',l'}^{\tau}$ as in Lemmas~\ref{M12rel_lmm} and~\ref{M03rel_lmm},
\eref{bigeq_e} translates into the relations between nodal curve counts in Figure~\ref{LiftedRel_fig};
see Proposition~\ref{LiftedRel_prp}.
Each diagram in this figure represents counts of curves of the corresponding shape 
constrained by~$\bfC$;
the labels $\ep_{\bfC}(\cS)\!\not\in\!2\Z$ and $\cdot\!\Ups$
under the diagrams on the right-hand side indicate that only intersections
of some strata of two-component maps with~$\Ups$ contribute to this relation.
These relations are the direct analogues of the relations of \cite[Fig.~1]{RealWDVV}.

\begin{figure}\begin{center}
\begin{tikzpicture}
\draw (0,0) circle [radius=0.5];
\draw (-0.5,0) arc [start angle=180, end angle=360, x radius=0.5, y radius=0.2];
\draw [dashed] (-0.5,0) arc [start angle=180, end angle=0, x radius=0.5, y radius=0.1];
\draw [fill] (-0.5,0) circle [radius=0.04];
\node [left] at (-0.5,0) {$x_1$};
\draw (0,1) circle [radius=0.5]; 
\draw (0,-1) circle [radius=0.5];
\draw [fill] (-0.2,1.25) circle [radius=0.02];
\node [above] at (-0.3,1.3) {$z_1^+$};
\draw [fill] (-0.2,-1.25) circle [radius=0.02];
\node [below] at (-0.3,-1.3) {$z_1^-$};
\draw [fill] (0.25,1.25) circle [radius=0.02];
\node [above] at (0.33,1.3) {$z_2^+$};
\draw [fill] (0.25,-1.25) circle [radius=0.02];
\node [below] at (0.33,-1.3) {$z_2^-$};
\node at (1,0) {$+$};
\end{tikzpicture}\hspace{-0.15cm}
\begin{tikzpicture}
\draw (0,0) circle [radius=0.5];
\draw (-0.5,0) arc [start angle=180, end angle=360, x radius=0.5, y radius=0.2];
\draw [dashed] (-0.5,0) arc [start angle=180, end angle=0, x radius=0.5, y radius=0.1];
\draw [fill] (-0.5,0) circle [radius=0.04];
\node [left] at (-0.5,0) {$x_1$};
\draw (0,1) circle [radius=0.5]; 
\draw (0,-1) circle [radius=0.5];
\draw [fill] (-0.2,1.25) circle [radius=0.02];
\node [above] at (-0.3,1.3) {$z_1^+$};
\draw [fill] (-0.2,-1.25) circle [radius=0.02];
\node [below] at (-0.3,-1.3) {$z_1^-$};
\draw [fill] (0.25,1.25) circle [radius=0.02];
\node [above] at (0.33,1.3) {$z_2^-$};
\draw [fill] (0.25,-1.25) circle [radius=0.02];
\node [below] at (0.33,-1.3) {$z_2^+$};
\node at (1.25,0) {=$~-2$};
\end{tikzpicture}\hspace{-0.3cm}
\begin{tikzpicture}
\draw (-0.5,0) circle [radius=0.5];
\draw (0.5,0) circle [radius=0.5];
\draw (0,0) arc [start angle=180, end angle=360, x radius=0.5, y radius=0.2];
\draw [dashed] (0,0) arc [start angle=180, end angle=0, x radius=0.5, y radius=0.1];
\draw (-1,0) arc [start angle=180, end angle=360, x radius=0.5, y radius=0.2];
\draw [dashed] (-1,0) arc [start angle=180, end angle=0, x radius=0.5, y radius=0.1];
\draw [fill] (-0.75,0.3) circle [radius=0.02];
\node [above left] at (-0.7,0.3) {$z_1^+$}; 
\draw [fill] (-0.75,-0.3) circle [radius=0.02];
\node [below left] at (-0.7,-0.3) {$z_1^-$}; 
\draw [fill] (-0.3,0.3) circle [radius=0.02];
\node [above] at (-0.3,0.4) {$z_2^\pm$}; 
\draw [fill] (-0.3,-0.3) circle [radius=0.02];
\node [below] at (-0.3,-0.4) {$z_2^\mp$};
\draw [fill] (1,0) circle [radius=0.04];
\node [right] at (1,0) {$x_1$};
\node [below] at (0.33,-1.7) { };
\node at (2,0) {$-$ 2};
\node at (0,-1.4){\sm{$\ep_{\bfC}(\cS)\!\not\in\!2\Z,~\cdot\!\Ups$}};
\end{tikzpicture}\hspace{-0.15cm}
\begin{tikzpicture}
\draw (-0.5,0) circle [radius=0.5];
\draw (0.5,0) circle [radius=0.5];
\draw (0,0) arc [start angle=180, end angle=360, x radius=0.5, y radius=0.2];
\draw [dashed] (0,0) arc [start angle=180, end angle=0, x radius=0.5, y radius=0.1];
\draw (-1,0) arc [start angle=180, end angle=360, x radius=0.5, y radius=0.2];
\draw [dashed] (-1,0) arc [start angle=180, end angle=0, x radius=0.5, y radius=0.1];
\draw [fill] (-1,0) circle [radius=0.04];
\node [left] at (-0.9,0) {$x_1$}; 
\draw [fill] (-0.3,0.3) circle [radius=0.02];
\node [above] at (-0.3,0.4) {$z_1^+$}; 
\draw [fill] (-0.3,-0.3) circle [radius=0.02];
\node [below] at (-0.3,-0.4) {$z_1^-$};
\draw [fill] (0.65,0.28) circle [radius=0.02];
\node [above right] at (0.6,0.28) {$z_2^+$};
\draw [fill] (0.65,-0.28) circle [radius=0.02];
\node [below right] at (0.6,-0.28) {$z_2^-$};
\node [below] at (0.33,-1.7) { };
\node at (0,-1.4){\sm{$\ep_{\bfC}(\cS)\!\not\in\!2\Z,~\cdot\!\Ups$}};
\end{tikzpicture}

\begin{tikzpicture}
\draw (0,0) circle [radius=0.5];
\draw (-0.5,0) arc [start angle=180, end angle=360, x radius=0.5, y radius=0.2];
\draw [dashed] (-0.5,0) arc [start angle=180, end angle=0, x radius=0.5, y radius=0.1];
\draw [fill] (-0.2,0.3) circle [radius=0.02];
\node [left] at (-0.25,0.3) {$z_2^\pm$};
\draw [fill] (-0.2,-0.3) circle [radius=0.02];
\node [left] at (-0.25,-0.4) {$z_2^\mp$};
\draw (0,1) circle [radius=0.5]; 
\draw (0,-1) circle [radius=0.5];
\draw [fill] (-0.2,1.25) circle [radius=0.02];
\node [above] at (-0.3,1.3) {$z_1^+$};
\draw [fill] (-0.2,-1.25) circle [radius=0.02];
\node [below] at (-0.3,-1.3) {$z_1^-$};
\draw [fill] (0.25,1.25) circle [radius=0.02];
\node [above] at (0.33,1.3) {$z_3^+$};
\draw [fill] (0.25,-1.25) circle [radius=0.02];
\node [below] at (0.33,-1.3) {$z_3^-$};
\node at (1,0) {+};
\end{tikzpicture}\hspace{-0.3cm}
\begin{tikzpicture}
\draw (0,0) circle [radius=0.5];
\draw (-0.5,0) arc [start angle=180, end angle=360, x radius=0.5, y radius=0.2];
\draw [dashed] (-0.5,0) arc [start angle=180, end angle=0, x radius=0.5, y radius=0.1];
\draw [fill] (-0.2,0.3) circle [radius=0.02];
\node [left] at (-0.25,0.3) {$z_2^\pm$};
\draw [fill] (-0.2,-0.3) circle [radius=0.02];
\node [left] at (-0.25,-0.4) {$z_2^\mp$};
\draw (0,1) circle [radius=0.5]; 
\draw (0,-1) circle [radius=0.5];
\draw [fill] (-0.2,1.25) circle [radius=0.02];
\node [above] at (-0.3,1.3) {$z_1^+$};
\draw [fill] (-0.2,-1.25) circle [radius=0.02];
\node [below] at (-0.3,-1.3) {$z_1^-$};
\draw [fill] (0.25,1.25) circle [radius=0.02];
\node [above] at (0.33,1.3) {$z_3^-$};
\draw [fill] (0.25,-1.25) circle [radius=0.02];
\node [below] at (0.33,-1.3) {$z_3^+$};
\node at (1,0) {$-$};
\end{tikzpicture}\hspace{-0.3cm}
\begin{tikzpicture}
\draw (0,0) circle [radius=0.5];
\draw (-0.5,0) arc [start angle=180, end angle=360, x radius=0.5, y radius=0.2];
\draw [dashed] (-0.5,0) arc [start angle=180, end angle=0, x radius=0.5, y radius=0.1];
\draw [fill] (-0.2,0.3) circle [radius=0.02];
\node [left] at (-0.25,0.3) {$z_3^\pm$};
\draw [fill] (-0.2,-0.3) circle [radius=0.02];
\node [left] at (-0.25,-0.4) {$z_3^\mp$};
\draw (0,1) circle [radius=0.5]; 
\draw (0,-1) circle [radius=0.5];
\draw [fill] (-0.2,1.25) circle [radius=0.02];
\node [above] at (-0.3,1.3) {$z_1^+$};
\draw [fill] (-0.2,-1.25) circle [radius=0.02];
\node [below] at (-0.3,-1.3) {$z_1^-$};
\draw [fill] (0.25,1.25) circle [radius=0.02];
\node [above] at (0.33,1.3) {$z_2^+$};
\draw [fill] (0.25,-1.25) circle [radius=0.02];
\node [below] at (0.33,-1.3) {$z_2^-$};
\node at (1,0) {$-$};
\end{tikzpicture}\hspace{-0.3cm}
\begin{tikzpicture}
\draw (0,0) circle [radius=0.5];
\draw (-0.5,0) arc [start angle=180, end angle=360, x radius=0.5, y radius=0.2];
\draw [dashed] (-0.5,0) arc [start angle=180, end angle=0, x radius=0.5, y radius=0.1];
\draw [fill] (-0.2,0.3) circle [radius=0.02];
\node [left] at (-0.25,0.3) {$z_3^\pm$};
\draw [fill] (-0.2,-0.3) circle [radius=0.02];
\node [left] at (-0.25,-0.4) {$z_3^\mp$};
\draw (0,1) circle [radius=0.5]; 
\draw (0,-1) circle [radius=0.5];
\draw [fill] (-0.2,1.25) circle [radius=0.02];
\node [above] at (-0.3,1.3) {$z_1^+$};
\draw [fill] (-0.2,-1.25) circle [radius=0.02];
\node [below] at (-0.3,-1.3) {$z_1^-$};
\draw [fill] (0.25,1.25) circle [radius=0.02];
\node [above] at (0.33,1.3) {$z_2^-$};
\draw [fill] (0.25,-1.25) circle [radius=0.02];
\node [below] at (0.33,-1.3) {$z_2^+$};
\node at (1.2,0) {=  2};
\end{tikzpicture}\hspace{-0.3cm}
\begin{tikzpicture}
\draw (-0.5,0) circle [radius=0.5];
\draw (0.5,0) circle [radius=0.5];
\draw (0,0) arc [start angle=180, end angle=360, x radius=0.5, y radius=0.2];
\draw [dashed] (0,0) arc [start angle=180, end angle=0, x radius=0.5, y radius=0.1];
\draw (-1,0) arc [start angle=180, end angle=360, x radius=0.5, y radius=0.2];
\draw [dashed] (-1,0) arc [start angle=180, end angle=0, x radius=0.5, y radius=0.1];
\draw [fill] (-0.75,0.3) circle [radius=0.02];
\node [above left] at (-0.7,0.3) {$z_1^+$}; 
\draw [fill] (-0.75,-0.3) circle [radius=0.02];
\node [below left] at (-0.7,-0.3) {$z_1^-$}; 
\draw [fill] (-0.3,0.3) circle [radius=0.02];
\node [above] at (-0.3,0.4) {$z_3^\pm$}; 
\draw [fill] (-0.3,-0.3) circle [radius=0.02];
\node [below] at (-0.3,-0.4) {$z_3^\mp$};
\draw [fill] (0.65,0.28) circle [radius=0.02];
\node [above right] at (0.6,0.28) {$z_2^+$};
\draw [fill] (0.65,-0.28) circle [radius=0.02];
\node [below right] at (0.6,-0.28) {$z_2^-$};
\node [below] at (0.33,-1.7) { };
\node at (1.5,0) {$+$ 2};
\node at (0,-1.4){\sm{$\ep_{\bfC}(\cS)\!\not\in\!2\Z,~\cdot\!\Ups$}};
\end{tikzpicture}\hspace{-0.3cm}
\begin{tikzpicture}
\draw (-0.5,0) circle [radius=0.5];
\draw (0.5,0) circle [radius=0.5];
\draw (0,0) arc [start angle=180, end angle=360, x radius=0.5, y radius=0.2];
\draw [dashed] (0,0) arc [start angle=180, end angle=0, x radius=0.5, y radius=0.1];
\draw (-1,0) arc [start angle=180, end angle=360, x radius=0.5, y radius=0.2];
\draw [dashed] (-1,0) arc [start angle=180, end angle=0, x radius=0.5, y radius=0.1];
\draw [fill] (-0.75,0.3) circle [radius=0.02];
\node [above left] at (-0.7,0.3) {$z_1^+$}; 
\draw [fill] (-0.75,-0.3) circle [radius=0.02];
\node [below left] at (-0.7,-0.3) {$z_1^-$}; 
\draw [fill] (-0.3,0.3) circle [radius=0.02];
\node [above] at (-0.3,0.4) {$z_2^\pm$}; 
\draw [fill] (-0.3,-0.3) circle [radius=0.02];
\node [below] at (-0.3,-0.4) {$z_2^\mp$};
\draw [fill] (0.65,0.28) circle [radius=0.02];
\node [above right] at (0.6,0.28) {$z_3^+$};
\draw [fill] (0.65,-0.28) circle [radius=0.02];
\node [below right] at (0.6,-0.28) {$z_3^-$};
\node [below] at (0.33,-1.7) { };
\node at (0,-1.4){\sm{$\ep_{\bfC}(\cS)\!\not\in\!2\Z,~\cdot\!\Ups$}};
\end{tikzpicture}
\end{center}
\caption{The relations on stable maps induced via~\eref{bigeq_e} by lifting 
codimension~2 relations from $\ov\cM_{1,2}^{\tau}$ and~$\ov\cM_{0,3}^{\tau}$;
the curves on the right-hand sides of the two relations are constrained
by the hypersurfaces $\Ups$ in $\ov\cM_{1,2}^{\tau}$ and~$\ov\cM_{0,3}^{\tau}$.}
\label{LiftedRel_fig}\end{figure}

\subsection{Splitting properties for disk invariants}
\label{SplitAxiom_subs}

The nodal curve counts appearing in~\eref{bigeq_e} and 
represented by the diagrams in Figure~\ref{LiftedRel_fig}
in general depend on the components of the constraint~$\bfC$ and 
not just on the homology classes represented by these components.
However, these counts depend only on the homology classes if $J\!\in\!\cJ_{\om;G}^{\phi}$ and
$\bfC$ is $G$-invariant for an averager~$G$ for $(X,\om,\phi;\wch{X}^{\phi})$ 
as in Definition~\ref{aveG_dfn}.
The right-hand side of~\eref{bigeq_e} then splits into invariant counts~\eref{Gnumsdfn_e} 
of irreducible $(\phi,\tau)$-real curves exactly as~\cite{RealWDVV}
(where an averager is not needed)
because of the vanishing of the intersection number in~\eref{GCancelDiag_e};
see Proposition~\ref{Rdecomp_prp}. 
The left-hand side of~\eref{bigeq_e} splits into invariant counts~\eref{Gnumsdfn_e} 
of irreducible $(\phi,\tau)$-real curves and the complex GW-invariants~\eref{CGWdfn_e}; 
see Proposition~\ref{Cdecomp_prp}. 
While the latter splitting is analogous to the splitting of \cite[Prop.~5.3]{RealWDVV}
in the WDVV sense,
its proof involves counts of $(\phi,\tau)$-real curves with insertions in~$H_*(X\!-\!\wch{X}^{\phi})$,
and not just in~$H_*(X)$.

The left-hand side of~\eref{bigeq_e} counts nodal curves with one real component and 
one conjugate pair of components; see Figure~\ref{LiftedRel_fig}. 
They arise from pairs $B_0,B'\!\in\!H_2(X)$ such~that 
$$B_0\!+\!\fd(B')=B\in H_2(X)$$
and decompositions $\{1,\ldots,l\}\!=\!L_0\!\sqcup\!L_{\C}$.
Let 
$$\bfC_0=q_1\!\times\!\ldots\!\times\!q_k\!\times\!
\prod_{i\in L_0}\!\!H_i\subset (\wch X^\phi)^k\!\times\!X^{L_0}
\qquad\hbox{and}\qquad
\bfC'=\prod_{i\in L_{\C}}\!\!H_i\subset X^{L_{\C}}.$$
We need to determine the signed number $N_{B_0,B'}(\bfC_0,\bfC')$
of nodal curves as on the left-hand side of
Figure~\ref{LiftedRel_fig} so~that 
\BEnum{$\bu$}

\item the real component has degree~$B_0$ and  passes through the constraints~$\bfC_0$, and

\item one of the conjugate components has degree~$B'$ and passes through the constraints~$\bfC'$
(with some components~$H_i$ replaced by~$\phi(H_i)$).

\EEnum
By~\eref{bfCcond_e2}, this number vanishes unless
\BEnum{(L\arabic*)}

\item\label{Decomp1_it} the number of $(\phi,\tau)$-real degree~$B_0$  curves 
passing through~$\bfC_0$ is finite, 
and the number of degree~$B'$ curves passing through~$\bfC'$ 
and another curve constraint is finite, or

\item\label{Decomp2_it} the number of degree~$B'$ curves passing through~$\bfC'$ is finite, 
and  the number of $(\phi,\tau)$-real degree~$B_0$ curves passing through~$\bfC_0$ 
and another curve constraint is also finite. 
\EEnum

\vspace{.15in}

In Case~\ref{Decomp1_it}, 
the number of $(\phi,\tau)$-real  degree~$B_0$ curves passing through~$\bfC_0$ is simply 
the corresponding disk GW-invariant~\eref{Gnumsdfn_e};
the $G$-invariance conditions ensure that this number is well-defined 
as in Section~\ref{sphbub_subs}.
The number $N_{B_0,B'}(\bfC_0,\bfC')$ is then this disk GW-invariant times
the number of degree~$B'$ curves passing through~$\bfC'$ and a curve constraint 
representing~$B_0$;
the latter number is just a complex GW-invariant~\eref{CGWdfn_e}.

In Case~\ref{Decomp2_it},
suppose $G$ is an averager, $J\!\in\!\cJ_{\om;G}^{\phi}$, and $\bfC$ is $G$-invariant.
Let
$$GB'=\big\{g_*B'\!:g\!\in\!G\big\} \qquad\hbox{and}\qquad
\lr{B'}_G=\frac{1}{|G|}\sum_{g\in G}\!g_*B'\in H_*(X)\,.$$
Let $C_1,\ldots,C_N$ be the curves that pass through~$\bfC'$ and whose degree is in~$GB'$; 
they lie in \hbox{$X\!-\!\wch{X}^{\phi}$}.
Their (standardly signed) number~is the sum 
$$\sum_{B''\in GB'}\!\!\!\!\!\blr{(H_i)_{i\in L_{\C}}}_{B''}^X
=|GB'|\blr{(H_i)_{i\in L_{\C}}}_{B'}$$
of the complex GW-invariants~\eref{CGWdfn_e};
the equality above holds because the constraint~$\bfC'$ is $G$-invariant.
Thus,
$$[C_1]_X\!+\!\ldots\!+[C_N]_X=|GB'|\blr{(H_i)_{i\in L_{\C}}}_{B'}\lr{B'}_G\,.$$
The sum of the numbers $N_{B_0,B''}(\bfC_0,\bfC')$ over all $B''\!\in\!GB'$ 
is the signed number of $(\phi,\tau)$-real degree~$B_0$ curves
passing through~$\bfC_0$ and $C_1\!\cup\!\ldots\!\cup\!C_N$.
While the homology class of each curve~$C_i$ in~$X$ is  a specific element~$B''$ of~$GB'$,
the usual count~\eref{RGWdfn_e} of $(\phi,\tau)$-real degree~$B_0$  curves
passing through~$\bfC_0$ and $C_i$ depends on the homology class of~$C_i$ 
in~$X\!-\!\wch{X}^{\phi}$.
However, $C_1\!\cup\!\ldots\!\cup\!C_N$ is a $G$-invariant curve constraint 
because the tuple~$\bfC'$ is $G$-invariant.
Since the first map in~\eref{aveG_e1} is an isomorphism, 
$$[C_1]_{X-\wch{X}^{\phi}}\!+\!\ldots\!+[C_N]_{X-\wch{X}^{\phi}}
=|GB'|\blr{(H_i)_{i\in L_{\C}}}_{B'}\big\{\io_G^*\big\}^{-1}\big(\lr{B'}_G\big).$$
Therefore, 
\begin{equation*}\begin{split}
\sum_{B''\in GB'}\!\!\!\!\!N_{B_0,B''}(\bfC_0,\bfC')
&=|GB'|\blr{(H_i)_{i\in L_{\C}}}_{B'}\blr{(H_i)_{i\in L_0},\lr{B'}_G}_{B_0;\wch{X}^{\phi};G}^{\phi,\os}\\
&=\blr{(H_i)_{i\in L_{\C}}}_{B'}\sum_{B''\in GB'}\!\!\!\!\!
\blr{(H_i)_{i\in L_0},\lr{B''}_G}_{B_0;\wch{X}^{\phi};G}^{\phi,\os}
\end{split}\end{equation*}
by the definition of the disk invariant~\eref{Gnumsdfn_e}.
Since $G$ acts trivially on~$H_2(X)^{\phi}_-$,
$B''\!-\!\lr{B''}_G$ lies in~$H_2(X)^{\phi}_+$.
Along with the second case of Proposition~\ref{GtauGW_prp}\ref{GtauGWvan_it},
this implies that the constraint~$\lr{B''}_G$ can be replaced by~$B''$ above.

Theorem~\ref{WDVVdim3_thm} follows from \eref{bigeq_e} and the splitting properties
provided Propositions~\ref{Rdecomp_prp} and~\ref{Cdecomp_prp};
see Section~\ref{SolWDVVpf_subs}.

\subsection{Computability of disk invariants}
\label{Compute_subs}

The WDVV relation~\cite{KM,RT} for the genus~0 GW-invariants~\eref{CGWdfn_e} 
is very effective in determining these invariants from basic low-degree input
whenever $H^2(X)$ generates $H^{2*}(X)$ and $\ell_{\om}(B)\!>\!0$ for all spherical $B\!\in\!H_2(X)$; 
see \cite[Section~10]{RT}.
A similar observation concerning the two ODEs of Theorem~\ref{WDVVdim3_thm}
and the disk invariants~\eref{Gnumsdfn_e} encoded by 
the generating function~$\Om_{\om;G}^{\phi,\os}$ is made in \cite[Prop.~17]{Adam}.
The latter do not include invariants with curve constraints that are not $G$-invariant. 
We next clarify the condition necessary for the conclusion of \cite[Prop.~17]{Adam}
and reduce all disk invariants~\eref{RGWdfn_e} 
of real symplectic sixfolds $(X,\om,\phi)$ with a choice of a connected component
$\wch{X}^{\phi}\!\subset\!X^{\phi}$ that admit an averager~$G$ as in Definition~\ref{aveG_dfn}
to disk invariants without curve constraints that are not $G$-invariant. 

The cup product of $H^*(X)$ with $H^*(X,\wch{X}^{\phi})$ and the Poincare Duality 
isomorphisms~\eref{PDisomdfn_e} induce an intersection homomorphism 
\BE{Hintersecdfn_e} H_*(X)^{\phi}_+\!\otimes\!H_*(X\!-\!\wch{X}^{\phi})^{\phi}_+
\lra H_*\big(X\!-\!\wch{X}^{\phi}\big)^{\phi}_- \,.\EE
By the reasoning in the next paragraph, the natural homomorphism
\BE{Hpincdfn_e} H_p(X\!-\!\wch{X}^{\phi})^{\phi}_+\lra H_p(X)^{\phi}_+\EE
is an isomorphism for all $p\!\in\!\Z$.
A sufficient condition for the conclusion of \cite[Prop.~17]{Adam},
beyond the complex case, is the surjectivity of the composition
\BE{Hcapdfn_e}
\cap\!: H_4(X)^{\phi}_+\!\otimes\!H_4(X)^{\phi}_+\lra H_2\big(X\!-\!\wch{X}^{\phi}\big)^{\phi}_-\EE
of~\eref{Hintersecdfn_e} in degree~$(4,4)$ with
the inverse of~\eref{Hpincdfn_e} for $p\!=\!4$.
This is the case in particular for $(\P^3,\tau_3)$ and $(\P^1)^3$ with 
the two natural involutions.

In order to reduce all disk invariants~\eref{RGWdfn_e} 
to disk invariants without curve constraints that are not $G$-invariant,
it is sufficient to show that insertions in the kernel of the first homomorphism 
in~\eref{aveG_e0} can be traded for real point insertions;
this is achieved by Proposition~\ref{SNXphi_prp} below.
The homology long exact sequence for the pair $(X,X\!-\!\wch{X}^{\phi})$
induces an exact sequence
$$\ldots\lra H_3(X,X\!-\!\wch{X}^{\phi})^{\phi}_{\pm}
\lra H_2(X\!-\!\wch{X}^{\phi})^{\phi}_{\pm} \lra H_2(X)^{\phi}_{\pm} 
\lra H_2(X,X\!-\!\wch{X}^{\phi})^{\phi}_{\pm}\lra\ldots$$
Since the action of $\phi$ on the normal bundle $\cN\wch{X}^{\phi}$ of
$\wch{X}^{\phi}$ in~$X$ is orientation-reversing,
excision and Thom isomorphism give
\BE{Hpvan_e}H_3(X,X\!-\!\wch{X}^{\phi})^{\phi}_-\approx H_0(\wch{X}^{\phi})\approx\Z,
\quad
H_3(X,X\!-\!\wch{X}^{\phi})^{\phi}_+,
H_2(X,X\!-\!\wch{X}^{\phi})^{\phi}_{\pm}=\{0\}.\EE
The first homomorphism in~\eref{aveG_e0} is thus surjective,
and its kernel is generated by the homology class of
a unit sphere $S(\cN_p\wch{X}^{\phi})$ in the fiber of~$\cN\wch{X}^{\phi}$ over any 
$p\!\in\!\wch{X}^{\phi}$.

An OSpin-structure $\os\!\equiv\!(\fo,\fs)$ on~$\wch{X}^{\phi}$ includes an orientation~$\fo$
on~$\wch{X}^{\phi}$.
Along with the symplectic orientation~$\fo_{\om}$ of~$X$, 
$\fo$ thus determines an orientation~$\fo_{\cN\wch{X}^{\phi}}$ 
of~$\cN\wch{X}^{\phi}$ via
the exact sequence
$$ 0\lra T\wch{X}^{\phi}\lra TX\big|_{\wch{X}^{\phi}}\lra \cN\wch{X}^{\phi}\lra 0.$$
Along with the orientation of the normal bundle $\cN(S(\cN_p\wch{X}^{\phi}))$ 
of $S(\cN_p\wch{X}^{\phi})$ in $\cN_p\wch{X}^{\phi}$ by the outward radial direction,
$\fo_{\cN\wch{X}^{\phi}}$ determines an orientation $\fo_{S(\cN_p\wch{X}^{\phi})}$ of
$S(\cN_p\wch{X}^{\phi})$ via the exact sequence
$$ 0\lra T\big(S(\cN_p\wch{X}^{\phi})\big)
\lra T(\cN_p\wch{X}^{\phi})\big|_{S(\cN_p\wch{X}^{\phi})}
\lra \cN\big(S(\cN_p\wch{X}^{\phi})\!\big)\lra0 \,.$$

\begin{prp}\label{SNXphi_prp}
Suppose $(X,\om,\phi)$ is a compact real symplectic sixfold, 
$\os$ is an OSpin-structure on 
a connected component~$\wch{X}^{\phi}$ of~$X^{\phi}$,
$B\!\in\!H_2(X)\!-\!\{0\}$, and $\mu_1,\ldots,\mu_l$ 
are elements of $H^2(X)$, $H^4(X,\wch{X}^{\phi})$, and $H^6(X)$.
If 
$$k\equiv \frac12\Big(\ell_{\om}(B)\!+\!2l-\sum_{i=1}^l\deg\mu_i\Big)-1 $$
and $B$ satisfy~\eref{noS2bub_e}, then 
\BE{SNXphi_e}\blr{\mu_1,\ldots,\mu_l,
\PD_{X,\wch{X}^{\phi}}\big([S(\cN_p\wch{X}^{\phi})]_{X-\wch{X}^{\phi}}
\big)}_{B;\wch{X}^{\phi}}^{\phi,\os}
=2\blr{\mu_1,\ldots,\mu_l}_{B;\wch{X}^{\phi}}^{\phi,\os}\,.\EE
\end{prp}

The motivation behind~\eref{SNXphi_e} is that a $J$-holomorphic curve passing 
though a point $p\!\in\!X^{\phi}$
intersects an infinitesimal sphere $S(\cN_p\wch X^\phi)$ at two points.
By its proof in Section~\ref{orient_subs},  
\eref{SNXphi_e} holds without the restriction~\eref{noS2bub_e} 
if the left-hand side is replaced by the real genus~0 GW-invariant of~\cite{Teh}
enumerating $(\phi,\tau)$- {\it and} $(\phi,\eta)$-real curves in~$(X,\om)$,
provided $(X,\phi)$ satisfies suitable topological conditions
so that this invariant is defined.
The projective space $(\P^3,\tau_3)$ satisfies such conditions.

An element $\psi$ of an averager $G$ for $(X,\om,\phi;\wch{X}^{\phi})$ 
as in Definition~\ref{aveG_dfn} reverses an orientation of~$\cN\wch{X}^{\phi}$
and thus does not fix any nonzero element in the kernel of the first map in~\eref{aveG_e0}.
Along with the paragraph containing~\eref{Hpvan_e}, this implies that 
the first map in~\eref{aveG_e1} is an isomorphism.
Since the two maps in~\eref{aveG_e1} are related by the Poincare Duality 
isomorphisms~\eref{PDisomdfn_e},
the second map in~\eref{aveG_e1} is also an isomorphism.

\subsection{Lower bounds for real curve counts}
\label{LowBnd_subs}

As only some elements of $H_2(X\!-\!X^{\phi})$ can be represented by holomorphic curves
in a real projective manifold~$(X,\om,\phi)$, 
Theorem~\ref{WDVVdim3_thm} and Proposition~\ref{SNXphi_prp} lead to lower bounds 
for counts of real algebraic curves in some real algebraic threefolds
through curve constraints.
This is explained below.

If $H\!\subset\!X$ is a (pseudo)cycle transverse to~$\wch{X}^{\phi}$, then 
\BE{wchHphidfn_e}\wch{H}^{\phi}\equiv H\!\cap\!\wch{X}^{\phi}=\phi(H)\!\cap\!\wch{X}^{\phi}\EE
inherits an orientation $\wch\fo_H^{\phi}$ from an orientation $\fo_{\cN\wch{X}^{\phi}}$ 
of~$\cN\wch{X}^{\phi}$  and the orientation~$\fo_H$ of~$H$ via the exact sequence 
$$0\lra T\wch{H}^{\phi}\lra TH|_{\wch{H}^{\phi}}
\lra \cN\wch{X}^{\phi}|_{\wch{H}^{\phi}}\lra0. $$
Since $\phi$ reverses $\fo_{\cN\wch{X}^{\phi}}$, 
the orientation $\wch\fo_{\phi(H)}^{\phi}$ 
of the intersection~\eref{wchHphidfn_e} inherited from the orientation 
\hbox{$\fo_{\phi(H)}\!\equiv\!\phi_*(\fo_H)$} of~$\phi(H)$
is the opposite of~$\wch\fo_H^{\phi}$.
It follows that the boundaries
$$\prt\big(H\!-\!B(\cN_H\wch{H}^{\phi})\big),\prt\big(H\!-\!B(\cN_{\phi(H)}\wch{H}^{\phi})\big)
\subset \cN\wch{X}^{\phi}\big|_{\wch{H}^{\phi}}$$
of the complements of small tubular neighborhoods $B(\cN_H\wch{H}^{\phi})$ of $\wch{H}^{\phi}$ 
in~$H$
and $B(\cN_{\phi(H)}\wch{H}^{\phi})$ of $\wch{H}^{\phi}$ in~$\phi(H)$ 
inherit opposite orientations
from~$\fo_H$ and $\fo_{\phi(H)}$, respectively.
We can thus glue the two complements along their boundaries to form 
a (pseudo)cycle 
$$(H,\fo_H)\!\#\!\big(\phi(H),\fo_{\phi(H)}\big)\subset X\!-\!\wch{X}^{\phi}. $$
The homology homomorphism induced by the inclusion of $X\!-\!\wch{X}^{\phi}$ into $X$ sends
the element of $H_*(X\!-\!\wch{X}^{\phi})$ represented by this (pseudo)cycle
to $[H]_X\!+\!\phi_*[H]_X$.

Suppose $H_1,H_2\!\subset\!X$ are (pseudo)cycles of dimension~4
so that $H_2$ is transverse to~$\wch{X}^{\phi}$ and $H_1$ is transverse to 
$H_2,\phi(H_2),\wch{H}_2^{\phi}$.
In particular,
$$H_1\!\cap\!H_2,H_1\!\cap\!\phi(H_2)\subset X\!-\!\wch{X}^{\phi}\,.$$
By the previous paragraph, the homomorphism~\eref{Hcapdfn_e} is described~by
\BE{Hcapdfn_e2}[H_1]_X\!\cap\!\big([H_2]_X\!+\!\phi_*([H_2]_X)\big)=
\big[H_1\!\cap\!H_2\big]_{X-\wch{X}^{\phi}}\!+\!
\big[H_1\!\cap\!\phi(H_2)\big]_{X-\wch{X}^{\phi}}\,.\EE

Suppose that  $H_1,H_2,H_2'\!\subset\!X$ are (pseudo)cycles of dimension~4
and $\Ups$ is a cobordism between~$H_2$ and~$H_2'$ so that 
$H_2,H_2',\Ups$ are transverse to~$\wch{X}^{\phi}$ and 
$H_1$ is transverse to $H_2,H_2',\wch{H}_2^{\phi},\wch{H}_2'^{\phi},\wch\Ups^{\phi}$.
Thus, $\wch{H}_1^{\phi}\!\cap\!\wch\Ups^{\phi}$ is a finite set of signed points and 
$$\prt\big(H_1\!\cap\!\Ups\!-\!B(\cN_{H_1\cap\Ups}(\wch{H}_1^{\phi}\!\cap\!\wch\Ups^{\phi})\big)
=H_1\!\cap\!H_2'\!-\!H_1\!\cap\!H_2\!+\!
S(\cN\wch{X}^{\phi})\big|_{\wch{H}_1^{\phi}\!\cap\!\wch\Ups^{\phi}}$$
for a small tubular neighborhood $B(\cN_{H_1\cap\Ups}(\wch{H}_1^{\phi}\!\cap\!\wch\Ups^{\phi}))$
of $\wch{H}_1^{\phi}\!\cap\!\wch\Ups^{\phi}$ in $H_1\cap\Ups$;
the equality above respects the orientations for suitable conventions for orienting 
the intersections of cycles.
Thus, 
\BE{capdiff_e}\big[H_1\!\cap\!H_2\big]_{X-\wch{X}^{\phi}}=
\big[H_1\!\cap\!H_2'\big]_{X-\wch{X}^{\phi}}\!+\!
\Lk_{\fo}(\wch{H}_1^{\phi},\wch{H}_2'^{\phi}\!-\!\wch{H}_2^{\phi})
\big[S(\cN_p\wch{X}^{\phi})\big]_{X-\wch{X}^{\phi}} 
\in H_2(X\!-\!\wch{X}^{\phi}),\EE
where $\Lk_{\fo}$ is the linking number with respect to the orientation~$\fo$ in~$\wch{X}^{\phi}$.

If $H_1,H_2\!\subset\!\P^3$ are generic complex hyperplanes, \eref{capdiff_e} gives
$$\big[H_1\!\cap\!H_2\big]_{\P^3-\R\P^3}=
\big[H_1\!\cap\!\tau_3(H_2)\big]_{\P^3-\R\P^3}
\pm \big[S(\cN_p\R\P^3)\big]_{\P^3-\R\P^3}  \in H_2(\P^3\!-\!\R\P^3)\,.$$
By~\eref{Hcapdfn_e2}, 
the sum of the first two homology classes above does not depend on the choices of~$H_1$ and~$H_2$.
Thus, only two classes, $\ell_-$ and $\ell_+$, in $H_2(\P^3\!-\!\R\P^3)$ can be 
represented by complex lines and 
$$\ell_+=\ell_-\!+\!\big[S(\cN_p\R\P^3)\big]_{\P^3-\R\P^3}  \in H_2(\P^3\!-\!\R\P^3).$$
The image of~\eref{Hcapdfn_e} in this case is generated by the averaged line class
$$\wt\ell\equiv\frac12\big([\ell_-]\!+\![\ell_+]\big)\in H_2(\P^3\!-\!\R\P^3).$$

\begin{table}[!htbp]
\begin{center}
\begin{tabular}{||c|c|c|c|c|c||}
\hline\hline
$d$& cond& $\lr{\wt\ell^a\pt^b}_d^{\tau_3,\os}$& 
$\lr{\ell_-^{a-i}\ell_+^i\pt^b}_d^{\tau_3,\os}$& min& $\C$GW\\
\hline
1& $\ell^0\pt^0$& 1& 1& 1& 1\\
\hline
1& $\ell^0\pt^1$& $-1$& $-1$& 1& 1\\
\hline
1& $\ell^1\pt^0$& 0& $-1,1$& 1& 1\\
\hline
1& $\ell^2\pt^0$& $-1$& $0,-2,0$& 0& 2\\
\hline
2& $\ell^0\pt^0$& 0& 0& 0& 0\\
\hline
2& $\ell^0\pt^1$& 0& 0& 0& 0\\
\hline
2& $\ell^0\pt^2$& 0& 0& 0& 0\\
\hline
2& $\ell^1\pt^0$& 1& 1,1& 1& 1\\
\hline
2& $\ell^1\pt^1$& $-1$& $-1,-1$& 1& 1\\
\hline
2& $\ell^2\pt^0$& 0& $-2,0,2$& 0& 4\\
\hline
2& $\ell^2\pt^1$& 0& $2,0,-2$& 0& 4\\
\hline
2& $\ell^3\pt^0$& $-3$& $0,-4,-4,0$& 0& 18\\
\hline
2& $\ell^4\pt^0$& 0& $8,8,0,-8,-8$& 0& 92\\
\hline
3& $\ell^0\pt^0$& $-1$& $-1$& 1& 1\\
\hline
3& $\ell^0\pt^1$& 1& 1& 1& 1\\
\hline
3& $\ell^0\pt^2$& $-1$& $-1$& 1& 1\\
\hline
3& $\ell^0\pt^3$& 1& 1& 1& 1\\
\hline
3& $\ell^1\pt^0$& 0& $1,-1$& 1& 5\\
\hline
3& $\ell^1\pt^1$& 0& $-1,1$& 1& 5\\
\hline
3& $\ell^2\pt^0$& 0& $1,-1$& 1& 5\\
\hline
3& $\ell^2\pt^0$& 5& $4,6,4$& {\bf 4}& 30\\
\hline
3& $\ell^2\pt^1$& $-3$& $-2,-4,-2$& {\bf 2}& 30\\
\hline
3& $\ell^2\pt^2$& 1& $0,2,0$& 0& 30\\
\hline
3& $\ell^3\pt^0$& 0& $-14,-6,6,14$& {\bf 6}& 190\\
\hline
3& $\ell^3\pt^1$& 0& $8,4,-4,-8$& {\bf 4}& 190\\
\hline
3& $\ell^4\pt^0$& $-13$& $16,-12,-24,-12,16$& {\bf 12}& 1312\\
\hline
3& $\ell^4\pt^1$& 1& $-16,0,8,0,-16$& 0& 1312\\
\hline
3& $\ell^5\pt^0$& 0& $16,48,24,-24,-48,-16$& {\bf 16}& 9864\\
\hline
3& $\ell^6\pt^0$& $-7$& $-128,-96,0,48,0,-96,-128$& 0& 80160\\
\hline\hline
\end{tabular}
\end{center}
\caption{The invariant count $\lr{\wt\ell^a\pt^b}_d^{\tau_3,\os}$ 
of $\tau_3$-real degree~$d$ rational curves in~$\P^3$ determined
by an $\OSpin$-structure $\os$ on~$\R\P^3$ 
through $a$ conjugate pairs of averaged lines $\wt\ell\!\equiv\!(\ell_-\!+\!\ell_+)/2$,
$b$ conjugate pairs of points in $\P^3\!-\!\R\P^3$,
and $2d\!-\!a\!-\!2b$ points in~$\R\P^3$,
the analogous counts with $a$ averaged lines replaced by $i$ lines~$\ell_-$
and $a\!-\!i$ lines~$\ell_+$,
the minimum of the absolute values of the latter counts,
and the associated count of complex curves.
The new lower bounds for real rational curves are in boldface.}
\label{P3nums_tbl}
\end{table}

Theorem~\ref{WDVVdim3_thm} applied to $(\P^3,\tau_3)$, 
an $\OSpin$-structure~$\os$ on~$\R\P^3$, and the subgroup $G$ of $\Aut(\P^3,\tau_3)$
generated by a real hyperplane reflection~$\psi_3$ as in~\eref{psi3dfn_3} 
 determines all open GW-invariants 
\BE{P3numsdfn_e}\blr{\wt\ell^a\pt^b}_d^{\tau_3,\os}\equiv
\blr{\underset{a}{\underbrace{\PD_{\P^3-\R\P^3}(\ell),\ldots,\PD_{\P^3-\R\P^3}(\ell)}},
\underset{b}{\underbrace{\PD_{\P^3}(\pt),\ldots,\PD_{\P^3}(\pt)}}}_{d\ell;\R\P^3;G}^{\tau_3,\os}\EE
enumerating real degree~$d$ holomorphic curves in~$\P^3$ that 
pass through generic representatives for $a$ averaged lines~$\wt\ell$,
$b$ general points in $\C\P^3\!-\!\R\P^3$, and $2d\!-\!a\!-\!2b$ general points in~$\R\P^3$
from the single number $\lr{\wt\ell^0\pt^0}_1^{\tau_3,\os}\!=\!\pm1$
(the sign depends on~$\os$).
These numbers in degrees 1-8 are shown in \cite[Table~4.2.2]{Adam};
the degree 1-3 numbers are reproduced in the third column of our Table~\ref{P3nums_tbl}.
Inline with G.~Mikhalkin's observation in the early 2000s and 
the first case of Proposition~\ref{GtauGW_prp}\ref{GtauGWvan_it},
the numbers~\eref{P3numsdfn_e} with $d\!+\!a$ even vanish.
The odd-degree $a\!=\!0$ numbers agree with \cite[Table~1]{BG} up to~sign.
Proposition~\ref{SNXphi_prp} then yields open GW-invariants of $(\P^3,\tau_3)$ with arbitrary
insertions in~$H_2(\P^3\!-\!\R\P^3)$.
The fourth column in Table~\ref{P3nums_tbl} shows all degree 1-3 numbers with
the insertions~$\ell_-$ and~$\ell_+$.
Taking the minimum of the absolute values of the numbers in each cell in
this column, we obtain lower bounds for the numbers of real rational curves
passing through generic complex lines in $\P^3\!-\!\R\P^3$, points in $\P^3\!-\!\R\P^3$,
and points in~$\R\P^3$.

\section{Topological preliminaries}
\label{TopolPrelim_sec}

For a real vector space or vector bundle~$V$, let
$\la(V)\!\equiv\!\La_{\R}^{\top}V$ be its top exterior power.
For a manifold $M$, possibly with nonempty boundary~$\prt M$, we denote by 
$$\la(M)\equiv\la(TM)\equiv\La^{\top}_{\R}TM\lra M$$
its \sf{orientation line bundle}.
An \sf{orientation} of~$M$ is a homotopy class of trivializations of~$\la(M)$.  
We identify the two orientations of any point with $\pm1$ in the obvious way.

For submanifolds $S'\!\subset\!S\!\subset\!M$, the short exact sequences
\begin{gather*}
0\lra TS\lra TM|_S\lra \cN S\!\equiv\!\frac{TM|_S}{TS}\lra 0  \qquad\hbox{and}\\
0\lra \cN_SS'\!\equiv\!\frac{TS|_{S'}}{TS'}\lra \cN S'\!\equiv\!\frac{TM|_{S'}}{TS'}
\lra \cN S|_{S'}\!\equiv\!\frac{TM|_{S'}}{TS|_{S'}} \lra0 
\end{gather*}
of vector spaces determine isomorphisms
\BE{lasplits_e} \la(M)\big|_S\approx \la(S)\!\otimes\!\la(\cN S)
\quad\hbox{and}\quad
\la(\cN S')\approx \la(\cN_SS')\!\otimes\!\la(\cN S)\big|_{S'}\EE
of line bundles over~$S$ and~$S'$, respectively.
A \sf{co-orientation} of $S$ in $M$ is an orientation of~$\cN S$.
We define the canonical co-orientation~$\fo_{\prt M}^c$ of~$\prt M$ in~$M$
to be given by the outer normal direction. 

If $\fo$ is an orientation of $M$ and $\fo^c_S$ is a co-orientation of~$S$ in~$M$, 
we denote by~$\fo^c_S\fo$ 
the orientation of $S$ induced by $\fo^c_S$ and $\fo$ via the first isomorphism in~\eref{lasplits_e}. 
If $M$ is a manifold with boundary, let 
\BE{prtcoorient_e}\prt\big(M,\fo\big)\equiv \big(\prt M,\prt\fo\big)
\equiv \big(\prt M,\fo_{\prt M}^c\fo\big)\,.\EE
If $\cS'\!\subset\!\cS$ is also a submanifold with a co-orientation $\fo_{\cS'}^c$ in~$\cS$,
then the co-orientations $\fo_{\cS}^c$ and  $\fo_{\cS'}^c$
induce a co-orientation $\fo_{\cS'}^c\fo_{\cS}^c$  of~$\cS'$ in~$\cZ$
via the second isomorphism in~\eref{lasplits_e}.
If $\cS$ has boundary, let
$$\prt\big(\cS,\fo_{\cS}^c\big)\equiv \big(\prt\cS,\prt\fo_{\cS}^c\big)
\equiv \big(\prt M,\fo_{\prt\cS}^c\fo_{\cS}^c\big).$$

For a fiber bundle $\ff_{\cM}\!:\cM\!\lra\!\cM'$, we denote by 
$T\cM^v\!\equiv\!\ker \nd\ff_{\cM}$ 
its vertical tangent bundle.
The short exact sequence
\BE{finses_e}0\lra T\cM^v\lra 
T\cM \xlra{\nd\ff_{\cM}} \ff_{\cM}^*T\cM'\lra0\EE
of vector bundles determines an isomorphism 
\BE{lasplits_e2} \la(\cM)\approx \ff_{\cM}^*\la(\cM')\!\otimes\!\la(T\cM^v)\EE
of line bundles over~$\cM$.
The switch of the ordering of the factors in~\eref{lasplits_e2} from~\eref{finses_e}
is motivated by \cite[Lemma~3.1(1)]{RealWDVV} and by the inductive construction
of the orientations~$\fo_{k,l}$ on the real Deligne-Mumford moduli spaces $\ov\cM_{k,l}^{\tau}$
in Section~\ref{cMstrata_subs}.
If $\fo'$ is an orientation of~$\cM'$ and $\fo_{\cM}^v$ is an orientation of~$T\cM^v$, 
we denote by $\fo_{\cM}^v\fo'$ 
the orientation of~$\cM$ induced by $\fo_{\cM}^v$ and~$\fo$ via~\eref{lasplits_e2}. 

Suppose $f\!:\cZ\!\lra\!\cM$ is a smooth map transverse to a submanifold $\Ups\!\subset\!\cM$.
The differential of~$f$  then induces an isomorphism from the
normal bundle $\cN(f^{-1}(\Ups))$ of the  submanifold
\hbox{$f^{-1}(\Ups)\!\subset\!\cZ$} to the normal bundle~$\cN\Ups$ of~$\Ups$.
The differential of~$f$ thus pulls back a co-orientation~$\fo_{\Ups}^c$ 
of~$\Ups$ in~$\cM$ to a co-orientation $\fo_{f^{-1}(\Ups)}^c\!\equiv\!f^*\fo_{\Ups}^c$
of~$f^{-1}(\Ups)$ in~$\cZ$.
The next observation is straightforward; see also the first diagram in Figure~\ref{degvsinter_fig}.

\begin{lmm}\label{cutfibr_lmm}
Suppose $\ff_{\cZ}\!:\cZ\!\lra\!\cZ'$ is a fiber bundle,
$f'\!:\cZ'\!\lra\!\cM$ is a smooth map transverse to a submanifold $\Ups\!\subset\!\cM$,
$\fo_{\cZ'}$ and $\fo_{\cZ}^v$ are orientations of~$\cZ'$ and $T\cZ^v$, respectively,
and $\fo_{\Ups}^c$ is a co-orientation of~$\Ups$ in~$\cM$.
The orientations $\{f'\!\circ\!\ff_{\cZ}\}^*\fo_{\Ups}^c(\fo_{\cZ}^v\fo_{\cZ'})$ and
$\fo_{\cZ}^v(f'^*\fo_{\Ups}^c\fo_{\cZ'})$ of $\{f'\!\circ\!\ff_{\cZ}\}^{-1}(\Ups)$ 
at $u\!\in\!\{f'\!\circ\!\ff_{\cZ}\}^{-1}(\Ups)$ are the same if and only if
$(\rk\,T\cZ^v)(\codim\,\Ups)$ is~even.
\end{lmm}

If $\fo_{\cZ},\fo_Y$ are orientations of smooth manifolds $\cZ$ and $Y$, respectively, 
$f\!:\cZ\!\lra\!Y$ is a smooth map, and $u\!\in\!\cZ$ is such that $\nd_uf$ is an isomorphism, 
we define
$$\fs_u(f,\fo_{\cZ};\fo_Y)=
\begin{cases}+1,&\hbox{if}~\{\nd_uf\}^*((\fo_Y)_{u(z)})\!=\!(\fo_{\cZ})_u;\\
-1,&\hbox{otherwise}.\end{cases}$$
If $y\!\in\!Y$ is a regular value of~$f$ and the set $f^{-1}(y)$ is finite, let
$$\big|f^{-1}(y)\big|^\pm_{\fo_{\cZ},\fo_Y}\equiv
\sum_{u\in f^{-1}(y)}\!\!\!\!\!\!\fs_u(f,\fo_{\cZ};\fo_Y) \,.$$
We abbreviate $\fs_u(f,\fo_{\cZ};\fo_Y)$ and $|f^{-1}(y)|^\pm_{\fo_{\cZ},\fo_Y}$
as $\fs_u(f,\fo_{\cZ})$ and $|f^{-1}(y)|^\pm_{\fo_{\cZ}}$, respectively,
whenever the orientation of~$Y$ is understood from the context.
If $\fo$ is an orientation~of
$$ \la(f)\equiv f^*\la(Y)^*\!\otimes\!\la(\cZ)\lra \cZ,$$ 
we define $\fs_u(f,\fo)$ as $\fs_u(f,\fo_{\cZ};\fo_Y)$ for some orientations 
$\fo_{\cZ}$ of~$T_u\cZ$ and $\fo_Y$ of~$T_{f(u)}Y$ inducing the orientation~$\fo$ of 
the fiber~$\la_u(f)$ of~$\la(f)$ at~$u$.
If $y\!\in\!Y$ is a regular value of~$f$ and the set $f^{-1}(y)$ is finite, let
$$\big|f^{-1}(y)\big|^\pm_{\fo}\equiv\sum_{u\in f^{-1}(y)}\!\!\!\!\!\!\fs_u(f,\fo) \,.$$

Let $\ff_{\cM}\!:\cM\!\lra\!\cM'$ be a fiber bundle. 
If $\Ups\!\subset\!\cM$ is a submanifold and $P\!\in\!\Ups$, then
the differential $\nd_P(\ff_{\cM}|_\Ups)$ is an isomorphism if and only~if
the composition
\BE{fibisom_e} T_P\cM^v\!\equiv\!\ker\nd_P\ff_{\cM}\lra 
T_P\cM\lra \frac{T_P\cM}{T_P\Ups}\!\equiv\!\cN_P\Ups\EE
is.
If $\fo^c_\Ups$ is a co-orientation of~$\Ups$ and 
$\fo^v_\cM$ is an orientation of~$T\cM^v$, we denote by 
\hbox{$\fs_P(\fo^c_\Ups\fo^v_\cM)\!\in\!\{\pm1\}$} 
the sign of \eref{fibisom_e} with respect to $\fo^c_\Ups$ and~$\fo^v_\cM$. 
By \cite[Lemma~3.1(1)]{RealWDVV}, $\fs_P(\fo^c_\Ups\fo^v_\cM)$ is 
the sign $\fs_P(\ff_{\cM}|_{\Ups},\fo^c_\Ups\fo^v_\cM)$ of~$\ff_{\cM}$ at~$P$
with respect to the orientation~$\fo^c_\Ups\fo^v_\cM$ of 
$$\la\big(\ff_{\cM}|_{\Ups}\big)\approx \la\big(T\cM^v\big)\!\otimes\!\la(\cN\Ups)^*$$
induced by $\fo^c_\Ups$ and $\fo^v_\cM$ via the first isomorphism in~\eref{lasplits_e}
and~\eref{lasplits_e2}.
If in addition \hbox{$f'\!:\cM'\!\lra\!Y$} is a smooth map
such that $\nd_{\ff_{\cM}(P)}f'$ is an isomorphism, $\fo'$ is an orientation of~$\cM'$,
and $\fo_Y$ is an orientation of~$Y$, then 
\BE{fsfoprod_e}
\fs_P\big(f'\!\circ\!\ff_{\cM}|_{\Ups},\fo^c_\Ups(\fo^v_\cM\fo');\fo_Y\big)
=\fs_P\big(\fo^c_\Ups\fo^v_\cM\big)\fs_{\ff_{\cM}(P)}\big(f',\fo';\fo_Y\big)\,.\EE

Suppose that $\ff_{\cZ}\!:\cZ\!\lra\!\cZ'$ is another fiber bundle,
the diagram
$$\xymatrix{\cZ\ar[rr]^f \ar[d]_{\ff_{\cZ}}&&  \cM\ar[d]^{\ff_{\cM}}\\
\cZ'\ar[rr]^{f'}&& \cM' }$$
of smooth maps commutes, and $\fo_{\cZ}^v$ and $\fo_{\cM}^v$ 
are orientations on~$T\cZ^v$ and~$T\cM^v$, respectively.
If $u\!\in\!\cZ$ is such that the restriction
\BE{fibisom_e2a}\nd_uf\!:T_u\cZ^v\!\equiv\!\ker\nd_u\ff_{\cZ} \lra T_{f(u)}\cM^v\EE
is an isomorphism, we define $\fs_u(f,\fo_{\cZ}^v,\fo_{\cM}^v)$ to be $+1$ 
if this isomorphism is orientation-preserving with respect to the orientations~$\fo_{\cZ}^v$
and~$\fo_{\cM}^v$ and to be $-1$ otherwise.

For continuous maps $f\!:\cZ\!\lra\!Y$ and $g\!:\Ups\!\lra\!Y$ between manifolds with boundary, define 
\begin{gather*}
M_{f,g}\equiv \cZ_f\!\!\times_g\!\Ups
=\big\{(u,P)\!\in\!\cZ\!\!\times\!\!\Ups\!-\!(\prt\cZ)\!\!\times\!\!(\prt \Ups)\!: 
f(u)\!=\!g(P)\big\}, \\
f\!\times_Y\!g\!:M_{f,g}\lra Y,\qquad f\!\times_Y\!g(u,P)=f(u)\!=\!g(P).
\end{gather*}
We call two such maps $f$ and $g$ \sf{strongly transverse} if  they are smooth and
the maps~$f$ and~$f|_{\prt\cZ}$ are transverse to the maps~$g$ and~$g|_{\prt \Ups}$.
The space $M_{f,g}$ is then a smooth manifold and 
\begin{gather}
\notag \dim\, M_{f,g}+\dim\,Y=\dim\,\cZ+\dim\,\Ups\,,\\
\label{prtMf1f2_e}
\prt M_{f,g}=\big(\cZ\!-\!\prt\cZ\big)\,_f\!\!\times_g\!(\prt \Ups)\sqcup 
(\prt \cZ)\,_f\!\!\times_g\!\big(\Ups\!-\!\prt \Ups\big)\,.
\end{gather}
If $\fo_\cZ$, $\fo_\Ups$, and $\fo_Y$ are orientations of $\cZ$, $\Ups$, and $Y$, respectively, 
and $(u,P)\!\in\!M_{f,g}$ is such~that the homomorphism
\BE{isominter_e}T_u\cZ\!\oplus\!T_P\Ups\lra T_{f(u)}Y\!=\!T_{g(P)}Y, \quad
(v,w)\lra \nd_uf(v)\!+\!\nd_Pg(w),\EE
is an isomorphism, we define $\fs_{u,P}(f,\fo_\cZ,g,\fo_\Ups;\fo_Y)$ to be $+1$ 
if this isomorphism is orientation-preserving with respect to 
$\fo_\cZ\!\oplus\!\fo_\Ups$ and $\fo_Y$ and to be $-1$ otherwise. 
If $f$ and $g$ are transverse and the set $M_{f,g}$ is finite, let
$$\big|M_{f,g}\big|_{\fo_\cZ,\fo_\Ups;\fo_Y}^{\pm}
=\sum_{(u,P)\in M_{f,g}}\!\!\!\!\!\!\!\!\fs_{u,P}(f,\fo_\cZ,g,\fo_\Ups;\fo_Y)\,.$$
We abbreviate $\fs_{u,P}(f,\fo_\cZ,g,\fo_\Ups;\fo_Y)$ and $|M_{f,g}|_{\fo_\cZ,\fo_\Ups;\fo_Y}^{\pm}$
as $\fs_{u,P}(f,\fo_\cZ,g,\fo_\Ups)$ and $|M_{f,g}|_{\fo_\cZ,\fo_\Ups}^{\pm}$, respectively,
whenever the orientation of~$Y$ is understood from the context.

Suppose that $\cZ$, $X$, $\cM$ are smooth manifolds, $\Ups\!\subset\!\cM$ is a submanifold, and 
$$ f\!=\!(f_1,f_2)\!:\cZ\lra Y\!\equiv\!X\!\times\!\cM \qquad\hbox{and}\qquad 
g\!=\!(g_1,g_2)\!:\Ups\lra X\!\times\!\cM$$
are smooth maps so that $g_2$ is the inclusion.
Let $\fo_\cZ$ and $\fo_X$ be orientations on~$\cZ$ and~$X$, respectively,
and $\fo^c_\Ups$ be a co-orientation of~$\Ups$ in~$\cM$.
For $(u,P)\!\in\!M_{f,g}$ such~that the homomorphism~\eref{isominter_e} is an isomorphism, 
we define $\fs_{u,P}(f,\fo_\cZ,g,\fo_\Ups^c;\fo_X)$ to be $+1$ 
if  the top exterior power~$\La^{\top}_{\R}$ of this isomorphism lies in the homotopy class
of isomorphisms
\begin{equation*}\begin{split}
\La^{\top}_{\R}\big(T_u\cZ\!\oplus\!T_P\Ups\big)
&\approx \La^{\top}_{\R}\big(T_u\cZ\big)\!\otimes\!\La^{\top}_{\R}\big(T_P\Ups\big)\\
&\approx \La^{\top}_{\R}\big(T_{f_1(u)}X\big)\!\otimes\!\La^{\top}_{\R}\big(T_P\cM\big)
\approx \La^{\top}_{\R}\big(T_{f_1(u)}X\!\oplus\!T_P\cM\big)
\end{split}\end{equation*}
determined by $(\fo_{\cZ})_u$, $(\fo_X)_{f_1(u)}$, and~$(\fo_{\Ups}^c)_P$ 
and to be $-1$ otherwise.
If $f$ and $g$ are transverse and the set $M_{f,g}$ is finite, let
$$\big|M_{f,g}\big|_{\fo_\cZ,\fo^c_\Ups;\fo_X}^{\pm}
\equiv\sum_{(u,P)\in M_{f,g}}\!\!\!\!\!\!\!\!\fs_{u,P}(f,\fo_\cZ,g,\fo_\Ups^c;\fo_X)\,.$$
Similarly to the above, we drop the orientation $\fo_X$ of~$X$ from the just introduced
notation if it is understood from the context.

Suppose $e_1\!:\cZ_1\!\lra\!X'$ and $e_2\!:\cZ_2\!\lra\!X'$ are strongly transverse maps 
from manifolds with boundary.
Thus,
$$\cZ\equiv M_{e_1,e_2}\equiv\big\{
(u_1,u_2)\!\in\!\cZ_1\!\times\!\cZ_2\!-\!(\prt\cZ_1)\!\times\!(\prt \cZ_2)\!: 
e_1(u_1)\!=\!e_2(u_2)\big\} \subset \cZ_1\!\times\!\cZ_2$$
is a smooth submanifold.
For each $u\!\equiv\!(u_1,u_2)\!\in\!\cZ$, the short exact sequence
\begin{gather*}
0\lra T_u\cZ\lra T_{u_1}\cZ_1\!\oplus\!T_{u_2}\cZ_2\lra T_{e_1(u_1)}X'\!=\!T_{e_2(u_2)}X'\lra0,\\
(v_1,v_2)\lra \nd_{u_2}e_2(v_2)\!-\!\nd_{u_1}e_1(v_1),
\end{gather*}
of vector spaces induces an isomorphism
$$\la_u(\cZ)\!\otimes\!\la\big(T_{e_2(u_2)}X'\big)\approx 
\la_{u_1}(\cZ_1)\!\otimes\!\la_{u_2}(\cZ_2)\,.$$
Orientations $\fo_1$, $\fo_2$, and $\fo'$ of $\cZ_1$, $\cZ_2$, and $X'$, respectively,
determine an orientation $(\!(\fo_1)_{e_1}\!\cdot\!_{e_2}\!(\fo_2)\!)_{\fo'}$ of~$\cZ$
via these isomorphisms.
We abbreviate this orientation as $(\fo_1)_{e_1}\!\cdot\!_{e_2}\!(\fo_2)$ whenever
the orientation of~$X'$ is implied by the context.

Suppose in addition that $f_1\!:\cZ_1\!\lra\!X_1$ and $f_2\!:\cZ_2\!\lra\!X_2$ are smooth maps;
see the second diagram of Figure~\ref{degvsinter_fig}.
For all $p_1\!\in\!X_1$ and $p_2\!\in\!X_2$,
\BE{degvsinter_e0}\big\{(f_1,f_2)|_{M_{e_1,e_2}}\big\}^{-1}(p_1,p_2)
=M_{e_1|_{f_1^{-1}(p_1)},e_2|_{f_2^{-1}(p_2)}}\,.\EE
The next observation is straightforward.

\begin{figure}
$$\xymatrix{  &&&
f_1^{-1}(p_1)\ar[d]\ar@{^(->}[rr] &&\cZ_1 \ar[rr]^{f_1} \ar[lld]|{e_1}&& X_1\\
\cZ \ar[d]_{\ff_{\cZ}}\ar[rd]^{f'}& &&
X' &&\cZ\ar[rr]^f \ar[u]\ar[d]&& X_1\!\times\!X_2\ar[d]\ar[u] \\
\cZ \ar[r]^f & \cM&&
f_2^{-1}(p_2)\ar@{^(->}[rr] \ar[u]&& \cZ_2 \ar[rr]^{f_2}\ar[llu]|{e_2}  && X_2 
}$$
\caption{The maps of Lemmas~\ref{cutfibr_lmm} and~\ref{degvsinter_lmm}.}
\label{degvsinter_fig}
\end{figure}
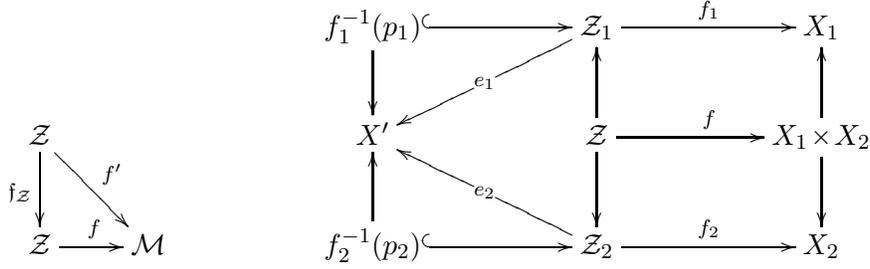

\begin{lmm}\label{degvsinter_lmm}
Suppose $\cZ_1,\cZ_2,X',X_1,X_2$ and $e_1,e_2,f_1,f_2$ are as above
and in the second diagram of Figure~\ref{degvsinter_fig} with
$$\dim\,\cZ_1+\dim\,\cZ_2=\dim\,X'+\dim\,X_1+\dim\,X_2$$
and $\fo_1,\fo_2,\fo',\fo_1',\fo_2'$ are orientations of $\cZ_1,\cZ_2,X',X_1,X_2$, respectively.
If $e_1$ and $e_2$ are strongly transverse, $p_1\!\in\!X_1$ is a regular value of~$f_1$,
$p_2\!\in\!X_2$ is a regular value of~$f_2$, and 
the differential of $(f_1,f_2)|_{M_{e_1,e_2}}$ at some point $(u_1,u_2)$ of 
the space~\eref{degvsinter_e0} is an isomorphism, then 
\begin{equation*}\begin{split}
&\fs_{(u_1,u_2)}\big((f_1,f_2)|_{M_{e_1,e_2}},(\!(\fo_1)_{e_1}\!\cdot\!_{e_2}\!(\fo_2)\!)_{\fo'};
\fo_1'\!\oplus\!\fo_2'\big)\\
&\hspace{1in}=(-1)^{\ep}
\fs_{u_1,u_2}\big(e_1|_{f_1^{-1}(p_1)},(f_1^*\fo_1')\fo_1,e_2|_{f_2^{-1}(p_2)},
(f_2^*\fo_2')\fo_2;\fo'\big),
\end{split}\end{equation*}
where $\ep\!=\!(\dim\,X')(\dim\,\cZ_1\!+\!\dim\,X_2)\!+\!(\dim\,\cZ_1)(\dim\,\cZ_2\!+\!\dim\,X_2)$.
\end{lmm}

Let $Y$ be a smooth manifold, possibly with boundary.
For a continuous map $f\!:\cZ\!\lra\!Y$,  let 
$$\Om(f)=\bigcap_{K\subset \cZ\text{~cmpt}}\!\!\!\!\!\!\!\!\ov{f(\cZ\!-\!K)}$$
be \sf{the limit set of~$f$}.
A \sf{$\Z_2$-pseudocycle} into~$Y$ is a continuous map \hbox{$f\!:\cZ\!\lra\!Y$} 
from a manifold, possibly with boundary, so that the closure of $f(\cZ)$ in~$Y$ is compact and
there exists a smooth map $h\!:\cZ'\!\lra\!Y$ such~that 
$$\dim\,\cZ'\le \dim\,\cZ\!-\!2, \qquad \Om(f)\subset h(\cZ'), \qquad 
f(\prt\cZ)\subset(\prt Y)\!\cup\!h(\cZ')\,.$$
The \sf{codimension} of such a $\Z_2$-pseudocycle is \hbox{$\dim\,Y\!-\!\dim\,\cZ$}.
A continuous map \hbox{$\wt{f}\!:\wt\cZ\!\lra\!Y$} is 
\sf{bordered $\Z_2$-pseudocycle with boundary} \hbox{$f\!:\cZ\!\lra\!Y$} if
the closure of $\wt{f}(\wt\cZ)$ in~$Y$ is compact, 
$$\cZ\subset\prt\wt\cZ, \qquad \wt{f}|_{\cZ}=f,$$
and there exists a smooth map $\wt{h}\!:\wt\cZ'\!\lra\!Y$ such~that 
$$\dim\,\wt\cZ'\le \dim\,\wt\cZ\!-\!2,\qquad
\Om(\wt{f})\subset\wt{h}(\wt\cZ'), \qquad 
\wt{f}\big(\prt\wt\cZ\!-\!\cZ\big)\subset(\prt Y)\!\cup\!\wt{h}(\wt\cZ') \,.$$
If $\wt\cZ$ is one-dimensional, then $\wt\cZ$ is compact and 
$\wt{f}(\prt\wt\cZ\!-\!\cZ)\!\subset\!\prt Y$.

Two bordered $\Z_2$-pseudocycles 
\hbox{$\wt{f}_1\!:\wt\cZ_1\!\lra\!Y$} and \hbox{$\wt{f}_2\!:\wt\cZ_2\!\lra\!Y$} 
as above are \sf{transverse}~if 
\BEnum{$\bu$}

\item the maps $\wt{f}_1$ and $\wt{f}_2$ are strongly transverse and

\item there exist smooth maps $\wt{h}_1\!:\wt\cZ_1'\!\lra\!Y$ and 
$\wt{h}_2\!:\wt\cZ_2'\!\lra\!Y$
such that $\wt{h}_1$ is transverse to~$\wt{f}_2$ and~$\wt{f}_2|_{\prt\wt\cZ_2}$, 
$\wt{h}_2$ is transverse to~$\wt{f}_1$ and~$\wt{f}_1|_{\prt\wt\cZ_1}$, and
$$\dim\,\wt\cZ_1'\le\dim\,\wt\cZ_1\!-\!2, \quad 
\dim\,\wt\cZ_2'\le\dim\,\wt\cZ_2\!-\!2, \quad 
\Om(\wt{f}_1)\subset\wt{h}_1(\wt\cZ_1'), \quad 
\Om(\wt{f}_2)\!\subset\!\wt{h}_2(\wt\cZ_2').$$

\EEnum
In such a case, 
$$\wt{f}_1\!\times_Y\!\!\wt{f}_2\!:M_{\wt{f}_1,\wt{f}_2}\lra Y$$ 
is a bordered $\Z_2$-pseudocycle with boundary~\eref{prtMf1f2_e}.

A \sf{Steenrod pseudocycle} into~$Y$ is a $\Z_2$-pseudocycle \hbox{$f\!:\cZ\!\lra\!Y$} 
along with an orientation~$\fo$ of~$\la(f)$.
A \sf{pseudocycle} into~$Y$ is a $\Z_2$-pseudocycle \hbox{$f\!:\cZ\!\lra\!Y$} 
along with an orientation~$\fo$ of~$\cZ$.
A bordered $\Z_2$-pseudocycle \hbox{$\wt{f}\!:\wt\cZ\!\lra\!Y$} with boundary~$f$
and an orientation~$\wt\fo$ of~$\wt\cZ$ is
a \sf{bordered pseudocycle} with boundary~$(f,\fo)$ if $\prt\wt\fo\!=\!\fo$.
If $(f,\fo)$ is a codimension~0 Steenrod pseudocycle, then the number
\BE{SteenDeg_e}\deg(f,\fo)\equiv \sum_{u\in f^{-1}(y)}\!\!\!\!\!\!\fs_u(f,\fo)\in\Z\EE
is well-defined for a generic choice of $y\!\in\!Y$ and is independent of such a choice.
We call~\eref{SteenDeg_e} the \sf{degree} of~$(f,\fo)$.
If $(f,\fo)$ is a codimension~0 pseudocycle and $\fo_Y$ is an orientation of~$Y$, then the~number
$$\deg(f,\fo;\fo_Y)\equiv \sum_{u\in f^{-1}(y)}\!\!\!\!\!\!\fs_u(f,\fo;\fo_Y)\in\Z$$
is well-defined for a generic choice of $y\!\in\!Y$ and is independent of such a choice.
We call this number the \sf{degree} of~$(f,\fo)$ with respect~$\fo_Y$.
If the orientation~$\fo_Y$ is understood from the context, we again drop it
from the notation.

\section{Moduli spaces of stable curves}
\label{DM_sec}

\subsection{Main stratum and orientations}
\label{cMstrata_subs}

For $l\!\in\!\Z^{\ge0}$, let $[l]\!=\!\{1,\ldots,l\}$.
For a finite set~$L$ with $|L|\!\ge\!3$, we denote by $\ov\cM_{0,L}$
the Deligne-Mumford moduli space of stable rational curves with $L$-marked points.
Let $\cM_{0,L}\!\subset\!\ov\cM_{0,L}$ be the \sf{main stratum} of~$\ov\cM_{0,L}$,
i.e.~the subspace parametrizing smooth curves.
For $l\!\in\!\Z^+$ with $l\!\ge\!3$, we write 
$\ov\cM_{0,l}$ and $\cM_{0,l}$ instead of $\ov\cM_{0,[l]}$ and $\cM_{0,[l]}$, respectively.

For finite sets~$K,L$ with $|K|\!+\!2|L|\!\ge\!3$,
we denote by $\ov\cM_{K,L}^{\tau}$ the Deligne-Mumford moduli space of 
stable real genus~0 curves
\BE{cCdfn0_e}\cC\equiv \big(\Si,(x_i)_{i\in K},(z_i^+,z_i^-)_{i\in L},\si\big)\EE
with $K$-marked real points, $L$-marked conjugate pairs of points, and 
an anti-holomorphic involution~$\si$ with separating fixed locus.
This space is a smooth manifold of dimension \hbox{$|K|\!+\!2|L|\!-\!3$},
without boundary if $K\!\neq\!\eset$ and with boundary if $K\!=\!\eset$.
The boundary of $\ov\cM_{0,L}^{\tau}$ parametrizes the curves with no irreducible 
component fixed by the involution;
the fixed locus of the involution on a curve in $\prt\ov\cM_{0,L}^{\tau}$ is a single node.
The \sf{main stratum}~$\cM_{K,L}^{\tau}$ of~$\ov\cM_{K,L}^{\tau}$ is the quotient~of
\begin{equation*}\begin{split}
\big\{\!\big((x_i)_{i\in K},(z_i^+,z_i^-)_{i\in L}\big)\!:\,
x_i\!\in\!S^1,\,z_i^{\pm}\!\in\!\P^1\!-\!S^1,\,z_i^+\!=\!\tau(z_i^-),&\\
x_i\!\neq\!x_j,\,z_i^+\!\neq\!z_j^+,z_j^-~\forall\,i\!\neq\!j&\big\}
\end{split}\end{equation*}
by the natural action of the subgroup $\PSL_2^{\tau}\C\!\subset\!\PSL_2\C$ 
of automorphisms of~$\P^1$ commuting with~$\tau$.
For $k,l\!\in\!\Z^{\ge0}$ with $k\!+\!2l\!\ge\!3$, 
we write $\ov\cM_{k,l}^{\tau}$ and $\cM_{k,l}^{\tau}$ instead of 
$\ov\cM_{[k],[l]}^{\tau}$ and $\cM_{[k],[l]}^{\tau}$, respectively.

If $|K|\!+\!2|L|\!\ge\!4$ and $i\!\in\!K$, let
\BE{ff1dfn_e}\ff_{K,L;i}^{\R}\!:\ov\cM_{K,L}^{\tau}\lra\ov\cM_{K-\{i\},L}^{\tau}\EE
be the forgetful morphism dropping the $i$-th real marked point.
The restriction of~$\ff_{K,L;i}^{\R}$ to the preimage of $\cM_{K-\{i\},L}^{\tau}$
is an $S^1$-fiber bundle.
The associated short exact sequence~\eref{finses_e} induces an isomorphism
\BE{cMorientR_e}
\la\big(\cM^\tau_{K,L}\big)\approx 
\ff_{K,L}^{\R\,*}\la\big(\cM^\tau_{K-\{i\},L}\big)\big|_{\cM^\tau_{K,L}}
\!\otimes\!\big(\!\ker\nd\ff_{K,L;i}^{\R}\big)\big|_{\cM^\tau_{K,L}}\,.\EE
If $|K|\!+\!2|L|\!\ge\!5$ and $i\!\in\!L$, we similarly denote~by
\BE{ff2dfn_e}\ff_{K,L;i}\!:\ov\cM_{K,L}^{\tau}\lra\ov\cM_{K,L-\{i\}}^{\tau}\EE
the forgetful morphism dropping the $i$-th conjugate pair of marked points.
The restriction of~$\ff_{K,L;i}$ to $\cM_{K,L}^{\tau}$
is a dense open subset of a $\P^1$-fiber bundle and thus induces an isomorphism
\BE{cMorientC_e}
\la\big(\cM^\tau_{K,L}\big)\approx 
\ff_{K,L;i}^{\,*}\la\big(\cM^\tau_{K,L-\{i\}}\big)\big|_{\cM^\tau_{K,L}}
\!\otimes\!\la\big(\!\ker\nd\ff_{K,L;i}\big)\big|_{\cM^\tau_{K,L}}\,.\EE
For each $\cC\!\in\!\cM^\tau_{K,L}$ as in~\eref{cCdfn0_e},
$$\ker\nd_{\cC}\ff_{K,L;i}\approx T_{z_i^+}\P^1 $$
is canonically oriented by the complex orientation of the fiber~$\P^1$ at~$z_i^+$.
We denote the resulting orientation of the last factor in~\eref{cMorientC_e} by~$\fo_i^+$.
For $k,l,i\!\in\!\Z^{\ge0}$ satisfying the appropriate conditions,
we write $\ff_{k,l;i}$ and $\ff_{k,l;i}^{\R}$ instead of $\ff_{[k],[l];i}^{\R}$ 
and~$\ff_{[k],[l];i}$, respectively.

Suppose $L$ contains a distinguished element~0 and $\cC\!\in\!\cM_{K,L}^{\tau}$ 
is as in~\eref{cCdfn0_e} with $\Si\!=\!\P^1$.
Let $\D^2_+\!\subset\!\C\!\subset\!\P^1$ be the disk cut out by 
the fixed locus~$S^1$ of~$\tau$ which contains~$z_0^+$. 
We orient \hbox{$S^1\!\subset\!\D_+^2\!\subset\!\C$} in the standard way
(this is the opposite of the boundary orientation of~$\D_+^2$ as defined in Section~\ref{TopolPrelim_sec}).
If $|K|\!+\!2|L|\!\ge\!4$ and $i\!\in\!K$,
this determines an orientation~$\fo_i^{\R}$ of the fiber
$$\ker\nd_{\cC}\ff_{K,L;i}^{\R}\approx T_{x_i}S^1 $$
of the last factor in~\eref{cMorientR_e} over~$\ff_{K,L;i}^{\R}(\cC)$.
This orientation extends over the subspace 
$$\ov\cM_{K,L;i}^{\tau;\st}\subset \ov\cM_{K,L}^{\tau}$$ 
consisting of curves~$\cC$ as in~\eref{cCdfn0_e} such that the real marked point $x_i$ of~$\cC$ 
lies on the same irreducible component of~$\Si$ as the marked point~$z_0^+$.
If $L$ is a nonempty subset of~$\Z^+$ containing~1, 
we take its distinguished element~0 to be $1\!\in\!L$.
For $k,l\!\in\!\Z^+$ with $k\!+\!2l\!\ge\!4$ and $i\!\in\![k]$,
we write $\ov\cM_{k,l;i}^{\tau;\st}$ instead of $\ov\cM_{[k],[l];i}^{\tau;\st}$.

We now define an orientation~$\fo_{k,l}$ on $\cM_{k,l}^{\tau}$ 
with $k\!\in\!\Z^{\ge0}$, $l\!\in\!\Z^+$, and $k\!+\!l\!\ge\!3$ inductively.
The space $\cM_{1,1}^{\tau}\!=\!\ov\cM_{1,1}^{\tau}$ is a single point;
we take $\fo_{1,1}\!\equiv\!+1$ to be its orientation as a plus point.
We identify the one-dimensional space 
$\ov\cM_{0,2}^{\tau}$ with $[0,\i]$ via the cross ratio 
\BE{cM02ident_e}
\vph_{0,2}\!:\ov\cM_{0,2}^{\tau}\lra[0,\i], \quad 
\vph\big([(z_1^+,z_1^-),(z_2^+,z_2^-)]\big)= 
\frac{z_2^+\!-\!z_1^-}{z_2^-\!-\!z_1^-}:\frac{z_2^+\!-\!z_1^+}{z_2^-\!-\!z_1^+}
=\frac{|1\!-\!z_1^+/z_2^-|^2}{|z_1^+\!-\!z_2^+|^2}\,;\EE
see Figure~\ref{fig_M_0,2}.
This identification determines an orientation~$\fo_{0,2}$ on~$\ov\cM_{0,2}^{\tau}$.
If $k\!+\!2l\!\ge\!4$ and $k\!\ge\!1$, we take~$\fo_{k,l}$ so that
the isomorphism~\eref{cMorientR_e} with $(K,L,i)$ replaced by~$([k],[l],k)$
 is compatible with the orientations~$\fo_{k,l}$,
$\fo_{k-1,l}$, and~$\fo_k^{\R}$ on the three line bundles involved.
If $k\!+\!2l\!\ge\!5$ and $l\!\ge\!2$, we take~$\fo_{k,l}$ so that 
the isomorphism~\eref{cMorientC_e} with $(K,L,i)$ replaced by~$([k],[l],l)$
 is compatible with the orientations~$\fo_{k,l}$,
$\fo_{k,l-1}$, and~$\fo_l^+$.
By a direct check, the orientations on~$\cM_{1,2}^{\tau}$ induced
from~$\cM_{0,2}^{\tau}$ via~\eref{cMorientR_e} and~$\cM_{1,1}^{\tau}$ via~\eref{cMorientC_e} are the same.
Since the fibers of $\ff_{k,l;l}|_{\cM_{k,l}^{\tau}}$ are even-dimensional,
it follows that the orientation~$\fo_{k,l}$ on~$\cM_{k,l}^{\tau}$ is well-defined
for all $l\!\in\!\Z^+$ and $k\!\in\!\Z^{\ge0}$ with $k\!+\!2l\!\ge\!3$.

\begin{figure}\begin{center}\begin{tikzpicture}
\draw (-4.5,0.5) circle [radius=0.5]; 
\draw (-4.5,-0.5) circle [radius=0.5];
\draw [fill] (-4.75,0.7) circle [radius=0.02];
\node [above left] at (-4.75,0.7) {$z_1^+$}; 
\draw [fill] (-4.36,0.7) circle [radius=0.02];
\node [above right] at (-4.36,0.7) {$z_2^-$};
\draw [fill] (-4.75,-0.7) circle [radius=0.02];
\node [below left] at (-4.75,-0.7) {$z_1^-$}; 
\draw [fill] (-4.36,-0.7) circle [radius=0.02];
\node [below right] at (-4.36,-0.7) {$z_2^+$};
\draw (-2.5,0) circle [radius=0.5];
\draw (-3,0) arc [start angle=180, end angle=360, x radius=0.5, y radius=0.2];
\draw [dashed] (-3,0) arc [start angle=180, end angle=0, x radius=0.5, y radius=0.1];
\draw [fill] (-2.7,0.3) circle [radius=0.02];
\node [above left] at (-2.7,0.3) {$z_1^+$}; 
\draw [fill] (-2.7,-0.3) circle [radius=0.02];
\node [below left] at (-2.7,-0.3) {$z_1^-$}; 
\draw [fill] (-2.3,0.3) circle [radius=0.02];
\node [above right] at (-2.3,0.3) {$z_2^-$}; 
\draw [fill] (-2.3,-0.3) circle [radius=0.02];
\node [below right] at (-2.3,-0.3) {$z_2^+$}; 
\draw (-0.5,0) circle [radius=0.5];
\draw (0.5,0) circle [radius=0.5];
\draw (0,0) arc [start angle=180, end angle=360, x radius=0.5, y radius=0.2];
\draw [dashed] (0,0) arc [start angle=180, end angle=0, x radius=0.5, y radius=0.1];
\draw (-1,0) arc [start angle=180, end angle=360, x radius=0.5, y radius=0.2];
\draw [dashed] (-1,0) arc [start angle=180, end angle=0, x radius=0.5, y radius=0.1];
\draw [fill] (-0.75,0.3) circle [radius=0.02];
\node [above left] at (-0.75,0.3) {$z_1^+$}; 
\draw [fill] (-0.75,-0.3) circle [radius=0.02];
\node [below left] at (-0.75,-0.3) {$z_1^-$}; 
\draw [fill] (0.6,0.28) circle [radius=0.02];
\node [above right] at (0.6,0.28) {$z_2^+$};
\draw [fill] (0.6,-0.28) circle [radius=0.02];
\node [below right] at (0.6,-0.28) {$z_2^-$};
\draw (2.5,0) circle [radius=0.5];
\draw (2,0) arc [start angle=180, end angle=360, x radius=0.5, y radius=0.2];
\draw [dashed] (2,0) arc [start angle=180, end angle=0, x radius=0.5, y radius=0.1];
\draw [fill] (2.3,0.3) circle [radius=0.02];
\node [above left] at (2.3,0.3) {$z_1^+$}; 
\draw [fill] (2.3,-0.3) circle [radius=0.02];
\node [below left] at (2.3,-0.3) {$z_1^-$}; 
\draw [fill] (2.7,0.3) circle [radius=0.02];
\node [above right] at (2.7,0.3) {$z_2^+$}; 
\draw [fill] (2.7,-0.3) circle [radius=0.02];
\node [below right] at (2.7,-0.3) {$z_2^-$}; 
\draw (4.5,0.5) circle [radius=0.5]; 
\draw (4.5,-0.5) circle [radius=0.5];
\draw [fill] (4.75,0.7) circle [radius=0.02];
\node [above right] at (4.75,0.7) {$z_2^+$}; 
\draw [fill] (4.36,0.7) circle [radius=0.02];
\node [above left] at (4.36,0.7) {$z_1^+$};
\draw [fill] (4.75,-0.7) circle [radius=0.02];
\node [below right] at (4.75,-0.7) {$z_2^-$}; 
\draw [fill] (4.36,-0.7) circle [radius=0.02];
\node [below left] at (4.36,-0.7) {$z_1^-$};
\draw [fill] (-4.5,-2) circle [radius=0.05];
\draw [fill] (4.5,-2) circle [radius=0.05];
\draw [fill] (0,-2) circle [radius=0.05];
\draw (-4.5,-2) to (4.5,-2); 
\node [below] at (-4.5,-2) {0}; 
\node [below] at (4.5,-2) {$\infty$};
\node [below] at (0,-2) {1};  
\node at (-3.5,0) {$\cdots$}; 
\node at (-1.5,0) {$\cdots$}; 
\node at (1.5,0) {$\cdots$}; 
\node at (3.5,0) {$\cdots$}; 
\end{tikzpicture}\end{center}
\caption{The structure of $\ov\cM_{0,2}^{\tau}$}
\label{fig_M_0,2}\end{figure}

For $L^*\!\subset\![l]$ and $\cC\!\in\!\cM_{k,l}^{\tau}$ as in~\eref{cCdfn0_e}, let 
$$\de_{L^*}^c(\cC)=
\big|\big\{i\!\in\![l]\!-\!L^*\!:z_i^+\!\not\in\!\D_+^2\big\}\big|+2\Z\in\Z_2.$$
In particular, $\de_{[l]}^c(\cC)\!=\!0$.
The functions~$\de_{L^*}$ is locally constant on~$\cM_{k,l}^{\tau}$.
We denote by $\fo_{k,l;L^*}$ the orientation on 
$\cM_{k,l}^{\tau}$ which equals~$\fo_{k,l}$ at $\cC$ if and only if 
$\de_{L^*}^c(\cC)\!=\!0$.
The next statement is straightforward.

\begin{lmm}\label{cMorient_lmm}
The orientations $\fo_{k,l;L^*}$ on $\cM_{k,l}^{\tau}$ with	
$k,l\!\in\!\Z^{\ge0}$ and $1\!\in\!L^*\!\subset\![l]$ such that \hbox{$k\!+\!2l\!\ge\!3$} satisfy the following properties:
\BEnum{($\fo_{\cM}\arabic*$)}

\item\label{cMorientR_it} the isomorphism~\eref{cMorientR_e} with $(K,L,i)$ 
replaced by~$([k\!+\!1],[l],k\!+\!1)$ respects the orientations 
$\fo_{k+1,l;L^*}$, $\fo_{k,l;L^*}$, and~$\fo_{k+1}^{\R}$; 

\item\label{cMorientC_it} 
the isomorphism~\eref{cMorientC_e} with $(K,L,i)$ replaced by~$([k],[l\!+\!1],l\!+\!1)$
respects the orientations $\fo_{k,l+1;L^*\!\cup\!\{l+1\}}$, $\fo_{k,l;L^*}$, and~$\fo_{l+1}^+$; 

\item the interchange of two real points $x_i$ and $x_j$ reverses $\fo_{k,l;L^*}$;

\item\label{Cijinter_it} if $i,j\!\in\![l]$, $1\!\in\!L'\!\subset\!L^*\!\cup\!\{i,j\}$,
$L^*\!\subset\!L'\!\cup\!\{i,j\}$,  $|L^*|\!=\!|L'|$ if $1\!\not\in\!\{i,j\}$, 
$\cC\!\in\!\cM_{k,l}^{\tau}$, 
and the marked points~$z_i^+$ and~$z_j^+$ 
are not separated by the fixed locus~$S^1$ of~$\cC$, then
the interchange of the conjugate pairs $(z_i^+,z_i^-)$ and $(z_j^+,z_j^-)$
respects the orientations~$\fo_{k,l;L^*}$ at~$\cC$ and~$\fo_{k,l;L'}$ 
at its image under the interchange;  

\item  the interchange of the points in a conjugate pair $(z_i^+,z_i^-)$ with $i\!\in\![l]\!-\!L^*$;
preserves~$\fo_{k,l;L^*}$;

\item the interchange of the points in a conjugate pair $(z_i^+,z_i^-)$ with $i\!\in\!L^*\!-\!\{1\}$
reverses~$\fo_{k,l;L^*}$;

\item the interchange of the points in the conjugate pair $(z_1^+,z_1^-)$ 
preserves~$\fo_{k,l;L^*}$ if and only~if 
$k\!+\!l\!-\!|L^*|\!\not\in\!2\Z$.

\EEnum
\end{lmm}

\vspace{.15in}

Suppose $K,L$ are finite sets so that $|K|\!+\!2|L|\!\ge\!3$,
$K$ is ordered, and $L$ contains a distinguished element~0.
Let $L^*\!\subset\!L$.
We then identify $K$ with $[|K|]$ as ordered sets and 
$(L,0)$ with $([|L|],1)$ as pointed sets.
Let $L'\!\subset\![|L|]$ be the image of~$L^*$ under the latter identification.
We denote by~$\fo_{K,L;L^*}$ the orientation on $\cM_{K,L}^{\tau}$  
obtained from the orientation $\fo_{k,l;L'}$ on $\cM_{k,l}^{\tau}$ via these identifications.
By Lemma~\ref{cMorient_lmm}\ref{Cijinter_it},
the orientation~$\fo_{K,L;L^*}$ does not depend on the choice of identification 
of $(L,0)$ with $([|L|],1)$ as pointed sets.
If $K\!\subset\!\Z$, we take $K$ to be ordered as a subset of~$\Z$.

\subsection{Codimension 1 strata and degrees}
\label{cMorient_subs}

The \sf{codimension~1 strata} of $\ov\cM_{k,l}^{\tau}\!-\!\prt\ov\cM_{k,l}^{\tau}$ correspond 
to the sets $\{(K_1,L_1),(K_2,L_2)\}$ such~that 
$$[k]=K_1\!\sqcup\!K_2, \quad  [l]=L_1\!\sqcup\!L_2, \quad
|K_1|\!+\!2|L_1|\ge2, \quad |K_2|\!+\!2|L_2|\ge2\,.$$
The open stratum $\oS$ corresponding to such a set  parametrizes marked curves~$\cC$ as in~\eref{cCdfn0_e} 
so that the underlying surface~$\Si$ consists of two real irreducible components
with one of them carrying the real marked points~$x_i$ with $i\!\in\!K_1$ and
the conjugate pairs of marked points $(z_i^+,z_i^-)$ with $i\!\in\!L_1$ and
the other component carrying the other marked points.
A \sf{closed codimension~1} stratum~$\ov{S}$ is the closure of such an open stratum~$\oS$.
Thus,
\BE{Ssplit_e0} \oS\approx\cM_{\{0\}\sqcup K_1,L_1}^{\tau}\!\times\!
\cM_{\{0\}\sqcup K_2,L_2}^{\tau},
\quad \ov{S}\approx\ov\cM_{\{0\}\sqcup K_1,L_1}^{\tau}\!\times\!
\ov\cM_{\{0\}\sqcup K_2,L_2}^{\tau},\EE
with the real marked point $x_0$ corresponding to the node on each 
of the two irreducible components.

Let $l\!\in\!\Z^+$. 
If $\oS$ is a codimension~1 stratum of $\ov\cM_{k,l}^{\tau}\!-\!\prt\ov\cM_{k,l}^{\tau}$ 
and 
$\cC\!\in\!\oS$, we denote by $\P^1_1$ the irreducible component of~$\cC$ containing
the marked points~$z_1^{\pm}$, by $\P^1_2$ the other irreducible component,
and by $S^1_1\!\subset\!\P^1_1$ and $S^1_2\!\subset\!\P^1_2$ the fixed loci of the involutions
on these components.
For $r\!=\!1,2$, we then take 
$$K_r\equiv K_r(S) \qquad\hbox{and}\qquad L_r\equiv L_r(S)$$ 
in~\eref{Ssplit_e0} to be the set of real marked points 
and the set of conjugate pairs of marked points, respectively, carried by~$\P^1_r$.
Let $\de_\R(S)\!\in\!\{0,1\}$ be the parity of the permutation 
$$[k]\lra K_1(S)\!\sqcup\!K_2(S)\!=\![k]$$ 
respecting the orders on the subsets $K_1(S),K_2(S)\!\subset\!\Z$.
For $L^*\!\subset\![l]$ and $r\!=\!1,2$, we define
$$L_r^*(S)=L_r(S)\!\cap\!L^*\subset[l].$$

An orientation~$\fo_{S;\cC}^c$ of the normal bundle~$\cN_{\cC}S$ of~$S$ 
in $\ov\cM_{k,l}^{\tau}$ at $\cC\!\in\!S$ determines a direction of degeneration of elements of 
$\cM_{k,l}^{\tau}$ to~$\cC$.
The orientation~$\fo_{k,l;L^*}$ on~$\cM_{k,l}^{\tau}$ limits to 
an orientation~$\fo_{k,l;L^*;\cC}$ of $\la_{\cC}(\ov\cM_{k,l}^{\tau})$
obtained by approaching~$\cC$ from this direction.
Along with~$\fo_{S;\cC}^c$, $\fo_{k,l;L^*;\cC}$ determines 
an orientation $\prt_{\fo_{S;\cC}^c}\fo_{k,l;L^*;\cC}$ of~$\la_{\cC}(S)$
via the first isomorphism in~\eref{lasplits_e}.
If in addition $L_1^*(S),L_2^*(S)\!\neq\!\eset$,  let $i^*_1\!\in\!L_1^*(S)$
and $i_2^*\!\in\!L_2^*(S)$ be the smallest elements of the two sets.
The two directions of degeneration of elements of $\cM_{k,l}^{\tau}$ to~$\cC$
are then distinguished by whether the marked points $z_{i_1^*}^+,z_{i_2^*}^+$ of
the degenerating elements lie on the same disk~$\D^2$ cut out by the fixed locus~$S^1$ or not.
We denote by~$\fo_{S;\cC}^{c;+}$ the orientation of~$\cN_{\cC}S$ 
which corresponds to the direction of degeneration for which $z_{i_1^*}^+,z_{i_2^*}^+$
lie on the same disk~$\D^2$
and by~$\fo_{S;\cC}^{c;-}$ the opposite orientation.
Let $\fo_{k,l;L^*;\cC}^{\pm}$ and $\fo_{S;L^*;\cC}^{\pm}$ 
be the orientations of $\la_{\cC}(\ov\cM_{k,l}^{\tau})$ and~$\la_{\cC}(S)$,
respectively, induced by~$\fo_{S;\cC}^{c;\pm}$ as~above. 
We denote by~$\fo_{S;L^*}$ the orientation on~$\oS$ obtained via 
the first identification in~\eref{Ssplit_e0} from the orientations 
$\fo_{\{0\}\sqcup K_1,L_1;L_1^*(S)}$ on $\cM_{\{0\}\sqcup K_1,L_1}^{\tau}$
with $i_1^*\!\in\!L_1$ as the distinguished point
and $\fo_{\{0\}\sqcup K_2,L_2;L_2^*(S)}$ on $\cM_{\{0\}\sqcup K_2,L_2}^{\tau}$
and  $i_2^*\!\in\!L_2$ as the distinguished point.

\begin{lmm}\label{DMboundary_lmm} 
Suppose $k,l\!\in\!\Z^{\ge0}$ with \hbox{$k\!+\!2l\!\ge\!3$},
$1\!\in\!L^*\!\subset\![l]$, and \hbox{$S\!\subset\!\ov\cM_{k,l}^{\tau}$} 
is a codimension~1 disk bubbling stratum~$S$ with $L_2^*(S)\neq\eset$.
The orientations~$\fo_{S;L^*}$, $\fo_{S;L^*}^+$, and $\fo_{S;L^*}^-$ on $\la(S)$  
satisfy
$$\fo_{S;L^*}=\begin{cases}
\fo_{S;L^*}^+&\hbox{iff}~\de_\R(S)\!\cong\!k\!+\!1~\hbox{mod~2};\\
\fo_{S;L^*}^-&\hbox{iff}~\de_\R(S)\!\cong\!|K_1(S)|\!+\!|L_2(S)\!-\!L^*_2(S)|.
\end{cases}$$
\end{lmm}

\begin{proof} For $r\!=\!1,2$, let 
$$L_r=L_r(S), \qquad L_r^*=L_r^*(S), \qquad K_r=K_r(S).$$
If $|L^*|\!=\!l\!=\!2$ and $k\!=\!0$, 
$S\!=\!S_1\!=\!S_2$ is a point and~$\fo_{S;L^*}\!=\!+1$.
The claim in this case thus holds by the definition of the orientations 
$\fo_{0,2;[2]}\!=\!\fo_{0,2}$ on~$\cM_{0,2}^{\tau}$ and $\fo_{S;\cC}^{c;\pm}$ on~$\cN S$.
Since the orientation \hbox{$\fo_{0,l;[l]}\!\equiv\!\fo_{0,l}$} with $l\!\ge\!3$ 
(resp.~$\fo_{1,l;[l]}\!\equiv\!\fo_{1,l}$ with $l\!\ge\!2$) is obtained
from the orientations~$\fo_{0,l-1;[l-1]}$ (resp.~$\fo_{1,l-1;[l-1]}$) and~$\fo_l^+$,
it follows that the claim holds whenever $L^*\!=\![l]$ and $k\!=\!0$.

Let $\cC\!\in\!S$ be as in~\eref{cCdfn0_e}.
Suppose $|L^*|\!<\!l$ and $k\!=\!0$.  
Let $l_1^c$ and~$l_2^c$ be the numbers of the marked points $z_i^-$ of~$\cC$ with 
\hbox{$i\!\in\![l]\!-\!L^*$}  on the same disk as $z_{i_1^*}^+\!\equiv\!z_1^+$ 
and on the same disk as~$z_{i_2^*}^+$,  respectively.
By definition,
\begin{alignat*}{2}
\fo_{1,L_1;L_1^*}\big|_{\cM_1}
&=(-1)^{l_1^c}\fo_{1,L_1;L_1}\big|_{\cM_1}\,, &\quad
\fo_{S;L^*}^+&=(-1)^{l_1^c+l_2^c}\fo_{S;[l]}^+,\\
\fo_{1,L_2;L_2^*}\big|_{\cM_2}
&=(-1)^{l_2^c}\fo_{1,L_2;L_2}\big|_{\cM_2}\,,&\quad
\fo_{S;L^*}^-&=(-1)^{l_1^c+(|L_2-L_2^*|-l_2^c)}\fo_{S;[l]}^-\,.
\end{alignat*}
Thus, the claim in this case follows from the $L^*\!=\![l]$ case above.

Suppose $k\!>\!0$, $S'\!\subset\!\cM_{0,l}^{\tau}$ is the image
of~$S$ under the forgetful morphism
$$\ff\!:\cM_{k,l}^{\tau}\lra\cM_{0,l}^{\tau}$$ 
dropping all real marked points, $\cC'\!=\!\ff(\cC)$, and 
$(\cC_1',\cC_2')\!\in\!\cM_1'\!\times\!\cM_2'$ 
is the corresponding pair of marked irreducible components
(with 1~real marked point each).
The orientation~$\fo_{S;L^*}$ on~$T_{\cC}\oS$ is obtained via isomorphisms
\BE{DMboundary_e5}\begin{split}
\big(T_{\cC}\oS,\fo_{S;L^*}\big)
&\approx \big(T_{\cC_1'}\cM_1',\fo_{1,L_1;L_1^*}\big)
\!\oplus\!\bigoplus_{i\in K_1}\!T_{x_i}\!S^1_1
\oplus \big(T_{\cC_2'}\cM_2',\fo_{1,L_2;L_2^*}\big)
\!\oplus\!\bigoplus_{i\in K_2}\!T_{x_i}\!S^1_2\\
&\approx \big(T_{\cC_1'}\cM_1',\fo_{1,L_1;L_1^*}\big)\!\oplus\!
\big(T_{\cC_2'}\cM_2',\fo_{1,L_2;L_2^*}\big)\oplus
\bigoplus_{i\in K_1}\!T_{x_i}\!S^1_1\!\oplus\!\bigoplus_{i\in K_2}\!T_{x_i}\!S^1_2\\
&\approx \big(T_{\cC'}S',\fo_{S';L^*}\big)\oplus
\bigoplus_{i\in K_1}\!T_{x_i}\!S^1_1\!\oplus\!\bigoplus_{i\in K_2}\!T_{x_i}\!S^1_2
\end{split}\EE
from  the standard orientations on $S^1_1$ and $S^1_2$
determined by the marked points~$z_1^+$ and~$z_{i_2^*}^+$.
The second isomorphism above is orientation-preserving because 
the dimension of $T_{\cC_2'}\cM_2'$ is even.

Let $\wt\cC\!\in\!\cM_{k,l}^{\tau}$ be a smooth marked curve 
close to~$\cC$ from the direction of degeneration
determined by~$\fo_S^{c;\pm}$ and $\wt\cC'\!=\!\ff(\wt\cC)$.
The orientation~$\fo_{S;L^*}^{\pm}$ at~$\cC$ is obtained via isomorphisms
\BE{DMboundary_e9}\begin{split}
\big(T_{\cC}\oS,\fo_{S;L^*}^{\pm}\big)
\!\oplus\!\big(\cN_{\cC}S,\fo_{S}^{c;\pm}\big)
&\approx\big(T_{\wt\cC}\cM_{k,l}^{\tau},\fo_{k,l;L^*}\big)
\approx  \big(T_{\wt\cC'}\cM_{0,l}^{\tau},\fo_{0,l;L^*}\big)\!\oplus\!
\bigoplus_{i=1}^{i=k}\!T_{x_i}\!S^1\\
&\approx \big(T_{\cC'}\oS',\fo_{S';L^*}^{\pm}\big)\!\oplus\!
\big(\cN_{\cC'}S',\fo_{S'}^{c;\pm}\big)\!\oplus\!
\bigoplus_{i=1}^{i=k}\!T_{x_i}\!S^1\\
&\approx(-1)^k\big(T_{\cC'}\oS',\fo_{S';L^*}^{\pm}\big)\!\oplus\!
\bigoplus_{i=1}^{i=k}\!T_{x_i}\!S^1\!\oplus\!
\big(\cN_{\cC}S,\fo_{S}^{c;\pm}\big).
\end{split}\EE
By \eref{DMboundary_e5}, \eref{DMboundary_e9}, and the $k\!=\!0$ case above,
the claim in the general case holds.
We note that the lines $T_{x_i}\!S^1$ with $i\!\in\!K_2$ have opposite orientations
in~\eref{DMboundary_e5} and~\eref{DMboundary_e9} in the minus case.  
\end{proof}

For $i\!\in\![l]$, we denote by 
$$\oS_i\subset\ov\cM_{k,l}^{\tau} \qquad\hbox{and}\qquad \ov{S}_i\subset \ov\cM_{k,l}^{\tau} $$
the open codimension~1 stratum parametrizing marked curves consisting of two real spheres 
with the marked points~$z_i^{\pm}$ on one of them and all other marked points on the other sphere
and its closure, respectively.

If $\ov{S}\!\subset\!\ov\cM_{k,l}^{\tau}\!-\!\prt\ov\cM_{k,l}^{\tau}$ 
is a closed codimension~1 stratum different from~$\ov{S}_1$, let
\BE{ff1dfn_e2b}\ff_{S;1}\!:\ov{S}\lra
\ov\cM_{K_1(S),L_1(S)}^{\tau}\!\times\!\ov\cM_{\{0\}\sqcup K_2(S),L_2(S)}^{\tau}\EE
denote the composition of the second identification in~\eref{Ssplit_e0}
with the forgetful morphism
$$\ff_{\nod}^{\R}\!:\ov\cM_{\{0\}\sqcup K_1(S),L_1(S)}^{\tau}\lra\ov\cM_{K_1(S),L_1(S)}^{\tau}$$  
as in~\eref{ff1dfn_e} dropping the marked point~$x_0$ corresponding to the node.
The vertical tangent bundle of $\ff_{S;1}|_{\oS}$ is a pullback
of the vertical tangent bundle of $\ff_{\nod}^{\R}|_{\cM_{\{0\}\sqcup K_1(S),L_1(S)}^{\tau}}$ 
and thus inherits an orientation from
the orientation~$\fo_{\nod}^{\R}$ of the latter specified in Section~\ref{cMstrata_subs};
we denote the induced orientation also by~$\fo_{\nod}^{\R}$.
It extends over the subspace 
$$S^{\st}\subset \ov{S}\subset \ov\cM_{k,l}^{\tau}$$
of curves~$\cC$ so that the marked point $x_0$ of the first component of the image of~$\cC$
under~\eref{Ssplit_e0} lies on the same irreducible component of the domain
as the marked point corresponding to~$z_1^+$.

Let $\Ups\!\subset\!\ov\cM_{k,l}^{\tau}$ be a bordered hypersurface.
If $k\!+\!2l\!\ge\!4$ and $i\!\in\![k]$, we call~$\Ups$ 
\sf{regular with respect to~}$\ff_{k,l;i}^{\R}$ if 
$\Ups\!\subset\!\ov\cM_{k,l;i}^{\tau;\st}$,
$\ff_{k,l;i}^{\R}(\ov\Ups\!-\!\Ups)$ is contained in the strata of codimension at least~2, 
i.e.~the subspace of $\ov\cM_{k-1,l}^{\tau}$ parametrizing curves with at least two nodes, and 
$\ff_{k,l;i}^{\R}(\prt\Ups)$ is contained in the union of $\prt\ov\cM_{k-1,l}^{\tau}$ 
and the strata of codimension at least~2.
By the last two assumptions, $\ff_{k,l;i}^{\R}|_\Ups$ is a $\Z_2$-pseudocycle of codimension~0;
see Section~\ref{TopolPrelim_sec}.
By the first assumption, the orientation~$\fo_i^{\R}$ of the last factor in~\eref{cMorientR_e}
and a co-orientation~$\fo_\Ups^c$ on~$\Ups$ induce a relative orientation~$\fo_\Ups^c\fo_i^{\R}$
of~$\ff_{k,l;i}^{\R}|_\Ups$. 
Let
$$\deg_i^{\R}\!\big(\Ups,\fo_\Ups^c\big)\equiv
\deg\!\big(\ff_{k,l;i}^{\R}|_\Ups,\fo_\Ups^c\fo_i^{\R}\big)$$
be the degree of the Steenrod pseudocycle $(\ff_{k,l;i}^{\R}|_\Ups,\fo_\Ups^c\fo_i^{\R})$;
see~\eref{SteenDeg_e}.

Suppose in addition that $S\!\subset\!\ov\cM_{k,l}^{\tau}\!-\!\prt\ov\cM_{k,l}^{\tau}$ 
is a codimension~1 stratum.
We call~$\Ups$ regular with respect to~$S$
if $\Ups$ and $\prt\Ups$ are transverse to~$\ov{S}$ in $\ov\cM_{k,l}^{\tau}$,
$$\Ups\!\cap\!\ov{S}\approx\Ups_1\!\times\!\ov\cM_{\{0\}\sqcup K_2(S),L_2(S)}^{\tau}$$
under the second identification in~\eref{Ssplit_e0} for some 
$\Ups_1\!\subset\!\ov\cM_{\{0\}\sqcup K_1(S),L_1(S);0}^{\tau;\st}$,
$\ff_{S;1}((\ov\Ups\!-\!\Ups)\!\cap\!\ov{S})$ is contained in
the strata of codimension at least~2 of the target of~$\ff_{S;1}$, 
and $\ff_{S;1}(\prt\Ups\!\cap\!\ov{S})$ is contained in the union of 
the boundary and the strata of codimension at least~2 of the target of~$\ff_{S;1}$.
By the first and the last two assumptions, 
$\ff_{S;1}|_{\Ups\cap \ov{S}}$  is a $\Z_2$-pseudocycle of codimension~0. 
By the first assumption, a co-orientation~$\fo_{\Ups}^c$ on~$\Ups$ in~$\ov\cM_{k,l}^{\tau}$ 
determines a co-orientation
$$\fo_{\Ups\cap S}^c\equiv\fo_{\Ups}^c\big|_{\Ups\cap \ov{S}}$$
on $\Ups\!\cap\!\ov{S}$ in~$\ov{S}$.
By the second assumption, $\Ups\!\cap\!\ov{S}\!\subset\!S^{\st}$.
By the first two assumptions, $S\!\neq\!S_1$ if $\Ups\!\cap\!\ov{S}\!\neq\!\eset$
and that $\fo_{\Ups}^c$ and
the orientation~$\fo_{\nod}^{\R}$ of the fibers of the restriction of~\eref{ff1dfn_e2b} to~$\oS$
specified above induce a
relative orientation $\fo_{\Ups\cap S}^c\fo_{\nod}^{\R}$ of~$\ff_{S;1}|_{\Ups\cap \ov{S}}$.
Let
$$\deg_S\!\big(\Ups,\fo_{\Ups}^c\big)\equiv
\deg\!\big(\ff_{S;1}|_{\Ups\cap \ov{S}},\fo_{\Ups}^c\fo_{\nod}^{\R}\big)
\equiv\deg\!\big(\ff_{S;1}|_{\Ups\cap \ov{S}},\fo^c_{\Ups\cap S}\fo_{\nod}^{\R}\big).$$

We call a bordered hypersurface $\Ups\!\subset\!\ov\cM_{k,l}^{\tau}$ \sf{regular}
if $\ov\Ups\!-\!\Ups$ is contained in the strata of codimension at least~2 and
$\Ups$ is regular with respect to the forgetful morphism~$\ff_{k,l;i}^{\R}$ for every $i\!\in\![k]$
and with respect to every codimension~1 stratum 
\hbox{$S\!\subset\!\ov\cM_{k,l}^{\tau}\!-\!\prt\ov\cM_{k,l}^{\tau}$}.
For such a hypersurface, \hbox{$\Ups\!\cap\!\ov{S}_1\!=\!\eset$}.

\subsection{Codimension 2 strata and bordisms}
\label{NBstrata_subs}

Suppose $l\!\ge\!2$ and $k\!+\!2l\!\ge\!5$.
The moduli space~$\ov\cM_{k,l}^{\tau}$ contains codimension~2 strata~$\oGa$
that parametrize marked curves~$\cC$ as in~\eref{cCdfn0_e} so that 
the underlying surface~$\Si$ consists of one real component~$\P^1_0$ and 
one pair~$\P^1_{\pm}$ of conjugate components; see Figure~\ref{LiftedRel_fig}.
We do not distinguish these strata based on the ordering of the marked points on 
the fixed locus $S^1_1\!\subset\!\P^1_0$ of the involution.
For such a stratum~$\oGa$, let 
$$L_0(\Ga),L_{\C}(\Ga)\subset\Z^+$$ 
be the subsets of the indices of the conjugate pairs of marked points carried by~$\P^1_0$ and 
$\P^1_-\!\cup\!\P^1_+$, respectively.
In particular, 
$$\big|L_{\C}(\Ga)\big|\ge2 \qquad\hbox{and}\qquad 
\big|L_0(\Ga)\big|\!+\!\big|L_{\C}(\Ga)\big|=l.$$
The closure~$\ov\Ga$ of~$\oGa$ decomposes~as 
\BE{Gasplit_e} \ov\Ga\approx \ov\cM_{[k],\{0\}\sqcup 
L_0(\Ga)}^{\tau}\!\times\!\ov\cM_{0,\{0\}\sqcup L_{\C}(\Ga)}\,.\EE
We call a codimension~2 stratum as above \sf{primary} if the marked point~$z_1^+$ of 
the curves~$\cC$ in~$\oGa$ is carried by $\P^1_-\!\cup\!\P^1_+$.

For a primary codimension~2 stratum $\oGa$ and $\cC\!\in\!\oGa$, 
we denote by $\P^1_+$ the irreducible component of~$\cC$ carrying the marked point~$z_1^+$.
If in addition $L^*\!\subset\![l]$, let 
$$L_0^*(\Ga)= L_0(\Ga)\!\cap\!L^*\subset [l].$$
We denote by $L_-^*(\Ga)\!\subset\!L_{\C}(\Ga)$
the subset of the indices of the marked points $z_i^-$ with $i\!\in\!L^*$ 
carried by~$\P^1_+$. 
The second factor in~\eref{Gasplit_e} is canonically oriented (being a complex manifold).
Let $\fo_{\Ga;L^*}$ be the orientation on~$\oGa$ obtained via 
the identification~\eref{Gasplit_e} from 
the orientation $\fo_{[k],\{0\}\sqcup L_0(\Ga);\{0\}\sqcup L_0^*(\Ga)}$ 
on $\cM_{[k],\{0\}\sqcup L_0(\Ga)}^{\tau}$ times $(-1)^{|L_-^*(\Ga)|}$.

With the identification as above, let 
$$\pi_1,\pi_2\!:\ov\Ga\lra \ov\cM_{[k],\{0\}\sqcup L_0(\Ga)}^{\tau},\ov\cM_{0,\{0\}\sqcup L_{\C}(\Ga)}$$
be the projections to the two factors.
Denote~by
$$\cL_{\Ga}^{\R}\lra \ov\cM_{[k],\{0\}\sqcup L_0(\Ga)}^{\tau} \qquad\hbox{and}\qquad
\cL_{\Ga}^{\C}\lra \ov\cM_{0,\{0\}\sqcup L_{\C}(\Ga)}$$
the universal tangent line bundles at the first point of the 0-th conjugate pair
of marked points and at the 0-th marked point, respectively.
The normal bundle~$\cN\Ga$ consists of conjugate smoothings of the two nodes 
of the curves in~$\oGa$.
Thus, it is canonically isomorphic to the complex line bundle
$$\cL_{\Ga}\equiv \pi_1^*\cL_{\Ga}^{\R}\!\otimes_{\C}\!\pi_2\cL_{\Ga}^{\C}\lra\Ga\,.$$
The next observation is straightforward.

\begin{lmm}\label{cNGa2_lmm}
Suppose $k,l\!\in\!\Z^{\ge0}$ and $1\!\in\!L^*\!\subset\![l]$ are 
such that \hbox{$k\!+\!2l\!\ge\!3$}.
Let \hbox{$\oGa\!\subset\!\ov\cM_{k,l}^{\tau}$} be a primary codimension~2 stratum.
The orientation~$\fo_{\Ga}^c$ on~$\cN\Ga$ induced by 
the orientations~$\fo_{k,l;L^*}$ on~$\cM_{k,l}^{\tau}$
and~$\fo_{\Ga;L^*}$ on~$\oGa$ agrees with the complex orientation of~$\cL_{\Ga}$.
\end{lmm}

The two relations of Theorem~\ref{WDVVdim3_thm} are proved by applying~\eref{bigeq_e} with 
the hypersurfaces $\Ups\!\subset\!\ov\cM_{1,2}^\tau$ and $\Ups\!\subset\!\ov\cM_{0,3}^{\tau}$
of Lemmas~\ref{M12rel_lmm} and~\ref{M03rel_lmm} below.
These hypersurfaces are \sf{regular}, in the sense defined at the end of Section~\ref{cMorient_subs},
and in particular are disjoint from the codimension~1 stratum~$S_1$ of the moduli space.
All notation for the codimension~1 strata and the degrees is as in Section~\ref{cMorient_subs}.
Since $\fo_{k,l}\!=\!\fo_{k,l;[l]}$,
\BE{M12rel_e0}\fo_{\Ga;[l]}=\fo_{\Ga}^c\fo_{k,l}\EE
in the cases of Lemmas~\ref{M12rel_lmm} and~\ref{M03rel_lmm}.
Let $P^{\pm}\!\in\!\ov\cM_{1,2}^{\tau}$ be 
the three-component curve so that $z_1^+$ and~$z_2^{\pm}$ lie on the same irreducible
component.

\begin{lmm}[{\cite[Lemma~4.4]{RealWDVV}}]\label{M12rel_lmm}
There exists an embedded closed path $\Ups\!\subset\!\ov\cM_{1,2}^\tau$ 
with a co-orientation~$\fo_\Ups^c$ so that $\Ups$ 
is a regular hypersurface and
\BE{M12rel_e}\prt\big(\Ups,\fo_\Ups^c\big)=\big(P^+,\fo_{P^+}^c\big)\!\sqcup\!\big(P^-,\fo_{P^-}^c\big),
\quad \deg_1^{\R}\!\big(\Ups,\fo_{\Ups}^c\big)=1, 
\quad \deg_{S_2}\!\!\big(\Ups,\fo_{\Ups}^c\big)=-1\,.\EE
\end{lmm}

The moduli space $\ov\cM_{0,3}^{\tau}$ is a 3-manifold with the~boundary
$$\prt\ov\cM_{0,3}^{\tau} =\ov{S}_{23}^{++}\sqcup \ov{S}_{23}^{+-}\sqcup \ov{S}_{23}^{-+}\sqcup \ov{S}_{23}^{--}\,,$$
where  $\ov{S}_{ij}^{\pm\pm}\!\approx\!\ov\cM_{0,4}\!\approx\!S^2$ 
is the closure of the open codimension~1 stratum  $\oS_{ij}^{\pm\pm}$ 
of curves consisting of a pair of conjugate spheres with the marked points
$z_i^{\pm}$ and $z_j^{\pm}$  on the same sphere as~$z_1^+$;
see \cite[Fig.~4]{RealEnum}
and the first diagram in Figure~\ref{fig_curveshapes}.
There are four primary codimension~2 strata $\Ga_i^{\pm}$, with $i\!=\!2,3$,
in $\ov\cM_{0,3}^{\tau}$.
The closed interval~$\ov\Ga_i^+$ (resp.~$\ov\Ga_i^-$) is the closure of the open codimension~2 
stratum~$\oGa_i^+$ (resp.~$\oGa_i^-$) of curves consisting of one real sphere and 
a conjugate pair of spheres so that the real sphere carries the marked points~$z_i^{\pm}$
and the decorations~$^{\pm}$ of the marked points on each of the conjugate spheres
are the same (resp.~different); 
see the last two diagrams in Figure~\ref{fig_curveshapes}.
Let 
$$\OGa_i^+\!=\!\Ga_i^+\!\cup\!\big(\ov\Ga_i^+\!\cap\!\ov{S}_i) \subset \ov\Ga_i^+$$
be the complement of the endpoints of~$\ov\Ga_i^+$.

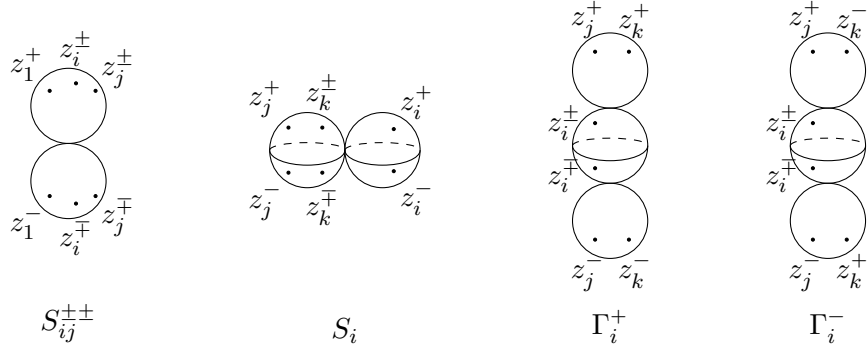
\begin{figure}
\begin{center}
\begin{tikzpicture}
\draw (0,0.5) circle [radius=0.5]; 
\draw (0,-0.5) circle [radius=0.5];
\draw [fill] (-0.25,0.7) circle [radius=0.02];
\node [above left] at (-0.2,0.7) {$z_1^+$}; 
\draw [fill] (0.1,0.8) circle [radius=0.02];
\node [above] at (0.1,0.9) {$z_i^\pm$}; 
\draw [fill] (0.36,0.7) circle [radius=0.02];
\node [above right] at (0.3,0.6) {$z_j^\pm$};
\draw [fill] (-0.25,-0.7) circle [radius=0.02];
\node [below left] at (-0.2,-0.7) {$z_1^-$}; 
\draw [fill] (0.1,-0.8) circle [radius=0.02];
\node [below] at (0.1,-0.9) {$z_i^\mp$}; 
\draw [fill] (0.36,-0.7) circle [radius=0.02];
\node [below right] at (0.3,-0.6) {$z_j^\mp$};
\node at (0,-2.4) {$\oS_{ij}^{\pm\pm}$}; 
\end{tikzpicture} \hspace{1cm}
\begin{tikzpicture}
\draw (-0.5,0) circle [radius=0.5];
\draw (0.5,0) circle [radius=0.5];
\draw (0,0) arc [start angle=180, end angle=360, x radius=0.5, y radius=0.2];
\draw [dashed] (0,0) arc [start angle=180, end angle=0, x radius=0.5, y radius=0.1];
\draw (-1,0) arc [start angle=180, end angle=360, x radius=0.5, y radius=0.2];
\draw [dashed] (-1,0) arc [start angle=180, end angle=0, x radius=0.5, y radius=0.1];
\draw [fill] (-0.75,0.3) circle [radius=0.02];
\node [above left] at (-0.7,0.3) {$z_j^+$}; 
\draw [fill] (-0.75,-0.3) circle [radius=0.02];
\node [below left] at (-0.7,-0.3) {$z_j^-$}; 
\draw [fill] (-0.3,0.3) circle [radius=0.02];
\node [above] at (-0.3,0.4) {$z_k^\pm$}; 
\draw [fill] (-0.3,-0.3) circle [radius=0.02];
\node [below] at (-0.3,-0.4) {$z_k^\mp$};
\draw [fill] (0.65,0.28) circle [radius=0.02];
\node [above right] at (0.6,0.28) {$z_i^+$};
\draw [fill] (0.65,-0.28) circle [radius=0.02];
\node [below right] at (0.6,-0.28) {$z_i^-$};
\node at (0,-2.4) {$\oS_i$};
\end{tikzpicture} \hspace{1cm}
\begin{tikzpicture}
\draw (0,0) circle [radius=0.5];
\draw (-0.5,0) arc [start angle=180, end angle=360, x radius=0.5, y radius=0.2];
\draw [dashed] (-0.5,0) arc [start angle=180, end angle=0, x radius=0.5, y radius=0.1];
\draw [fill] (-0.2,0.3) circle [radius=0.02];
\node [left] at (-0.25,0.3) {$z_i^\pm$};
\draw [fill] (-0.2,-0.3) circle [radius=0.02];
\node [left] at (-0.25,-0.4) {$z_i^\mp$};
\draw (0,1) circle [radius=0.5]; 
\draw (0,-1) circle [radius=0.5];
\draw [fill] (-0.2,1.25) circle [radius=0.02];
\node [above] at (-0.3,1.3) {$z_j^+$};
\draw [fill] (-0.2,-1.25) circle [radius=0.02];
\node [below] at (-0.3,-1.3) {$z_j^-$};
\draw [fill] (0.25,1.25) circle [radius=0.02];
\node [above] at (0.33,1.3) {$z_k^+$};
\draw [fill] (0.25,-1.25) circle [radius=0.02];
\node [below] at (0.33,-1.3) {$z_k^-$};
\node at (0, -2.4) {$\oGa_i^+$};
\end{tikzpicture} \hspace{1cm}
\begin{tikzpicture}
\draw (0,0) circle [radius=0.5];
\draw (-0.5,0) arc [start angle=180, end angle=360, x radius=0.5, y radius=0.2];
\draw [dashed] (-0.5,0) arc [start angle=180, end angle=0, x radius=0.5, y radius=0.1];
\draw [fill] (-0.2,0.3) circle [radius=0.02];
\node [left] at (-0.25,0.3) {$z_i^\pm$};
\draw [fill] (-0.2,-0.3) circle [radius=0.02];
\node [left] at (-0.25,-0.4) {$z_i^\mp$};
\draw (0,1) circle [radius=0.5]; 
\draw (0,-1) circle [radius=0.5];
\draw [fill] (-0.2,1.25) circle [radius=0.02];
\node [above] at (-0.3,1.3) {$z_j^+$};
\draw [fill] (-0.2,-1.25) circle [radius=0.02];
\node [below] at (-0.3,-1.3) {$z_j^-$};
\draw [fill] (0.25,1.25) circle [radius=0.02];
\node [above] at (0.33,1.3) {$z_k^-$};
\draw [fill] (0.25,-1.25) circle [radius=0.02];
\node [below] at (0.33,-1.3) {$z_k^+$};
\node at (0, -2.4) {$\oGa_i^-$};
\end{tikzpicture}
\end{center}
\caption{Elements of open codimension 1 and 2 strata of $\ov\cM_{0,3}^\tau$,
with $\{i,j\}\!=\!\{2,3\}$ in the first diagram and 
$\{i,j,k\}\!=\!\{1,2,3\}$ in the other four.}
\label{fig_curveshapes}
\end{figure}

\begin{lmm}[{\cite[Lemma~4.4]{RealWDVV}}]\label{M03rel_lmm}
There exist a bordered surface $\Ups\!\subset\!\ov\cM_{0,3}^{\tau}$ 
with a co-orientation~$\fo_\Ups^c$ and a one-dimensional manifold 
$\ga'\!\subset\!\ov\cM_{0,3}^{\tau}$ with a co-orientation~$\fo_{\ga'}^c$
so~that $\Ups$ is transverse to all open strata of~$\ov\cM_{0,3}^{\tau}$
not contained in any~$\ov\Ga_i^{\pm}$ with $i\!=\!2,3$,
$\Ups$ is a regular hypersurface, and
\begin{gather}\label{M03rel_e0}
\prt\big(\Ups,\fo_\Ups^c\big)=\big(\OGa_2^+,\fo_{\Ga_2^+}^c\big)\!\cup\!
\big(\OGa_3^+,-\fo_{\Ga_3^+}^c\big)\!\cup\!
\big(\OGa_2^-,\fo_{\Ga_2^-}^c\big)\!\cup\!
\big(\OGa_3^-,-\fo_{\Ga_3^-}^c\big)\!\cup\!
\big(\ga',\fo_{\ga'}^c\big),\\
\notag
\ga'\subset\prt\ov\cM_{0,3}^{\tau},  
\qquad \deg_{S_2}\!\!\big(\Ups,\fo_{\Ups}^c\big)=1, \qquad 
\deg_{S_3}\!\!\big(\Ups,\fo_{\Ups}^c\big)=-1.
\end{gather}
\end{lmm}

\section{Real GW-invariants}
\label{RealGWs_sec}

We introduce notation for moduli spaces of stable maps to a real symplectic manifold
$(X,\om,\phi)$ and for their strata in Section~\ref{MapSpaces_subs}.
We then formulate three key structural propositions in Section~\ref{DecompForm_subs}
and deduce Theorem~\ref{WDVVdim3_thm} from them in Section~\ref{SolWDVVpf_subs}. 
The evaluation maps from the moduli spaces of  stable maps
take values in ordered products of copies~$X$ and~$X^{\phi}$.
For the remainder of the paper,
we take the default orientations of these products to be given by 
the symplectic orientation~$\fo_{\om}$ and the orientation~$\fo$ of~$\wch{X}^{\phi}$ 
encoded by the OSpin-structure \hbox{$\os\!\equiv\!(\fo,\fs)$} on~$\wch{X}^{\phi}$ 
under consideration.

\subsection{Moduli spaces of stable maps}
\label{MapSpaces_subs}

Let $(X,\om,\phi)$ be a real symplectic manifold,  
$\wch X^{\phi}$ be a topological component of~$X^{\phi}$,
and $G$ be a finite subgroup of $\Aut(X,\om,\phi;\wch{X}^{\phi})$. 
For finite sets $K,L$ with $|K|\!+\!2|L|\!\ge\!3$,
we denote by $\cH_{K,L;G}^{\om,\phi}$ the space of pairs $(J,\nu)$ 
consisting of $J\!\in\!\cJ_{\om;G}^{\phi}$ and 
a real $G$-invariant perturbation~$\nu$ of the $\dbar_J$-equation 
associated with $\ov\cM_{K,L}^{\tau}$ as in \cite[Section~2]{Penka2}.
For $k,l\!\in\!\Z^{\ge0}$ with $k\!+\!2l\!\ge\!3$, we write $\cH_{k,l;G}^{\om,\phi}$
instead of $\cH_{[k],[l];G}^{\om,\phi}$;
the same applies to all spaces of maps and morphisms defined below.

For $(J,\nu)\!\in\!\cH_{K,L;G}^{\om,\phi}$,
a \sf{real genus~0 $(J,\nu)$-map with $K$-marked real points 
and $L$-marked conjugate pairs of points} is a tuple
\BE{udfn_e} \u=\big(u\!:\Si\!\lra\!X,(x_i)_{i\in K},(z_i^+,z_i^-)_{i\in L},\si\big)\EE
such that 
\BE{cCdfn_e}\cC\equiv \big(\Si,(x_i)_{i\in K},(z_i^+,z_i^-)_{i\in L},\si\big)\EE
is a real genus~0 nodal curve with complex structure~$\fj$, 
$K$-marked real points, and $L$-marked conjugate pairs of points
and $u$ is a smooth map satisfying
$$u\!\circ\!\si=\phi\!\circ\!u, \qquad
\dbar_Ju|_z\!\equiv\!\frac12\big(\nd_zu\!+\!J\!\circ\!\nd_zu\!\circ\!\fj\big)
=\nu\big(z,u(z)\big)~~\forall~z\!\in\!\Si.$$
A map~$\u$ is called \sf{simple} if  the restriction 
of~$u$ to each unstable irreducible component of the domain is simple (i.e.~not multiply covered)
and no two such restrictions have the same image.
The fixed locus $\Si^{\si}$ of~$\si$ in~\eref{cCdfn_e} is either a single point or
a tree of circles (possibly a single circle).
We call a map~$\u$ as in~\eref{cCdfn_e} \sf{$\Z_2$-pinchable} if 
$K\!=\!\eset$ and
either $\Si^{\si}$ is a single point or  the element of $H_1(X^{\phi};\Z_2)$ 
determined by $u|_{\Si^{\si}}$ is~0.
Two tuples as in~\eref{udfn_e} are \sf{equivalent} if they differ by
a reparametrization of the domain.

Let $B\!\in\!H_2(X)$ and $(J,\nu)\!\in\!\cH_{K,L;G}^{\om,\phi}$.
We denote the moduli space of the equivalence classes of stable real genus~0 degree~$B$ 
$(J,\nu)$-maps with $K$-marked real points and $L$-marked conjugate pairs of points 
as in~\eref{udfn_e} such~that
$$\Si^{\si}\neq \eset  \qquad\hbox{and}\qquad u\big(\Si^{\si}\big)\subset \wch X^\phi$$ 
by $\ov\M_{K,L}(B;J,\nu;\wch X^\phi)$.
Let
$$\ov\M_{K,L}^*(B;J,\nu;\wch X^\phi)\subset\ov\M_{K,L}(B;J,\nu;\wch X^\phi) \quad\hbox{and}\quad 
\M_{K,L}(B;J,\nu;\wch X^\phi)\subset\ov\M_{K,L}^*(B;J,\nu;\wch X^\phi)$$
be the subspace of simple maps and the (\sf{virtually}) \sf{main stratum}, 
i.e.~the subspace consisting of maps as in~\eref{udfn_e} from smooth domains~$\Si$, respectively.

The forgetful morphisms~\eref{ff1dfn_e} and~\eref{ff2dfn_e} induce maps
$$\ff_{K,L;i}^{\R\,*}\!:\cH_{K-\{i\},L;G}^{\om,\phi}\!\lra\!\cH_{K,L;G}^{\om,\phi}
\qquad\hbox{and}\qquad
\ff_{K,L;i}^{\,*}\!:\cH_{K,L-\{i\};G}^{\om,\phi}\!\lra\!\cH_{K,L;G}^{\om,\phi},$$
respectively.
For each $\nu\!\in\!\cH_{K-\{i\},L;G}^{\om,\phi}$ and 
$\nu\!\in\!\cH_{K,L-\{i\};G}^{\om,\phi}$, we also denote by 
\BE{ffMdfn_e}\begin{split}
\ff_{K,L;i}^{\R}\!:\ov\M_{K,L}\big(B;J,\ff_{K,L;i}^{\R\,*}\nu;\wch X^\phi\big)
&\lra\ov\M_{K-\{i\},L}\big(B;J,\nu;\wch X^\phi\big) \qquad\hbox{and}\\
\ff_{K,L;i}\!:\ov\M_{K,L}\big(B;J,\ff_{K,L;i}^{\,*}\nu;\wch X^\phi\big)
&\lra\ov\M_{K,L-\{i\}}\big(B;J,\nu;\wch X^\phi\big)
\end{split}\EE
the forgetful morphisms dropping the $i$-th real marked point and
the $i$-th conjugate pair of marked points, respectively.
The restriction of the second morphism in~\eref{ffMdfn_e} to 
$$\M_{K,L}\big(B;J,\ff_{K,L;i}^{\,*}\nu;\wch{X}^{\phi}\big)\subset 
\ov\M_{K,L}\big(B;J,\ff_{K,;i}^{\,*}\nu;\wch{X}^{\phi}\big)$$ 
is a dense open subset of a $\P^1$-fiber bundle.
We denote by~$\fo_i^+$ the relative orientation of this restriction induced by
the position of the marked point~$z_i^+$.
The restriction of the first morphism in~\eref{ffMdfn_e} to the preimage of 
$$\M_{K-\{i\},L}\big(B;J,\nu;\wch X^\phi\big)\subset 
\ov\M_{K-\{i\},L}\big(B;J,\nu;\wch X^\phi\big)$$
is an $S^1$-fiber bundle.
If $L$ contains a distinguished element~0, 
we denote by~$\fo_i^{\R}$ the relative orientation of this restriction 
defined as in Section~\ref{cMstrata_subs}.

For $\fc\!\in\!\Z^+$, a (\sf{virtually}) \sf{codimension~$\fc$ stratum}~$\cS$ 
of $\ov\M_{K,L}(B;J,\nu;\wch X^\phi)$ is a subspace of maps from domains~$\Si$ with precisely~$\fc$ 
nodes and thus with $\fc\!+\!1$ irreducible components isomorphic to~$\P^1$. 
It is characterized~by  the distributions~of
\BEnum{$\bu$}

\item the degree~$B$ of the map components~$u$ of its elements~$\u$ 
as in~\eref{udfn_e},

\item the $K$-marked real points, and

\item the $l$-marked conjugate pairs of points
\EEnum
between the irreducible components of~$\Si$.
There are two types of codimension~1 strata distinguished by whether
the fixed locus~$\Si^{\si}$ of~$\si$ consists of a single point or a wedge of two circles.
These two types are known as \sf{sphere bubbling} and \sf{disk bubbling},
respectively.
If \eref{noS2bub_e} holds,
no element~\eref{udfn_e} of $\ov\M_{K,L}(B;J,\nu;\wch X^\phi)$ is $\Z_2$-pinchable
and sphere bubbling does not~occur.

For each $i\!\in\!K$, let 
\BE{fMRevdfn_e}
\ev_i^{\R}\!:\ov\M_{K,L}(B;J,\nu;\wch X^\phi)\lra X^{\phi},\quad
\ev_i^{\R}\big([u,(x_j)_{j\in K},(z_j^+,z_j^-)_{j\in L},\si]\big)=u(x_i),\EE
be the evaluation morphism for the $i$-th real marked point.
For each $i\!\in\!L$, let 
\BE{fMCevdfn_e}
\ev_i^+\!:\ov\M_{K,L}(B;J,\nu;\wch X^\phi)\lra X,\quad
\ev_i^+\big([u,(x_j)_{j\in K},(z_j^+,z_j^-)_{j\in L},\si]\big)=u(z_i^+),\EE
be the evaluation morphism for the positive point of the $i$-th conjugate pair of marked points.
Let
\BE{fMevdfn_e}\ev\!\equiv\!\prod_{i\in K}\!\ev_i^{\R}\times\!\prod_{i\in L}\!\ev_i^+\!:
\ov\M_{K,L}(B;J,\nu;\wch X^\phi)\lra \wch X_{K,L}\!\equiv\!(\wch X^{\phi})^K\!\times\!X^L\EE
be the \sf{total evaluation} map.
We will use the same notation for the compositions of these evaluation maps
with all obvious maps to~$\ov\M_{K,L}(B;J,\nu;\wch X^\phi)$.

Suppose $k\!\in\!\Z^{\ge0}$ and $l\!\in\!\Z^+$ with $k\!+\!2l\!\ge\!3$.
Let $\cS$ be a codimension~1 disk bubbling stratum of 
$\ov\M_{k,l}(B;J,\nu;\wch X^\phi)$. 
We call the irreducible component~$\P_1^1$ of the domain~$\Si$ of an element~$\u$
of~$\cS$ carrying the marked points $z_1^\pm$ the \sf{first bubble} and
the other irreducible component~$\P^1_2$ the \sf{second bubble}. 
For $r\!=\!1,2$, let
\BE{KrLrcSdfn_e}
K_r(\cS)\subset[k] \qquad\hbox{and}\qquad  L_r(\cS)\subset[l]\EE
to be the subsets of the indices of real marked points and conjugate pairs of marked points, 
respectively, carried by $\P^1_r$.
We denote by $B_r(\cS)\!\in\!H_2(X)$ the degree of the restriction  
of the map components~$u$ of the elements~$\u$ of~$\cS$ to $\P^1_r$. 
In particular,
$$[k]=K_1(\cS)\!\sqcup\!K_2(\cS),\quad  [l]=L_1(\cS)\!\sqcup\!L_2(\cS),\quad
\ell_{\om}(B)=\ell_{\om}\big(B_1(\cS)\big)\!+\!\ell_{\om}\big(B_2(\cS)\big).$$
Let
$$\ov\cS\subset \ov\M_{k,l}(B;J,\nu;\wch X^\phi)$$
be the \sf{virtual closure} of~$\cS$, i.e.~the subspace of maps~$\u$ as in~\eref{udfn_e}
so that the domain~$\Si$ can be split at a node into two connected (possibly reducible) surfaces,
$\Si_1$ and~$\!\Si_2$, so that
the degree of the restriction of the map component~$u$ of~$\u$ to~$\Si_1$ is~$B_1(\cS)$,
the real marked points $x_i$ with $i\!\in\!K_1(\cS)$ lie on~$\Si_1$,
and so do the conjugate pairs of marked points~$z_i^{\pm}$ with $i\!\in\!L_1(\cS)$.

If in addition $1\!\in\!L^*\!\subset\![l]$, let
\BE{KrLrcSdfn2_e}\begin{split}
&~~L_1^*(\cS)=L_1(\cS)\!\cap\!L^*, \qquad L_2^*(\cS)=L_2(\cS)\!\cap\!L^*, \\
&\ve_{L^*}(\cS)=\frac{\ell_{\om}(B_2(\cS))}{2}-\big|K_2(\cS)\big|
\!-\!\big|L_2(\cS)\!-\!L_2^*(\cS)\big|\,.
\end{split}\EE
We denote by
$$\M_{k,l;L^*}^{\st}(B;J,\nu;\wch X^\phi)\subset\ov\M_{k,l}^*(B;J,\nu;\wch X^\phi)$$
the subspace of simple maps that are not $\Z_2$-pinchable and
\BEnum{$\bu$}

\item have no nodes, or

\item lie in a codimension~1 disk bubbling stratum~$\cS$ 
with $\ve_{L^*}(\cS)\!\in\!2\Z$, or

\item  have only one conjugate pair of nodes.

\EEnum

\vspace{.15in}

Let $\wh\M_{k,l;L^*}(B;J,\nu;\wch X^\phi)$ be the space obtained 
by cutting $\ov\M_{k,l}(B;J,\nu;\wch X^\phi)$ 
along the closures $\ov\cS$ of the codimension~1 strata~$\cS$ with $\ve_{L^*}(\cS)\!\not\in\!2\Z$.
Thus, $\wh\M_{k,l;L^*}(B;J,\nu;\wch X^\phi)$ contains a double cover of~$\ov\cS$
for each codimension~1 stratum~$\cS$ of $\ov\M_{k,l}(B;J,\nu;\wch X^\phi)$ 
with $\ve_{L^*}(\cS)\!\not\in\!2\Z$. 
The union of these covers and the sphere bubbling strata, if any, 
form the (virtual) boundary of $\wh\M_{k,l;L^*}(B;J,\nu;\wch X^\phi)$.
Let
\BE{qdfn_e}q\!:\wh\M_{k,l;L^*}(B;J,\nu;\wch X^\phi)\lra \ov\M_{k,l}(B;J,\nu;\wch X^\phi)\EE
be the quotient map and
\begin{gather}\notag
\ev_i^{\R}\!:\wh\M_{k,l;L^*}(B;J,\nu;\wch X^\phi)\lra \wch X^{\phi},\qquad
\ev_i^+\!:\wh\M_{k,l;L^*}(B;J,\nu;\wch X^\phi)\lra X, \\
\label{whfMevdfn_e}
\ev\!:\wh\M_{k,l;L^*}(B;J,\nu;\wch X^\phi)\lra \wch X_{k,l}
\end{gather}
be the compositions of the evaluation maps in~\eref{fMevdfn_e} with the quotient map~$q$ 
in~\eref{qdfn_e}.

We denote by
\BE{whMstdfn_e}\wh\M_{k,l;L^*}^{\st}(B;J,\nu;\wch X^\phi)\subset\wh\M_{k,l;L^*}(B;J,\nu;\wch X^\phi)\EE
the subspace of simple maps that are not $\Z_2$-pinchable and 
\BEnum{$\bu$}

\item have no nodes, or

\item have only one real node, or

\item have only one conjugate pair of nodes.

\EEnum
The boundary $\prt\wh\M_{k,l;L^*}^{\st}(B;J,\nu;\wch X^\phi)$ of this subspace consists of 
double covers~$\wh\cS^*$ of the subspaces~$\cS^*$ of simple maps
of the codimension~1 strata~$\cS$ of $\ov\M_{k,l}(B;J,\nu;\wch X^\phi)$
with \hbox{$\ve_{L^*}(\cS)\!\not\in\!2\Z$}.

An $\OSpin$-structure $\os$ on $\wch{X}^{\phi}$ is a pair $(\fo,\fs)$ consisting 
of an orientation~$\fo$ on~$\wch{X}^{\phi}$ and a Spin-structure~$\fs$ on
the oriented vector bundle~$(T\wch{X}^{\phi},\fo)$,
i.e.~a compatible collection of homotopy classes of trivializations
of~$(T\wch{X}^{\phi},\fo)$ over loops in~$\wch{X}^{\phi}$;
see \cite[Def.~1.3]{SpinPin}. 
We identify homotopy classes of trivializations for different orientations if
they differ by a composition with an isomorphism of~$\R^3$;
this convention identifies Spin-structures for different orientations of~$\wch{X}^{\phi}$.
For an $\OSpin$-structure $\os\!\equiv\!(\fo,\fs)$ on~$\wch{X}^{\phi}$,
we denote by $\os\!\equiv\!(\ov\fo,\fs)$ the $\OSpin$-structure on~$\wch{X}^{\phi}$
obtained from~$\os$ by reversing its orientation component~$\fo$ only.
Lemma~\ref{orient_lmm} and Proposition~\ref{JakePseudo_prp} below
follow readily from~\cite{Jake}; see Section~\ref{orient_subs}.

\begin{lmm}\label{orient_lmm}
Suppose $(X,\om,\phi)$ is a real symplectic sixfold, 
$\wch X^\phi$ is a connected component of~$X^\phi$, 
\BE{orienlmm_e0}
k,l\in\Z^{\ge0}~~\hbox{with}~~k\!+\!2l\ge3,  \qquad
1\in L^*\subset[l],  \qquad B\!\in\!H_2(X),\EE
$G$ is a finite subgroup of $\Aut(X,\om,\phi;\wch{X}^{\phi})$,
and $(J,\nu)\!\in\!\cH_{k,l;G}^{\om,\phi}$ is generic.
An $\OSpin$-structure~$\os$ on~$\wch X^{\phi}$ determines orientations~$\fo_{\os;L^*}$ 
and~$\wh\fo_{\os;L^*}$ of $\M_{k,l;L^*}^{\st}(B;J,\nu;\wch X^\phi)$ 
and $\wh\M_{k,l;L^*}^{\st}(B;J,\nu;\wch X^\phi)$, 
respectively, with the following properties:
\BEnum{($\fo_{\os}\arabic*$)}

\item the restrictions of~$\fo_{\os;L^*}$ and $\wh\fo_{\os;L^*}$ to
$\M_{k,l}(B;J,\nu;\wch X^\phi)$ are the same;

\item\label{fforientR_it} 
the restrictions of $\fo_{\os;L^*}$ and $\fo_{k+1}^{\R}\fo_{\os;L^*}$ 
to $\M_{k+1,l}(B;J,\ff_{k+1,l;k+1}^{\,*}\nu;\wch X^\phi)$ are the same;

\item\label{fforient_it} 
the restrictions of $\fo_{\os;L^*\cup\{l+1\}}$ and $\fo_{l+1}^+\fo_{\os;L^*}$ 
to $\M_{k,l+1}(B;J,\ff_{k,l+1;l+1}^{\,*}\nu;\wch X^\phi)$ are the same;

\item the interchange of two real points $x_i$ and $x_j$ reverses~$\fo_{\os;L^*}$;

\item\label{Cijinter2_it} if $i,j\!\in\![l]$, $1\!\in\!L'\!\subset\!L^*\!\cup\!\{i,j\}$,
$L^*\!\subset\!L'\!\cup\!\{i,j\}$, $|L^*|\!=\!|L'|$ if $1\!\not\in\!\{i,j\}$, 
$\u\!\in\!\M_{k,l}(B;J,\nu;\wch X^\phi)$, and 
the marked points $z_i^+$ and $z_j^+$ are not separated by the fixed locus~$S^1$
of the domain of~$\u$, then
the interchange of the conjugate pairs $(z_i^+,z_i^-)$ and $(z_j^+,z_j^-)$
respects the orientations~$\fo_{\os;L^*}$  at~$\u$ and~$\fo_{\os;L'}$
at its image under the interchange; 

\item\label{orientpm0_it} the interchange of the points in a conjugate pair $(z_i^+,z_i^-)$
with $i\!\in\![l]\!-\!L^*$ preserves~$\fo_{\os;L^*}$;

\item the interchange of the points in a conjugate pair $(z_i^+,z_i^-)$
with $i\!\in\!L^*\!-\!\{1\}$ reverses~$\fo_{\os;L^*}$;

\item\label{orient1pm_it} 
the interchange of the points in the conjugate pair  $(z_1^+,z_1^-)$ preserves $\fo_{\os;L^*}$
if and only~if 
$$\ell_{\om}(B)\big/2\!+\!k\!+\!l\!-\!|L^*|\not\in2\Z;$$

\item\label{orient0_it} if $k,l\!=\!1$ and $B\!=\!0$, then
$(\ev_1^{\R},\fo_{\os;L^*})$ is a pseudocycle of degree~1;

\item\label{fMosch_it} if $\os'$ is another $\OSpin$-structure on~$\wch X^{\phi}$,
$\u\!\in\!\M_{k,l}(B;J,\nu;\wch X^\phi)$ is as in~\eref{udfn_e},
and the pullbacks of~$\os'$ and~$\ov\os$ by the restriction of~$u$ to the fixed locus
of its domain are the same, then
the orientations  $\fo_{\os;L^*}$ and $\fo_{\os';L^*}$ at~$\u$ are opposite.

\EEnum
\end{lmm}

\vspace{.15in}

Let $k,l,L^*,B$ and $(J,\nu)$ be as in Lemma~\ref{orient_lmm}.
For a tuple 
\BE{bhdfn_e}\bh\!\equiv\!(h_i\!:H_i\!\lra\!X)_{i\in[l]}\EE
of maps, define
\begin{gather}
\label{Mhdfn_e}f_{\bh}\!:M_{\bh}\equiv \prod_{i\in[l]}\!H_i\lra X^l,\qquad
f_{\bh}\big((y_i)_{i\in[l]}\big)=\big(h_i(y_i)\!\big)_{i\in[l]},\\
\notag
\cZ_{k,\bh;L^*}^{\st}(B;J,\nu;\wch X^\phi)=\big\{\big(\u,(y_i)_{i\in[l]}\big)\!\in\!
\M_{k,l;L^*}^{\st}(B;J,\nu;\wch X^\phi)\!\times\!M_{\bh}\!:
\ev_i^+(\u)\!=\!h_i(y_i)\,\forall\,i\!\in\![l]\big\}\,.
\end{gather}
Let 
\BE{JakePseudo_e} 
\ev_{k,\bh;L^*}\!: \cZ_{k,\bh;L^*}^{\st}\big(B;J,\nu;\wch X^\phi\big)
\lra \big(\wch X^\phi\big)^k\EE
be the map induced by~\eref{fMevdfn_e}.
Orientations on~$H_i$ determine an orientation~$\fo_{\bh}$ on~$M_{\bh}$.
Along an orientation $\fo_{\M}$ of $\M_{k,l;L^*}^{\st}(B;J,\nu;\wch X^\phi)$, 
the orientation~$\fo_{\bh}$ determines an orientation~$\fo_{\M}\fo_{\bh}$
of~$\cZ_{k,\bh;L^*}^{\st}\big(B;J,\nu;\wch X^\phi \big)$.

If $Y$ is a smooth manifold, a dimension~$p$ pseudocycle $h\!:H\!\lra\!Y$ 
determines an element~$[h]_Y$ of $H_p(Y;\Z)$; see~\cite{pseudo}.
If $Y\!=\!X$ and $B$ is a homology class in~$X$ in the complementary dimension, let
$$h\!\cdot_X\!\!B\equiv \blr{\PD_X([h]_X),B}\in\Z$$
denote the homology intersection product of~$[h]_X$ with~$B$. 
If $h$ and $B$ are not of complementary dimensions, we set $h\!\cdot_X\!\!B\!=\!0$.
For a tuple~$\bh$ of maps from smooth manifolds as in~\eref{bhdfn_e}, let
\begin{gather}\label{bhprpdnf_e}
\codim_{\C}\bh=\frac12\sum_{i=1}^l\!\big(\dim_{\R}X\!-\!\dim_{\R}H_i\big),\quad
L^*(\bh)=\{1\}\!\cup\!\big\{i\!\in\![l]\!:\dim\,H_i\in\{0,4\}\!\big\},\\
\notag
L_+^*(\bh)=\big\{i\!\in\![l]\!:\dim\,H_i\!\in\!\{0,4\}\!\big\}, \qquad
L_-^*(\bh)=\big\{i\!\in\![l]\!:\dim\,H_i\!=\!2\big\}.
\end{gather}
We denote the orientation $\fo_{\os;L^*(\bh)}\fo_{\bh}$ of the domain of~\eref{JakePseudo_e}
with $L^*\!=\!L^*(\bh)$ by~$\fo_{\os;\bh}$.

\begin{prp}\label{JakePseudo_prp}
Let $(X,\om,\phi)$, $\wch X^{\phi}$, $\os$, $B$, and $G$ be as in Lemma~\ref{orient_lmm}
and $l\!\in\!\Z^+$.
Suppose \hbox{$\bh\!\equiv\!(h_i)_{i\in[l]}$} is a tuple of 
pseudocycles into~$X$ of dimensions~0,2,4 in general position so~that
\BE{JakePseudo_e0} 
k\!\equiv\!\frac{\ell_{\om}(B)}{2}\!+\!l\!-\!\codim_{\C}\bh
\ge\max(0,3\!-\!2l)\,.\EE
\BEnum{(\arabic*)}

\item\label{RGWdfn_it}
For a generic choice of \hbox{$(J,\nu)\!\in\!\cH_{k,l;G}^{\om,\phi}$},  
the map~\eref{JakePseudo_e} with the orientation~$\fo_{\os;L^*(\bh)}\fo_{\bh}$ on its domain
is a codimension~0 pseudocycle and its degree
\BE{JakePseudo_e0b}\blr{(h_i)_{i\in[l]}}^{\phi,\os}_{B;\wch X^\phi}
\equiv \deg\!\big(\ev_{k,\bh;L^*(\bh)},\fo_{\os;\bh}\big)\EE
does not depend on the choice of $(J,\nu)$, $h_i\!\in\![h_i]_X$ with $i\!\in\!L_+^*(\bh)$,
or $h_i\!\in\![h_i]_{X-\wch X^{\phi}}$ with \hbox{$i\!\in\!L_-^*(\bh)$}.

\item\label{RGWsym_it} The number~\eref{JakePseudo_e0b} is invariant under the permutations
of the components~$h_i$ of~$\bh$.

\item\label{RGWvan_it} The number~\eref{JakePseudo_e0b} vanishes if 
$[h_i]_X\!\in\!H_4(X)^{\phi}_-$ for some $i\!\in\!L_+^*(\bh)$ or
$[h_i]_{X-\wch X^{\phi}}\!\in\!H_2(X\!-\!\wch{X}^{\phi})^{\phi}_+$
for some $i\!\in\!L_-^*(\bh)$.

\item\label{RdivRel_it}
If $k\!+\!2l\!\ge\!5$ and $i^*\!\in\![l]$ with $\dim\,h_{i^*}\!=\!4$, then
\BE{RdivRel_e}
\blr{(h_i)_{i\in[l]}}^{\phi,\os}_{B;\wch X^\phi}
=\big(h_{i^*}\!\cdot_X\!B\big)\blr{(h_i)_{i\in[l]-\{i^*\}}}^{\phi,\os}_{B;\wch X^\phi}\,.\EE 

\EEnum
\end{prp}

\vspace{.15in}

The assumption that the pseudocycles $h_i$ are in general position in 
Proposition~\ref{JakePseudo_prp} implies that each 
two-dimensional pseudocycle~$h_i$ is in fact a pseudocycle into $X\!-\!\wch{X}^{\phi}$.
By Proposition~\ref{JakePseudo_prp}\ref{RGWdfn_it}, the number 
\BE{RGWdfn_e0a}\blr{\big(\PD_X([h_i]_X)\!\big)_{\!i\in L_+^*(\bh)},
\big(\PD_{X,\wch{X}^{\phi}}([h_i]_{X-\wch X^{\phi}})\!
\big)_{\!i\in L_-^*(\bh)}}^{\phi,\os}_{B;\wch X^\phi}
=\blr{(h_i)_{i\in[l]}}^{\phi,\os}_{B;\wch X^\phi}\EE
of real $(J,\nu)$-holomorphic curves meeting the pseudocycles~$h_i$ and passing 
through $k$ general points in~$\wch{X}^{\phi}$
is well-defined, i.e.~it depends only on the homology classes on the left-hand side.
Thus, we obtain a well-defined number
\BE{RGWdfn_e0}\lr{(\mu_i)_{i\in[l]}}^{\phi,\os}_{B;\wch X^\phi}  \in\Q\EE
if $l\!\in\!\Z^+$ and
\BE{JakePseudo_e0a}k\!\equiv\!\frac{\ell_{\om}(B)}{2}\!+\!l\!-\!
\frac12\sum_{i\in[l]}\!\dim\,\mu_i\ge\max(0,3\!-\!2l).\EE
Below we drop the conditions on $k$ and~$l$.

We assume that $B\!\neq\!0$ and can be represented by a $J$-holomorphic map;
thus, $\lr{\om,B}\!\neq\!0$.
Let $H\!\in\!H^2(X;\Z)$ be such $\phi^*H\!=\!-H$ and $\lr{H,B}\!\neq\!0$;
such a class~$H$ can be obtained by slightly deforming~$\om$ so that it represents a rational class,
taking a multiple of the deformed class that represents an integral class,
and then taking the anti-invariant part of the multiple.
Let $l$, $L^*$, and $\bh\!\equiv\!(h_i)_{i\in[l]}$ be as in Proposition~\ref{JakePseudo_prp}
so that $h_1$ and~$h_2$ represent the Poincare dual of~$H$.
We define
\BE{RdivRel2_e}\blr{(h_i)_{i\in[l]-\{1,2\}}}^{\phi,\os}_{B;\wch X^\phi}
\equiv \frac{1}{\lr{H,B}^2}
\deg\!\big(\ev_{k,\bh;L^*(\bh)},\fo_{\os;\bh}\big).\EE
By~\ref{RGWsym_it} and~\ref{RdivRel_it} in Proposition~\ref{JakePseudo_prp},
this definition does not depend on the choice of~$H$,
agrees with~\eref{JakePseudo_e0b} in the overlapping cases,
determines the numbers~\eref{RGWdfn_e0} without any conditions on $l\!\in\!\Z^{\ge0}$ 
or $k\!\in\!\Z$ (if $k\!<\!0$, we take the number~\eref{JakePseudo_e0b} to be~0).
By~\eref{RdivRel_e},
\BE{RdivRel_e2}
\blr{(\mu_i)_{i\in[l]}}^{\phi,\os}_{B;\wch X^\phi} 
=\lr{\mu_{i^*},B}\blr{(\mu_i)_{i\in[l]-\{i^*\}}}^{\phi,\os}_{B;\wch X^\phi}\EE
if $B\!\neq\!0$,  
$\mu_i\!\in\!H^2(X)\!\cup\!H^6(X)\!\cup\!H^4(X,\wch{X}^{\phi})$ for all $i\!\in\![l]$,
and $\mu_{i^*}\!\in\!H^2(X)$.

Suppose $K,L$ are finite sets so that $|K|\!+\!2|L|\!\ge\!3$,
$K$ is ordered, and $L$ contains a distinguished element~0.
Let $0\!\in\!L^*\!\subset\!L$.
We then identify $K$ with $[|K|]$ as ordered sets and 
$(L,0)$ with $([|L|],1)$ as pointed sets.
Let $L'\!\subset\![|L|]$ be the image of~$L^*$ under the latter identification
and $\os$ be an $\OSpin$-structure on~$\wch X^{\phi}$. 
For $(J,\nu)\!\in\!\cH_{K,L;G}^{\om,\phi}$ generic,
we denote by~$\fo_{\os;L^*}$ the orientation on $\M_{K,L}(B;J,\nu;\wch X^\phi)$
obtained from the orientation $\fo_{k,l;L'}$ on $\M_{k,l}(B;J,\nu;\wch X^\phi)$
via these identifications.
By Lemma~\ref{orient_lmm}\ref{Cijinter2_it},
the orientation~$\fo_{\os;L^*}$ does not depend on the choice of identification 
of $(L,0)$ with $([|L|],1)$ as pointed~sets.
If in addition $\mu_i\!\in\!H^2(X)\!\cup\!H^6(X)\!\cup\!H^4(X,\wch{X}^{\phi})$ 
for $i\!\in\!L$, we denote by
$$\blr{(\mu_i)_{i\in L}}^{\phi,\os}_{B;\wch X^\phi} \in\Q$$
the number~\eref{RdivRel_e2} arising under the above identification of 
$(L,0)$ with $([|L|],1)$.

\subsection{Structural propositions}
\label{DecompForm_subs}

We next formulate three propositions which together imply Theorem~\ref{WDVVdim3_thm}.
Proposition~\ref{LiftedRel_prp} relates counts of two- and three-component real curves
passing through {\it fixed} constraints by lifting the bordisms of 
Lemmas~\ref{M12rel_lmm} and~\ref{M03rel_lmm}.
The two relations of Proposition~\ref{LiftedRel_prp}, 
depicted in Figure~\ref{LiftedRel_fig} on page~\pageref{LiftedRel_fig}, 
have the exact same form as in~\cite{RealWDVV}.
The counts of curves represented by the individual terms in these relations
generally depend on the choices of the constraints.
An averager~$G$ as in Definition~\ref{aveG_dfn} eliminates this dependence 
for $G$-invariant constraints and leads to splittings of the two types of counts
into invariant counts of irreducible curves in Propositions~\ref{Rdecomp_prp} and~\ref{Cdecomp_prp}.
These two propositions are the analogues of Propositions~5.7 and~5.3 in~\cite{RealWDVV},
but now depend on the use of an averager~$G$.

We fix a compact real symplectic sixfold $(X,\om,\phi)$,
an $\OSpin$-structure $\os$ on a connected component~$\wch{X}^{\phi}$ of~$X^{\phi}$,
a finite subgroup $G$ of $\Aut(X,\om,\phi;\wch{X}^{\phi})$,
$k,k',l,l'\!\in\!\Z^{\ge0}$, and $L^*\!\subset\![l]$ with 
\BE{klcond_e} k'\le k, \quad  l'\le l, \quad k'\!+\!2l'\ge3, \quad
1\in L^*.\EE
Let  $B\!\in\!H_2(X)$ and $(J,\nu)\!\in\!\cH_{k,l;G}^{\om,\phi}$.
There is then a well-defined \sf{forgetful morphism}
\BE{ffdfn_e}\ff_{k',l'}\!:\ov\M_{k,l}(B;J,\nu;\wch X^\phi)\lra\ov\cM_{k',l'}^{\tau}\EE
which drops the last $k\!-\!k'$ real marked points and the last $l\!-\!l'$ conjugate pairs
from the nodal marked curve~\eref{cCdfn_e} associated with each tuple~$\u$ as in~\eref{udfn_e}
and contracts the unstable irreducible components of the resulting curve.
We also fix a tuple~$\bh$ as in~\eref{bhdfn_e} of smooth maps from oriented manifolds and
a $k$-tuple \hbox{$\bp\!\equiv\!(p_i)_{i\in[k]}$} of points in~$\wch{X}^{\phi}$.
Let $L^*(\bh)\!\subset\![l]$ be as in~\eref{bhprpdnf_e}.

Suppose $\cS$ is an open codimension~1 disk-bubbling
stratum of $\ov\M_{k,l}(B;J,\nu;\wch X^\phi)$.
For $r\!=\!1,2$, let
$$K_r(\cS)\subset[k],~~L_r(\cS)\subset[l],~~L_r^*(\cS)\subset L^*,~~
\ep_{L^*}(\cS)\in\Z,~~B_r(\cS)\in H_2(X)$$
be as in Section~\ref{MapSpaces_subs} and $\cS^*\!\subset\!\cS$ be the subspace of simple maps.
With $M_{\bh}$ given by~\eref{Mhdfn_e}, define 
$$\cS_{\bh}^*=
\big\{\big(\u,(y_i)_{i\in[l]}\big)\!\in\!\cS^*\!\times\!M_{\bh}\!:
\ev_i^+(\u)\!=\!h_i(y_i)\,\forall\,i\!\in\![l]\big\}.$$ 
The (virtual) normal bundles $\cN\cS$ of $\cS$ in $\ov\M_{k,l}(B;J,\nu;\wch X^\phi)$ and  
$\cN\cS_{\bh}^*$ of $\cS_{\bh}^*$ in
$$\ov\cZ_{k,\bh}\big(B;J,\nu;\wch X^{\phi}\big)\equiv\big\{\big(\u,(y_i)_{i\in[l]}\big)
\!\in\!\ov\M_{k,l}(B;J,\nu;\wch X^\phi)\!\times\!M_{\bh}\!:
\ev_i^+(\u)\!=\!h_i(y_i)\,\forall\,i\!\in\![l]\big\}$$
are canonically isomorphic. Let
\BE{evcSdfn_e}\ev_{\cS;\bh}\!:\cS_{\bh}^*\lra \big(\wch{X}^{\phi}\big)^k\EE
be the map induced by~\eref{fMevdfn_e}.

If $\u\!\in\!\cS_{\bh}^*$, an orientation~$\fo_{\cS;\u}^c$ of~$\cN_{\u}\cS$ 
determines a direction of degeneration of elements of 
the main stratum of $\cZ_{k,\bh;L^*}^{\st}(B;J,\nu;\wch X^\phi)$ to~$\u$.
The orientation~$\fo_{\os;L^*}\fo_{\bh}$ of~$\cZ^\st_{k,\bh;L^*}(B;J,\nu;\wch X^\phi)$
limits to an orientation $\fo_{\os;L^*;\bh;\u}$ of $\ov\cZ_{k,\bh}(B;J,\nu;\wch X^\phi)$
at~$\u$ obtained by approaching~$\u$ from this direction.
Along with~$\fo_{\cS;\u}^c$,  $\fo_{\os;L^*;\bh;\u}$ determines an 
orientation $\prt_{\fo_{\cS;\u}^c}\fo_{\os;L^*;\bh;\u}$ of $\cS^*_\bh$
via the first isomorphism in~\eref{lasplits_e}.

\begin{lmm}\label{orient_lmm2}
Suppose $(X,\om,\phi)$, $\wch{X}^{\phi}$, 
$\os,k,l,L^*,B$, and $(J,\nu)$ are as in Lemma~\ref{orient_lmm}
and $\bh$ as in~\eref{bhdfn_e} is a generic tuple of smooth maps from oriented manifolds. 
If $\cS$ is an open codimension~1 disk bubbling stratum of 
$\ov\M_{k,l}(B;J,\nu;\wch X^\phi)$ and $\u\!\in\!\cS_{\bh}^*$, then
the orientation $\prt_{\fo_{\cS;\u}^c}\fo_{\os;L^*;\bh;\u}$ of~$\cS^*_\bh$ at~$\u$
does not depend on the choice of~$\fo_{\cS;\u}^c$ if and only if $\ep_{L^*}(\cS)\!\not\in\!2\Z$. 
\end{lmm}

The orientation~$\fo_{\os;L^*}\fo_{\bh}$ of 
$$\cZ_{k,\bh}(B;J,\nu;\wch X^\phi)\equiv\big\{\big(\u',(y_i)_{i\in[l]}\big)
\!\in\!\M_{k,l}(B;J,\nu;\wch X^\phi)\!\times\!M_{\bh}\!:
\ev_i^+(\u')\!=\!h_i(y_i)\,\forall\,i\!\in\![l]\big\}$$
extends across $\cS_{\bh}^*$ if and only if
$\prt_{\fo_{\cS;\u}^c}\fo_{\os;L^*;\bh;\u}$ {\it depends}  on the choice of~$\fo_{\cS;\u}^c$ 
for every $\u\!\in\!\cS_{\bh}^*$.
In particular, the first statement of Lemma~\ref{orient_lmm} is an immediate consequence
of Lemma~\ref{orient_lmm2}.
If $\ov\M_{k,l}(B;J,\nu;\wch X^\phi)$ is cut along~$\ov\cS$ and
$\wh\cS^*$ is the double cover of~$\cS^*$ in the~cut, then 
$\prt_{\fo_{\cS;\u}^c}\fo_{\os;L^*;\bh;\u}$ is the boundary orientation 
induced by~$\fo_{\os;L^*}\fo_{\bh}$ at one of the copies~$\wh\u$ of~$\u$ in 
\BE{pHSinc_e2wh}
\wh\cS_{\bh}^*=
\big\{\big(\wh\u',(y_i)_{i\in[l]}\big)\!\in\!\wh\cS^*\!\times\!M_{\bh}\!:
\ev_i^+(\wh\u')\!=\!h_i(y_i)\,\forall\,i\!\in\![l]\big\};\EE
we then denote it by $\prt\fo_{\os;L^*;\bh;\wh\u}$. 
If $\ep_{L^*}(\cS)\!\not\in\!2\Z$, we abbreviate $\prt_{\fo_{\cS;\u}^c}\fo_{\os;L^*;\bh;\u}$
as~$\prt\fo_{\os;L^*;\bh;\u}$.
We denote the orientation~$\prt\fo_{\os;L^*(\bh);\bh}$ by~$\prt\fo_{\os;\bh}$.

\begin{rmk}\label{orient2_rmk}
While Lemma~\ref{orient_lmm2} follows readily from \cite[Prop.~5.3]{Jake},
it is also immediately implied by our Lemmas~\ref{DMboundary_lmm} and~\ref{DorientComp_lmm} 
(which are also needed to establish Proposition~\ref{Rdecomp_prp} below). 
Our $\ep_{L^*}(\cS)$ equals to $\mathfrak{s}^\#\!-\!1$ in \cite[(22)]{Jake}. 
\end{rmk}

If in addition $\Ups\!\subset\!\ov\cM_{k',l'}^{\tau}$, we define
\begin{gather}\label{pHSinc_e}
f_{\bp;\Ups}\!:\Ups\lra \big(\wch{X}^{\phi}\big)^k\!\times\!\ov\cM_{k',l'}^{\tau},
\quad f_{\bp;\Ups}(P)=(\bp,P),\\
\label{pHSinc_e2}
\cS_{\bh,\bp;\Ups}^*=\big\{(\u,P)\!\in\!\cS_{\bh}^*\!\times\!\Ups\!:
\ev_{\cS;\bh}(\u)\!=\!\bp,\,\ff_{k',l'}(\u)=P\big\}.
\end{gather}

Suppose next that $\oGa\!\subset\!\ov\cM_{k',l'}^{\tau}$ is a primary codimension~2 stratum
and $\fo_{\Ga}^c$ is its canonical co-orientation as in Lemma~\ref{cNGa2_lmm}.
We denote~by 
$$\M_{\Ga;k,l}(B;J,\nu;\wch X^\phi)\subset \ff_{k',l'}^{-1}(\oGa) \subset\ov\M_{k,l}(B;J,\nu;\wch X^\phi)$$
the subspace consisting of maps from three-component domains.
The domain of every element~$\u$ of $\M_{\Ga;k,l}(B;J,\nu;\wch X^\phi)$ is stable 
and thus $\u$ is automatically a simple~map.
Define
\begin{equation*}\begin{split}
\cZ_{\Ga;k,\bh}^{\st}(B;J,\nu;\wch X^\phi)
&=\big\{\big(\u,(y_i)_{i\in[l]}\big)\!\in\!
\ov\cZ_{k,\bh}\big(B;J,\nu;\wch X^{\phi}\big)\!: 
\u\!\in\!\M_{\Ga;k,l}(B;J,\nu;\wch X^\phi)\big\}\\
&\subset\cZ_{k,\bh;L^*}^{\st}\big(B;J,\nu;\wch X^\phi\big).
\end{split}\end{equation*}
Let
\BE{evGadfn_e}\ev_{\Ga;\bh}:\cZ_{\Ga;k,\bh}^{\st}(B;J,\nu;\wch X^\phi)\lra
\big(\wch{X}^{\phi}\big)^k\EE
be the map induced by~\eref{fMevdfn_e};
it is the restriction of~\eref{JakePseudo_e}.

For generic choices of~$(J,\nu)$ and $\bh$, 
\BE{cZGadfn_e2} \cZ_{\Ga;k,\bh}^{\st}(B;J,\nu;\wch X^\phi)
\subset \cZ_{k,\bh;L^*}^{\st}(B;J,\nu;\wch X^\phi)\EE
is a smooth submanifold of a smooth manifold with the normal bundle canonically isomorphic 
to~$\ff_{k',l'}^*\cN\Ga$.
We denote~by
$$\fo_{\Ga;\os;\bh}\equiv \big(\ff_{k',l'}^*\fo_{\Ga}^c\big)\big(\fo_{\os;\bh}\big)$$ 
the orientation of the left-hand side in~\eref{cZGadfn_e2} 
determined by~$\ff_{k',l'}^*\fo_{\Ga}^c$ and 
the orientation~$\fo_{\os;\bh}$ of the right-right side in~\eref{cZGadfn_e2}
with $L^*\!=\!L^*(\bh)$. 

\begin{prp}\label{LiftedRel_prp}
Suppose $(X,\om,\phi)$ is a real symplectic sixfold,
$\os$ is an $\OSpin$-structure on a connected component~$\wch{X}^{\phi}$ of~$X^{\phi}$,
$G$ is a finite subgroup of $\Aut(X,\om,\phi;\wch{X}^{\phi})$,
$l\!\in\!\Z^{\ge0}$, and $B\!\in\!H_2(X)$.
Let \hbox{$\bh\!\equiv\!(h_i)_{i\in[l]}$} be a tuple of 
pseudocycles into~$X$ of dimensions~0,2,4 in general position so~that
\BE{Cdecomp_e0}
k\!\equiv\!\frac{\ell_{\om}(B)}2\!+\!l\!-\!\codim_{\C}\bh\!-\!1\ge\max(0,3\!-\!2l)\EE
and $(J,\nu)\!\in\!\cH_{k,l;G}^{\om,\phi}$ be generic.
\BEnum{(\arabic*)}

\item\label{M12lift_it} If $k\!\ge\!1$, $l\!\ge\!2$, and
$P^\pm\!\in\!\ov\cM_{1,2}$, $\Ups\!\subset\!\ov\cM_{1,2}$,
$\fo^c_{P^{\pm}}$, and $\fo^c_\Ups$ are as in Lemma~\ref{M12rel_lmm}, then
\BE{M12lift_e0a}\big|\ev^{-1}_{P^+;\bh}(\bp)\big|^\pm_{\fo_{P^+;\os;\bh}}
+\big|\ev^{-1}_{P^-;\bh}(\bp)\big|^\pm_{\fo_{P^-;\os;\bh}}
=-2\sum_{\cS}\big|\cS^*_{\bh,\bp;\Ups}\big|^\pm_{\prt\fo_{\os;\bh},\fo^c_\Ups},\EE
where the sum on the right-hand side is over all codimension 1 disk bubbling 
strata $\cS$ of $\ov\M_{k,l}(B;J,\nu;\wch{X}^{\phi})$
with $\ep_{L^*(\bh)}(\cS)\!\not\in\!2\Z$ and 
\BE{M12lift_e0b}\hbox{either}\qquad 2\in L_1(\cS),~1\in K_2(\cS)  
\quad\hbox{or}\quad 1\in K_1(\cS),~2\in L_2(\cS).\EE

\item\label{M03lift_it} If $l\!\ge\!3$ and $\Ga_2^\pm,\Ga_3^\pm,\Ups\!\subset\!\ov\cM_{0,3}$,
$\fo_{\Ga^\pm_2},\fo_{\Ga^\pm_3}$, and $\fo^c_\Ups$ are as in Lemma~\ref{M03rel_lmm}, 
then
\BE{M03lift_e0a}\begin{split}
&\big|\ev^{-1}_{\Ga^+_2;\bh}(\bp)\big|^\pm_{\fo_{\Ga_2^+;\bh}}
+\big|\ev^{-1}_{\Ga_2^-;\bh}(\bp)\big|^\pm_{\fo_{\Ga_2^-;\os;\bh}}\\
&\qquad-\big|\ev^{-1}_{\Ga^+_3;\bh}(\bp)\big|^\pm_{\fo_{\Ga_3^+;\os;\bh}}
-\big|\ev^{-1}_{\Ga_3^-;\bh}(\bp)\big|^\pm_{\fo_{\Ga_3^-;\os;\bh}}
=2\sum_{\cS}\big|\cS^*_{\bh,\bp;\Ups}\big|^\pm_{\prt\fo_{\os;\bh},\fo^c_\Ups},
\end{split}\EE
where the sum on the right-hand side is over all codimension 1 disk bubbling strata 
$\cS$ of $\ov\M_{k,l}(B;J,\nu;\wch{X}^{\phi})$
with $\ep_{L^*(\bh)}(\cS)\!\not\in\!2\Z$ and either
\BE{M03lift_e0b}\hbox{either}\qquad 3\in L_1(\cS),~2\in L_2(\cS)  
\quad\hbox{or}\quad 2\in L_1(\cS),~3\in L_2(\cS).\EE

\EEnum
\end{prp}

\vspace{.2in}

Let $G$ be an averager for $(X,\om,\phi;\wch{X}^{\phi})$ as in Definition~\ref{aveG_dfn}.
We call a tuple~$\bh$ as in~\eref{bhdfn_e} \sf{$G$-invariant} if
there exists a $G$-action by orientation-preserving diffeomorphisms on each $H_i$ such~that 
\BE{Gactbhdfn_e}g\!\circ\!h_i=h_i\!\circ\!g^{-1}.\EE
The Poincare dual of an integer multiple of every element of 
$H^*(X)^{\phi}_{\pm}$ and $H^*(X,X^{\phi})^{\phi}_{\pm}$ can be represented by a pseudocycle~$h_i$
satisfying the above condition.
Propositions~\ref{Rdecomp_prp} and~\ref{Cdecomp_prp} below, which split counts
of two- and three-component real curves with $G$-invariant insertions,
thus imply that these counts are in fact well-defined on the $G$-invariant cohomology insertions.
If $\bh$ is as in Proposition~\ref{LiftedRel_prp}, $L'\!\subset\![l]$, 
and $B'\!\in\!H_2(X)$, we define
\BE{OWGhomdfn_e}
\blr{(h_i)_{i\in L'}}_{B'}^X
=\blr{\big(\PD_X([h_i]_X)\!\big)_{i\in L'}}_{B'}^X, ~~
\blr{(h_i)_{i\in L'}}^{\phi,\os}_{B';\wch X^\phi;G}
=\blr{\big(\PD_X([h_i]_X)\!\big)_{i\in L'}}^{\phi,\os}_{B';\wch X^\phi;G}\EE
to be the invariant count of rational degree~$B'$ $J$-holomorphic curves in~$X$
meeting the pseudocycles~$h_i$ as in~\eref{CGWdfn_e} and
the invariant count of real rational degree~$B'$ $J$-holomorphic curves in~$X$
meeting the $G$-averages of the pseudocycles~$h_i$ as in~\eref{Gnumsdfn_e}, respectively.

Suppose $\cS$ is an open codimension~1 disk bubbling stratum of $\ov\M_{k,l}(B;J,\nu;\wch X^\phi)$.
It satisfies exactly one of the following conditions:
\BEnum{($\cS\arabic*$)}

\setcounter{enumi}{-1}

\item\label{cS0_it} $K_2(\cS)\!\cap\![k']\!=\!\eset$ and $L_2(\cS)\!\cap\![l']\!=\!\eset$;

\item\label{cS1_it} $|K_2(\cS)\!\cap\![k']|\!=\!1$ and $L_2(\cS)\!\cap\![l']\!=\!\eset$;

\item\label{cS2_it} there exists a codimension~1 stratum $S\!\subset\!\ov\cM_{k',l'}^{\tau}$
such that $\ff_{k',l'}(\cS)\!\subset\!S$.

\EEnum
We call a pair $(\cS,\Ups)$ consisting of $\cS$ as above and 
a (possibly bordered) hypersurface $\Ups\!\subset\!\ov\cM_{k,l}^{\tau}$ \sf{admissible}
if one of the following conditions holds:
\BEnum{($\cS\arabic*\Ups$)}

\item\label{cSUps_it1} $K_2(\cS)\!\cap\![k']\!=\!\{i\}$, $L_2(\cS)\!\cap\![l']\!=\!\eset$,
and $\Ups$ is regular with respect to~$\ff_{k',l';i}^{\R}$;  

\item\label{cSUps_it2} there exists a codimension~1 stratum 
$S\!\subset\!\ov\cM_{k',l'}^{\tau}$
such that $\ff_{k',l'}(\cS)\!\subset\!S$ and $\Ups$ is regular with respect to~$S$.

\EEnum
The notions of $\Ups$ being \sf{regular} with respect to $\ff_{k',l';i}^{\R}$ and~$S$
are defined in Section~\ref{cMorient_subs}. 
If $(\cS,\Ups)$ is an admissible pair and $\fo_\Ups^c$ is a co-orientation on~$\Ups$, 
we denote by $\deg(\cS,\fo_\Ups^c)\!\in\!\Z$
the corresponding degree $\deg_i^{\R}(\Ups,\fo_\Ups^c)$ or $\deg_S(\Ups,\fo_\Ups^c)$
defined in Section~\ref{cMorient_subs}.

\begin{prp}\label{Rdecomp_prp}
Let $(X,\om,\phi)$, $\os$, $l,B,\bh,k,\bp$ be as in 
Proposition~\ref{LiftedRel_prp} and $G$ be an averager for $(X,\om,\phi;\wch{X}^{\phi})$
so that $\bh$ is $G$-invariant. 
Suppose $k',l'\!\in\!\Z^{\ge0}$ satisfy the conditions in~\eref{klcond_e} and
$$\cS\subset\ov\M_{k,l}\big(B;J,\nu;\wch X^\phi)
\qquad\hbox{and}\qquad \Ups\subset\ov\cM_{k',l'}^{\tau}$$ 
form an admissible pair.
\BEnum{(\arabic*)}

\item\label{RdecompTrans_it} If \hbox{$(J,\nu)\!\in\!\cH_{k,l;G}^{\om,\phi}$} is generic, then 
$$\big(\ev_{\cS;\bh},\ff_{k',l'}\big)\!:
\cS_{\bh}^*\lra \big(\wch{X}^{\phi}\big)^k\!\times\!\ov\cM_{k',l'}^{\tau}
\quad\hbox{and}\quad
f_{\bp;\Ups}\!:\Ups \lra  \big(\wch{X}^{\phi}\big)^k\!\times\!\ov\cM_{k',l'}^{\tau}$$
are transverse maps from manifolds 
of complementary dimensions and the set~$\cS_{\bh,\bp;\Ups}^*$ is finite.

\item\label{RdecompEmpt_it} The set~$\cS_{\bh,\bp;\Ups}^*$ is empty unless $\ve_{L^*(\bh)}(\cS)\!\in\!2\Z$ or
\BE{Rdecomp_e0}\ve_{L^*(\bh)}(\cS)=2
\big|\big\{i\!\in\!L_2(\cS)\!:\dim\,h_i\!=\!0\big\}\big|\!+\!1\,.\EE

\item\label{Rdecomp_it} If \eref{Rdecomp_e0} holds and 
$\fo_\Ups^c$ is a co-orientation on~$\Ups$, then 
\BE{Rdecomp_e}\begin{split}
\big|\cS_{\bh,\bp;\Ups}^*\big|_{\prt\fo_{\os;\bh},\fo_\Ups^c}^{\pm}
=&-(-1)^{\dim\,\Ups}\!\deg(\cS,\fo_\Ups^c)\\
&\qquad\times\blr{(h_i)_{i\in L_1(\cS)}}^{\phi,\os}_{B_1(\cS);\wch X^\phi;G}
\blr{(h_i)_{i\in L_2(\cS)}}^{\phi,\os}_{B_2(\cS);\wch X^\phi;G}.
\end{split}\EE

\EEnum
\end{prp}

\vspace{.1in}

The condition~\eref{Rdecomp_e0} implies that
\begin{equation*}\begin{split}  
\big|K_1(\cS)\big|&=\frac{\ell_{\om}(B_1(\cS))}{2}\!+\!\big|L_1(\cS)\big|
\!-\!\codim_{\C}(h_i)_{i\in L_1(\cS)} \qquad\hbox{and}\\
\big|K_2(\cS)\big|\!+\!1&=\frac{\ell_{\om}(B_2(\cS))}{2}\!+\!\big|L_2(\cS)\big|
\!-\!\codim_{\C}(h_i)_{i\in L_2(\cS)},
\end{split}\end{equation*}
i.e.~the second irreducible component of the maps in~$\cS$ passes through an extra real point.

If \hbox{$h\!:H\!\lra\!X$} and \hbox{$h'\!:H'\!\lra\!X$} are transverse pseudocycles into~$X$, 
we define
$$h\!\cap\!h'\!:\big\{(y,y')\!\in\!H\!\times\!H'\!:h(y)\!=\!h'(y')\big\}\lra X,
\qquad h\!\cap\!h'(y,y')\!=\!h(y).$$
This is a pseudocycle representing $\PD_X^{-1}(\PD_X([h]_X)\!\cup\!\PD_X([h']_X))$.

Suppose $\oGa\!\subset\!\ov\cM_{k',l'}^{\tau}$ is a primary codimension~2 stratum.
Let $L_0(\Ga),L_{\C}(\Ga)\!\subset\![l]$  be as in Section~\ref{NBstrata_subs}.
With $B$ and $\bh$ as in Proposition~\ref{Rdecomp_prp}, we define
$$\lr{k}_{\Ga}=\begin{cases}1,&\hbox{if}~
k'\!=\!k\!=\!1,\,L_0(\Ga)\!=\!\eset;\\
0,&\hbox{otherwise};\end{cases}\qquad
\lr{\bh}_{\Ga}=\begin{cases}h_i\!\cap\!h_j,&\hbox{if}~L_{\C}(\Ga)\!=\!\{i,j\},\,i\!\neq\!j;\\
0,&\hbox{if}~\big|L_{\C}(\Ga)\big|\!\neq\!2.\end{cases}$$

\begin{prp}\label{Cdecomp_prp}
Let $(X,\om,\phi)$, $\os$, $l,B,\bh,k,\bp,G,k',l'$ be as in Proposition~\ref{Rdecomp_prp} with
\BE{Cdecomp_e0c}\phi_*\big([h_i]_X\big)=[h_i]_X~~\forall\,i\!\in\!L^*(\bh)\!-\!\{1\},  \quad
\phi_*\big([h_i]_X\big)=-[h_i]_X~~\forall\,i\!\in\![l]\!-\!L^*(\bh).\EE 
Suppose $\oGa\!\subset\!\ov\cM_{k',l'}^{\tau}$ is a primary codimension~2 stratum.
If \hbox{$(J,\nu)\!\in\!\cH_{k,l;G}^{\om,\phi}$} is generic, 
then $\bp$ is a regular value of~\eref{evGadfn_e} 
and the set $\ev_{\Ga;\bh}^{-1}(\bp)$ is finite.
Furthermore,
\BE{Cdecomp_e}\begin{split}
&\big|\ev_{\Ga;\bh}^{-1}(\bp)\big|_{\fo_{\Ga;\os;\bh}}^{\pm}
=2^{l-l'}\lr{k}_{\Ga}\!\!\!\!\!\!
\sum_{\begin{subarray}{c}B'\in H_2(X)\\ \fd(B')=B\end{subarray}}\!\!\!\!\!\!\!
\blr{(h_i)_{i\in[l]},\pt}^X_{B'}+
\blr{(h_i)_{i\in[l]-L_{\C}(\Ga)},\lr{\bh}_{\Ga}}^{\phi,\os}_{B;\wch X^\phi;G}\\
&\hspace{.2in}+
\sum_{\begin{subarray}{c}B_0,B'\in H_2(X)-\{0\}\\ B_0+\fd(B')=B\end{subarray}}
\sum_{L_{\C}(\Ga)\subset L'\subset[l]-L_0(\Ga)}
\hspace{-.48in}2^{|L'-L_{\C}(\Ga)|}\!\Bigg(\!\!
\blr{(h_i)_{i\in L'}}^X_{B'}\blr{(h_i)_{i\in[l]-L'},B'}^{\phi,\os}_{B_0;\wch X^\phi;G}\\
&\hspace{2.9in}
+\!\blr{(h_i)_{i\in L'},B_0}^X_{B'}\blr{(h_i)_{i\in[l]-L'}}^{\phi,\os}_{B_0;\wch X^\phi;G}\!\!\Bigg)\!.
\end{split}\EE
\end{prp}

For dimensional reasons, at most one of the two terms in the last sum in~\eref{Cdecomp_e}
is nonzero
for each fixed pair $(B_0,B')$ of nonzero curve degrees and each fixed subset 
$L'\!\subset\![l]$.

\subsection{Proofs of Proposition~\ref{LiftedRel_prp} and Theorem~\ref{WDVVdim3_thm}}
\label{SolWDVVpf_subs}

We continue with the notation and assumptions of Proposition~\ref{LiftedRel_prp} and just above. 
For $k'$, $l'$, and $L^*$ as in~\eref{klcond_e},
we denote~by
$$\ff_{k',l'}\!:\wh\M_{k,l;L^*}(B;J,\nu;\wch X^\phi)\lra\ov\cM_{k',l'}^{\tau}$$
the composition of~\eref{ffdfn_e} and the quotient map~$q$ in~\eref{qdfn_e}.
For a stratum~$\cS$ of $\ov\M_{k,l}(B;J,\nu;\wch X^\phi)$, let
$$\wh\cS^*=q^{-1}(\cS^*)\subset\wh\M_{k,l;L^*}(B;J,\nu;\wch X^\phi).$$
With the notation as in~\eref{whMstdfn_e}, let
\begin{gather*}
M_{\bh}=H_1\!\times\!\ldots\!\times\!H_{l}, \\
\wh\cZ_{k;\bh}^{\st}(B;J,\nu;\wch X^\phi)=\big\{(\u,y_1,\ldots,y_{l})\!\in\!
\wh\M_{k,l;L^*(\bh)}^{\st}(B;J,\nu;\wch X^\phi)\!\times\!M_{\bh}\!:
\ev_i^+(\u)\!=\!h_i(y_i)~\forall\,i\!\in\![l]\big\}.
\end{gather*}
For $(J,\nu)\!\in\!\cH_{k,l;G}^{\om,\phi}$ generic,
the orientation $\wh\fo_{\os;L^*(\bh)}$ of Lemma~\ref{orient_lmm} and 
the orientation~$\fo_{\bh}$ of~$M_{\bh}$ 
determine an orientation~$\wh\fo_{\os;\bh}$
of $\wh\cZ_{k;\bh}^{\st}\big(B;J,\nu;\wch X^\phi\big)$. Let 
$$\ev_{k;\bh}\!: \wh\cZ_{k;\bh}^{\st}\big(B;J,\nu;\wch X^\phi\big)\lra(\wch X^\phi)^k$$
be the map induced by~\eref{fMevdfn_e}.

We take $(k',l')\!=\!(1,2),(0,3)$ and $\Ups\!\subset\!\ov\cM_{k',l'}^{\tau}$ 
to be the bordered compact hypersurfaces
of Lemmas~\ref{M12rel_lmm} and~\ref{M03rel_lmm} with their co-orientations~$\fo_\Ups^c$.
For a stratum~$\cS$ of $\ov\M_{k,l}(B;J,\nu;\wch X^\phi)$, let 
$$\cS_{\bh,\bp;\Ups}^*\subset \cS^*\!\times\!M_{\bh}\!\times\!\Ups
\quad\hbox{and}\quad
\wh\cS_{\bh}^*\equiv  
\wh\cZ_{k;\bh}^{\st}\big(B;J,\nu;\wch X^\phi\big)\!\cap\!\big(\wh\cS^*\!\times\!M_{\bh}\big)$$
be as in~\eref{pHSinc_e2} and~\eref{pHSinc_e2wh}, respectively, and
$$\wh\cS_{\bh,\bp;\Ups}^*=
\big\{\big(\wh\u,P\big)\!\in\!\wh\cS_{\bh}^*\!\times\!\Ups\!:
\ev_{k;\bh}(\wh\u)\!=\!\bp,\,\ff_{k',l'}(\wh\u)\!=\!P\big\}\,.$$
We establish the next statement at the end of this section.

\begin{lmm}\label{LiftedRel_lmm}
With the assumptions as in Proposition~\ref{LiftedRel_prp}, the~map 
$$ \big(\ev_{k;\bh},\ff_{k',l'}\big)\!: 
\wh\cZ_{k;\bh}^{\st}(B;J,\nu;\wch X^\phi)\lra (\wch X^\phi)^k\!\times\!\ov\cM_{k',l'}^{\tau}$$
is a bordered $\Z_2$-pseudocycle of dimension $3k\!+\!2$ transverse to~\eref{pHSinc_e}.
Furthermore, 
\BE{mainsetup_e2}
\big(\prt\wh\cZ_{k;\bh}^{\st}(B;J,\nu;\wch X^\phi)\!\big)
_{(\ev_{k;\bh},\ff_{k',l'})}\!\!\times\!_{f_{\bp;\Ups}}\Ups
=\bigsqcup_{\cS}\wh\cS_{\bh,\bp;\Ups}^*,\EE
with the union taken over the codimension~1 disk bubbling strata $\cS$ of 
$\ov\M_{k,l}(B;J,\nu;\wch X^\phi)$ that satisfy~\eref{Rdecomp_e0} 
and either~\ref{cS1_it} or~\ref{cS2_it} above Proposition~\ref{Rdecomp_prp}.
\end{lmm}

\begin{proof}[{\bf{\emph{Proof of Proposition~\ref{LiftedRel_prp}}}}]
By the first statement of Lemma~\ref{LiftedRel_lmm} and \cite[Lemma~3.5]{RealWDVV},
\BE{bndsplit_e}\begin{split}
&\big|\wh\cZ_{k;\bh}^{\st}(B;J,\nu;\wch X^\phi)\,_{(\ev_{k;\bh},\ff_{k',l'})}\!\!\times_{f_{\bp;\Ups}}\!
\prt \Ups\big|_{\wh\fo_{\os;\bh},\prt\fo_\Ups^c}^{\pm}\\
&\hspace{.7in}
=(-1)^{\dim\,\Ups}
\big|\big(\prt\wh\cZ_{k;\bh}^{\st}(B;J,\nu;\wch X^\phi)\!\big)\,_{(\ev_{k;\bh},\ff_{k',l'})}
\!\!\times_{f_{\bp;\Ups}}\!\Ups\big|_{\prt\fo_{\os;\bh},\fo_\Ups^c}^{\pm}\,,
\end{split}\EE
By \cite[Lemma~3.3(1)]{RealWDVV} and the choice of~$\Ups$, 
the left-hand side of~\eref{bndsplit_e} equals to the left-hand side
of~\eref{M12lift_e0a} if $(k,l)\!=\!(1,2)$ and of~\eref{M03lift_e0a} if $(k,l)\!=\!(0,3)$.

The right-hand side of~\eref{bndsplit_e} is the signed cardinality of~\eref{mainsetup_e2}
times~$(-1)^{\dim\,\Ups}$ and
$$\big|\wh\cS_{\bh,\bp;\Ups}^*\big|^{\pm}_{\prt\wh\fo_{\os;\bh},\fo^c_\Ups}  
= 2 \big|\cS^*_{\bh,\bp;\Ups}\big|^\pm_{\prt\fo_{\os;\bh},\fo^c_\Ups}$$
for each codimension~1 disk bubbling strata $\cS$ as in~\eref{mainsetup_e2}.
In our case, $\Ups\!\cap\!\ov{S}_1\!=\!\eset$.
If $\cS_{\bh,\bp;\Ups}^*\!\neq\!\eset$ and $\cS$ satisfies~\ref{cS2_it},
this implies that $S\!\neq\!S_1$ and thus $\cS$ satisfies the second condition 
in~\eref{M12lift_e0b} if $(k,l)\!=\!(1,2)$ and 
one of the conditions in~\eref{M03lift_e0b} if $(k,l)\!=\!(0,3)$.
If $\cS$ satisfies~\ref{cS1_it}, then $(k,l)\!=\!(1,2)$ and $\cS$ satisfies
the first condition in~\eref{M12lift_e0b}.
Thus, the right-hand side of~\eref{bndsplit_e} equals to the right-hand side
of~\eref{M12lift_e0a} if $(k,l)\!=\!(1,2)$ and of~\eref{M03lift_e0a} if $(k,l)\!=\!(0,3)$.
\end{proof}

\begin{proof}[{\bf{\emph{Proof of \eref{WDVVodeM12_e}}}}]
For an element $\mu\!\in\!H^{2p}(X)$, let $|\mu|\!\equiv\!p$.
With $N$ as above Theorem~\ref{WDVVdim3_thm}, let \hbox{$\La\!\equiv\!(\Z^{\ge 0})^N$}.
For elements $\la\!\equiv\!(\la_1,\ldots,\la_N)$ and
$\al\!\equiv\!(\al_1,\ldots,\al_N)$ of~$\La$, we define
$$|\la|\equiv\sum_{j=1}^N\la_j, ~~ 
\|\la\|\equiv\sum_{j=1}^N\la_j|\mu^\st_j|, ~~ 
\binom{\la}{\al}\equiv\prod_{j=1}^N\binom{\la_j}{\al_j}\,, ~~
\mu^{\st\la}\equiv \underset{\la_1}{\underbrace{\mu^\st_1,\ldots,\mu^\st_1}},
\ldots,\underset{\la_N}{\underbrace{\mu^\st_N,\ldots,\mu^\st_N}}\,.$$
The ODE~\eref{WDVVodeM12_e}	is equivalent to
\BE{WDVVodeM12_e0}\begin{split}
&\sum_{\begin{subarray}{c}B_0,B'\in H_2(X)\\ B_0+\fd(B')=B\\ \al,\be\in\La,\al+\be=\la\end{subarray}}
\!\!\!\!\!2^{|\al|}\binom{\la}{\al}\!\!
\sum_{i,j\in[N]}\!\!\!\!
\blr{\mu^\st_a,\mu^\st_b,\mu^\st_i,\mu^{\st\al}}^X_{B'}g^{ij}
\blr{\mu^\st_j,\mu^{\st^\be}}^{\phi,\os}_{B_0,k;G}\\
&\hspace{.6in}
+\sum_{\begin{subarray}{c}B_1,B_2\in H_2(X)^{\phi}_-\\ B_1+B_2=B\\ 
k_1,k_2\in\Z^{\ge0},k_1+k_2=k-1\\ \al,\be\in\La,\al+\be=\la\end{subarray}}
\hspace{-.25in}\binom{k\!-\!1}{k_1}\binom{\la}{\al}
\blr{\mu^\st_a,\mu^\st_b,\mu^{\st\al}}^{\phi,\os}_{B_1,k_1;G}
\blr{\mu^{\st\be}}^{\phi,\os}_{B_2,k_2+2;G}\\
&\hspace{.6in}=\sum_{\begin{subarray}{c}B_1,B_2\in H_2(X)^{\phi}_-\\ B_1+B_2=B\\ 
k_1,k_2\in\Z^{\ge0},k_1+k_2=k-1\\ \al,\be\in\La,\al+\be=\la\end{subarray}}
\hspace{-.25in}
\binom{k\!-\!1}{k_1}\binom{\la}{\al}
\blr{\mu^\st_a,\mu^{\st\al}}^{\phi,\os}_{B_1,k_1+1;G}
\blr{\mu^\st_b,\mu^{\st\be}}^{\phi,\os}_{B_2,k_2+1;G}
\end{split}\EE
for all $B\!\in\!H_2(X)^\phi_-$, $k\!\in\!\Z^+$, and $\la\!\in\!\La$.
All summands above vanish unless
\BE{M12dim_e}
\ell_{\om}(B)/2\!-\!k\!+\!|\la|\!+\!1=|\mu^\st_a|\!+\!|\mu^\st_b|\!+\!\|\la\|.\EE
We thus need to establish~\eref{WDVVodeM12_e0} under the assumption that \eref{M12dim_e} holds.

We take $l\!=\!|\la|\!+\!2$ and $\bh$ as in~\eref{bhdfn_e} to be 
an $l$-tuple of $G$-invariant pseudocycles in general position so~that 
\BE{WDVVodeM12_e3}\begin{split}
&\quad\PD_X\big([h_1]_X\big)=\mu^{\st}_a,\qquad \PD_X\big([h_2]_X\big)=\mu^\st_b,\\
&\big|\big\{i\!\in\![l]\!-\![2]\!:\PD_X\big([h_i]_X\big)\!=\!\mu^\st_j\big\}\big|=\la_j
~~\forall\,j\!\in\![N].
\end{split}\EE
Let $L^*(\bh)\!\subset\![l]$ be as in~\eref{bhprpdnf_e}.
By~\eref{M12dim_e}, $k$ satisfies~\eref{Cdecomp_e0}.
Since $k\!\ge\!1$ and $l\!\ge\!2$, \eref{M12lift_e0a} applies.

Let $\cA_1^{\R}$ (resp.~$\cA_2$) be the collection of the codimension~1 disk bubbling strata~$\cS$
of the moduli space $\ov\M_{k,l}(B;J,\nu;\wch{X}^{\phi})$ with $\ep_{L^*(\bh)}(\cS)\!\not\in\!2\Z$ 
that satisfy the first (resp.~second) condition in~\eref{M12lift_e0b}.
Define
$$\big(L_1'(\cS),L_2'(\cS)\big)=\begin{cases}
(L_1(\cS)\!-\!\{1,2\},L_2(\cS)),&\hbox{if}~\cS\!\in\!\cA_1^{\R};\\
(L_1(\cS)\!-\!\{1\},L_2(\cS)\!-\!\{2\}),&\hbox{if}~\cS\!\in\!\cA_2.
\end{cases}$$
By Proposition~\ref{Rdecomp_prp} and~\eref{WDVVodeM12_e3},
\begin{equation*}\begin{split}
&\sum_{\cS\in\cA_1^{\R}}\!\!\!
\big|\cS_{\bh,\bp;\Ups}^*\big|_{\prt\fo_{\os;\bh},\fo^c_\Ups}^{\pm}
=\sum_{\cS\in\cA_1^{\R}}\!\!\!
\blr{\mu_a^{\st},\mu_b^{\st},(h_i)_{i\in L_1'(\cS)}}^{\phi,\os}_{B_1(\cS);\wch X^\phi;G}
\blr{(h_i)_{i\in L_2'(\cS)}}^{\phi,\os}_{B_2(\cS);\wch X^\phi;G}\\
&\qquad\qquad=\sum_{\begin{subarray}{c}B_1,B_2\in H_2(X)^{\phi}_-\\ B_1+B_2=B\\ 
k_1,k_2\in\Z^{\ge0},k_1+k_2=k-1\\ \al,\be\in\La,\al+\be=\la\end{subarray}}
\hspace{-.25in}\binom{k\!-\!1}{k_1}\binom{\la}{\al}
\blr{\mu^\st_a,\mu^\st_b,\mu^{\st\al}}^{\phi,\os}_{B_1,k_1;G}
\blr{\mu^{\st\be}}^{\phi,\os}_{B_2,k_2+2;G}\,.
\end{split}\end{equation*}
The second equality above is obtained by summing over all splittings 
of $B\!\in\!H_2(X)^{\phi}_-$ into~$B_1$ and~$B_2$,
of $\la\!\in\!\La$ into~$\al$ and~$\be$,
each set on the second line in~\eref{WDVVodeM12_e3} into two subsets of cardinalities~$\al_j$
and~$\be_j$,
of $k\!-\!1$ real points into sets of cardinalities~$k_1$ and~$k_2$.
The first real marked point of~$\cS$ goes to the $B_2$-invariant above,
which also gains an additional real marked point;
see the sentence after Proposition~\ref{Rdecomp_prp}.
Similarly,
\begin{equation*}\begin{split}
&\sum_{\cS\in\cA_2}\!\!\!\big|\cS_{\bh,\bp;\Ups}^*\big|_{\prt\fo_{\os;\bh},\fo^c_\Ups}^{\pm}
=-\!\!\!\sum_{\cS\in\cA_2}\!\!\!
\blr{\mu_a^{\st},(h_i)_{i\in L_1'(\cS)}}^{\phi,\os}_{B_1(\cS);\wch X^\phi;G}
\blr{\mu_b^{\st},(h_i)_{i\in L_2'(\cS)}}^{\phi,\os}_{B_2(\cS);\wch X^\phi;G}\\
&\qquad\qquad=-\!\!\!\sum_{\begin{subarray}{c}B_1,B_2\in H_2(X)^{\phi}_-\\ B_1+B_2=B\\ 
k_1,k_2\in\Z^{\ge0},k_1+k_2=k-1\\ \al,\be\in\La,\al+\be=\la\end{subarray}}
\hspace{-.25in}
\binom{k\!-\!1}{k_1}\binom{\la}{\al}
\blr{\mu^\st_a,\mu^{\st\al}}^{\phi,\os}_{B_1,k_1+1;G}
\blr{\mu^\st_b,\mu^{\st\be}}^{\phi,\os}_{B_2,k_2+1;G}\,.
\end{split}\end{equation*}
In this case,  the first real marked point of~$\cS$ goes to the $B_1$-invariant above,
while the $B_2$-invariant still gains an additional real marked point.
Thus,
\BE{WDVVodeM12_e5}\hbox{RHS of~\eref{M12lift_e0a}}=
2\Big(\hbox{RHS of~\eref{WDVVodeM12_e0}}-
\hbox{2nd $\sum$ on LHS of~\eref{WDVVodeM12_e0}}\Big).\EE

By Proposition~\ref{Cdecomp_prp}, 
\begin{equation*}\begin{split}
&\big|\ev_{P^{\pm};\bh}^{-1}(\bp)\big|_{\fo_{P^{\pm};\os;\bh}}^{\pm}
\!\!\!\!=2^{|\la|}\lr{k}_{P^{\pm}}\!\!\!\!\!\!
\sum_{\begin{subarray}{c}B'\in H_2(X)\\ \fd(B')=B\end{subarray}}\!\!\!\!\!\!\!
\blr{\mu_a^{\st},\mu_b^{\st},\mu^{\st\la},\pt}^X_{B'}+
\blr{\mu^{\st\la},\mu_a^\st\mu^\st_b}^{\phi;\os}_{B;\wch X^\phi;G}\\
&+\!\!\!\!\sum_{\begin{subarray}{c}B_0,B'\in H_2(X)-\{0\}\\ B_0+\fd(B')=B\\
\al,\be\in\La,\al+\be=\la\end{subarray}}\!\!\!\!\!\!\!\!\!\!\!\!
2^{|\al|}\binom{\la}{\al}\!\Bigg(\!\!\blr{\mu_a^{\st},\mu_b^{\st},\mu^{\st\al}}^X_{B'}
\blr{\mu^{\st\be},B'}^{\phi,\os}_{B_0;\wch X^\phi;G}
+\!\blr{\mu_a^{\st},\mu_b^{\st},\mu^{\st\al},B_0}^X_{B'}
\blr{\mu^{\st\be}}^{\phi,\os}_{B_0;\wch X^\phi;G}\!\!\Bigg).
\end{split}\end{equation*}
In light of~\eref{M12dim_e}, the invariants 
$\lr{\ldots}^{\phi,\os}_{B_0;\wch X^\phi;G}$ can be replaced by
the invariants $\lr{\ldots}^{\phi,\os}_{B_0;k;G}$. 
We note~that 
\begin{alignat}{2}
\label{WDVVodeM12_e8}
[\pt]_X&=\sum_{i,j\in[N]}\!\!\!\!\PD_X(\mu_i^{\st})g^{ij}
\blr{\mu_j^{\st}}^{\phi,\os}_{0;1;G},
&\quad
\mu_a^\st\mu^\st_b&=
\sum_{i,j\le[N]}\!\!\!\!\blr{\mu_a^\st,\mu^\st_b,\mu_i^{\st}}^X_0g^{ij}\mu_j^{\st}\,,\\
\notag
\frac12\fd(B')&=
\sum_{i,j\in[N]}\!\!\!\!\blr{\mu_i^{\st},B'}g^{ij}\PD_X(\mu_j^{\st}),&\quad
B_0&=\sum_{i,j\in[N]}\!\!\!\!\PD_X(\mu_i^{\st})g^{ij}\blr{\mu_j^{\st},B_0}\,.
\end{alignat}
By the second case of Proposition~\ref{GtauGW_prp}\ref{GtauGWvan_it},
the $B'$ relation above, and the divisor relation for complex GW-invariants,
\BE{WDVVodeM12_e8a}
\blr{\mu_a^{\st},\mu_b^{\st},\mu^{\st\al}}^X_{B'}
\blr{\mu^{\st\be},B'}^{\phi,\os}_{B_0;k;G}
=\sum_{i,j\in[N]}\!\!\!\!
\blr{\mu_a^{\st},\mu_b^{\st},\mu^{\st\al},\mu_i^{\st}}^X_{B'}g^{ij}
\blr{\mu^{\st\be},\mu_j^{\st}}^{\phi,\os}_{B_0;k;G}\,.\EE
By the $B_0$ relation above and the divisor relation~\eref{RdivRel_e},
\BE{WDVVodeM12_e8b}
\blr{\mu_a^{\st},\mu_b^{\st},\mu^{\st\al},B_0}^X_{B'}
\blr{\mu^{\st\be}}^{\phi,\os}_{B_0;k;G}=
\sum_{i,j\in[N]}\!\!\!\!
\blr{\mu_a^{\st},\mu_b^{\st},\mu^{\st\al},\mu_i^{\st}}^X_{B'}g^{ij}
\blr{\mu^{\st\be},\mu_j^{\st}}^{\phi,\os}_{B_0;k;G}\,.\EE
As noted after Proposition~\ref{Cdecomp_prp}, 
at most one of~\eref{WDVVodeM12_e8a} and~\eref{WDVVodeM12_e8b} is nonzero.
Combining~\eref{WDVVodeM12_e8}-\eref{WDVVodeM12_e8b} with the expression
for $|\ev_{P^{\pm};\bh}^{-1}(\bp)|_{\fo_{P^{\pm};\os;\bh}}^{\pm}$,
we obtain
$$\hbox{LHS of~\eref{M12lift_e0a}}=
2\Big(\hbox{1st $\sum$ on LHS of~\eref{WDVVodeM12_e0}}\Big).$$
Along with~\eref{WDVVodeM12_e5}, this gives~\eref{WDVVodeM12_e0}. 
\end{proof}

\begin{proof}[{\bf{\emph{Proof of \eref{WDVVodeM03_e}}}}]
We continue with the notation at the beginning of the proof of~\eref{WDVVodeM12_e}.
For  \hbox{$B\!\in\!H_2(X)^{\phi}_-$}, $k\!\in\!\Z^{\ge0}$, \hbox{$a,b,c\!\in\![N]$},
and $\la\!\in\!\La$, define
\begin{equation*}\begin{split}
\Psi_{a,b;c}^{B,k}(\la)=&
\sum_{\begin{subarray}{c}B_0,B'\in H_2(X)\\ B_0+\fd(B')=B\\ \al,\be\in\La,\al+\be=\la\end{subarray}}
\!\!\!\!\!2^{|\al|}\binom{\la}{\al}\!\!
\sum_{i,j\in[N]}\!\!\!\!
\blr{\mu^\st_a,\mu^\st_b,\mu^\st_i,\mu^{\st\al}}^X_{B'}g^{ij}
\blr{\mu^\st_j,\mu^\st_c,\mu^{\st^\be}}^{\phi,\os}_{B_0,k;G}\\
&+\sum_{\begin{subarray}{c}B_1,B_2\in H_2(X)^{\phi}_-\\ B_1+B_2=B\\ 
k_1,k_2\in\Z^{\ge0},k_1+k_2=k\\ \al,\be\in\La,\al+\be=\la\end{subarray}}
\hspace{-.2in}
\binom{k}{k_1}\binom{\la}{\al}
\blr{\mu^\st_a,\mu^\st_b,\mu^{\st\al}}^{\phi,\os}_{B_1,k_1;G}
\blr{\mu^\st_c,\mu^{\st\be}}^{\phi,\os}_{B_2,k_2+1;G}\,.
\end{split}\end{equation*}
The ODE~\eref{WDVVodeM03_e}	is equivalent to
\BE{WDVVodeM03_e0}\Psi_{a,b;c}^{B,k}(\la)=\Psi_{a,c;b}^{B,k}(\la)\EE
for all $B\!\in\!H_2(X)^\phi_-$, $k\!\in\!\Z^{\ge0}$, and $\la\!\in\!\La$.
Both sides of~\eref{WDVVodeM03_e0} vanish unless
\BE{M03dim_e}
\ell_{\om}(B)/2\!-\!k\!+\!|\la|\!+\!2=|\mu^\st_a|\!+\!|\mu^\st_b|\!+\!|\mu^\st_c|\!+\!\|\la\|.\EE
We thus need to establish~\eref{WDVVodeM03_e0} under the assumption~\eref{M03dim_e} holds.  

We take $l\!=\!|\la|\!+\!3$ and $\bh$ as in~\eref{bhdfn_e} to be 
an $l$-tuple of $G$-invariant pseudocycles in general position so~that 
\BE{WDVVodeM03_e3}\begin{split}
&\PD_X\big([h_1]_X\big)=\mu^{\st}_a,\quad \PD_X\big([h_2]_X\big)=\mu^\st_b,
\quad \PD_X\big([h_3]_X\big)=\mu^\st_c,\\
&\qquad\quad\big|\big\{i\!\in\![l]\!-\![3]\!:\PD_X\big([h_i]_X\big)\!=\!\mu^\st_j\big\}\big|=\la_j
~~\forall\,j\!\in\![N].
\end{split}\EE
Let $L^*(\bh)\!\subset\![l]$ be as in~\eref{bhprpdnf_e}.
By~\eref{M03dim_e}, $k$ satisfies~\eref{Cdecomp_e0}.
Since $l\!\ge\!3$, \eref{M03lift_e0a} applies.

Let $\cA_2$ (resp.~$\cA_3$) be the collection of the codimension~1 disk bubbling strata~$\cS$
of the moduli space $\ov\M_{k,l}(B;J,\nu;\wch{X}^{\phi})$ with $\ep_{L^*(\bh)}(\cS)\!\not\in\!2\Z$ 
that satisfy the first (resp.~second) condition in~\eref{M03lift_e0b}.
Define
$$\big(L_1'(\cS),L_2'(\cS)\big)=\begin{cases}
(L_1(\cS)\!-\!\{1,3\},L_2(\cS)\!-\!\{2\}),&\hbox{if}~\cS\!\in\!\cA_2;\\
(L_1(\cS)\!-\!\{1,2\},L_2(\cS)\!-\!\{3\}),&\hbox{if}~\cS\!\in\!\cA_3.
\end{cases}$$
By Proposition~\ref{Rdecomp_prp} and~\eref{WDVVodeM03_e3},
\begin{equation*}\begin{split}
&\sum_{\cS\in\cA_2}\!\!\!
\big|\cS_{\bh,\bp;\Ups}^*\big|_{\prt\fo_{\os;\bh},\fo^c_\Ups}^{\pm}
=-\!\!\sum_{\cS\in\cA_2}\!\!\!
\blr{\mu_a^{\st},\mu_c^{\st},(h_i)_{i\in L_1'(\cS)}}^{\phi,\os}_{B_1(\cS);\wch X^\phi;G}
\blr{\mu_b^{\st},(h_i)_{i\in L_2'(\cS)}}^{\phi,\os}_{B_2(\cS);\wch X^\phi;G}\\
&\qquad\qquad=-\!\!\!\!\!\sum_{\begin{subarray}{c}B_1,B_2\in H_2(X)^{\phi}_-\\ B_1+B_2=B\\ 
k_1,k_2\in\Z^{\ge0},k_1+k_2=k\\ \al,\be\in\La,\al+\be=\la\end{subarray}}
\hspace{-.2in}\binom{k}{k_1}\binom{\la}{\al}
\blr{\mu^\st_a,\mu^\st_c,\mu^{\st\al}}^{\phi,\os}_{B_1,k_1;G}
\blr{\mu_b^{\st},\mu^{\st\be}}^{\phi,\os}_{B_2,k_2+1;G}\,.
\end{split}\end{equation*}
The second equality above is obtained by summing over all splittings 
of $B\!\in\!H_2(X)^{\phi}_-$ into~$B_1$ and~$B_2$,
of $\la\!\in\!\La$ into~$\al$ and~$\be$,
each set on the second line in~\eref{WDVVodeM03_e3} into two subsets of cardinalities~$\al_j$
and~$\be_j$,
of $k$~real points into sets of cardinalities~$k_1$ and~$k_2$.
The $B_2$-invariant above gains an additional real marked point;
see the sentence after Proposition~\ref{Rdecomp_prp}.
Similarly,
\begin{equation*}\begin{split}
&\sum_{\cS\in\cA_3}\!\!\!\big|\cS_{\bh,\bp;\Ups}^*\big|_{\prt\fo_{\os;\bh},\fo^c_\Ups}^{\pm}
=\sum_{\cS\in\cA_3}\!\!\!
\blr{\mu_a^{\st},\mu_b^{\st},(h_i)_{i\in L_1'(\cS)}}^{\phi,\os}_{B_1(\cS);\wch X^\phi;G}
\blr{\mu_c^{\st},(h_i)_{i\in L_2'(\cS)}}^{\phi,\os}_{B_2(\cS);\wch X^\phi;G}\\
&\qquad\qquad=\sum_{\begin{subarray}{c}B_1,B_2\in H_2(X)^{\phi}_-\\ B_1+B_2=B\\ 
k_1,k_2\in\Z^{\ge0},k_1+k_2=k\\ \al,\be\in\La,\al+\be=\la\end{subarray}}
\hspace{-.2in}
\binom{k}{k_1}\binom{\la}{\al}
\blr{\mu^\st_a,\mu_b^{\st},\mu^{\st\al}}^{\phi,\os}_{B_1,k_1;G}
\blr{\mu^\st_c,\mu^{\st\be}}^{\phi,\os}_{B_2,k_2+1;G}\,.
\end{split}\end{equation*}
In this case, the $B_2$-invariant still gains an additional real marked point.
Thus,
\BE{WDVVodeM03_e5}\hbox{RHS of~\eref{M03lift_e0a}}=
2\Big(\hbox{2nd $\sum$ in~$\Psi_{a,b;c}^{B,k}$}-
\hbox{2nd $\sum$ in~$\Psi_{a,c;b}^{B,k}$}\Big).\EE

By Proposition~\ref{Cdecomp_prp}, 
\begin{equation*}\begin{split}
&\big|\ev_{\Ga_2^{\pm};\bh}^{-1}(\bp)\big|_{\fo_{\Ga_2^{\pm};\os;\bh}}^{\pm}
\!\!\!\!=\blr{\mu_b^\st,\mu^{\st\la},\mu_a^\st\mu^\st_c}^{\phi;\os}_{B;\wch X^\phi;G}\\
&\hspace{1in}+\!\!\!\!\sum_{\begin{subarray}{c}B_0,B'\in H_2(X)-\{0\}\\ B_0+\fd(B')=B\\
\al,\be\in\La,\al+\be=\la\end{subarray}}\!\!\!\!\!\!\!\!\!\!\!\!
2^{|\al|}\binom{\la}{\al}\!\Bigg(\!\!\blr{\mu_a^{\st},\mu_c^{\st},\mu^{\st\al}}^X_{B'}
\blr{\mu_b^\st,\mu^{\st\be},B'}^{\phi,\os}_{B_0;\wch X^\phi;G}\\
&\hspace{2.5in}
+\!\blr{\mu_a^{\st},\mu_c^{\st},\mu^{\st\al},B_0}^X_{B'}
\blr{\mu_b^\st,\mu^{\st\be}}^{\phi,\os}_{B_0;\wch X^\phi;G}\!\!\Bigg).
\end{split}\end{equation*}
The number $|\ev_{\Ga_3^{\pm};\bh}^{-1}(\bp)|_{\fo_{\Ga_3^{\pm};\os;\bh}}^{\pm}$ 
is given by the same expression with~$b$ and~$c$ interchanged.
In light of~\eref{M03dim_e}, the invariants 
$\lr{\ldots}^{\phi,\os}_{B_0;\wch X^\phi;G}$ can be replaced by
the invariants $\lr{\ldots}^{\phi,\os}_{B_0;k;G}$. 
Combining these statements with \eref{WDVVodeM12_e8}-\eref{WDVVodeM12_e8b}, 
we obtain
$$\hbox{LHS of~\eref{M03lift_e0a}}=
2\Big(\hbox{1st $\sum$ in~$\Psi_{a,c;b}^{B,k}$}-
\hbox{1st $\sum$ in~$\Psi_{a,c;b}^{B,k}$}\Big).$$
Along with~\eref{WDVVodeM03_e5}, this gives~\eref{WDVVodeM03_e0}. 
\end{proof}

\begin{proof}[{\bf{\emph{Proof of Lemma~\ref{LiftedRel_lmm}}}}]
Let $L^*\!=\!L^*(\bh)$ be as in~\eref{bhprpdnf_e}.
For the purposes of the first statement of this lemma, it is sufficient to show~that 
\BE{evfkl_e2}   \big(\ev,\ff_{k',l'}\big)\!: 
\wh\M_{k,l;L^*}^{\st}(B;J,\nu)\lra \wch X_{k,l}\!\times\!\ov\cM_{k',l'}^{\tau}\EE
is a bordered $\Z_2$--pseudocycle of dimension $3k\!+\!2\,\codim_{\C}\bh\!+\!2$ transverse to
\BE{pHSinc_e4}
f_{\bh;\bp;\Ups}\!: M_{\bh}\!\times\!\Ups\lra \wch X_{k,l}\!\times\!\ov\cM_{k',l'}^{\tau},
\quad
f_{\bh;\bp;\Ups}\big((y_i)_{i\in[l]},P)=\big(\bp,(h_i(y_i))_{i\in[l]},P\big).\EE
We omit the proof of this statement since it is a direct adaptation
of the proof of \cite[Lemma~5.9]{RealWDVV}.
By the first statement of Lemma~\ref{LiftedRel_lmm}, 
$$\big(\prt\wh\cZ_{k;\bh}^{\st}(B;J,\nu;\wch X^\phi)\!\big)
_{(\ev_{k;\bh},\ff_{k',l'})}\!\!\times\!_{f_{\bp;\Ups}}\Ups
=\bigsqcup_{\cS}\wh\cS_{\bh,\bp;\Ups}^*,$$
with the union taken over the codimension~1 disk bubbling strata $\cS$ of 
$\ov\M_{k,l}(B;J,\nu;\wch X^\phi)$ that lie in the image of the boundary 
of $\wh\M_{k,l}(B;J,\nu;\wch X^\phi)$ under the projection~\eref{qdfn_e}.

For a codimension~1 disk bubbling stratum~$\cS$ of $\ov\M_{k,l}(B;J,\nu;\wch X^\phi)$ and $r\!=\!1,2$,
let 
$$K_r(\cS)\subset[k], \qquad L_r^*(\cS)\subset L_r(\cS)\subset[l],
\qquad\hbox{and}\qquad B_r(\cS)\in H_2(X) $$
be as in~\eref{KrLrcSdfn_e} and~\eref{KrLrcSdfn2_e}.
We~set
\BE{JakePseudo_e8} K_r=K_r(\cS), ~~ k_r=|K_r|, ~~
L_r=L_r(\cS), ~~ l_r=|L_r(\cS)|,~~  
L_r^*=L_r^*(\cS), ~~ B_r=B_r(\cS).\EE
We note that 
\BE{JakePseudo_e8b}\begin{split}
&k_1\!+\!k_2=k, \quad l_1\!+\!l_2=l, \quad 
\codim_{\C}\bh=\codim_{\C}(h_i)_{i\in L_1}\!+\!\codim_{\C}(h_i)_{i\in L_2};\\
&\hspace{.5in}\ell_{\om}(B_1)\!+\!\ell_{\om}(B_2)= 
\ell_{\om}(B)=2\big(k\!+\!\codim_{\C}\bh\!-\!l\!+\!1\big).
\end{split}\EE

Suppose that $\wh\cS^*$ is a stratum of $\prt\wh\M_{k,l;L^*}^{\st}(B;J,\nu;\wch X^\phi)$, i.e.
$$\ve_{L^*}(\cS)\equiv 
\frac{\ell_{\om}(B_2)}{2}\!-\!k_2\!-\!\big(l_2\!-\!|L_2^*|\big)\not\in2\Z,$$ 
and $\wh\cS_{\bh,\bp;\Ups}^*\!\neq\!\eset$.
By the definition of $L^*\!=\!L^*(\bh)$ in~\eref{bhprpdnf_e} and 
the above condition on~$\ve_{L^*}(\cS)$,
\BE{JakePseudo_e8c}\frac{\ell_{\om}(B_2)}{2}\!-\!k_2\!-\!
\big(\codim_{\C}(h_i)_{i\in L_2}\!-\!l_2\big)\not\in2\Z.\EE

Since $\wh\cS_{\bh,\bp;\Ups}^*\!\neq\!\eset$ and $\Ups\!\cap\!S_1\!=\!\eset$, 
$(l_1,k_1)\!\neq\!(1,0)$.
If $B_2\!=\!0$, $l_2,|L_2^*|\!=\!1$, and $k_2\!=\!0$, then \hbox{$\ve_{L^*}(\cS)\!=\!0$},
contrary to the assumption on~$\cS$ above.
Suppose $B_2\!=\!0$, $l_2\!=\!1$, and $|L_2^*|,k_2\!=\!0$.
For good choices of~$\nu$ (still sufficiently generic), the restriction to~$\wh\cS^*$
of~\eref{whfMevdfn_e} then factors~as 
$$\wh\cS^*\lra  \M_{k+1,l-1}\big(B;J,\nu_1;\wch X^\phi\big)\!\times\!\M_{1,1}(0;J,0)
\lra \wch X_{k,l-1}\!\times\!\wch X^{\phi} \lra \wch X_{k,l}\,.$$
Thus, $\wh\cS_{\bh,\bp;\Ups}^*\!=\!\eset$ for generic choices of~$\bh$ and~$\bp$.
Suppose $B_2\!=\!0$, $l_2\!=\!0$, and $k_2\!=\!2$.
For good choices of~$\nu$, the restriction to~$\wh\cS^*$
of~\eref{whfMevdfn_e} then factors~as 
$$\wh\cS^*\lra  \M_{k-1,l}\big(B;J,\nu_1;\wch X^\phi\big)\!\times\!\M_{3,0}(0;J,0)
\lra \wch X_{k-2,l}\!\times\!\De_{\wch X^\phi} \lra \wch X_{k,l}\,,$$
where $\De_{\wch X^\phi}\!\subset\!(\wch X^{\phi})^2$ is the diagonal.
Thus, $\wh\cS_{\bh,\bp;\Ups}^*\!=\!\eset$ for generic choices of~$\bh$ and~$\bp$.

We can thus assume that either $B_r\!\neq\!0$ or $2l_r\!+\!k_r\!\ge\!3$ for $r\!=\!1,2$.
For good choices of~$\nu$, the restriction to~$\wh\cS^*$ of~\eref{whfMevdfn_e} then factors~as 
\begin{equation*}\begin{split}
\wh\cS^*&\lra  \M_{k_1+1,l_1}\!\big(B_1;J,\nu_1;\wch X^\phi\big)\!\times\!
\M_{k_2+1,l_2}\!\big(B_2;J,\nu_2;\wch X^\phi\big)\\
&\lra\M_{k_1,l_1}\!\big(B_1;J,\nu_1';\wch X^\phi\big)\!\times\!
\M_{k_2,l_2}\big(B_2;J,\nu_2';\wch X^\phi\big)
\lra \wch X_{k_1,l_1}\!\times\!\wch X_{k_2,l_2} \lra \wch X_{k,l}\,.
\end{split}\end{equation*}
Thus, $\wh\cS_{\bh,\bp;\Ups}^*\!=\!\eset$ for generic choices of $\bh$, $\bp$, and $(J,\nu)$ unless
$$\ell_{\om}\big(B_r\big)\!+\!2l_r\!+\!k_r\ge 
3k_r\!+\!2\,\codim_{\C}(h_i)_{i\in L_r}\qquad\forall\,r\!=\!1,2.$$
Along with~\eref{JakePseudo_e8b}, this implies that either
\begin{gather}\label{mainsetup_e25}
\ell_{\om}(B_1)=2\big(k_1\!+\!\codim_{\C}(h_i)_{i\in L_1}\!-\!l_1\big),
~~\ell_{\om}(B_2)=2\big(k_2\!+\!\codim_{\C}(h_i)_{i\in L_2}\!-\!l_2\!+\!1\big),\\
\notag
\hbox{or}\quad
\ell_{\om}(B_1)=2\big(k_1\!+\!\codim_{\C}(h_i)_{i\in L_1}\!-\!l_1\!+\!1\big),
~~\ell_{\om}(B_2)=2\big(k_2\!+\!\codim_{\C}(h_i)_{i\in L_2}\!-\!l_2\big).
\end{gather}
In light of~\eref{JakePseudo_e8c}, \eref{mainsetup_e25} is the case.
By the definition of $L^*\!=\!L^*(\bh)$ in~\eref{bhprpdnf_e},
the second equation in~\eref{mainsetup_e25} is equivalent to~\eref{Rdecomp_e0}.
Thus, the union in~\eref{mainsetup_e2} is over the codimension~1 disk bubbling strata $\cS$ of 
$\ov\M_{k,l}(B;J,\nu;\wch X^\phi)$ that satisfy~\eref{Rdecomp_e0}.
 
If $\cS$ satisfies~\ref{cS0_it} above Proposition~\ref{Rdecomp_prp},
the restriction to~$\wh\cS^*$ of the composition of~\eref{evfkl_e2}
with the projection to the product
\hbox{$X_{k_1,l_1}\!\times\!\ov\cM_{k',l'}^{\tau}$} factors~as 
\begin{equation*}\begin{split}
\wh\cS^*&\lra  \M_{k_1+1,l_1}\!\big(B_1;J,\nu_1;\wch X^\phi\big)\!\times\!\M_{k_2+1,l_2}\!\big(B_2;J,\nu_2;\wch X^\phi\big)\\
&\lra\M_{k_1,l_1}\!\big(B_1;J,\nu_1';\wch X^\phi\big)\lra\wch X_{k_1,l_1}\!\times\!\ov\cM_{k',l'}^{\tau}.
\end{split}\end{equation*}
Since the restriction of~\eref{evfkl_e2}
to~$\wh\cS^*$ is transverse to~\eref{pHSinc_e4}
and $\Ups$ is a real hypersurface, \eref{mainsetup_e25} then implies that 
$\wh\cS_{\bh,\bp;\Ups}^*\!=\!\eset$.
\end{proof}

\section{Proofs of structural statements}
\label{proofs_sec}

\subsection{Orienting the linearized $\dbar$-operator}
\label{CRdet_subs}

For $\u$ as in~\eref{udfn_e}, let 
\begin{equation*}\begin{split}
D_{J,\nu;\u}^{\phi}\!:  \Ga(\u)
\equiv&\big\{\xi\!\in\!\Ga(\Si;u^*TX)\!:\,\xi\!\circ\!\si\!=\!\nd\phi\!\circ\!\xi\big\}\\
&\lra
\Ga^{0,1}(\u)\equiv
\big\{\ze\!\in\!\Ga(\Si;(T^*\Si,\fj)^{0,1}\!\otimes_{\C}\!u^*(TX,J))\!:\,
\ze\!\circ\!\nd\si=\nd\phi\!\circ\!\ze\big\}
\end{split}\end{equation*}
be the linearization of the $\{\dbar_J\!-\!\nu\}$-operator on the space of 
real maps from~$(\Si,\si)$ with its complex structure~$\fj$.
We define
$$\la_{\u}\big(D_{J,\nu}^{\phi}\big)=\det D_{J,\nu;\u}^{\phi}.$$
By~\cite[Appendix]{RealGWsI}, the projection
\BE{Dorient_e} \la\big(D_{J,\nu}^{\phi}\big)\equiv
\bigcup_{\u\in\ov\M_{k,l}(B;J,\nu;\wch X^\phi)}\hspace{-.32in}
\big\{\u\}\!\times\!\la_{\u}\big(D_{J,\nu}^{\phi}\big)\lra \ov\M_{k,l}(B;J,\nu;\wch X^\phi)\EE
is a line orbi-bundle with respect to a natural topology on its domain.

The next statement is a consequence of the orienting construction of \cite[Prop.~3.1]{Jake},
a more systematic perspective of which appears in the proof 
of \cite[Thm.~7.1]{SpinPin}.

\begin{lmm}\label{Dorient_lmm} 
Suppose $(X,\om,\phi)$ is a real symplectic sixfold, 
$\wch X^\phi$ is a connected component of $X^\phi$, 
$l\!\in\!\Z^+$, $k\!\in\!\Z^{\ge0}$ with $k\!+\!2l\!\ge\!3$,
$B\!\in\!H_2(X)$, and $(J,\nu)\!\in\!\cH_{k,l;\{1\}}^{\om,\phi}$.
An $\OSpin$-structure~$\os$ on~$\wch X^{\phi}$
determines an orientation~$\fo_{\os}^D$ on the restriction of $\la(D_{J,\nu}^{\phi})$ to 
$\M_{k,l}(B;J,\nu;\wch X^\phi)$ with the following properties:  
\BEnum{($\fo_{\os}^D\arabic*$)}

\item  the interchange of two real points $x_i$ and $x_j$ preserves $\fo_{\os}^D$;

\item if $\u\!\in\!\M_{k,l}(B;J,\nu;\wch X^\phi)$ and 
the marked points $z_i^+$ and $z_j^+$ are not separated by the fixed locus~$S^1$
of the domain of~$\u$, then
the interchange of the conjugate pairs $(z_i^+,z_i^-)$ and $(z_j^+,z_j^-)$
preserves~$\fo_{\os}^D$ at~$\u$; 

\item  the interchange of the points in a conjugate pair $(z_i^+,z_i^-)$ with $1\!<\!i\!\le\!l$
preserves~$\fo_{\os}^D$;

\item\label{halfchoiceD_it} the interchange of the points in the conjugate pair $(z_1^+,z_1^-)$
preserves $\fo_{\os}^D$  if and only~if $\ell_\om(B)/2$ is even;

\item\label{Dev_it} if $k,l\!=\!1$, $B\!=\!0$, and $\nu$ is small, 
then $\fo_{\os}^D$ is the orientation induced
by the evaluation at~$x_1$ and the orientation of $\wch X^\phi$ determined by~$\os$;

\item\label{Dosch_it} if $\os'$ is another $\OSpin$-structure on~$\wch X^{\phi}$,
$\u\!\in\!\M_{k,l}(B;J,\nu;\wch X^\phi)$ is as in~\eref{udfn_e},
and the pullbacks of~$\os'$ and~$\ov\os$ by the restriction of~$u$ to the fixed locus
of the domain are the same, then
the orientations $\fo_{\os}^D$ and $\fo_{\os'}^D$ at~$\u$ are opposite.

\EEnum
\end{lmm}

\begin{proof}
Let $\u$ be as in~\eref{udfn_e}.
For the purposes of applying \cite[Thm.~7.1]{SpinPin},
we take the distinguished half-surface \hbox{$\D^2\!\subset\!\P^1$} to be the disk 
so that $\prt\D^2$ is the fixed locus~$S^1$ of~$\tau$ and $z_1^+\!\in\!\D^2$.
An $\OSpin$-structure~$\os$ on~$\wch X^{\phi}$ then determines an orientation~$\fo_{\os}^D$
on the line  \hbox{$\la_{\u}(D_{J,\nu}^{\phi})$}
varying continuously with~$\u$.
The first three properties of this lemma are clear, since $\fo^D_\os$ does not depend on the marked points, except for the conjugate pair~$z^\pm_1$ which determines $\D^2$. 
By the CROrient~1$\os$(1) properties in \cite[Section~7.2]{SpinPin}, 
$\fo^D_\os$ satisfies~\ref{halfchoiceD_it} and~\ref{Dosch_it}, respectively.
By the CROrient~5a and~6a properties in \cite[Section~7.2]{SpinPin},
it also satisfies~\ref{Dev_it}.
\end{proof}

Suppose now that $l\!\in\!\Z^+$ and $\cS$ is an open codimension~1 disk 
bubbling stratum of $\ov\M_{k,l}(B;J,\nu;\wch X^\phi)$.
An orientation~$\fo_{\cS;\u}^c$ of~$\cN_{\u}\cS$ determines a direction of degeneration 
of elements of $\M_{k,l}(B;J,\nu;\wch X^\phi)$ to~$\u$.
The orientation~$\fo_{\os}^D$ on~\eref{Dorient_e} limits to an orientation 
$\fo_{\os;\u}^D$ of $\la_{\u}(D_{J,\nu}^{\phi})$ by approaching~$\u$ from this direction.
The orientation $\fo_{\os;\u}^D$ is called the \sf{limiting} orientation induced by~$\os$ 
and~$\fo_{\cS;\u}^c$ in \cite[Section~7.3]{SpinPin}.
If in addition $L^*\!\subset\![l]$ and $L_1^*(\cS),L_2^*(\cS)\!\neq\!\eset$,  
the possible orientations~$\fo_{\cS;\u}^{c;\pm}$ of~$\cN_{\u}\cS$ 
are distinguished as above Lemma~\ref{DMboundary_lmm}. 
We denote by $\fo_{\os;\u}^{D;\pm}$ the limiting orientation of $\la_{\u}(D_{J,\nu}^{\phi})$
induced by~$\os$ and~$\fo_{\cS;\u}^{c;\pm}$.

For good choices of~$\nu$, there is a natural embedding
\BE{cSsplit_e}\cS \lhra{~~~} 
\M_{\{0\}\sqcup K_1(\cS),L_1(\cS)}\big(B_1(\cS);J,\nu_1;\wch X^\phi\big)\!\times\!
\M_{\{0\}\sqcup K_2(\cS),L_2(\cS)}\big(B_2(\cS);J,\nu_2;\wch X^\phi\big).\EE
If $|K_1(\cS)|\!+\!2|L_1(\cS)|\!\ge\!3$, there is also a forgetful morphism
\BE{cSsplit_e2a}
\ff_{\nod}\!:\M_{\{0\}\sqcup K_1(\cS),L_1(\cS)}\big(B_1(\cS);J,\nu_1;\wch X^\phi\big)
\lra \M_{K_1(\cS),L_1(\cS)}\big(B_1(\cS);J,\nu_1';\wch X^\phi\big)\EE
dropping the real marked point corresponding to the nodal point~$\nod$ on the first component.
If $|K_2(\cS)|\!+\!2|L_2(\cS)|\!\ge\!3$, there is then a forgetful morphism
\BE{cSsplit_e2b}
\ff_{\nod}\!:\M_{\{0\}\sqcup K_2(\cS),L_2(\cS)}\big(B_2(\cS);J,\nu_2;\wch X^\phi\big)
\lra \M_{K_2(\cS),L_2(\cS)}\big(B_2(\cS);J,\nu_2';\wch X^\phi\big)\EE
dropping the real marked point corresponding to the nodal point~$\nod$ on the first component.

For an element $\u\!\in\!\cS$, we denote~by  
$$\u_1\in\M_{\{0\}\sqcup K_1(\cS),L_1(\cS)}\big(B_1(\cS);J,\nu_1;\wch X^\phi\big)
\quad\hbox{and}\quad 
\u_2\in\M_{\{0\}\sqcup K_2(\cS),L_2(\cS)}\big(B_2(\cS);J,\nu_2;\wch X^\phi\big)$$
the pair of maps corresponding to~$\u$ via~\eref{cSsplit_e}.
Let
$$\u_1'\in\M_{K_1(\cS),L_1(\cS)}\big(B_1(\cS);J,\nu_1';\wch X^\phi\big)
\quad\hbox{and}\quad 
\u_2'\in\M_{K_2(\cS),L_2(\cS)}\big(B_2(\cS);J,\nu_2';\wch X^\phi\big)$$
be the images of~$\u_1$ and~$\u_2$ under~\eref{cSsplit_e2a} and~\eref{cSsplit_e2b}.
For $r\!=\!1,2$, the determinants $\la_{\u_r}(D_{J,\nu_r}^{\phi})$ and 
$\la_{\u_r'}(D_{J,\nu'_r;\u_r'}^{\phi})$ are canonically the same. 

For each $\u\!\in\!\cS$, the exact sequence
\begin{equation}\label{Dses_e}
0\lra D_{J,\nu;\u}^{\phi}\lra 
D_{J,\nu_1;\u_1}^{\phi}\!\oplus\!D_{J,\nu_2;\u_2}^{\phi}\lra T_{u(\nod)}\wch X^{\phi}\lra0, \quad 
\big(\xi_1,\xi_2\big)\lra \xi_2(\nod)\!-\!\xi_1(\nod),
\end{equation}
of Fredholm operators determines an isomorphism 
\BE{Du0isom_e}
\la_{\u}\big(D_{J,\nu}^{\phi}\big)\!\otimes\!\la\big(T_{u(\nod)}\wch X^{\phi}\big)
\approx \la_{\u_1}\big(D_{J,\nu_1}^{\phi}\big)\!\otimes\!\la_{\u_2}\big(D_{J,\nu_2}^{\phi}\big).\EE
If $L_1^*(\cS),L_2^*(\cS)\!\neq\!\eset$ with the smallest elements~$i_1^*$ and~$i_2^*$,
respectively, 
an $\OSpin$-structure~$\os$ on~$\wch X^{\phi}$ 
determines orientations~$\fo_{\os}^{D_1}$ on $\la_{\u_1}(D_{J,\nu_1}^{\phi})$ 
and~$\fo_{\os}^{D_2}$ on $\la_{\u_2}(D_{J,\nu_2}^{\phi})$ via identifications of 
$(L_1(\cS),i_1^*)$ with $(|L_1(\cS)|,1)$ and of $(L_2(\cS),i_2^*)$ with $(|L_2(\cS)|,1)$.
By the first two statements of Lemma~\ref{Dorient_lmm}, these orientations do not depend 
on these identification or on identifications of $\{0\}\!\sqcup\!K_1(\cS)$ with $1\!+\!|K_1(\cS)|$
and $\{0\}\!\sqcup\!K_2(\cS)$ with $1\!+\!|K_2(\cS)|$.
Combining the orientations $\fo_{\os}^{D_1}$ on $\la_{\u_1}(D_{J,\nu_1}^{\phi})$, 
and~$\fo_{\os}^{D_2}$ on $\la_{\u_2}(D_{J,\nu_2}^{\phi})$, and 
the orientation on~$\wch X^\phi$ determined $\os$, 
we obtain an orientation $\fo^D_{\os;\u}$ on $\la_{\u}(D_{J,\nu}^{\phi})$
via~\eref{Du0isom_e}. 

\begin{lmm}\label{DorientComp_lmm}
Suppose $(X,\om,\phi)$, $\wch{X}^{\phi}$, $\os$, $k,l,B$, and $(J,\nu)$ 
are as in Lemma~\ref{Dorient_lmm}, 
$\cS$ is  a codimension~1 disk bubbling stratum of 
$\ov\M_{k,l}(B;J,\nu;\wch X^\phi)$ with $L_2^*(\cS)\!\neq\!\eset$, 
and $\u\!\in\!\cS$.
The orientations~$\fo_{\os}^{D;+}$ and~$\fo_{\os}^D$ on $\la_{\u}(D_{J,\nu}^{\phi})$ are the same; 
the orientations~$\fo_{\os}^{D;-}$ and~$\fo_{\os}^D$ on $\la_{\u}(D_{J,\nu}^{\phi})$ are the same
if and only~if $\ell_\om(B_2(\cS))/2$ is even. 
\end{lmm}

\begin{proof} 
In the terminology of \cite[Section~7.4]{SpinPin},
$\fo_{\os}^D$  is the \sf{split} orientation of~$D_{J,\nu;\u}^{\phi}$.
Thus, the first comparison is a special case of \cite[Cor.~7.4(a)]{SpinPin}.
The second comparison follows from the first and 
Lemma~\ref{Dorient_lmm}\ref{halfchoiceD_it}. 
\end{proof}

\subsection{Proofs of Lemmas~\ref{orient_lmm} and~\ref{orient_lmm2} and 
Propositions~\ref{GtauGW_prp}, \ref{SNXphi_prp}, and~\ref{JakePseudo_prp}}
\label{orient_subs}

Let $(X,\om,\phi)$, $\wch{X}^{\phi}$, $\os,k,l,L^*,B$, and $(J,\nu)$
be as in Lemma~\ref{orient_lmm}.
The exact sequences
$$0\lra\ker D_{J,\nu;\u}^{\phi}\lra T_{\u}\M_{k,l}^*(B;J,\nu;\wch X^\phi)\lra
T_{\ff_{k,l}(\u)}\cM_{k,l}^{\tau}\lra0$$
with \hbox{$\u\in\M_{k,l}^*(B;J,\nu;\wch X^\phi)$} induced by the forgetful morphism~$\ff_{k,l}$ 
determine an isomorphism
\BE{OrientSubs_e3}
\la\big(\M_{k,l}^*(B;J,\nu;\wch X^\phi)\big)\approx \la\big(D_{J,\nu}^{\phi}\big)\!\otimes\!
\ff_{k,l}^*\la\big(\cM_{k,l}^{\tau}\big)\EE
of line bundles over $\M_{k,l}^*(B;J,\nu;\wch X^\phi)$.
By Lemma~\ref{Dorient_lmm}, the $\OSpin$-structure~$\os$ on~$\wch X^{\phi}$ induces 
an orientation~$\fo_{\os}^D$ on the first factor on the right-hand side above.
Along with the orientation~$\fo_{k,l;L^*}$ on the second factor defined in 
Section~\ref{cMstrata_subs}, it determines an orientation~$\fo_{\os;L^*}$  
on $\M_{k,l}^*(B;J,\nu;\wch X^\phi)$
via~\eref{OrientSubs_e3}.

\begin{proof}[{\bf{\emph{Proofs of Lemmas~\ref{orient_lmm} and~\ref{orient_lmm2}}}}]
By Lemmas~\ref{cMorient_lmm} and~\ref{Dorient_lmm}, the orientation~$\fo_{\os;L^*}$ 
above satisfies all properties listed in Lemma~\ref{orient_lmm} wherever it is defined.
Every (continuous) extension of~$\fo_{\os;L^*}$ to subspaces of 
$\M_{k,l;L^*}^{\st}(B;J,\nu;\wch X^\phi)$ and $\wh\M_{k,l;L^*}^{\st}(B;J,\nu;\wch X^\phi)$
satisfies the same properties.
The orientation~$\fo_{\os;L^*}$ automatically extends over all strata 
of codimension~2 and higher.
By Lemma~\ref{orient_lmm2}, it extends over the codimension~1 strata of 
$\M_{k,l;L^*}^{\st}(B;J,\nu;\wch X^\phi)$ and $\wh\M_{k,l;L^*}^{\st}(B;J,\nu;\wch X^\phi)$ as well.
Lemma~\ref{orient_lmm2} in turn follows immediately from 
Lemmas~\ref{DMboundary_lmm} and~\ref{DorientComp_lmm}.
\end{proof}
 
\begin{proof}[{\bf{\emph{Proof of Proposition~\ref{JakePseudo_prp}}}}]
We continue with the notation and assumptions of this proposition and just above
and take $L^*\!=\!L^*(\bh)$ as in~\eref{bhprpdnf_e}.

\ref{RGWdfn_it} Let \hbox{$h\!:Z\!\lra\!X^l$} be a smooth map from a manifold of dimension
$6l\!-\!2\,\codim_{\C}\bh\!-\!2$ that covers~$\Om(f_{\bh})$ and
$$\ev^+=\prod_{i\in[l]}\!\ev^+_i\!: \ov\M_{k,l}\big(B;J,\nu;\wch X^\phi\big)\lra X^l\,.$$
We denote~by 
$$\ov\M_{k,l}^{\st}\big(B;J,\nu;\wch X^\phi\big)\subset 
\ov\M_{k,l}\big(B;J,\nu;\wch X^\phi\big)$$
the subspace of maps that are not $\Z_2$-pinchable.
This is a union of topological components of the entire moduli space and
is thus compact.
Let
\begin{alignat}{2}
\label{JakePseudo_e3}
\ev_{k,\bh}\!:\ov\cZ_{k,\bh}^{\st}(B;J,\nu;\wch{X}^{\phi})&\!\equiv\!\big\{
(\u,\y)\!\in\!\ov\M_{k,l}(B;J,\nu;\wch{X}^{\phi})\!\times\!M_{\bh}\!:
\ev^+(\u)\!=\!f_{\bh}(\y)\big\} \lra \big(\wch{X}^{\phi}\big)^k,\\
\notag
\ev_{k,h}\!:\ov\cZ_{k,h}^{\st}(B;J,\nu;\wch{X}^{\phi})&\!\equiv\!\big\{
(\u,z)\!\in\!\ov\M_{k,l}(B;J,\nu;\wch{X}^{\phi})\!\times\!Z\!:\ev^+(\u)\!=\!h(z)\big\}
\lra \big(\wch{X}^{\phi}\big)^k
\end{alignat}
be the maps induced by~\eref{fMevdfn_e}.
For each stratum~$\cS$ of $\ov\M_{k,l}(B;J,\nu;\wch X^\phi)$, 
let $\cS^*\!\subset\!\cS$ be the subspace of simple maps, 
$$\cS_{\bh}=\ov\cZ^\st_{k,\bh}(B;J,\nu;\wch X^\phi)\!\cap\!\big(\cS\!\times\!M_{\bh}\big), \quad
\cS^*_{\bh}=\ov\cZ^\st_{k,\bh}(B;J,\nu;\wch X^\phi)\!\cap\!\big(\cS^*\!\times\!M_{\bh}\big),$$
and $\fc(\cS)\!\in\!\Z^{\ge0}$ be the number of nodes of the domains of the elements of~$\cS$.

By~\eref{JakePseudo_e0} and the reasoning in the proof of \cite[Prop.~5.2]{RealWDVV},
the domain of~\eref{JakePseudo_e} is a smooth manifold of dimension
$$\big(\ell_{\om}(B)\!+\!k\!+\!2l\big)\!+\!\big(6l\!-\!2\,\codim_{\C}\bh\big)\!-\!6l
=3k$$
for a generic choice of \hbox{$(J,\nu)\!\in\!\cH_{k,l;G}^{\om,\phi}$}.
The orientation orientation~$\fo_{\os;L^*}$ of $\M_{k,l;L^*}^{\st}(B;J,\nu;\wch X^\phi)$
provided by Lemma~\ref{orient_lmm} and the orientation~$\fo_{\bh}$ of~$M_{\bh}$
determine an orientation~$\fo_{\os;L^*}\fo_{\bh}$ on the domain of~\eref{JakePseudo_e}.
Since the space $\ov\M_{k,l}^{\st}(B;J,\nu;\wch X^\phi)$ is compact,
\BE{JakePseudo_e5}\begin{split}
\Om\big(\ev_{k,\bh;L^*}\big|_{\cZ_{k,l;\bh}^{\st}(B;J,\nu;\wch X^\phi)}\big)
\subset \ev_{k,\bh}\big(\ov\cZ_{k,\bh}^{\st}(B;J,\nu;\wch{X}^{\phi})
\!-\!\cZ_{k,\bh;L^*}^{\st}(B;J,\nu;\wch{X}^{\phi})\big)\qquad&\\
\cup  \ev_{k,h}\big(\ov\cZ_{k,h}^{\st}(B;J,\nu;\wch{X}^{\phi})\big)&\,.
\end{split}\EE
In order to show that \eref{JakePseudo_e} is a dimension~0 pseudocycle, 
it is thus sufficient to show that the right-hand side above 
can be covered by smooth maps from manifolds of dimension $3k\!-\!2$.

The subspace
$$\ov\cZ_{k,\bh}^{\st}(B;J,\nu;\wch{X}^{\phi})
\!-\!\cZ_{k,\bh;L^*}^{\st}(B;J,\nu;\wch{X}^{\phi})\subset
\ov\M_{k,l}^{\st}(B;J,\nu)\!\times\!M_{\bh}$$
consists of the subspaces~$\cS_{\bh}^*$ corresponding to the strata~$\cS$ 
of $\ov\M_{k,l}(B;J,\nu)$ with either $\fc(\cS)\!\ge\!2$ nodes or $\ep_{L^*}(\cS)\!\not\in\!2\Z$
and of the subspaces $\cS_{\bh}\!-\!\cS_{\bh}^*$ with $\fc(\cS)\!\ge\!1$.
By~\eref{JakePseudo_e0} the reasoning in the proof of \cite[Prop.~5.2]{RealWDVV},
the subsets
$\ev_{k,\bh}(\cS_{\bh}^*)$ of $(\wch{X}^{\phi})^k$ with $\fc(\cS)\!\ge\!2$
and $\ev_{k,\bh}(\cS_{\bh}\!-\!\cS_{\bh}^*)$ with $\fc(\cS)\!\ge\!1$
can be covered by smooth maps from manifolds of dimension $3k\!-\!2$.
The same applies to the last set in~\eref{JakePseudo_e5}.

Suppose $\cS$ is a codimension~1 disk bubbling stratum of $\ov\M_{k,l}(B;J,\nu;\wch X^\phi)$ 
with $\ep_{L^*}(\cS)\!\not\in\!2\Z$.
For $r\!=\!1,2$, let $K_r,k_r,L_r,l_r,L_r^*,B_r$ be as in~\eref{JakePseudo_e8}
and $\bh_r\!=\!(h_i)_{i\in L_r}$.  
Similarly to~\eref{JakePseudo_e8b},
\BE{JakePseudo2_e8b}\begin{split}
&k_1\!+\!k_2=k, \quad l_1\!+\!l_2=l, \quad 
\codim_{\C}\bh=\codim_{\C}\bh_1\!+\!\codim_{\C}\bh_2;\\
&\hspace{.5in}\ell_{\om}(B_1)\!+\!\ell_{\om}(B_2)= 
\ell_{\om}(B)=2\big(k\!+\!\codim_{\C}\bh\!-\!l\big).
\end{split}\EE
Similarly to~\eref{JakePseudo_e8c} and the preceding equation,
\BE{JakePseudo2_e8c}
\frac{\ell_{\om}(B_2)}{2}\!-\!k_2\!-\!\big(l_2\!-\!|L_2^*|\big),
\frac{\ell_{\om}(B_2)}{2}\!-\!k_2\!-\!
\big(\codim_{\C}\bh_2\!-\!l_2\big)\not\in2\Z.\EE
Along with \eref{JakePseudo2_e8b} and the definition of $L^*\!=\!L^*(\bh)$ in~\eref{bhprpdnf_e},
the second equality in~\eref{JakePseudo2_e8c} give 
\BE{JakePseudo2_e8d}
\frac{\ell_{\om}(B_1)}{2}\!-\!k_1\!-\!
\big(\codim_{\C}\bh_1\!-\!l_1\big),
\frac{\ell_{\om}(B_1)}{2}\!-\!k_1\!-\!\big(l_1\!-\!|L_1^*\!\cap\!L^*_+(\bh)|\big)\not\in2\Z.\EE
By the first statement in~\eref{JakePseudo2_e8c} and
the second in~\eref{JakePseudo2_e8d},
$$\big(B_r,k_r,l_r,|L_r^*\!\cap\!L^*_+(\bh)|\big)\neq(0,0,0,0),(0,2,0,0),(0,0,1,1) 
\qquad\forall\,r=1,2.$$
Since the image of~$h_i$ with $i\!\in\!L_-^*(\bh)$ is disjoint from $\wch X^\phi$,
$\cS^*_\bh\!=\!\eset$ if 
$$\big(B_r,k_r,l_r,|L_r^*\!\cap\!L^*_+(\bh)|\big)=(0,0,1,0)$$ 
and $\nu$ is small.
Thus, we can thus assume either
$B_r\!\neq\!0$ or $k_r\!+\!2l_r\!\ge\!3$ for each $r\!=\!1,2$.

For good choices of~$\nu$, the restriction of~\eref{JakePseudo_e3} to~$\cS_{\bh}^*$ factors~as 
$$\xymatrix{\cS^*_{\bh}\ar[r] &
\cZ_{k_1,\bh_1;L^*_1}^{\st}(B_1;J,\nu_1';\wch X^\phi)
\times\cZ_{k_2,\bh_2;L^*_2}^{\st}(B_2;J,\nu_2';\wch X^\phi)
\ar@<-7ex>[d]_{\ev_{k_1,\bh_1;L^*_1}}\ar@<5ex>[d]^{\ev_{k_2,\bh_2;L^*_2}}\\
&(\wch X^\phi)^{k_1}\times (\wch X^\phi)^{k_2} \ar[r]& (\wch X^\phi)^{k}.}$$
Thus, $\ev_{k,\bh;L^*}(\cS_{\bh}^*)$ is covered by a smooth map from a manifold
of dimension
\begin{equation*}\begin{split}
\dim\,\cZ_{k_r,\bh_r;L^*_r}^{\st}(B_r;J,\nu_r;\wch X^\phi)\!+\!
\dim\,(\wch X^\phi)^{k_{3-r}}
&=\ell_{\om}(B_r)\!+\!k_r\!\!+\!2l_r\!-\!2\,\codim_{\C}\bh_r\!+\!3k_{3-r}\\
&=\ell_{\om}(B_r)\!-\!2k_r\!\!+\!2l_r\!-\!2\,\codim_{\C}\bh_r\!+\!3k
\end{split}\end{equation*}
for $r\!=\!1,2$.
By~\eref{JakePseudo2_e8b}, the second statement in~\eref{JakePseudo2_e8c},
and the first in~\eref{JakePseudo2_e8d},
$$\ell_{\om}(B_r)\!-\!2k_r\!\!+\!2l_r\!-\!2\,\codim_{\C}\bh_r\le -2$$
for either $r\!=\!1$ or $r\!=\!2$.
Thus, $\ev_{k,\bh;L^*}(\cS^*)$ is covered by a smooth map from a manifold
of dimension $3k\!-\!2$ if $\fc(\cS)\!\in\!2\Z$.
This confirms the first statement in~\ref{RGWdfn_it}.

By definition,
\BE{degRGWdfn_e} \deg\!\big(\ev_{k,\bh;L^*},\fo_{\os;\bh}\big)
=\big|\ev_{k,\bh;L^*}^{-1}(\bp)\big|^{\pm}_{\fo_{\os;\bh}}\EE
for a generic choice of $\bp\!\in\!(\wch X^{\phi})^k$.
Let $(J_t,\nu_t)_{t\in[0,1]}$ be a generic path in $\cH_{k,l;G}^{\om,\phi}$
between two generic values pairs $(J_0,\nu_0)$ and~$(J_1,\nu_1)$.
By the same reasoning as for the statement that the right-hand side 
of~\eref{JakePseudo_e5} can be covered by smooth maps from manifolds of dimension $3k\!-\!2$,
the~subsets 
\begin{gather*}
\big\{\ev_{k,\bh}(\u,\y)\!:t\!\in\![0,1],\,(\u,\y)\!\in\!
\ov\cZ_{k,\bh}^{\st}(B;J_t,\nu_t;\wch{X}^{\phi})
\!-\!\cZ_{k,\bh;L^*}^{\st}(B;J_t,\nu_t;\wch{X}^{\phi})\big\}\subset (\wch X^{\phi})^k,\\
\big\{\ev_{k,h}(\u,z)\!:t\!\in\![0,1],\,(\u,z)\!\in\!
\ov\cZ_{k,h}^{\st}(B;J_t,\nu_t;\wch{X}^{\phi})\big\}\subset (\wch X^{\phi})^k
\end{gather*}
can be covered by smooth maps from manifolds of dimension $3k\!-\!1$.
Since the space $\ov\M_{k,l}^{\st}(B;J,\nu;\wch X^\phi)$ is compact, 
it follows that the~space
$$\wt\cZ_{\bh,\bp}\equiv\big\{(t,\u,\y)\!:t\!\in\![0,1],\,(\u,\y)\!\in\!
\cZ_{k,\bh}^{\st}(B;J_t,\nu_t;\wch{X}^{\phi}),\,
\ev_{k,l;\bh}(\u,\y)\!=\!\bp\big\}$$
is a compact one-dimensional manifold with boundary
\BE{JakePseudo2_e9}\begin{split}
\prt\wt\cZ_{\bh,\bp}=&
\big\{\ev_{k,\bh}(\u,\y)\!:(\u,\y)\!\in\!
\cZ_{k,\bh}^{\st}(B;J_0,\nu_0;\wch{X}^{\phi}),\,\ev_{k,l;\bh}(\u,\y)\!=\!\bp\big\}\\
&\sqcup
\big\{\ev_{k,\bh}(\u,\y)\!:(\u,\y)\!\in\!
\cZ_{k,\bh}^{\st}(B;J_1,\nu_1;\wch{X}^{\phi}),\,\ev_{k,l;\bh}(\u,\y)\!=\!\bp\big\}.
\end{split}\EE
The orientations of Lemmas~\ref{cMorient_lmm} and~\ref{Dorient_lmm} and the orientation of $[0,1]$
determine an orientation on $\wt\cZ_{\bh,\bp}$ so that~\eref{JakePseudo2_e9} induces the
signs on the first set on the right-hand side which are opposite to the signs determined
by~$\fo_{\os;L^*}$ and~$\os$ and
the signs on the second set determined by~$\fo_{\os;L^*}$ and~$\os$.
Thus,  $\wt\cZ_{\bh,\bp}$ is an oriented cobordism between the signed subsets 
$$\ev_{k,\bh;L^*}^{-1}(\bp)\subset\cZ_{k,\bh}^{\st}\big(B;J_0,\nu_0;\wch{X}^{\phi}\big)
\quad\hbox{and}\quad
\ev_{k,\bh;L^*}^{-1}(\bp)\subset\cZ_{k,\bh}^{\st}\big(B;J_1,\nu_1;\wch{X}^{\phi}\big).$$
This establishes the independence of the signed cardinality in~\eref{degRGWdfn_e} 
of the choice of generic~$(J,\nu)$ in~$\cH_{k,l;G}^{\om,\phi}$.

Let $i_0\!\in\![l]$ and 
$$\wt{h}_{i_0}\!: \wt{H}_{i_0}\lra \begin{cases}X,&\hbox{if}~i_0\!\in\!L^*_+(\bh);\\
X\!-\!\wch{X}^{\phi},&\hbox{if}~i_0\!\in\!L^*_-(\bh);
\end{cases}$$
be a pseudocycle equivalence between two generic pseudocycle representatives, 
$$h_i\!:\wt{H}_{i_0}\lra \begin{cases}X,&\hbox{if}~i_0\!\in\!L^*_+(\bh);\\
X\!-\!\wch{X}^{\phi},&\hbox{if}~i_0\!\in\!L^*_-(\bh);
\end{cases} \quad\hbox{and}\quad
h_i'\!:\wt{H}_{i_0}'\lra \begin{cases}X,&\hbox{if}~i_0\!\in\!L^*_+(\bh);\\
X\!-\!\wch{X}^{\phi},&\hbox{if}~i_0\!\in\!L^*_-(\bh);
\end{cases}$$
for $[h_{i_0}]_X$ if $i_0\!\in\!L^*_+(\bh)$
and $[h_{i_0}]_{X-\wch{X}^{\phi}}$ if $i_0\!\in\!L^*_-(\bh)$.
We define 
$$\bh'\equiv(h_i')_{i\in[l]},~~\wt\bh\equiv(\wt{h}_i)_{i\in[l]}
\qquad\hbox{by}\qquad
h_i',\wt{h}_i=h_i~~\hbox{if}~i\!\neq\!i_0.$$
Let \hbox{$\wt{h}\!:\wt{Z}\!\lra\!X^l$} be a smooth map from a manifold of dimension
$6l\!-\!2\,\codim_{\C}\bh\!-\!1$ that covers~$\Om(f_{\wt\bh})$.

By the same reasoning as for the statement that the right-hand side of~\eref{JakePseudo_e5} 
can be covered by smooth maps from manifolds of dimension $3k\!-\!2$,
the~subsets 
\begin{gather*}
\big\{\ev_{k,\wt\bh}(\u,\y)\!:(\u,\y)\!\in\!\ov\cZ_{k,\wt\bh}^{\st}(B;J,\nu;\wch{X}^{\phi})
\!-\!\cZ_{k,\wt\bh;L^*}^{\st}(B;J,\nu;\wch{X}^{\phi})\big\}\subset (\wch X^{\phi})^k,\\
\big\{\ev_{k,\wt{h}}(\u,z)\!:(\u,z)\!\in\!
\ov\cZ_{k,\wt{h}}^{\st}(B;J,\nu;\wch{X}^{\phi})\big\}\subset (\wch X^{\phi})^k
\end{gather*}
can be covered by smooth maps from manifolds of dimension $3k\!-\!1$.
It follows that the~space
$$\cZ_{\wt\bh,\bp}\equiv\big\{(\u,\y)\!:(\u,\y)\!\in\!
\cZ_{k,\wt\bh}^{\st}(B;J,\nu;\wch{X}^{\phi}),\,
\ev_{k,l;\wt\bh}(\u,\y)\!=\!\bp\big\}$$
is a compact one-dimensional manifold with boundary
\BE{JakePseudo2_e11}\begin{split}
\prt\cZ_{\wt\bh,\bp}=&
\big\{\ev_{k,\bh}(\u,\y)\!:(\u,\y)\!\in\!
\cZ_{k,\bh}^{\st}(B;J,\nu;\wch{X}^{\phi}),\,\ev_{k,l;\bh}(\u,\y)\!=\!\bp\big\}\\
&\sqcup
\big\{\ev_{k,\bh'}(\u,\y)\!:(\u,\y)\!\in\!
\cZ_{k,\bh'}^{\st}(B;J,\nu;\wch{X}^{\phi}),\,\ev_{k,l;\bh'}(\u,\y)\!=\!\bp\big\}.
\end{split}\EE
The orientations of Lemmas~\ref{cMorient_lmm} and~\ref{Dorient_lmm} 
determine an orientation on $\cZ_{\wt\bh,\bp}$ so that~\eref{JakePseudo2_e11} induces the
signs on the first set on the right-hand side which are opposite to the signs determined
by~$\fo_{\os;L^*}$ and~$\os$ and 
the signs on the second set determined by~$\fo_{\os;L^*}$ and~$\os$.
Thus,  $\cZ_{\wt\bh,\bp}$ is an oriented cobordism between the signed subsets 
$$\ev_{k,\bh;L^*}(\bp)\subset\cZ_{k,\bh}^{\st}\big(B;J,\nu;\wch{X}^{\phi}\big)
\quad\hbox{and}\quad
\ev_{k,\bh';L^*}^{-1}(\bp)\subset\cZ_{k,\bh'}^{\st}\big(B;J,\nu;\wch{X}^{\phi}\big).$$
This establishes the independence of the signed cardinality in~\eref{degRGWdfn_e} 
of the choices of $h_i\!\in\![h_i]_X$ with $i\!\in\!L_+^*(\bh)$,
or $h_i\!\in\![h_i]_{X-\wch X^{\phi}}$ with \hbox{$i\!\in\!L_-^*(\bh)$}.

\ref{RGWvan_it} Let $i_0\!\in\![l]$, $\bh'\!\equiv\!(h_i')_{i\in[l]}$
be the tuple of maps obtained from $\bh$ by replacing the $i_0$-component 
with~$\phi\!\circ\!h_{i_0}$, and
$$\Psi_{i_0}\!:\ov\M_{k,l}\big(B;J,\nu;\wch{X}^{\phi}\big)
\lra \ov\M_{k,l}\big(B;J,\nu;\wch{X}^{\phi}\big)$$
be the automorphism induced by the interchange of 
the points in the conjugate pair $(z_{i_0}^+,z_{i_0}^-)$.
It induces a bijection
$$\Psi_{i_0;\bh}\!: \ev_{k,\bh;L^*}^{-1}(\bp)\lra\ev_{k,\bh';L^*}^{-1}(\bp).$$
By \ref{orientpm0_it}-\ref{orient1pm_it} in Lemma~\ref{orient_lmm} and~\eref{JakePseudo_e0},
\BE{JakePseudo2_e21a}\Psi_{i_0}^*\fo_{\os;L^*}
=\begin{cases}-\fo_{\os;L^*},&\hbox{if}~i_0\!\in\!L_+^*(\bh);\\
\fo_{\os;L^*},&\hbox{if}~i_0\!\in\!L_-^*(\bh).\end{cases}\EE
Since the action of~$\phi$ on~$X$ is orientation-reversing, it follows~that 
\BE{JakePseudo2_e21} 
\big|\ev_{k,\bh';L^*}^{-1}(\bp)\big|^{\pm}_{\fo_{\os;\bh'}}
=\begin{cases}
|\ev_{k,\bh;L^*}^{-1}(\bp)|^{\pm}_{\fo_{\os;\bh}},&
\hbox{if}~~i_0\!\in\!L_+^*(\bh);\\
-|\ev_{k,\bh;L^*}^{-1}(\bp)|^{\pm}_{\fo_{\os;\bh}},&
\hbox{if}~~i_0\!\in\!L_-^*(\bh).\end{cases}\EE
Along with the independence of the signed cardinality in~\eref{degRGWdfn_e} 
of the choices of $h_i\!\in\![h_i]_X$ with $i\!\in\!L_+^*(\bh)$
and $h_i\!\in\![h_i]_{X-\wch X^{\phi}}$ with \hbox{$i\!\in\!L_-^*(\bh)$},
this implies~\ref{RGWvan_it}.

\ref{RGWsym_it} Let $i_1,i_2\!\in\![l]$, 
$\bh'\!\equiv\!(h_i')_{i\in[l]}$
be the tuple of maps obtained from $\bh$ by interchanging 
the $i_1$ and $i_2$-components, and
$$\Psi_{i_1,i_2}\!:\ov\M_{k,l}\big(B;J,\nu;\wch{X}^{\phi}\big)
\lra \ov\M_{k,l}\big(B;J,\nu;\wch{X}^{\phi}\big)$$
be the automorphism induced by the interchange of 
the conjugate pairs $(z_{i_1}^+,z_{i_1}^-)$ and~$(z_{i_2}^+,z_{i_2}^-)$.
Along with the interchange of the $i_1$ and $i_2$-components of~$\bh$,
it induces a bijection 
$$\Psi_{i_1,i_2;\bh}\!: \ev_{k,\bh;L^*}^{-1}(\bp)\lra\ev_{k,\bh';L^*(\bh')}^{-1}(\bp).$$
By Lemma~\ref{orient_lmm}\ref{Cijinter2_it}, 
\eref{JakePseudo2_e21a} with~$i_0$ replaced by~$i_1$, and
\eref{JakePseudo2_e21a} with~$L^*$ and~$i_0$ replaced by~$L^*(\bh')$
and~$i_2$, respectively,
$$\Psi_{i_1,i_2}^{\,*}\fo_{\os;L^*(\bh')}=\fo_{\os;L^*}\,.$$
Since the interchange of the $i_1$ and $i_2$-components of~$M_{\bh}$ respects 
the orientations~$\fo_{\bh}$ and~$\fo_{\bh'}$, it follows~that 
$$\big|\ev_{k,\bh';L^*(\bh')}^{-1}(\bp)\big|^{\pm}_{\fo_{\os;\bh'}}
=\big|\ev_{k,\bh;L^*}^{-1}(\bp)\big|^{\pm}_{\fo_{\os;\bh}}\,.$$
This establishes the invariance of the numbers~\eref{JakePseudo_e0b}
under the permutations of the components~$h_i$ of~$\bh$.

\ref{RdivRel_it} Since the proof of this statement is identical to the proof
of the last statement of \cite[Prop.~5.2]{RealWDVV}, we omit~it.
\end{proof}

\begin{proof}[{\bf{\emph{Proof of Proposition~\ref{GtauGW_prp}}}}]
The functional~\eref{GtauGWdfn_e0} is specified by~\eref{Gnumsdfn_e};
the numbers on the right-hand side of~\eref{Gnumsdfn_e} are special cases
of the invariants~\eref{RGWdfn_e0} arising from Proposition~\ref{JakePseudo_prp}\ref{RGWdfn_it}. 
Its multilinearity and the vanishing property~\eref{GtauGWdfn_e0a}
are immediate from the definition of the invariants~\eref{RGWdfn_e0}.
The symmetry property of~\eref{GtauGWdfn_e0},  
the divisor relation~\eref{GtauGWdfn_e0b}, and
the vanishing in the second case in Proposition~\ref{GtauGW_prp}\ref{GtauGWvan_it} 
are direct consequences of~\ref{RGWsym_it}, \ref{RdivRel_it}, and~\ref{RGWvan_it}, respectively,
in Proposition~\ref{JakePseudo_prp}, along with~\eref{RdivRel2_e}.

It remains to establish the vanishing in the first case in~\ref{GtauGWvan_it}.
Let $\psi\!\in\!G$ be an automorphism of $(X,\om,\phi;\wch{X}^{\phi})$ 
which restricts to an orientation-reversing diffeomorphism of~$\wch{X}^{\phi}$,
$\bh\!\equiv\!(h_i)_{i\in[l]}$ be a $G$-invariant tuple of pseudocycles,
as in Section~\ref{DecompForm_subs}, representing (multiples~of)
Poincare duals of the cohomology classes~$\wt\mu_i$ in~\eref{Gnumsdfn_e},
$L^*\!=\!L^*(\bh)$ be as in~\eref{bhprpdnf_e}, and $k$ be as in~\eref{dimcond_e}.
Let 
$$\Psi\!:\ov\M_{k,l}\big(B;J,\nu;\wch{X}^{\phi}\big)
\lra \ov\M_{k,l}\big(B;J,\nu;\wch{X}^{\phi}\big)$$
be the automorphism induced by replacing the map component~$u$ in each tuple~$\u$
as in~\eref{udfn_e} with~$\psi\!\circ\!u$.
Along with the orientation-preserving action~\eref{Gactbhdfn_e},
it induces a bijection 
\BE{GtauGW_e3}\Psi_{\bh}\!: \ev_{k,\bh;L^*}^{-1}(\bp)\lra\ev_{k,\bh;L^*}^{-1}(\bp)\EE
for every tuple $\bp\!=\!(p_i)_{i\in[k]}$ of $G$-orbits 
of points in $\wch{X}^{\phi}$.

By the SpinPin~2a property in \cite[Section~1.2]{SpinPin}, 
there is a natural free action of $H^1(\wch{X}^{\phi};\Z_2)$ on 
the set of $\OSpin$-structures on~$\wch{X}^{\phi}$ which acts transitively 
on the set of Spin-structures for a fixed orientation on~$\wch{X}^{\phi}$.
Thus, $\psi^*\os\!=\!\mu\!\cdot\!\ov\os$ for some $\mu\!\in\!H^1(\wch{X}^{\phi};\Z_2)$.
Since the degree~$B$ of each element~$\u$ of $\M_{k,l}(B;J,\nu;\wch{X}^{\phi})$
is $(\wch{X}^{\phi},\Z_2)$-trivial,
the pullbacks of~$g^*\os$ and~$\ov\os$ by the restriction of~$u$ to the fixed locus
of the domain are the~same.
Along with Lemma~\ref{orient_lmm}\ref{fMosch_it}, this implies that 
$$\Psi^*\fo_{\os;L^*}=-\fo_{\os;L^*}\,.$$
Since the action of~$\psi$ on $X^k$ is orientation-preserving if and only if $k\!\in\!2\Z$,
it follows~that the bijection~\eref{GtauGW_e3} respects the signs of each point determined 
by~$\fo_{\os;\bh}$ and~$\os$ if and only~if $k\!\not\in\!2\Z$.
Thus,
$$\big|\ev_{k,\bh;L^*}^{-1}(\bp)\big|^{\pm}_{\fo_{\os;\bh}}
=(-1)^{k+1}\big|\ev_{k,\bh;L^*}^{-1}(\bp)\big|^{\pm}_{\fo_{\os;\bh}}\,.$$
We conclude that
the signed cardinality $\ev_{k,l;\bh}^{-1}(\bp)$ and the number~\eref{Gnumsdfn_e} 
vanish if $k\!\in\!2\Z$.
By~\eref{dimcond_e},  $k\!\in\!2\Z$ if and only if the last condition 
in~\eref{GtauGWvan_e0} holds.
\end{proof}

\begin{rmk}\label{GtauGW_rmk}
With the notation as in~\eref{Mhdfn_e} and just below, let
$$\cZ_{k,\bh}(B;J,\nu;\wch X^\phi)=\big\{\big(\u,(y_i)_{i\in[l]}\big)\!\in\!
\M_{k,l}(B;J,\nu;\wch X^\phi)\!\times\!M_{\bh}\!:
\ev_i^+(\u)\!=\!h_i(y_i)\,\forall\,i\!\in\![l]\big\}.$$
If $\bh$ is $G$-invariant and $k$ as in~\eref{dimcond_e} is~zero,
\eref{GtauGW_e3} restricts to a bijection
$$\cZ_{k,\bh}(B;J,\nu;\wch X^\phi)\!-\!\cZ_{k,\bh;L^*}^{\st}(B;J,\nu;\wch X^\phi)
\lra \cZ_{k,\bh}(B;J,\nu;\wch X^\phi)\!-\!\cZ_{k,\bh;L^*}^{\st}(B;J,\nu;\wch X^\phi)$$
between the sets of $\Z_2$-pinchable degree~$B$ maps meeting the pseudocycles~$h_i$.
By the proof of the last statement of Proposition~\ref{GtauGW_prp},
this bijection is sign-reversing.
Thus, the signed cardinality of the above set is~zero. 
It follows that we can define the numbers~\eref{Gnumsdfn_e} via~\eref{RGWdfn_e0a},
\eref{JakePseudo_e0b}, and~\eref{degRGWdfn_e} with the domain of 
the evaluation map~$\ev_{k,\bh;L^*}$
taken to~be $\cZ_{k,\bh}(B;J,\nu;\wch X^\phi)$ if $\bh$ is $G$-invariant.
\end{rmk}

\begin{proof}[{\bf{\emph{Proof of Proposition~\ref{SNXphi_prp}}}}]
Let $\bh\!\equiv\!(h_i)_{i\in[l]}$ be a tuple of pseudocycles
into~$X$ and~$X\!-\!\wch{X}^{\phi}$, as appropriate, 
representing (multiples~of) Poincare duals of the cohomology classes~$\mu_i$,
$L^*\!=\!L^*(\bh)$ be as in~\eref{bhprpdnf_e}, and 
$\bp\!\equiv\!(p_i)_{i\in[k]}$ be a $k$-tuple of general points in~$\wch{X}^{\phi}$.
In light of Proposition~\ref{JakePseudo_prp}\ref{RdivRel_it}, 
we can assume that $l\!\in\!\Z^+$.
For a generic choice of $(J,\nu)\!\in\!\cH_{k,l;\{1\}}^{\om,\phi}$, the space 
$$\cZ_{\bh,\bp}\equiv\big\{\big(\u,(y_i)_{i\in[l]}\big)\!\in\!
\M_{k,l}(B;J,\nu;\wch X^\phi)\!\times\!M_{\bh}\!:
\ev_i^+(\u)\!=\!h_i(y_i)\,\forall\,i\!\in\![l],\,
\ev_i^{\R}(\u)\!=\!p_i\,\forall\,i\!\in\![k]\big\}$$
is a two-dimensional manifold.
The orientations $\fo_{\os;L^*}$ of $\M_{k,l}(B;J,\nu;\wch X^\phi)$,
$\fo_{\bh}$ of~$M_{\bh}$, and~$\fo$ of~$\wch{X}^{\phi}$
determine an orientation~$\fo_{\bh,\bp}$ on~$\cZ_{\bh,\bp}$ as
the preimage of~$\bp$ under the restriction of~\eref{JakePseudo_e} to
the main stratum of its domain.

With the notation as in~\eref{ffMdfn_e}, define 
\begin{equation*}\begin{split}
\cZ_{\bh,\bp}^{\R}&=\big\{\big(\u,(y_i)_{i\in[l]}\big)\!\in\!
\M_{k+1,l}(B;J,\ff_{k+1,l;k+1}^{\R\,*}\nu;\wch X^\phi)\!\times\!M_{\bh}\!:
\big(\ff_{k+1,l;k+1}^{\R}(\u),(y_i)_{i\in[l]}\big)\!\in\!\cZ_{\bh,\bp}\big\},\\
\cZ_{\bh,\bp}^+&=\big\{\big(\u,(y_i)_{i\in[l]}\big)\!\in\!
\M_{k,l+1}(B;J,\ff_{k,l+1;l+1}^*\nu;\wch X^\phi)\!\times\!M_{\bh}\!:
\big(\ff_{k,l+1;l+1}(\u),(y_i)_{i\in[l]}\big)\!\in\!\cZ_{\bh,\bp}\big\}.
\end{split}\end{equation*}
The projections from $\cZ_{\bh,\bp}$, $\cZ_{\bh,\bp}^{\R}$, and $\cZ_{\bh,\bp}^+$
to the first factor induce commutative diagrams
$$\xymatrix{ \cZ_{\bh,\bp}^{\R} \ar[d]^{\ff_{\bh,\bp}^{\R}} \ar[r]^<<<<{\pi_1^{\R}}&
\M_{k+1,l}(B;J,\ff_{k+1,l;k+1}^{\R\,*}\nu;\wch X^\phi)\ar[d]_{\ff_{k+1,l;k+1}}& 
\cZ_{\bh,\bp}^+\ar[d]^{\ff_{\bh,\bp}^+} \ar[r]^<<<<<{\pi^+_1}&
\M_{k,l+1}(B;J,\ff_{k,l+1;l+1}^*\nu;\wch X^\phi)\ar[d]_{\ff_{k,l+1;l+1}}\\
\cZ_{\bh,\bp}\ar[r]^<<<<<<<<<{\pi_1} & \M_{k,l}(B;J,\nu;\wch X^\phi)&
\cZ_{\bh,\bp}\ar[r]^<<<<<<<<<{\pi_1} & \M_{k,l}(B;J,\nu;\wch X^\phi)\,.}$$
Since $\pi_1^{\R}$ and~$\pi_1^+$ induce isomorphisms between the vertical tangent bundles
of their domains and targets, they 
pull back  the orientations~$\fo_{k+1}^{\R}$ and~$\fo_{k+1}^+$ 
of the fibers of~$\ff_{k+1,l;k+1}$ and~$\ff_{k,l+1;l+1}$ to
orientations~$\fo^{\R}$ and $\fo^+$ of the fibers of~$\ff_{\bh,\bp}^{\R}$
and~$\ff_{\bh,\bp}^+$, respectively. 

We denote by \hbox{$\fo_{\bh,\bp}^{\R}\!\equiv\!\fo^{\R}\fo_{\bh,\bp}$}
the orientation of $\cZ_{\bh,\bp}^{\R}$ induced by~$\fo^{\R}$ and~$\fo_{\bh,\bp}$.
Let $\fo_{\bh,\bp}^+$ be the orientation of $\cZ_{\bh,\bp}^+$ which restricts to 
$\fo^+\fo_{\bh,\bp}$ on the subspace of maps from $(\P^1,\tau)$
with the marked points $z_1^+$ and~$z_{l+1}^+$ not separated by the fixed locus $S^1\!\subset\!\P^1$
and to the opposite orientation on the complement of this subspace.
In particular, the orientation~$\fo_{\bh,\bp}^+$ is preserved by the interchange
of the marked points~$z_{l+1}^+$ and~$z_{l+1}^-$.
We denote by 
$$\ev_{k+1}^{\R}\!:\cZ_{\bh,\bp}^{\R}\lra\wch{X}^{\phi} \quad\hbox{and}\quad
\ev_{l+1}^+\!:\cZ_{\bh,\bp}^+\lra X$$
the maps induced by~\eref{fMRevdfn_e} and~\eref{fMCevdfn_e}, respectively.

Let $p\!\in\!\wch{X}^{\phi}$ be another general point and 
$S_p\!\subset\!X\!-\!\wch{X}^{\phi}$ be a sphere in the fiber~$\cN_p\wch{X}^{\phi}$ 
of a tubular neighborhood~$\cN\wch{X}^{\phi}$ of~$\wch{X}^{\phi}$ in~$X$ over~$p$.
We denote the inclusion of $S_p$ into $X\!-\!\wch{X}^{\phi}$ by~$\io_p$.
By~\eref{RGWdfn_e0a}, \eref{JakePseudo_e0b}, and~\eref{degRGWdfn_e},  
\BE{SNXphi_e5a}\begin{split}
\blr{\mu_1,\ldots,\mu_l}_{B;\wch{X}^{\phi}}^{\phi,\os}
&=\big|\ev_{k+1,\bh;L^*}^{-1}(\bp)\big|^{\pm}_{\fo_{\os;\bh}},\\
\blr{\mu_1,\ldots,\mu_l,\PD_{X,\wch{X}^{\phi}}\big([S_p]_{X-\wch{X}^{\phi}}
\big)}_{B;\wch{X}^{\phi}}^{\phi,\os}
&=\big|\ev_{k,\bh\io_{S_p};L^*}^{-1}(\bp)\big|^{\pm}_{\fo_{\os;\bh\io_{S_p}}}\,.
\end{split}\EE
By \ref{fforientR_it}, \ref{fforient_it}, \ref{Cijinter2_it}, 
and~\ref{orientpm0_it} in Lemma~\ref{orient_lmm}, 
\BE{SNXphi_e5}\begin{split}
\big|\ev_{k+1,\bh;L^*}^{-1}(\bp)\big|^{\pm}_{\fo_{\os;\bh}}
&=\big|\{\ev_{k+1}^{\R}\}^{-1}(p)\big|^{\pm}_{\fo_{\bh,\bp}^{\R}}\,,\\
\big|\ev_{k,\bh\io_{S_p};L^*}^{-1}(\bp)\big|^{\pm}_{\fo_{\os;\bh\io_{S_p}}}
&=\big|\{\ev_{l+1}^+\}^{-1}(S_p)\big|^{\pm}_{\fo_{\bh,\bp}^+,\fo_{S_p}}\,.
\end{split}\EE
As $S_p$ shrinks to~$p$, the elements of $\{\ev_{l+1}^+\}^{-1}(S_p)$ converge to
maps from two-component domains sending the marked point~$z_{l+1}^+$ to~$p$.
The restriction of any such limiting map to one of the components is constant,
and this component carries the marked points~$z_{l+1}^{\pm}$ only. 
The restriction to the other component represents an element of~$\{\ev_{k+1}^{\R}\}^{-1}(p)$.
We show below that there are precisely two elements of $\{\ev_{l+1}^+\}^{-1}(S_p)$
near each element $(\u,\y)$ of~$\{\ev_{k+1}^{\R}\}^{-1}(p)$ if~$S_p$ is sufficiently small.
Furthermore, the signs of these two elements are the same as the sign of~$(\u,\y)$.
Along with~\eref{SNXphi_e5a} and~\eref{SNXphi_e5}, this implies~\eref{SNXphi_e}.

We denote by $\prt_r,\prt_{\th}\!\in\!T_1\C$ the outward unit radial vector and 
the counterclockwise unit rotation vector so that \hbox{$\prt_{\th}\!=\!\fI\prt_r$}.
We identify a neighborhood of~$p$ in~$\wch{X}^{\phi}$ with~$T_p\wch{X}^{\phi}$
and a neighborhood of~$p$ in~$X$ with~$T_p\wch{X}^{\phi}\!\oplus\!\cN_p\wch{X}^{\phi}$.
Let $\u$ be a stable map representative for an element of $\{\ev_{k+1}^{\R}\}^{-1}(p)$
so that its map component~$u$ takes $1\!\in\!\P^1$ to~$p$ and its marked point~$z_1^+$
is $0\!\in\!\C$.
Since $p$ is a regular value of~$\ev_{k+1}^{\R}$, the differential 
\BE{SNXphi_e7}T_{\u}\cZ_{\bh,\bp}\!\oplus\!\R\lra T_p\wch{X}^{\phi},
\qquad (\xi,t)\lra \xi(1)\!+\!t(\prt_{\th}u),\EE
of~$\ev_{k+1}^{\R}$ at $(\u,x_{k+1}\!=\!1)$ is an isomorphism and $\nd_1u$ is injective.
The sign of~$(\u,x_1)$ as an element of $\{\ev_{k+1}^{\R}\}^{-1}(p)$
 in~\eref{SNXphi_e5} is the sign of the isomorphism~\eref{SNXphi_e7}
with respect to the orientation~$\fo_{\bh,\bp}$ on~$T_{\u}\cZ_{\bh,\bp}$,
the standard orientation on~$\R$, and the orientation~$\fo$ on~$T_p\wch{X}^{\phi}$.
Since the homomorphism~\eref{SNXphi_e7} is an isomorphism and the differential
\BE{SNXphi_e9}T_{\u}\cZ_{\bh,\bp}\!\oplus\!\R\!\oplus\!\R\lra 
T_pX\!=\!T_p\wch{X}^{\phi}\!\times\!\cN_p\wch{X}^{\phi},
\quad (\xi,t,s)\lra \big(\xi(1)\!+\!t(\prt_{\th}u),0\big)\!+\!s(\prt_ru),\EE
of~$\ev_{k+1}^+$ is injective, the equation
\BE{SNXphi_e11}\ev_{k+1}^+\big(\u',z_{l+1}^+\!=\!(1\!+\!s)\ne^{\fI t}\big)\in \{0\}\!\times\!S_x\EE
has two solutions with $\u'\!\in\!\cZ_{\bh,\bp}$ near~$\u$ and small $(s,t)$, 
one with $s\!<\!0$ and one with $s\!>\!0$,
if the radius of $S_x$ is sufficiently small.

If $s\!<\!0$, the marked points $z_1^+\!=\!0$ and~$z_{l+1}^+$ are not separated 
by the fixed locus $S^1\!\subset\!\P^1$ of~$\tau$.
In this case, the orientation~$\fo_{\bh,\bp}^+$ of $\cZ_{\bh,\bp}^+$ at $(\u',z_{l+1}^+)$
is the opposite of the orientation of the left-hand side of~\eref{SNXphi_e9}
given by the orientation~$\fo_{\bh,\bp}$ and the standard orientations of the two factors
of~$\R$ (with the first factor corresponding to~$\prt_{\th}$ and the second to~$\prt_r$).
Furthermore, $\nd_{z_{k+1}^+}u'(\prt_r)$ points inward from~$S_x$.
If $s\!>\!0$, the marked points $z_1^+\!=\!0$ and~$z_{l+1}^+$ are separated by~$S^1$.
In this case, the orientation~$\fo_{\bh,\bp}^+$ of $\cZ_{\bh,\bp}^+$ at $(\u',z_{l+1}^+)$
is the orientation of the left-hand side of~\eref{SNXphi_e9}
given by~$\fo_{\bh,\bp}$ and the standard orientations of the two factors of~$\R$.
Furthermore, $\nd_{z_{k+1}^+}u'(\prt_r)$ points outward from~$S_x$.
Thus, the sign of~$(\u',z_{l+1}^+)$ as an element of $\{\ev_{l+1}^+\}^{-1}(S_p)$
in~\eref{SNXphi_e5}
in either case is the sign of the isomorphism~\eref{SNXphi_e7}
with respect to the orientation~$\fo_{\bh,\bp}$ on~$T_{\u}\cZ_{\bh,\bp}$,
the standard orientation on~$\R$, and the orientation~$\fo$ on~$T_p\wch{X}^{\phi}$.
We conclude that the sign of each of the two solutions of~\eref{SNXphi_e11} 
as an element of $\{\ev_{l+1}^+\}^{-1}(S_p)$ in~\eref{SNXphi_e5}
is 
the sign of~$(\u,x_1)$ as an element of $\{\ev_{k+1}^{\R}\}^{-1}(p)$
in~\eref{SNXphi_e5}.
\end{proof}

\subsection{Proof of Proposition~\ref{Rdecomp_prp}}
\label{Rdecomp_subs}

We continue with the notation in the statement of this proposition and just above.
For finite sets $K',L'$ with $|K'|\!+\!2|L'|\!\le\!2$,
we denote by $\cH_{K',L';G}^{\om,\phi}$ the set of pairs $(J,0)$ with $J\!\in\!\cJ_{\om;G}^{\phi}$.
Let \hbox{$L^*\!\equiv\!L^*(\bh)$} be as in~\eref{bhprpdnf_e}, 
$K_r,L_r,L_r^*,B_r$ for $r\!=\!1,2$ be as in~\eref{JakePseudo_e8}, and
$$\M^{\st}=\M_{k,l;L^*}^{\st}(B;J,\nu;\wch X^\phi).$$
Since $(\cS,\Ups)$ is admissible, \hbox{$|K_1|\!+\!2|L_1|\!\ge\!3$} and 
either $K_2\!\neq\!\eset$ or $L_2\!\neq\!\eset$. 
We assume that there exist \hbox{$\nu_1'\!\in\!\cH_{K_1,L_1;G}^{\om,\phi}$} and 
$\nu_2\!\in\!\cH_{\{0\}\sqcup K_2,L_2;G}^{\om,\phi}$
so that every $\cS$ admits an embedding as in~\eref{cSsplit_e}
with \hbox{$\nu_1\!=\!\ff_{\{0\}\sqcup K_1,L_1;\{0\}}^*\nu_1'$}.

The first claim of Proposition~\ref{Rdecomp_prp}\ref{RdecompTrans_it}
follows from standard transversality arguments, as in 
the part of the proof of \cite[Prop.~5.2]{RealWDVV} concerning the transversality on
the subspaces of simple maps.
The second claim then follows from~\eref{Cdecomp_e0}.
By the proof of Lemma~\ref{LiftedRel_lmm},
$$\ve_{L^*}(\cS)-2\big|\big\{i\!\in\!L_2(\cS)\!:\dim\,h_i\!=\!0\big\}\big|\in\{0,1\}$$
if $\cS_{\bh,\bp;\Ups}^*\!\neq\!\eset$; 
see the second equation in~\eref{mainsetup_e25} and on the following line.
This establishes Proposition~\ref{Rdecomp_prp}\ref{RdecompEmpt_it}.
We establish Proposition~\ref{Rdecomp_prp}\ref{Rdecomp_it} below 
under the assumption that $L_2^*\!\neq\!\eset$.
The $L_2^*\!=\!\eset$ case then follows by the reasoning
in the proof of \cite[Prop.~5.7]{RealWDVV}.
Since $L_2^*\!\neq\!\eset$, the image of~$\cS$ under the forgetful morphism~$\ff_{k,l}$
is contained in a codimension~1 stratum~$\cS^{\vee}$ of~$\ov\cM_{k,l}^{\tau}$.
By Lemma~\ref{orient_lmm2}, we can also assume that the orientation~$\fo_{\cS}^c$
of~$\cN\cS$ used to define the orientation
$\prt\fo_{\os;L^*}\!\equiv\!\prt_{\fo_{\cS}^c}\fo_{\os;L^*}$ of $\cS$ 
is~$\fo_{\cS}^{c;+}$ in the notation of Lemma~\ref{DorientComp_lmm}.
 
For $\u\in\cS$, let 
\begin{gather*}\begin{aligned}
\u_1\in  \M_1&\equiv\!\M_{\{0\}\sqcup K_1,L_1}\big(B_1;J,\nu_1;\wch X^\phi\big), &\qquad
\u_1'\in \M_1'&\equiv\!\M_{K_1,L_1}\big(B_1;J,\nu_1';\wch X^\phi\big),\\
\u_2\in \M_2&\equiv\!\M_{\{0\}\sqcup K_2,L_2}\big(B_2;J,\nu_2;\wch X^\phi\big), &\qquad
\nod&\in\P^1_1,\P^1_2,~~S^1_1\subset\P^1_1,
\end{aligned}\\
D_{\u}^{\phi}=D_{J,\nu;\u}^{\phi}, \qquad 
D_{\u_1}^{\phi}=D_{J,\nu_1;\u_1}^{\phi}=D_{J,\nu_1';\u_1'}^{\phi},\qquad
D_{\u_2}^{\phi}=D_{J,\nu_2;\u_2}^{\phi}\
\end{gather*}
be as above Lemma~\ref{DorientComp_lmm} and in Section~\ref{cMorient_subs}.
We denote by
\begin{alignat*}{2}
\cC&\equiv\!\ff_{k,l}(\u)\in\cS^{\vee}\subset\ov\cM\!\equiv\!\ov\cM_{k,l}^{\tau}, &\quad
\cC_1&\equiv\!\ff_{\{0\}\sqcup K_1,L_1}(\u_1)\in\cM_1\!\equiv\!\cM_{\{0\}\sqcup K_1,L_1}^{\tau}, \\
\cC_1'&\equiv\!\ff_{K_1,L_1}(\u_1')\in\cM_1'\!\equiv\!\cM_{K_1,L_1}^{\tau},&\quad
\cC_2&\equiv\!\ff_{\{0\}\sqcup K_2,L_2}(\u_2)\in\!\cM_2\!\equiv\!\cM_{\{0\}\sqcup K_2,L_2}^{\tau} 
\end{alignat*}
the marked domains of the maps $\u$, $\u_1$, $\u_1'$, and $\u_2$, respectively.

The exact sequence
\BE{TcSses_e}
0\lra T_{\u}\cS\lra T_{\u_1}\M_1\!\oplus\!T_{\u_2}\M_2\lra T_{u(\nod)}\wch X^{\phi}\lra0, ~~ 
\big(\xi_1,\xi_2\big)\lra \xi_2(\nod)\!-\!\xi_1(\nod),\EE
of vector spaces determines an isomorphism 
\BE{Rdecomppf_e3}
\la_{\u}(\cS)\!\otimes\!\la\big(T_{u(\nod)} \wch X^{\phi}\big)
\approx \la_{\u_1}(\M_1)\!\otimes\!\la_{\u_2}(\M_2).\EE
The $\OSpin$-structure~$\os$ on~$\wch X^{\phi}$ determines orientations~$\fo_{\os;L_1^*}$
and~$\fo_{\os;L_2^*}$ of $\la_{\u'_1}(\M'_1)$ and $\la_{\u_2}(\M_2)$
respectively; see Lemma~\ref{orient_lmm}.
The $S^1$-fibration in~\eref{cSsplit_e2a} determines a homotopy class of isomorphisms 
\BE{Rdecomppf_e5a}
\la_{\u_1}(\M_1) \approx \la_{\u_1'}(\M_1')\!\otimes\!T_{\nod}S^1_1\,.\EE
Together with the orientation~$\fo_{\nod}^{\R}$ on its vertical tangent bundle 
\hbox{$T_{\u_1}\M_1^v\!=\!T_{\nod}S^1_1$},
we obtain an orientation $\wt\fo_{\os;L_1^*;\u_1}\!\equiv\!\fo_{\nod}^{\R}\fo_{\os;L_1^*;\u'_1}$ 
of $\la_{\u_1}(\M_1)$. 
We denote by~$\fo_{\os;L^*;\u}^{\cS}$ the orientation on $\la_{\u}(\cS)$ 
determined by~$\wt\fo_{\os;L^*_1;\u_1}$ and~$\fo_{\os;L^*_2;\u_2}$ via~\eref{Rdecomppf_e3}.

We define $\de_{\R}(\cS)\!\in\!\{0,1\}$ as at the beginning of Section~\ref{cMorient_subs}.
The next lemma is deduced from Lemmas~\ref{DMboundary_lmm}  and~\ref{DorientComp_lmm} 
at the end of this section.

\begin{lmm}\label{TMcomp_lmm}
Let $\u\!\in\!\cS$.
The orientations $\prt\fo_{\os;L^*}$ and $\fo_{\os;L^*}^{\cS}$ of $\la_{\u}(\cS)$ are opposite
if and only if $\de_{\R}(\cS)\!=\!0$.
\end{lmm}

We take $\bh_1$ and~$\bh_2$ to be the components of~$\bh$ as in the proof 
of Proposition~\ref{JakePseudo_prp} and
$$\bp_1\in (\wch X^\phi)^{k_1} \qquad\hbox{and}\qquad \bp_2\in (\wch X^\phi)^{k_2}$$
to be the components of $\bp\!\in\!(\wch X^\phi)^k$ defined analogously.
Let
\begin{gather*}
\cZ_1=\cZ_{\{0\}\sqcup K_1,\bh_1;L^*_1}^{\st}(B_1;J,\nu_1;\wch X^\phi)
\!\cap\!\big(\M_1\!\times\!M_{\bh_1}\big), ~~
\cZ_1'=\cZ_{K_1,\bh_1;L^*_1}^{\st}(B_1;J,\nu_1';\wch X^\phi)
\!\cap\!\big(\M_1'\!\times\!M_{\bh_1}\big),\\
\cZ_2=\cZ_{\{0\}\sqcup K_2,\bh_2;L^*_2}^{\st}(B_2;J,\nu_2;\wch X^\phi)
\!\cap\!\big(\M_2\!\times\!M_{\bh_2}\big).
\end{gather*}
We denote by 
$$\ev_{\bh_1}\!:\cZ_1\lra (\wch X^\phi)^{k_1}, \quad \ev_{\bh_1}'\!:\cZ_1'\lra (\wch X^\phi)^{k_1}, 
\quad\hbox{and}\quad \ev_{\bh_2}\!:\cZ_2\lra (\wch X^\phi)^{k_2}\,,$$
the maps induced by~\eref{fMRevdfn_e}.
By~\eref{Rdecomp_e0}, Remark~\ref{GtauGW_rmk},  and~\eref{OWGhomdfn_e}, 
\BE{Rdecomp_e2}\begin{split}
\deg\!\big(\ev_{\bh_1}',\fo_{\os;\bh_1}\big)&=
\blr{\big(\PD_X([h_i]_X)\!\big)_{i\in L_1(\cS)}}^{\phi,\os}_{B_1(\cS);\wch X^\phi;G},\\
\deg\!\big(\ev_{\bh_2},\fo_{\os;\bh_2}\!\big)&=
\blr{\big(\PD_X([h_i]_X)\!\big)_{i\in L_2(\cS)}}^{\phi,\os}_{B_2(\cS);\wch X^\phi;G}\,.
\end{split}\EE

The forgetful morphism~\eref{cSsplit_e2a} induces a fibration~$\ff_{\cZ_1}$
so that the diagram 
$$\xymatrix{ \cZ_1 \ar[d]_{\ff_{\cZ_1}} \ar[rr]^{\pi_{\cZ}}&& \M_1 \ar[d]^{\ff_{\nod}}\\
\cZ_1' \ar[rr]^{\pi_{\cZ'}} && \M_1'}$$
commutes.
Since $\pi_{\cZ}$ induces an isomorphism between the vertical tangent bundles
$T\cZ_1^v$ of~$\ff_{\cZ_1}$ and $T\M_1^v$ of~$\ff_{\nod}$, 
it pulls back $\fo_{\nod}^{\R}$ to an orientation~$\fo_{\cZ_1}^v$ on the fibers of~$\ff_{\cZ_1}$.
The orientations $\wt\fo_{\os;L_1^*}$, $\fo_{\os;L_1^*}$, and $\fo_{\os;L_2^*}$ 
on $\M_1,\M'_1$ and $\M_2$, respectively, and
the orientations $\fo_{h_i}$ of $H_i$, determine 
orientations $\wt\fo_{\os;\bh_1}$, $\fo_{\os;\bh_1}$, and $\fo_{\os;\bh_2}$ of 
$\cZ_1,\cZ_1'$ and $\cZ_2$, respectively.
Since the dimensions of $X$ and $H_i$ are even, the isomorphism
\BE{Rdecomppf_e8} \la_{\u_1}(\cZ_1) \approx \la_{\u_1'}(\cZ_1')\!\otimes\!T_{\nod}S^1_1\EE
respects the orientations
$\wt\fo_{\os;\bh_1}$,  $\fo_{\os;\bh_1}$, and 
$\fo^v_{\cZ_1}\!=\!\pi^*_\cZ\fo^\R_\nod$.

For $\wt\u\!\in\!\cS_{\bh}^*$, we denote by 
$$\wt\u_1\in\cZ_1, \qquad  \wt\u_1'\in\cZ_1', \qquad  \wt\u_2\in\cZ_2$$
the images of~$\wt\u$ under the projections induced by the embedding~\eref{cSsplit_e},
the forgetful morphism~\eref{cSsplit_e2a},  and the decomposition 
$$M_{\bh}\approx M_{\bh_1}\!\times\!M_{\bh_2}\,.$$
The exact sequence
$$0\lra T_{\wt\u}\cS_{\bh}^*\lra T_{\wt\u_1}\cZ_1\!\oplus\!T_{\wt\u_2}\cZ_2\lra 
T_{u(\nod)}\wch X^{\phi}\lra0, ~~ 
\big(\xi_1,\xi_2\big)\lra \xi_2(\nod)\!-\!\xi_1(\nod),$$
of vector spaces determines an isomorphism 
$$\la_{\wt\u}(\cS_{\bh}^*)\!\otimes\!\la\big(T_{u(\nod)}\wch X^{\phi}\big)
\approx \la_{\wt\u_1}(\cZ_1)\!\otimes\!\la_{\wt\u_2}(\cZ_2).$$
Along with the orientations $\wt\fo_{\os;\bh_1}$ and $\fo_{\os;\bh_2}$  
and the orientation~$\fo$ of $\wch X^\phi$ determined by~$\os$, 
this isomorphism determines an orientation $\fo_{\cS;\bh}$ of $\cS^*_{\bh}$. 
Since the dimensions of $X$ and $H_i$ are even, 
Lemma~\ref{TMcomp_lmm} implies~that 
\BE{Rdecomppf_e11} \big|\cS_{\bh,\bp;\Ups}^*\big|_{\prt\fo_{\os;\bh},\fo_\Ups^c}^{\pm}
=-(-1)^{\de_{\R}(\cS)}\big|\cS_{\bh,\bp;\Ups}^*\big|_{\fo_{\cS;\bh},\fo_\Ups^c}^{\pm}\,.\EE

If $\cS$ and $\Ups$ satisfy~\ref{cSUps_it1} above Proposition~\ref{Rdecomp_prp}
with $i\!\in\![k]$ as in~\ref{cSUps_it1}, let
$$K_1'=K'\!-\!\{i\}, \quad L_1'=L', \quad K_2'=\{i\}, \quad L_2'=\eset, \quad \Ups_1=\Ups,
\quad 0=i\,.$$
If $\cS$ and $\Ups$ satisfy~\ref{cSUps_it2} and $S\!\subset\!\ov\cM_{k',l'}^{\tau}$ 
as in~\ref{cSUps_it2}, let
$$K_1'=K_1(S), \quad L_1'=L_1(S), \quad K_2'=K_2(S), \quad L_2'=L_2(S)$$
and denote by 
$$\pi_1\!: S\!\approx\!\ov\cM_{\{0\}\sqcup K_1',L_1'}^{\tau}
\!\times\!\ov\cM_{\{0\}\sqcup K_2',L_2'}^{\tau}
\lra \ov\cM_{\{0\}\sqcup K_1',L_1'}^{\tau}$$
the projection to the first component in the second identification in~\eref{Ssplit_e0}.
In this case,
$$\Ups\!\cap\!\ov{S}\approx\Ups_1\!\times\!\ov\cM_{\{0\}\sqcup K_2',L_2'}^{\tau}
\subset \ov\cM_{\{0\}\sqcup K_1',L_1'}^{\tau}\!\times\!\ov\cM_{\{0\}\sqcup L_2',L_2'}^{\tau}$$
for some $\Ups_1\!\subset\!\ov\cM_{\{0\}\sqcup K_1',L_1';\nod}^{\tau;\st}$.
The co-orientation $\fo_{\Ups\cap S}^c$ on 
$\Ups\!\cap\!\ov{S}$ in~$\ov{S}$ induced by~$\fo_{\Ups}^c$ is the pullback by~$\pi_1$
of a co-orientation $\fo_{\Ups_1}^c$ on $\Ups_1$ in~$\ov\cM_{\{0\}\sqcup K_1',L_1'}^{\tau}$.
Let
$$\ff_{\{0\}\sqcup K_1',L_1'}\!=\!\pi_1\!\circ\!\ff_{k',l'}\!:
\cS_{\bh}^*\lra S\lra \ov\cM_{\{0\}\sqcup K_1',L_1'}^{\tau}\,.$$
In both cases, 
\BE{Rdecomppf_e9}\dim\,\Ups_1=\dim\,\Ups\!+\!1\!-\!|K_2'|\!-\!2|L_2'|\EE
and the forgetful morphism $\ff_{\{0\}\sqcup K_1',L_1'}$ factors as 
$$\cS_{\bh}^* \lhra{~~~} \cZ_1\!\times\!\cZ_2\lra \cZ_1 
\xlra{\ff_{\{0\}\sqcup K_1',L_1'}}\ov\cM_{\{0\}\sqcup K_1',L_1'}^{\tau}\,.$$
We define
$$\ff_{\cM}\!=\!\ff_{\{0\}\sqcup K_1',L_1';0}\!:
\ov\cM_{\{0\}\sqcup K_1',L_1'}^{\tau}\lra\ov\cM_{K_1',L_1'}^{\tau}\,.$$

If $\cS$ and $\Ups$ satisfy~\ref{cSUps_it2}, (2) and~(3) in \cite[Lemma~3.3]{RealWDVV} give
\BE{Rdecomppf_e21}\begin{split}
\big|\cS_{\bh,\bp;\Ups}^*\big|_{\fo_{\cS;\bh},\fo_\Ups^c}^{\pm}
&=-\big|M_{(\ev_{\cS;\bh},\ff_{k',l'}),f_{\bp;\Ups}}
\big|_{\fo_{\cS;\bh},\pi_1^*\fo_{\Ups_1}^c|_{\Ups\cap\ov{S}}}^{\pm}\\
&=-(-1)^{|K_2'|}\big|M_{(\ev_{\cS;\bh},\ff_{\{0\}\sqcup K_1',L_1'}),f_{\bp;\Ups_1}}
\big|_{\fo_{\cS;\bh},\fo_{\Ups_1}^c}^{\pm};
\end{split}\EE
the signed fiber products in the second and third expressions above are taken 
with respect to $(\wch X^\phi)^k\!\times\!S$ and 
\hbox{$(\wch X^\phi)^k\!\times\!\ov\cM_{\{0\}\sqcup K_1',L_1'}^{\tau}$}, respectively.
The first and last expressions in~\eref{Rdecomppf_e21} are the same if 
$\cS$ and $\Ups$ satisfy~\ref{cSUps_it1}.
Since the diffeomorphism 
$$(\wch X^\phi)^k\lra(\wch X^\phi)^{K_1}\!\times\!(\wch X^\phi)^{K_2}$$
respecting the ordering of the elements of $K_1$ and $K_2$ has sign $(-1)^{\de_\R(\cS)}$, 
the definition of $\fo_{\cS;\bh}$ and \cite[Lemma~3.4]{RealWDVV} give
\BE{Rdecomppf_e23}\begin{split}
&\big|M_{(\ev_{\cS;\bh},\ff_{\{0\}\sqcup K_1',L_1'}),f_{\bp;\Ups_1}}
\big|_{\fo_{\cS;\bh},\fo_{\Ups_1}^c}^{\pm}\\
&\qquad\qquad
=(-1)^{\de_\R(\cS)}\big|M_{(\ev_{\bh_1},\ff_{\{0\}\sqcup K_1',L_1'}),f_{\bp_1;\Ups_1}}
\big|_{\wt\fo_{\os;\bh_1},\fo_{\Ups_1}^c}^{\pm}
\!\!\deg\!\big(\ev_{\bh_2},\fo_{\os;\bh_2}\big).
\end{split}\EE
By \cite[Lemma~3.3(1)]{RealWDVV},
$$\big|M_{(\ev_{\bh_1},\ff_{\{0\}\sqcup K_1',L_1'}),f_{\bp_1;\Ups_1}}
\big|_{\wt\fo_{\os;\bh_1},\fo_{\Ups_1}^c}^{\pm}=
(-1)^{\dim\,\Ups_1}\deg\!\big(\ev_{\bh_1}|_{\ff_{\{0\}\sqcup K_1',L_1'}^{-1}\!(\Ups_1)},
(\ff_{\{0\}\sqcup K_1',L_1'}^*\fo_{\Ups_1}^c)\wt\fo_{\os;\bh_1}\big)\,.$$
By the sentence containing~\eref{Rdecomppf_e8} and \eref{fsfoprod_e}, 
\begin{equation*}\begin{split}
&\deg\!\big(\ev_{\bh_1}|_{\ff_{\{0\}\sqcup K_1',L_1'}^{-1}\!(\Ups_1)},
(\ff_{\{0\}\sqcup K_1',L_1'}^*\fo_{\Ups_1}^c)\wt\fo_{\os;\bh_1}\big)\\
&\qquad\qquad
=\deg\!\big(\ff_{\cZ_1}|_{\ff_{\{0\}\sqcup K_1',L_1'}^{-1}\!(\Ups_1)},
(\ff_{\{0\}\sqcup K_1',L_1'}^*\fo_{\Ups_1}^c)\fo_{\cZ_1}^v\big)
\deg\!\big(\ev_{\bh_1}',\fo_{\os;\bh_1}\big).
\end{split}\end{equation*}
Since $\fo_{\cZ_1}^v\!=\!\pi_{\cZ}^*\fo_{\nod}^{\R}$, \cite[Lemma~3.2]{RealWDVV} gives
\begin{equation*}\begin{split}
\fs_{\wt\u}\!\big(\ff_{\cZ_1}|_{\ff_{\{0\}\sqcup K_1',L_1'}^{-1}\!(\Ups_1)},
(\ff_{\{0\}\sqcup K_1',L_1'}^*\fo_{\Ups_1}^c)\fo_{\cZ_1}^v\big)
&=\fs_{\wt\u}\!\big(\ff_{\{0\}\sqcup K_1',L_1'},\pi_{\cZ}^*\fo_{\nod}^{\R},\fo_{\nod}^{\R}\big)
\fs_{\ff_{\{0\}\sqcup K_1',L_1'}(\wt\u)}\!\big(\fo_{\Ups_1}^c\fo_{\nod}^{\R}\big)\\
&=\fs_{\ff_{\{0\}\sqcup K_1',L_1'}(\wt\u)}\!\big(\fo_{\Ups_1}^c\fo_{\nod}^{\R}\big)
\end{split}\end{equation*}
for a generic $\wt\u\!\in\!\ff_{\{0\}\sqcup K_1',L_1'}^{-1}\!(\Ups_1)$.

Combining the last three equations with~\eref{Rdecomppf_e9}, we obtain
\begin{equation*}\begin{split}
\big|M_{(\ev_{\bh_1},\ff_{\{0\}\sqcup K_1',L_1'}),f_{\bp_1;\Ups_1}}
\big|_{\wt\fo_{\os;\bh_1},\fo_{\Ups_1}^c}^{\pm}
&=-(-1)^{\dim\,\Ups+|K_2'|}\deg\!\big(\ff_{\cM}|_{\Ups_1},\fo_{\Ups_1}^c\fo_{\nod}^{\R}\big)
\deg\!\big(\ev_{\bh_1}',\fo_{\os;\bh_1}\big)\\
&=-(-1)^{\dim\,\Ups+|K_2'|}\deg_S\!\big(\Ups,\fo_{\Ups}^c\big)
\deg\!\big(\ev_{\bh_1}',\fo_{\os;\bh_1}\big).
\end{split}\end{equation*}
Along with~\eref{Rdecomppf_e21} and~\eref{Rdecomppf_e23}, this gives
$$\big|\cS_{\bh,\bp;\Ups}^*\big|_{\fo_{\cS;\bh},\fo_\Ups^c}^{\pm}
=(-1)^{\dim\,\Ups+\de_\R(\cS)}\deg_S\!\big(\Ups,\fo_{\Ups}^c\big)
\deg\!\big(\ev_{\bh_1}',\fo_{\os;\bh_1}\big)
\deg\!\big(\ev_{\bh_2},\fo_{\os;\bh_2}\!\big)\,.$$
Combining this equation with~\eref{Rdecomppf_e11} and~\eref{Rdecomp_e2}, 
we obtain~\eref{Rdecomp_e}.

\begin{proof}[{\bf{\emph{Proof of Lemma~\ref{TMcomp_lmm}}}}]
The differential of the forgetful morphism $\ff_{k,l}$ induces 
the first exact square of Figure~\ref{TMcomp_fig}.
The two spaces in the bottom row are oriented by~$\fo_{\cS}^{c;+}$ and~$\fo_{\cS^{\vee}}^{c;+}$ 
with the isomorphism between them being orientation-preserving.
These orientations and the orientations~$\fo_{\os}^D$, $\fo_{\os;L^*}$, and $\fo_{k,l;L^*}$
determine the limiting orientations $\fo_{\os}^{D;+}$ on $\ker D_{\u}^{\phi}$,
$\fo_{\os;L^*}^+$ on $T_{\u}\M^{\st}$, and $\fo_{k,l;L^*}^+$ on $T_{\cC}\ov\cM$, respectively.
By~\eref{OrientSubs_e3}, the middle row respects these orientations.
The middle (resp.~right) column respects the orientations 
$\prt\fo_{\os;L^*}$ on~$T_{\u}\cS$, $\fo_{\os;L^*}^+$ on $T_{\u}\M^{\st}$, and
$\fo_{\cS}^{c;+}$ on~$\cN_{\u}\cS$
(resp.~$\fo_{\cS^{\vee};L^*}^+$ on~$T_{\u}\cS^{\vee}$, 
$\fo_{k,l;L^*}^+$ on $T_{\cC}\ov\cM$, and
$\fo_{\cS^{\vee}}^{c;+}$ on~$\cN_{\u}\cS^{\vee}$).
Thus, the top row in the first exact square of Figure~\ref{TMcomp_fig}
{\it respects} the orientations
$\fo_{\os}^{D;+}$ on $\ker D_{\u}^{\phi}$,
$\prt\fo_{\os;L^*}$ on~$T_{\u}\cS$, and $\fo_{\cS^{\vee};L^*}^+$ on~$T_{\u}\cS^{\vee}$;
see \cite[Lemma~6.3]{RealWDVV}, for example.

\begin{figure}
$$\xymatrix{ &0\ar[d] &0\ar[d] &0\ar[d]\\
0\ar[r]& \ker D_{\u}^{\phi} \ar[r]\ar@{=}[d]&
T_{\u}\cS\ar[r]\ar[d]& T_{\cC}\cS^{\vee}\ar[r]\ar[d]& 0\\
0\ar[r]& \ker D_{\u}^{\phi} \ar[r]\ar[d]&
T_{\u}\M^{\st}\ar[r]\ar[d]& T_{\cC}\ov\cM\ar[r]\ar[d]& 0\\
&0\ar[r]& \cN_{\u}\cS \ar[r]\ar[d]& \cN_{\cC}\cS^{\vee}\ar[r]\ar[d]& 0\\
&&0&0}$$
$$\xymatrix{ &&0\ar[d] &0\ar[d]\\
&0\ar[d]\ar[r]& T_{\nod}S^1_1\ar[d]\ar[r]& T_{\nod}S^1_1\ar[d]\ar[r]& 0\\
0\ar[r]& \ker D_{\u_1}^{\phi} \ar[r]\ar@{=}[d]&T_{\u_1}\M_1 \ar[r]\ar[d]& 
T_{\cC_1}\cM_1 \ar[r]\ar[d]& 0\\
0\ar[r]& \ker D_{\u_1}^{\phi} \ar[r]\ar[d]&T_{\u_1'}\M_1' \ar[r]\ar[d]& 
T_{\cC_1'}\cM_1' \ar[r]\ar[d]& 0\\
&0&0&0}$$
$$\xymatrix{ &0\ar[d] &0\ar[d] &0\ar[d]\\
0\ar[r]& \ker D_{\u}^{\phi} \ar[r]\ar[d]&
T_{\u}\cS\ar[r]\ar[d]& T_{\cC}\cS^{\vee}\ar[r]\ar[d]& 0\\
0\ar[r]& \ker D_{\u_1}^{\phi}\!\oplus\!\ker D_{\u_2}^{\phi} \ar[r]\ar[d]&
T_{\u_1}\M_1\!\oplus\!T_{\u_2}\M_2 \ar[r]\ar[d]& 
T_{\cC_1}\cM_1\!\oplus\!T_{\cC_2}\cM_2 \ar[r]\ar[d]& 0\\
0\ar[r]& T_{u(\nod)}\wch X^{\phi} \ar@{=}[r]\ar[d]& T_{u(\nod)}\wch X^{\phi}\ar[r]\ar[d]& 0\\
&0&0}$$
\caption{Commutative squares of vector spaces with
exact rows and columns for the proof of Lemma~\ref{TMcomp_lmm}}
\label{TMcomp_fig}
\end{figure}

The differentials of forgetful morphisms induce the second exact square of Figure~\ref{TMcomp_fig}.
The two spaces in the top row are oriented by~$\fo_{\nod}^{\R}$ as in Section~\ref{cMorient_subs} 
with the isomorphism between them being orientation-preserving.
The first real marked point of~$\u_1$ is the node. 
By~\ref{cMorientR_it} in Lemma~\ref{cMorient_lmm}, 
the right column thus does not respect the orientations
$\fo_{\nod}^\R$ on $T_{\nod}S^1_1$, $\fo_{\{0\}\sqcup K_1,L_1;L_1^*}$ on~$T_{\cC_1}\cM_1$,
and $\fo_{K_1,L_1;L_1^*}$ on~$T_{\cC_1'}\cM_1'$ because  
$$|K_1| +\dim\,\cM_1'=2|K_1|\!+\!2|L_1|\!-\!3 \not\in2\Z.$$
By~\eref{OrientSubs_e3},
the bottom row respects the orientations $\fo_{\os}^D$ on $\ker D_{\u_1}^{\phi}$,
$\fo_{\os;L_1^*}$ on $T_{\u_1'}\M_1'$, and $\fo_{K_1,L_1;L_1^*}$ on~$T_{\cC_1'}\cM_1'$.
By~\eref{Rdecomppf_e5a}, 
the middle column respects the orientations $\fo_{\nod}^\R$ on $T_{\nod}S^1_1$,
$\wt\fo_{\os;L_1^*}$ on $T_{\u_1}\M_1$, and $\fo_{\os;L_1^*}$ on $T_{\u_1'}\M_1'$ 
if and only if $|K_1|\!\in\!2\Z$ because 
$$\dim\,\M_1'=\ell_{\om}(B_1)\!+\!2|L_1|\!+\!|K_1| \qquad\hbox{and}\qquad 
\ell_{\om}(B_1)\in2\Z.$$
Thus, the middle row respects the orientations
$\fo_{\os}^D$ on $\ker D_{\u_1}^{\phi}$,
$\wt\fo_{\os;L_1^*}$ on $T_{\u_1}\M_1$, and $\fo_{k_1+1,l_1;L_1^*}$ on~$T_{\cC_1}\cM_1$
if and only~if
$$1\!+\!\dim\ker D_{\u_1}^{\phi}\!+\!|K_1|=1\!+\!3\!+\!\ell_{\om}(B_1)\!+\!|K_1|$$
is even.  The last condition is equivalent to $|K_1|\!\in\!2\Z$.

The short exact sequences~\eref{Dses_e} and~\eref{TcSses_e} and 
the differential of the forgetful morphism~$\ff_{k,l}$
induce the third exact square of Figure~\ref{TMcomp_fig}.
By~\eref{OrientSubs_e3}, the short exact sequence of the second summands in the middle row respects
the orientations $\fo_{\os}^D$ on $\ker D_{\u_2}^{\phi}$,
$\fo_{\os;L_2^*}$ on $T_{\u_2}\M_2$, and $\fo_{\{0\}\sqcup K_2,L_2;L_2^*}$ on~$T_{\cC_2}\cM_2$.
Along with the conclusion of the previous paragraph,
this implies that the middle row respects the orientations 
$\fo_{\os}^D\!\oplus\!\fo_{\os}^D$,  $\wt\fo_{\os;L_1^*}\!\oplus\!\fo_{\os;L_2^*}$,
and $\fo_{\{0\}\sqcup K_1+1,L_1;L_1^*}\oplus\!\fo_{\{0\}\sqcup K_2,L_2;L_2^*}$ because 
$$|K_1|+\big(\dim\ker D_{\u_2}^{\phi}\big)\big(\dim\,\cM_1\big)
=|K_1|\!+\!\big(3\!+\!\lr{c_1(X,\om),B_2}\big)\big(|K_1|\!+\!2|L_1|\!-\!2\big)\in2\Z.$$
By Lemma~\ref{DorientComp_lmm}, the left column respects the orientations 
$\fo_{\os}^{D;+}$, $\fo_{\os}^D\!\oplus\!\fo_{\os}^D$, and
the orientation of~$T_{u(\nod)}\wch X^{\phi}$ in $\os$. 
By Lemma~\ref{DMboundary_lmm}, the non-trivial isomorphism in the right column respects 
the orientations $\fo_{\cS^{\vee};L^*}^+$ and 
$\fo_{\{0\}\sqcup K_1,L_1;L_1^*}\oplus\!\fo_{\{0\}\sqcup K_2,L_2;L_2^*}$ 
if and only~if $\de_{\R}(\cS)\!\cong\!k\!+\!1$ mod~2.
By the definition of $\fo^\cS_{\os;L^*}$ via \eref{Rdecomppf_e3}, 
the middle column respects the orientations $\fo_{\os;L^*}^{\cS}$, 
$\wt\fo_{\os;L_1^*}\!\oplus\!\fo_{\os;L_2^*}$, and~$\fo$.
Combining these statements with \cite[Lemma~6.3]{RealWDVV},
we conclude that the top row respects the orientations 
$\fo_{\os}^{D;+}$, $\fo_{\os;L^*}^{\cS}$, and $\fo_{\cS^{\vee};L^*}^+$ if and only~if
$$\big(k\!+\!1\!+\!\de_{\R}(\cS)\big)\!+\!(\dim\,\cS^{\vee})(\dim\,\wch X^{\phi})
=\big(k\!+\!1\!+\!\de_{\R}(\cS)\big)\!+\!3(k\!+\!2l\!-\!3\!-\!1)$$
is even.
Comparing this conclusion with the conclusion concerning the top row  
in the first exact square of Figure~\ref{TMcomp_fig} above, we obtain the claim.
\end{proof}

\subsection{Proof of Proposition~\ref{Cdecomp_prp}}
\label{Cdecomp_subs}

Let $L$ be a finite set. 
We denote by $\cH_{L;G}^{\om}$ the space of pairs $(J,\nu')$ 
consisting of $J\!\in\!\cJ_{\om;G}^{\phi}$ and a $G$-invariant Ruan-Tian perturbation~$\nu'$ of 
the $\dbar_J$-equation associated with $\ov\cM_{0,L}$ 
if $|L|\!\ge\!3$ and the set of pairs $(J,0)$ with $J\!\in\!\cJ_{\om;G}^{\phi}$ otherwise.
For $B'\!\in\!H_2(X)$ and $\nu'\!\in\!\cH_{L;G}^{\om}$, we denote by $\M^\C_L(B';J,\nu')$
the moduli space of (complex) genus~0 degree~$B'$ $(J,\nu')$-holomorphic maps 
from smooth domains with $L$-marked points and~by
$$\ev_i\!: \M^\C_L(B';J,\nu')\lra X, \qquad i\!\in\!L,$$
the evaluation maps at the marked points.
For $I\!\subset\!L$, let $\fo_{\C;L}$ be the orientation of $\M^\C_L(B';J,\nu')$ 
obtained by twisting the standard complex orientation by~$(-1)^{|I|}$.
Define 
\begin{alignat}{2}\notag
\Th_i^I\!:X&\lra X,&\quad \Th_i^I&=\begin{cases}\id_X,&\hbox{if}~i\!\not\in\!I;\\
\phi,&\hbox{if}~i\!\in\!I;\end{cases}\\
\label{CevIdfn_e}
\ev^I\!:\M^\C_L(B';J,\nu')&\lra X^L, &\quad 
\ev^I(\u)&=\big(\big(\Th^I_i(\ev_i(\u))\big)_{i\in L}\big).
\end{alignat}

We continue with the notation in the statement of Proposition~\ref{Cdecomp_prp} 
and just above and take \hbox{$L^*\!\equiv\!L^*(\bh)$} as in~\eref{bhprpdnf_e}.
The co-orientation~$\fo_{\Ga}^c$ of~$\Ga$ in~$\ov\cM_{k',l'}^{\tau}$ and
the orientation~$\fo_{\os;L^*}$ of Lemma~\ref{orient_lmm} induce
an orientation $(\ff^*_{k',l'}\fo^c_\Ga)\fo_{\os;L^*}$ 
of $\M_{\Ga;k,l}(B;J,\nu;\wch{X}^{\phi})$. 

Fix a stratum $\cS\!\subset\!\M_{\Ga;k,l}(B;J,\nu)$ of maps that are not~$\Z_2$-pinchable. 
Let $B_0$ be the degree of the restrictions of the maps in~$\cS$ to the real component $\P^1_0$ 
of the domain and $B'$ be the degree of their restrictions to the component~$\P^1_+$ 
of the domain carrying the marked point~$z^+_1$. 
Since \hbox{$B_0\!\in\!H_2(X)^{\phi}_-$}, $B_0$ is $G$-invariant.
Denote by $L_0,L_\C\!\subset\![l]$ the subsets indexing the conjugate pairs 
of marked points carried by~$\P^1_0$ and $\P^1_+$, respectively. 
Let $I\!\subset\!L_\C$ be the subset indexing the conjugate pairs 
of marked points $(z^+_i,z^-_i)$ of curves in~$\cS$ with
$z^-_i\!\in\!\P^1_+$.
Define
$$L^*_0=L_0\!\cap\!L^*, \quad L^*_{\C}=L_{\C}\!\cap\!L^*, \quad
L^*_-=I\!\cap\!L^*, \quad
\bh_0=(h_i)_{i\in L_0},  \quad \bh_{\C}=(h_i)_{i\in L_{\C}}.$$
By~\eref{Cdecomp_e0} and~\eref{Cdecomp_e0c}, 
\begin{gather}\label{Cdecomp_e3}
\big(\ell_{\om}(B_0)\!-\!2(k\!+\!\codim_{\C}\bh_0\!-\!|L_0|)\big)
+2\big(\ell_{\om}(B')\!-\!(\codim_{\C}\bh_{\C}\!-\!|L_{\C}|)\big)=2,\\
\label{Cdecomp_e3b}
\big[\Th^I_i\!\circ\!h_i\big]_X=
\begin{cases}[h_i]_X,&\hbox{if}~i\!\in\!L_{\C}\!-\!(I\!-\!L^*_-);\\
-[h_i]_X,&\hbox{if}~i\!\in\!I\!-\!L^*_-\,.
\end{cases}
\end{gather}

For a good choice of $\nu$, there exist \hbox{$\nu_{\R}\!\in\!\cH_{[k],\{0\}\sqcup L_0;G}^{\om,\phi}$},
\hbox{$\nu_{\C}\!\in\!\cH_{\{0\}\sqcup L_\C;G}^{\om}$}, and a natural embedding
\BE{Gasplit_e1} 
\io_{\cS}\!:\cS\lhra{~~~} \M_\R\times\!\M_\C\!\equiv\!
\M_{[k],\{0\}\sqcup L_0;L_0^*}^{\st}\big(B_0;J,\nu_\R;\wch{X}^{\phi}\big)
\!\times\!\M^\C_{\{0\}\sqcup L_\C}\big(B';J,\nu_\C\big)\,.\EE
If $B_0\!\neq\!0$, we also assume that there exists 
\hbox{$\nu_{\R}'\!\in\!\cH_{[k],L_0;G}^{\om,\phi}$} so that the forgetful morphism
\BE{Gasplit_e2a}\ff_{\nod}\!:\M_\R
\lra \M_\R'\equiv\!\M_{[k],L_0;L_0^*}^{\st}\big(B_0;J,\nu'_\R;\wch{X}^{\phi}\big)\EE
dropping the conjugate pair corresponding to the node~$\nod$
is defined.
If $B'\!\neq\!0$, we similarly assume that there exists $\nu_{\C}'\!\in\!\cH_{L_{\C};G}^{\om}$
so that the analogous forgetful morphism
\BE{Gasplit_e2b}\ff_{\nod}\!:\M_\C \lra \M_\C'\equiv\!\M^\C_{L_\C}\big(B';J,\nu'_\C\big)\EE
is defined.  

For an element $\u\!\in\!\cS$, we denote~by $\u_0\!\in\!\M_\R$ and $\u_+\!\in\!\M_\C$
the pair of maps corresponding to~$\u$ via~\eref{Gasplit_e1}.
Let $\u_0'\!\in\!\M'_\R$ and $\u_+'\!\in\!\M'_\C$
be the image of~$\u_0$ under~\eref{Gasplit_e2a} if $B_0\!\neq\!0$ and 
the image of~$\u_+$ under~\eref{Gasplit_e2b} if $B'\!\neq\!0$, respectively.
The exact sequence
$$0\lra T_{\u}\cS\lra T_{\u_0}\M_\R\!\oplus\!T_{\u_+}\M_\C\lra T_{u(\nod)}X\lra0,~~
\big(\xi_1,\xi_2\big)\lra \xi_2(\nod)\!-\!\xi_1(\nod),$$
of vector spaces determines an isomorphism 
\BE{Cdecomp_e4}\la_{\u}(\cS)\!\otimes\!\la\big(T_{u(\nod)}X\big)
\approx \la_{\u_0}(\M_\R)\!\otimes\!\la_{\u_+}(\M_\C).\EE
By Lemma~\ref{orient_lmm},
the $\OSpin$-structure~$\os$ on~$\wch X^{\phi}$ determines an orientation 
$\fo_{\os;L^*_0}$ of $\la_{\u_0}(\M_\R)$. 
This orientation, the complex orientation of $\la(T_{u(\nod)}X)$,
and  the orientation~$\fo_{\C;L^*_-}$ of~$\la_{\u_+}(\M_\C)$
induce an orientation~$\fo_{\os;L^*;\u}^{\Ga}$ of~$\la_{\u}(\wt\cS)$
via the isomorphism~\eref{Cdecomp_e4}. 

With the notation as in~\eref{Mhdfn_e}, let 
$$\cS_\bh={\cS}{}_{\ev}\!\!\times_{f_{\bh}}\!M_{\bh},\quad
\cZ_\R=({\M_\R})_{\ev}\!\!\times_{f_{\bh_0}}\!M_{\bh_0}, \quad
\cZ_\C=(\M_\C)_{\ev^I}\!\!\times_{\bh_\C}\!M_{f_{\bh_\C}}$$
be the spaces cut out by~$\bh$, $\bh_0$,  and~$\bh_{\C}$, respectively, and  
$$\ev_{\R}\!:\cZ_\R\lra (\wch X^\phi)^k, 
\quad \ev_{\R;\nod}\!\equiv\!\ev_0\!:\cZ_\R\lra X, 
\quad\hbox{and}\quad \ev_{\C;\nod}\!\equiv\!\ev_0\!:\cZ_\C\lra X,$$
be the induced evaluation maps. 
The orientations~$\fo_{\os;L_0^*}$ of $\M_\R$, $\fo_{\C;L^*_-}$ of $\M_\C$,
and $\fo_{h_i}$ of~$H_i$ determine orientations~$\fo_{\os;\bh_0}$ of~$\cZ_\R$ 
and~$\fo_{\bh_{\C}}$ of~$\cZ_\C$ via the evaluation maps~$\ev$ as in~\eref{whfMevdfn_e}
and~$\ev^I$ as in~\eref{CevIdfn_e}, respectively.

\begin{lmm}\label{TMcompGa_lmm}
\BEnum{(\arabic*)}

\item\label{TMcompGa_it1} The orientations $(\ff^*_{k',l'}\fo^c_\Ga)\fo_{\os;L^*}$ 
and $\fo_{\os;L^*}^{\Ga}$ of $\la(\cS)$ are the same.

\item\label{TMcompGa_it2} The orientation~$\fo_{\bh_{\C}}$ of~$\cZ_\C$ at $\wt\u\!\in\!\cZ_\C$ 
is the orientation induced by the complex orientation of $\M^\C_L(B';J,\nu')$ and
the orientations~$\fo_{h_i}$ of~$H_i$ via the intersection of the smooth~maps
$$\ev\!\equiv\!\prod_{i\in L_{\C}}\!\!\ev_i\!:\M^\C_L(B';J,\nu')\lra X^{L_{\C}}
\quad\hbox{and}\quad
\prod_{i\in L_{\C}}\!\!\big\{\Th^I_i\!\circ\!h_i\big\}\!:M_{\bh_{\C}}\lra X^{L_{\C}}$$
if and only if $|I\!-\!L^*_-|\!\in\!2\Z$.

\EEnum
\end{lmm}

\begin{proof} 
The first statement holds for the same reasons as \cite[Lemma~6.5]{RealWDVV}.
The second statement holds because the action of~$\phi$ on~$X$ is orientation-reversing. 
\end{proof}

For $\wt\u\!\in\!\cS_\bh$, let 
$\wt\u_0\!\in\!\cZ_\R$ and $\wt\u_+\!\in\!\cZ_\C$ be the components 
of $\wt\u$ in the corresponding spaces. 
The exact sequence 
\BE{Cdecomppf_e4}0\lra T_{\wt\u}\cS_{\bh}\lra T_{\wt\u_0}\cZ_\R\!\oplus\!T_{\wt\u_+}\cZ_\C
\lra T_{u(\nod)}X\lra0\EE
of vector spaces determines an isomorphism
$$\la_{\wt\u}(\cS_{\bh})\!\otimes\!\la\big(T_{u(\nod)}X\big)
\approx \la_{\wt\u_0}(\cZ_\R)\!\otimes\!\la_{\wt\u_+}(\cZ_\C).$$
The orientations $\fo_{\os;\bh_0}$ of $\la_{\wt\u_0}(\cZ_\R)$
and $\fo_{\bh_{\C};L^*_-}$ of $\la_{\wt\u_+}(\cZ_\C)$ 
and the complex orientation of~$\la(T_{u(\nod)}X)$
determine an orientation $\fo_{\os;\bh}^{\Ga}$ of $\cS_{\bh}$ via
this isomorphism. 
Since the dimensions of $H_i$ and $X$ are even, 
Lemma~\ref{TMcompGa_lmm}\ref{TMcompGa_it1} implies~that 
\BE{Cdecomppf_e11}\big|\ev_{\Ga;\bh}^{-1}(\bp)\!\cap\!\cS_{\bh}\big|_{\fo_{\Ga;\os;\bh}}^{\pm}
=\big|\ev_{\Ga;\bh}^{-1}(\bp)\!\cap\!\cS_{\bh}\big|_{\fo_{\os;\bh}^{\Ga}}^{\pm}\,.\EE
By Lemma~\ref{degvsinter_lmm},
\BE{Cdecomppf_e11b}
\big|\ev_{\Ga;\bh}^{-1}(\bp)\!\cap\!\cS_{\bh}\big|_{\fo_{\os;\bh}^{\Ga}}^{\pm}
=\big|M_{\ev_{\R;\nod}|_{\ev_{\R}^{-1}(\bp)},\ev_{\C;\nod}}\big|_{
(\ev_{\R}^*\fo_k)\fo_{\os;\bh_0},\fo_{\bh_{\C};L^*_-}}^{\pm}\,,\EE
where $\fo_k$ is the orientation of $(\wch{X}^{\phi})^k$ induced by~$\os$.

If $B_0\!\neq\!0$ (resp.~$B'\!\neq\!0$), we also define
$$\cZ'_\R=(\M'_\R)_{\ev}\!\!\times_{\bh_0}\!M_{\bh_0}
\qquad\big(\hbox{resp.}~\cZ'_\C=(\M'_\C)_{\ev^I}\!\!\times_{\bh_\C}\!M_{\bh_\C}\big)$$
and denote by $\wt\u'_0\!\in\!\cZ'_\R$ (resp.~$\wt\u'_+\!\in\!\cZ'_\C$)
the image of~$\wt\u_0$ (resp.~$\wt\u_+$) under the forgetful morphism
\BE{fnd_e}
\ff_{\R}\!:\cZ_{\R}\lra \cZ_{\R}' \qquad\big(\hbox{resp.}~\ff_{\C}\!:\cZ_{\C}
\lra \cZ_{\C}'\big)\EE
dropping the marked points corresponding to the nodes.
Let 
$$\ev'_\R:\cZ'_\R\lra(\wch X^\phi)^k$$
be the evaluation map induced by~\eref{fMRevdfn_e}. 

Since the projections $\pi_{\R}$ and~$\pi_{\C}$ in the commutative diagrams
$$\xymatrix{ \cZ_{\R} \ar[rr]^{\pi_{\R}} \ar[d]_{\ff_{\R}}&& \M_{\R}\ar[d]^{\ff_{\nod}}&& 
\cZ_{\C} \ar[rr]^{\pi_{\C}} \ar[d]_{\ff_{\C}}&& \M_{\C}\ar[d]^{\ff_{\nod}}\\
\cZ_{\R}' \ar[rr]^{\pi_{\R}'} && \M_{\R}'&&
\cZ_{\C}' \ar[rr]^{\pi_{\C}'} && \M_{\C}'}$$
induce isomorphisms between the vertical tangent bundles of $\ff_{\R}$, $\ff_{\C}$,
and $\ff_{\nod}$, they pull back the orientations~$\fo_{\nod}^+$ 
of the fibers of~$\ff_{\nod}$ to 
orientations~$\fo_{\R}^+$ and $\fo_\C^+$ of the fibers of~$\ff_{\R}$ and~$\ff_{\C}$,
respectively. 
If $B_0\!\neq\!0$, $\fo_{\os;L_0^*}$ 
determines an orientation~$\fo_{\os;\bh_0}'$ of $\cZ'_\R$. 
If $B'\!\neq\!0$, $\fo_{\C;L_-^*}$ determines an orientation~$\fo_{\bh_{\C};L^*_-}'$ 
of~$\cZ'_\C$. 
Since the dimensions of $X$ and $H_i$ are~even, 
\BE{Cdecomppf_e8} 
\fo_{\os;\bh_0}=\fo_{\R}^+\fo_{\os;\bh_0}'=
\big(\pi_{\R}^*\fo_{\nod}^+\big)\fo_{\os;\bh_0}' \quad\hbox{and}\quad
\fo_{\bh_{\C};L^*_-}=\fo_{\C}^+\fo_{\bh_{\C};L^*_-}'=
\big(\pi_{\C}^*\fo_{\nod}^+\big)\fo_{\bh_{\C};L^*_-}'\,,\EE
whenever $B_0\!\neq\!0$ and $B'\!\neq\!0$, respectively;
see Lemma~\ref{cutfibr_lmm}.

We first consider the case $B_0,B'\!\neq\!0$.
By~\eref{Cdecomp_e3}, we can assume~that either 
\begin{alignat}{2}
\label{Cdecomppf1_e10}
\ell_{\om}(B_0)&=2\big(k\!+\!\codim_{\C}\bh_0\!-\!|L_0|\big) &~~\hbox{and}~~
\ell_{\om}(B')&=\big(\codim_{\C}\bh_{\C}\!-\!|L_{\C}|\big)\!+\!1, \quad\hbox{or}\\
\label{Cdecomppf2_e10}
\ell_{\om}(B_0)&=2\big(k\!+\!\codim_{\C}\bh_0\!-\!|L_0|\!+\!1\big) &~~\hbox{and}~~
\ell_{\om}(B')&=\codim_{\C}\bh_{\C}\!-\!|L_{\C}|\,;
\end{alignat}
otherwise, either $\ev'^{-1}_\R(\bp)\!=\!\eset$ or $\cZ_\C'\!=\!\eset$ for generic $\bh,\bp$.
 
Suppose that \eref{Cdecomppf1_e10} is the case.
This implies that $\ev'^{-1}_\R(\bp)$ is a finite set of points~$\wt\u_{\R}'$.
The fiber of~$\ff_{\R}$ over each~$\wt\u_{\R}'$ is the complement of finitely many points
in~$\P^1\!-\!S^1$.
The restriction of~$\ev_{\R;\nod}$ to~$\ev^{-1}_\R(\bp)$ extends to a smooth~map
\BE{eRfn_e} e_{\R}\!:M_{\R}\lra X\EE
over the natural compactification $M_{\R}$ of $\ev^{-1}_\R(\bp)$.
By Lemmas~\ref{orient_lmm} and~\ref{cutfibr_lmm},
$$\big[e_{\R}\big]_X=
\big|\ev_{\R}'^{-1}(\bp)\big|^{\pm}_{\fo'_{\os;\bh_0}}B_0\in H_2(X).$$
By~\eref{OWGhomdfn_e} and the assumption that $\bh$ is $G$-invariant,
$$\big|\ev_{\R}'^{-1}(\bp)\big|^{\pm}_{\fo'_{\os;\bh_0}}=
\blr{(h_i)_{i\in L_0}}^{\phi,\os}_{B_0;\wch X^\phi;G}\,.$$
By~\eref{OWGhomdfn_e}, the evenness of the dimension of~$M_{\R}$,
Lemma~\ref{TMcompGa_lmm}\ref{TMcompGa_it2}, and~\eref{Cdecomp_e3b},
\BE{Cdecomp_e5a}\begin{split}
\big|M_{\ev_{\R;\nod}|_{\ev_{\R}^{-1}(\bp)},\ev_{\C;\nod}}\big|_{
(\ev_{\R}^*\fo_k)\fo_{\os;\bh_0},\fo_{\bh_{\C};L^*_-}}^{\pm}
&=(-1)^{|I-L^*_-|}
\blr{e_{\R},(\Th^I_i\!\circ\!h_i)_{i\in L_{\C}}}^X_{B'}\\
&=\blr{(h_i)_{i\in L_{\C}},e_{\R}}^X_{B'}\,.
\end{split}\EE
Combining the last three equations with~\eref{Cdecomppf_e11} and~\eref{Cdecomppf_e11b}, 
we obtain
\BE{Cdecomp_e6a}
\big|\ev_{\Ga;\bh}^{-1}(\bp)\!\cap\!\cS_{\bh}\big|_{\fo_{\Ga;\os;\bh}}^{\pm}
=\blr{(h_i)_{i\in L_{\C}},B_0}^X_{B'}\blr{(h_i)_{i\in[l]-L_{\C}}}^{\phi,\os}_{B_0;\wch X^\phi;G}\,.\EE
This is the last term in~\eref{Cdecomp_e} with $L'\!=\!L_{\C}$ and a choice of $I\!\subset\!L'$
so that $I\!\cap\!L_{\C}(\Ga)$ is the subset $L_-^*(\Ga)\!\subset\![l']$ defined
in Section~\ref{NBstrata_subs}. 

Suppose instead that \eref{Cdecomppf2_e10} is the case.
This implies that $\cZ_\C'$ is a finite set of points~$\wt\u_{\C}'$.
The fiber of~$\ff_{\C}$ over each~$\wt\u_{\C}'$ is the complement of finitely many points
in~$\P^1\!-\!S^1$.
The restriction of~$\ev_{\C;\nod}$ to~$\cZ_\C$ extends to a smooth~map
\BE{eCfn_e} e_{\C;B'}\!:M_{\C;B'}\lra X\EE
over the natural compactification $M_{\C;B'}$ of~$\cZ_\C$.
By Lemma~\ref{orient_lmm},
$$\big[e_{\C;B'}\big]_X=
\big|\cZ_\C'\big|^{\pm}_{\fo'_{\bh_{\C};L^*_-}}B'\in H_2(X).$$
By~\eref{OWGhomdfn_e}, Lemma~\ref{TMcompGa_lmm}\ref{TMcompGa_it2}, and~\eref{Cdecomp_e3b},
\BE{Cdecomp_e5c}\big|\cZ_\C'\big|^{\pm}_{\fo'_{\bh_{\C};L^*_-}}
=(-1)^{|I-L^*_-|}\blr{(\Th^I_i\!\circ\!h_i)_{i\in L_{\C}}}^X_{B'}
=\blr{(h_i)_{i\in L_{\C}}}^X_{B'}\,.\EE
Since $\bh$ is $G$-invariant, the number~\eref{Cdecomp_e5c} does not depend on
the choice of $B'$ in $GB'$ and the sum of the cycles~\eref{eCfn_e} over 
these elements is $G$-invariant.
Along with~\eref{OWGhomdfn_e} and  Lemmas~\ref{orient_lmm} and~\ref{cutfibr_lmm},
the latter implies~that
\BE{Cdecomp_e5b}
\sum_{B''\in GB'}\!\!\!\!\!
\big|M_{\ev_{\R;\nod}|_{\ev_{\R}^{-1}(\bp)},\ev_{\C;\nod}}\big|_{
(\ev_{\R}^*\fo_k)\fo_{\os;\bh_0},\fo_{\bh_{\C};L^*_-}}^{\pm}
=\sum_{B''\in GB'}\!\!\!\!\!
\blr{e_{\C;B''},(h_i)_{i\in L_0}}^{\phi,\os}_{B_0;\wch X^\phi;G},\EE
with the $B''$ summand on the left-hand side corresponding to~$B''$ instead of~$B'$.  
Combining the last three equations with~\eref{Cdecomppf_e11} and~\eref{Cdecomppf_e11b}
and the first observation after~\eref{Cdecomp_e5c}, 
we obtain
\BE{Cdecomp_e6b}
\sum_{B''\in GB'}\!\!\!\!\!
\big|\ev_{\Ga;\bh}^{-1}(\bp)\!\cap\!\cS_{\bh}\big|_{\fo_{\Ga;\os;\bh}}^{\pm}
=\sum_{B''\in GB'}\!\!\!\!\!
\blr{(h_i)_{i\in L_{\C}}}^X_{B''}\blr{(h_i)_{i\in[l]-L_{\C}},B''}^{\phi,\os}_{B_0;\wch X^\phi;G}\,.\EE
This is the sum of the penultimate terms in~\eref{Cdecomp_e} with $L'\!=\!L_{\C}$ and 
a choice of $I\!\subset\!L'$ so that \hbox{$I\!\cap\!L_{\C}(\Ga)$} equals~$L_-^*(\Ga)$
over all $B''\!\in\!GB'$. 
Summing~\eref{Cdecomp_e6a} and~\eref{Cdecomp_e6b} over all possibilities 
for~$\cS$ with $B_0,B'\!\neq\!0$, we obtain the $(B_0,B')$-sum in~\eref{Cdecomp_e}.

We next consider the case $B'\!=\!0$ and thus $B_0\!=\!B$.
We can assume~that
$$|L_{\C}|=|L_{\C}(\Ga)|=2;$$
otherwise, either $\ev'^{-1}_\R(\bp)\!=\!\eset$ or $\cZ_\C\!=\!\eset$ for generic~$\bh,\bp$. 
Let $L_{\C}\!=\!\{i_1,i_2\}$ and
$e_{\C}$ as in~\eref{eCfn_e} be the restriction of~$\ev_{\C;\nod}$ to~$\cZ_\C$.
By standard properties of (complex) GW-invariants, Lemma~\ref{TMcompGa_lmm}\ref{TMcompGa_it2}, 
and~\eref{Cdecomp_e3b}, 
$$\big[e_{\C}\big]_X=(-1)^{|I-L^*_-|}
\lr{(\Th^I_{i_1}\!\circ\!h_{i_1})\!\cap\!(\Th^I_{i_2}\!\circ\!h_{i_2})}_{\Ga}
=\lr{\bh}_{\Ga}\in H_{2*}(X).$$
Combining this with~\eref{Cdecomp_e5b}, \eref{Cdecomppf_e11}, 
and~\eref{Cdecomppf_e11b}, we obtain
$$\big|\ev_{\Ga;\bh}^{-1}(\bp)\!\cap\!\cS_{\bh}\big|_{\fo_{\Ga;\os;\bh}}^{\pm}
=\blr{(h_i)_{i\in[l]-L_{\C}(\Ga)},\lr{\bh}_{\Ga}}^{\phi,\os}_{B;\wch X^\phi;G}\,.$$
This is the second term on the right hand side of~\eref{Cdecomp_e}. 

We finally consider the case $B_0\!=\!0$ and thus $\fd(B')\!=\!B$.
We can assume that $k\!=\!k'\!=\!1$ and \hbox{$L_0\!=\!\eset$};
otherwise, either $\ev^{-1}_\R(\bp)\!=\!\eset$ or $\cZ_\C'\!=\!\eset$ for generic~$\bh,\bp$. 
Let $e_{\R}$ as in~\eref{eRfn_e} be the restriction of~$\ev_{\R;\nod}$ to~$\cZ_\R$.
By Lemma~\ref{orient_lmm}\ref{orient0_it}, 
$$\big[e_{\R}\big]_X=\big[\pt\big]_X\in H_0(X).$$
Combining this with~\eref{Cdecomp_e5a}, \eref{Cdecomppf_e11}, 
and~\eref{Cdecomppf_e11b}, we obtain
$$\big|\ev_{\Ga;\bh}^{-1}(\bp)\!\cap\!\cS_{\bh}\big|_{\fo_{\Ga;\os;\bh}}^{\pm}
=\blr{(h_i)_{i\in[l]},\pt}^X_{B'}\,.$$
Summing this over all possibilities for~$\cS$ with $B_0\!=\!0$, 
we obtain the first sum on the right-hand side of~\eref{Cdecomp_e}.

\vspace{.3in}

{\it Department of Mathematics, Stony Brook University, Stony Brook, NY 11794\\
xujia@math.stonybrook.edu, azinger@math.stonybrook.edu}

\clearpage

\end{document}